\documentclass{amsart}
\pdfoutput=1
\usepackage{amsmath,amssymb}
\usepackage{mathtools}
\usepackage[utf8]{inputenc} 

\usepackage[T1]{fontenc}    % use new font-encoding (to have bold
\usepackage{ae,aecompl}     % changes to Type 1 fonts
\usepackage{url}

\usepackage{pdfsync}

\usepackage{tensor}

\input xypic

\xyoption{all}

\xyoption{arc}

\xyoption{2cell}

\usepackage{amsmath}
\usepackage{amsfonts}
\usepackage{amssymb}
\usepackage{amsthm} % 
 \usepackage{amscd}
\usepackage{mathrsfs} % 
\usepackage{enumerate} 
\usepackage{verbatim}	%To comment parts

\input xypic
\xyoption{all}
\xyoption{arc}
\xyoption{2cell}

\usepackage{caption}
\usepackage[hyperfootnotes]{hyperref}

\theoremstyle{plain} % style plain
\newtheorem{thm}[subsection]{Theorem}
\newtheorem{cor}[subsection]{Corollary}
\newtheorem{prop}[subsection]{Proposition}
\newtheorem{lem}[subsection]{Lemma}

\theoremstyle{definition}
\newtheorem{defi}[subsection]{Definition}

\newtheorem{assumption}[subsection]{Assumption}
\theoremstyle{remark}
\newtheorem{remark}[subsection]{Remark}\newtheorem*{remark*}{Remark}
\newtheorem*{remarks*}{Remarks}

\newtheorem{ex}[subsection]{Example}
\newtheorem{case}{Case}

\numberwithin{equation}{subsection}

\newcommand{\nc}{\newcommand}

\nc{\dotcup}{\,\dot\cup\,}
\nc{\seq}{\subseteq}
\nc{\h}{\widehat}
\nc{\cone}{\operatorname{cone}}
\nc{\conv}{\operatorname{conv}}
\nc{\acc}{\operatorname{Acc}}
\nc{\Span}{\operatorname{span}}
\nc{\Gar}{\operatorname{Gar}}
\nc{\dep}{\textnormal{dp}}
\nc{\dom}{\operatorname{Dom}}
\nc{\mpair}[1]{\langle\, #1\,\rangle}
\nc{\supp}{\operatorname{supp}}
\nc{\relint}{\operatorname{relint}}
\nc{\mor}{\operatorname{mor}}
\nc{\ob}{\operatorname{ob}}
\nc{\sm}{\setminus}
\nc{\inft}{\mathrm{inf}}
\nc{\fin}{\mathrm{fin}}
\nc{\wh}{\widehat}
\nc{\nts}{\negthinspace}

\def\lsub#1#2{\tensor*[_{#1}]{#2}{}}

\def\lrsub#1#2#3{\tensor*[_{#1}]{#2}{_{#3}}}

\nc{\gr}{{\mathrm{Gram}}}

\nc{\mset}[1]{\set{\,#1\,}}

\nc{\op}{\mathrm{op}}

\nc{\pair}[1]{\langle #1\rangle}

\nc{\round}[1]{(\,#1\,)}

\nc{\set}[1]{\{#1\}}

\nc{\wt}{\widetilde}

\nc{\sneq}{{\,\subsetneq\,}}

\nc{\sreq}{{\,\supseteq\,}}

\nc{\srneq}{{\,\supsetneq\,}}

\nc{\mcc}{{\mathcal{C}}}

\nc{\mcx}{{\mathcal{X}}}

\nc{\mcy}{{\mathcal{Y}}}

\nc{\mck}{{\mathcal{K}}}

\nc{\mcz}{{\mathcal{Z}}}

\nc{\mcv}{{\mathcal{V}}}

\nc{\rec}{{\mathrm{rec}}}

\nc{\re}{{\mathrm{re}}}

\nc{\im}{{\mathrm{im}}}

\nc{\cc}{{\mathrm{c}}}

\DeclareMathOperator{\sgn}{\mathrm{sgn}}

\DeclareMathOperator{\Stab}{{\mathrm{Stab}}}

\DeclareMathOperator{\corank}{\mathrm{corank}}

\DeclareMathOperator{\rank}{\mathrm{rank}}

\DeclareMathOperator{\pc}{\mathrm{ParCl}}

\DeclareMathOperator{\height}{\mathrm{ht}}

\nc{\Set}{{\mathbf{Set}}}

\nc{\Adj}{{\mathbf{Adj}}}

\nc{\Adjs}{{\mathbf{Adj^{*}}}}

\nc{\Adjp}{{\mathbf{Adj'}}}
 
\nc{\eset}{{\emptyset}}

\newenvironment{num}{
                      
                                          \begin{enumerate} }
                    {\end{enumerate} }
 \newenvironment{conds}{
                       
                        \begin{enumerate} }
                     {\end{enumerate} }

\author[M.~Dyer]{Matthew Dyer}
\address[M.~Dyer]{Department of Mathematics 
\\ 255 Hurley Building\\ University of Notre Dame \\
Notre Dame, Indiana 46556, U.S.A.}
\email{dyer.1@nd.edu}

\author[H.~Gimenez]{Harrison Gimenez}
\address[H.~Gimenez]{Department of Mathematics \\ B26 Hayes-Healy Building \\ University of Notre Dame, Indiana 46556, U.S.A.}
\email{hgimenez@nd.edu}

\title[Root systems of Brink-Howlett groupoids]{Root systems, Tits cones and imaginary cones of Brink-Howlett groupoids}

\begin{document}

\begin{abstract}
We extend the basic theory of  the groupoids introduced by Brink and Howlett in their study of normalizers of parabolic 
subgroups of Coxeter groups, by studying both their abstract root systems    and root systems realized in real vector spaces. Such root systems have some properties   formally analogous to those of root systems of Borcherds-Kac-Moody Lie algebras; in particular, some contain imaginary simple roots.  Further, positive roots correspond to certain reflection subgroups.
We also extend the most basic properties of the Tits cone and imaginary cone of Coxeter groups to corresponding cones defined for Brink-Howlett groupoids.  The results linearize the study of certain classes of reflection subgroups of Coxeter groups in a similar way as root systems of Coxeter groups linearize the study of reflections.   
 \end{abstract}
\date{\today}

%%%%%%%%%%%%%%%%%%%%%%%%%%%%%%%%%

 \maketitle

\renewcommand*\contentsname{Table of Contents}

 \tableofcontents

\addtocontents{toc}{\protect\setcounter{tocdepth}{1}}

%%%%%%%%%%%%%%%%%%%%%%%%%%%%%%%
%\section{Introduction}
%%%%%%%%%%%%%%%%%%%%%%%%%%%%%%%
\section*{Introduction}
  Let $(W,S)$ be a Coxeter system with  standard root system $\Phi$.   The subgroups  generated by subsets  of the set $S$ of 
  simple reflections are called standard parabolic subgroups, and 
  their conjugates are called parabolic subgroups. Both types of 
  subgroups are  important  in the study of Coxeter groups and in areas in which Coxeter groups find application, such as in the 
  structure and representation theory of  semisimple algebraic 
  groups and Kac-Moody groups and Lie algebras.

In their study \cite{BrHo99} of normalizers of standard  parabolic subgroups, Brink and Howlett studied the components of  a 
groupoid  in which the objects are the subsets of the simple roots and a morphism $J\to K$ is
given by an element $w\in W$ satisfying $w(J)=K$. We will call 
a component (or union of components) of this groupoid a Brink Howlett groupoid.

This paper considers various notions of root systems of Brink-Howlett groupoids. We describe both   abstract root systems
  and root systems  realized as subsets of real vector spaces.
 The most interesting and useful of these root systems have a set of roots based at each object $J$ of the groupoid, with the positive 
 roots based at $J$ corresponding to a set of corank one  reflection overgroups (in $W$) of the standard parabolic 
 subgroup $W_{J}$ with $J$ as its set of  simple roots.  The study of these root systems leads to  study of versions ``relative to $J$'' of various basic facts about Coxeter systems which were previously known in the case $J=\eset$. For example, Tits theorem that a finite subgroup of a Coxeter group is contained in a finite parabolic subgroup is generalized to the statement that a finite index overgroup of $W_{J}$ in $W$ is contained in a parabolic subgroup of $W$ which is also a finite index overgroup of $W_{J}$.

 Such root systems 
 in real vector spaces linearize the study of certain classes of 
 reflection subgroups in a similar way as root systems linearize 
 the study of reflections.   We  give a 
 direct elementary   proof that a Brink-Howlett groupoid has 
 certain favorable lattice theoretic properties (in regard to its 
 weak orders, defined by inclusion of inversion sets in its root system) which may be summarized by the statement that 
 it underlies a principal rootoid (see e.g. \cite{DyWa})  and to prove that  the Brink-Howlett generators of Brink-Howlett  groupoids are the specialization to those groupoids of the  standard (simple) generators of principal rootoids.  This makes the extensive (but largely unpublished) general theory of principal rootoids directly applicable to the study of Brink-Howlett groupoids, and provides interesting and important examples for that general theory.

Some of the  root systems we consider have certain features  formally 
analogous to those of root systems of Borcherds-Kac-Moody 
Lie algebras (see \cite{Kac}). In particular, they may contain imaginary simple 
roots, in a natural sense. In the case of root systems realized in 
real vector spaces, inclusion of such  imaginary simple roots in 
the root system is necessary in order to ensure that  each 
positive real root is expressible as a non-negative linear 
combination of simple roots and thereby to generalize the basic notion of root coefficient.  

 We also define, for suitable root systems of  Brink-Howlett groupoids in real vector spaces, natural analogues   of the Tits cone and imaginary cone of a Coxeter group.  The definitions involve both real and imaginary roots. The basic fact that the stabilizer of a subset of the fundamental chamber of the Tits cone of a Coxeter group is a standard parabolic subgroup is generalized to the fact that 
the ``pointwise stabilizer groupoid'' of a subset of  a fundamental chamber of a Brink-Howlett groupoid's Tits cone is another Brink-Howlett groupoid, which one may view as an analogue for Brink-Howlett groupoids of a  standard parabolic subgroup.  

 This paper  emphasizes basic results and phenomena for  Brink-Howlett groupoids for which there are no extensions known (or in some cases, possible) to principal rootoids in general. The focus is mainly  on imaginary roots and realizations of the root systems in real vector spaces. The results suggest many natural directions for further study of  Brink-Howlett (and other) groupoids with root systems, and applications to the study of Coxeter groups.

 \section{Overview of results}
 \label{sec:1}
\subsection{Coxeter groups and root systems}  Let $V$ be a real vector space. 
For subsets $X$ and $Y$ of $V$, we write $X+Y
:=\mset{x+y\mid x\in X, y\in Y}$, $cX:=\mset{cx\mid x\in X}$ and $-X:=(-1)X$. For $X\seq V$, let $\Span(X)$ denote the linear span of $X\cup\set{0}$ and  $\cone(X)$ denote the convex cone in $V$ spanned by $X\cup\{0\}$.   

As a general convention, in  writing a  sum  $\sum_{i\in I}a_{i}$ of elements $(a_{i})_{i\in I}$ of  an abelian group (e.g. a vector space), we tacitly assume that $a_{i}=0$ for almost all (i.e. all but finitely many) $i\in I$. When  such a sum is, more specifically, an $R$-linear combination $\sum_{i\in I}r_{i}m_{i}$ of  a family of elements $(m_{i})_{i\in I}$ of some left $R$-module $M$ over  a ring $R$, we assume more strongly that $r_{i}=0$ for almost all $i\in I$.  In both cases, empty sums have value $0$, by convention. 
In particular, for $X\seq V$, \[\Span(X)=\mset{\sum_{\alpha\in X} c_{\alpha}\alpha\mid  c_{\alpha}\in\mathbb{R} \text{ \rm for all  $\alpha\in X$}}\] and  \[\cone(X)=\mset{\sum_{\alpha\in X} c_{\alpha}\alpha\mid c_{\alpha}\in\mathbb{R}_{\geq 0} \text{ \rm for all $\alpha\in X$}}.\]
As references for elementary background we use from convex geometry, see for instance \cite{Web},  \cite[Ch II, \S 2]{BouEsp} and \cite{Bar}.

\subsection{} By a \emph{real quadratic space}, we shall mean in this paper  a pair  $(V,B)$ consisting of  a real vector space $V$ and a symmetric bilinear form $B\colon V\times V\to \mathbb{R}$. 

Fix a real quadratic space $(V,B)$.
Let $\mathrm{O}=\mathrm{O}(V,B)$ denote the \emph{orthogonal group} of $(V,B)$, which is the group of automorphisms of $(V,B)$. The elements of $\mathrm{O}$  are the invertible linear maps $g\colon V\to V$ which preserve the form $B$, in the sense that $B(gx,gy)=B(x,y)$ for all $x,y\in V$.  An element  $\alpha$ of $V$ is said to be \emph{non-isotropic} if $B(\alpha,\alpha)\neq 0$.  For non-isotropic $\alpha\in V$, write  $\alpha^{\vee}:=\frac{2}{B(\alpha,\alpha)}\alpha$ and define the     \emph{reflection} $s_{\alpha}\in O$    by $s_{\alpha}(v)=v-B(v,\alpha^{\vee})\alpha$ for $v\in V$. 
\subsection{} \label{rootbasis} Let  $(\Phi,\Pi)$ be a based root system in  $(V,B)$ in the sense of \cite[2.3]{DyHo}, with positive roots $\Phi^{+}$ and associated Coxeter system $(W,S)$. This means that:
\begin{conds}\item $(V,B)$ is a real quadratic space
\item $\Pi\subseteq V$ satisfies $\cone(\Pi)\cap \cone(-\Pi)=\{0\}$.
\item for all $\alpha\in \Pi$, one has  $\alpha\not\in \cone(\Pi\setminus\{\alpha\})$ and  $B(\alpha,\alpha)=1$. 
\item  $S:=\{s_{\alpha}\mid \alpha\in \Pi\}$ and $W:=\langle S\rangle$ is the subgroup of $O$ generated by $S$.
\item $\Phi:=\{w\alpha\mid w\in W,\alpha\in \Pi\}$ satisfies 
  $\Phi=\Phi^{+}\cup \Phi^{-}$ where   $\Phi^{+}:=\Phi\cap \cone(\Pi)$ and  $\Phi^{-}:=\Phi\cap \cone(-\Pi)$.   \end{conds} 
  
  These conditions imply that $\Pi$ is \emph{positively independent}. That is, if $\sum_{\alpha\in \Pi}c_{\alpha}\alpha=0$ with $c_{\alpha}\geq 0$ for all $\alpha\in \Pi$, then $c_{\alpha}=0$ for all $\alpha$.

See \cite{DyHo},  \cite{BDy} and \cite{Dy19} for  more details.  For example, $\Phi$ could be the standard root 
  system of $(W,S)$, with simple roots $\Pi$ forming a basis of $V$ and  $B$ the standard symmetric bilinear form on $V$, as 
  described in \cite{Bou}, \cite{Hu90}, \cite{BjBr} and \cite{AbBr}.   We assume the 
 reader is familiar with basic facts concerning 
  Coxeter systems and their  root systems,  as may be found in those references.

 \subsection{} Let $T:=\{wsw^{-1}\mid w\in W, s\in S\}
=\mset{s_{\alpha}\mid \alpha\in \Phi}$. Elements of $\Pi$, $\Phi$, $\Phi^{+}$, $\Phi^{-}$,  $T$    and $S$ are called 
\emph{simple roots}, \emph{roots}, \emph{positive roots},  
 \emph{negative roots}, \emph{reflections}  and \emph{simple 
 reflections}, respectively.

  By definition, the \emph{rank}   of the Coxeter system $(W,S)$ 
  is  $\vert S\vert$. From \cite{DyRef}, for example, the rank is the minimum of the cardinalities  
   of sets of reflections which generate $W$. Unless 
  otherwise stated, Coxeter systems are not assumed in this paper to be of 
  finite rank.

The \emph{standard  length function} of $(W,S)$ is denoted 
$\ell$ or $\ell_{(W,S)}$.   
The map $\alpha\mapsto s_{\alpha}\colon  \Phi^{+}\to T$ is a 
bijection.  
 For $w\in W$, the  \emph{left  inversion set} $\Phi_{w}$
  of $w$ is \begin{equation}
 \Phi_{w}:=\Phi^{+}\cap w(-\Phi^{+}).
 \end{equation} Define \begin{equation}
N(w):=\mset{s_{\alpha}\mid \alpha\in \Phi_{w}}=\set{t\in T\mid \ell(tw)<\ell(w)}.
\end{equation} 
If $x,y\in W$, we have 
\begin{equation}
N(xy)=N(x)+xN(y)x^{-1},\qquad \Phi_{xy}=(\Phi_{x}\sm(-x\Phi_{y}))\cup (x(\Phi_{y})\sm -\Phi_{x})
\end{equation}
where $+$ denotes symmetric difference on the power set $\mathcal{P}(T)$ of $T$ (that is,  $A+B=(A\cup B)\sm 
(A\cap B)=(A\sm B)\cup(B\sm A)$ for $A,B\seq T$) . Further, \begin{equation}
\ell(x)=\vert N(x)\vert =\vert \Phi_{x}\vert, \qquad \text{\rm if $x\in W$}.
\end{equation}
\subsection{}\label{weakorder} The \emph{weak (right) order}  of $W$ is the partial order $\leq$ on $W$ defined by
\begin{equation*}
x\leq y\iff \Phi_{x}\seq \Phi_{y}\iff N(x)\seq N(y)\iff \ell(y)=\ell(x)+\ell(x^{-1}y),\end{equation*}
for $ x,y\in W$. One has
\begin{equation*}
x\leq xy\iff \Phi_{x^{-1}}\cap \Phi_{y}=\eset\iff N(x^{-1})\cap  N(y)=\eset\iff \ell(xy)=\ell(x)+\ell(y).\end{equation*} 
The poset $(W,\leq)$ is a complete meet semilattice (see \cite[Theorem 3.2.1]{BjBr}).

\subsection{} \label{abreal} We say loosely that the pair $(\Phi,\Pi)$, with $\Pi\seq \Phi$ regarded as  subsets of the real vector
 space $V$ equipped with its bilinear form $B$ and $W$-action, 
is a \emph{realized root system} of $(W,S)$. There may be uncountably many pairwise non-isomorphic (in the natural 
sense) realized root systems of a fixed (finite rank) Coxeter 
system $(W,S)$, even restricting to those for which $\Pi$ is a basis of $V$.

On the other hand,  one may regard $(\Phi,\Pi)$ more abstractly 
as a pair consisting of a $W\times\set{\pm 1}$-set $\Phi$ 
together with a subset $\Pi$ of $\Phi$ and a specified bijection 
$\alpha\mapsto s_{\alpha}\colon \Pi\to S$. As such, it depends 
only on  $(W,S)$, up to the natural notion of isomorphism of 
such pairs.   We shall call such a pair an \emph{abstract root 
system} of $(W,S)$.   Two well-known descriptions  of the
(unique up to isomorphism) abstract root system directly in 
terms of the Coxeter system $(W,S)$ are recalled in \ref{abrs}--\ref{abrs2}. 

Many of the results discussed in \ref{parabolic}--\ref{infTits}  are of a purely algebraic and combinatorial nature, in that they can 
be formulated in terms of the abstract root system of $(W,S)$, though we do not always do so when possible; however,  
currently known proofs of many of them  require use of realized root systems. Later results  of this section typically depend in an 
essential way on a choice of realized root system.

\subsection{}\label{parabolic} The subgroups $W_{J}:=\mpair{J}$ of $W$, for $J\subseteq S$, are called \emph{standard 
parabolic subgroups}. Conjugates $wW_{J}w^{-1}$, where $w\in W$ and $J\seq S$,  of standard parabolic subgroups are called 
\emph{parabolic subgroups}. 

In this and similar notation,
we  use indexing by sets of reflections, or by the corresponding  sets of positive roots, interchangeably. For example, 
if $\Gamma\seq \Pi$, then $W_{\Gamma}:=W_{J}$ where $J:=\{\,s_{\alpha}\mid \alpha\in \Gamma\,\}$. 
More generally, for any subset $\Gamma$ of roots, we write $W_{\Gamma}=W_{R}=\mpair{R}$ where $R=\mset{s_{\alpha}
\mid \alpha\in \Gamma}$.

 A subset $J$ of $\Pi$, the corresponding subset $\mset{s_{\alpha}\mid \alpha\in J}$ of simple reflections, or the 
 subgroup $W_{J}$ itself, is said to be \emph{spherical} if $W_{J}$ is finite.
In that case, we denote the \emph{longest element} of $W_{J}$ as $w_{J}$; it satisfies $(w_{J})^{2}=1$, $N(w_{J})=W_{J}\cap T$ and  $w_{J}J=-J$.

\subsection{Brink-Howlett groupoids} Recall that a \emph{groupoid} is a small category $G$ in which every morphism is invertible. Such a groupoid is  determined by its set 
$\mathrm{Ob}(G)$ of objects, the hom-sets $\lrsub{b}{G}{a}:=\mathrm{Hom}_{G}(a,b)$ for $a,b\in \mathrm{Ob}(G)$ and composition maps \[(g,f)\mapsto gf\colon \lrsub{c}{G}{b}\times 
\lrsub{b}{G}{a}\to \lrsub{c}{G}{a}\] subject to suitable (associativity, unit and inverse) conditions.

A groupoid $H$ is said to be \emph{connected} if it has at least 
one object and there is a morphism $a\to b$ for any objects $a$ and $b$ of $H$. The 
inclusion-maximal connected subgroupoids $H$ of $G$  are called the \emph{connected components}, or  \emph{components},  of $G$.

By a \emph{representation} of the groupoid $G$ in a category $C$, we mean a functor $F\colon G\to C$. In case $C$ is a 
concrete category, given by a faithful ``underlying set'' functor $U\colon C\to \mathrm{Set}$
(which is frequently suppressed from the notation)
  to the category of sets, we often abbreviate $gx:=(U(F(g)))(x)\in U(F(b))$ for $g\in \lrsub{b}{G}{a}$ and $x\in U(F(a))$. Then 
  $F$ is characterized by ``action maps''  \[(g,x)\mapsto gx\colon \lrsub{b}{G}{a}\times U(F(a))\to U(F(b))\] subject to suitable (associativity and unit) conditions.

\subsection{}\label{abbeg}  In their study \cite{BrHo99} of normalizers of standard  parabolic subgroups, Brink and Howlett implicitly defined a groupoid $\wh G$  as follows.

The set of objects of $\wh G$ is  $\ob(\wh G):=\mathcal{P}(\Pi)$, the power set of $\Pi$. For two subsets $J$ and $K$ of $\Pi$, a morphism $J\to K$ in $\wh G$ is by definition a triple
$(K,w,J)$ where $K=w(J)$; we may write such a morphism simply as $w\colon J\to K$ or even just $w$ if $J$ and/or $K$ 
are understood. Composition of morphisms is induced by multiplication in $W$:  $(L,v,K)(K,w,J)=(L,vw,J)$.  

We call $\wh G$ the \emph{full  Brink-Howlett groupoid} of $(W,S)$. Brink and Howlett mainly studied the connected 
components of  $\wh G$. By a \emph{Brink-Howlett groupoid} (of $(W,S)$), we mean any groupoid $G$ which is a union of 
components of $\wh G$.  

\begin{ex}\label{morphex}  Let $J\seq \Pi$ and $\alpha\in \Phi^{+}\cap J^{\perp}$, where for any subset $U$ of $V$, \[U^{\perp}:=\mset{v\in V\mid B(v,u)=0\text{ \rm for all $u\in U$}}\]  Then $(J,s_{\alpha},J)$ is a morphism in $\wh G$. Even for 
$W$ of type $B_{2}$, this morphism does not in general lie in the subgroupoid generated by morphisms $(J,s_{\beta},J)$ with $\beta\in \Pi\,\cap\, \alpha^{\perp} $.
\end{ex}

\subsection{} Brink and Howlett \cite{BrHo99} showed that Brink-Howlett groupoids have several Coxeter group-like 
features, including  an analogue of the  presentation of Coxeter groups by simple generators subject to braid relations.   
They also showed  that for any object $J$, a certain subgroup of the vertex group of morphisms $J\to J$ (containing in particular 
all morphisms discussed in Example \ref{morphex}) is a Coxeter 
group with  root system naturally constructed from a realized 
root system  $\Phi$ of $(W,S)$.  However, Brink and Howlett's paper does not develop notions of  root systems of  Brink-Howlett groupoids.
 
  Similar Coxeter group-like groupoids also appear, with root systems, in the literature as Coxeter (and Weyl) groupoids (see 
  for instance \cite{CH} and \cite{HY}), although the Brink-Howlett groupoids  of infinite Coxeter groups are not of that type in general. 
  
  Rudiments of a  more general framework (rootoids, especially principal rootoids) including the Brink Howlett groupoids,  and Coxeter and Weyl groupoids,  are indicated in 
  \cite{DyGrp1}--\cite{DyGrp2} and \cite{DyWa}. The main novelties in the theory of rootoids are its lattice- and category-theoretic underpinnings, and consequent closure of  classes of 
  rootoids under various  natural category-theoretic constructions. 
      
  The underlying groupoid of  a principal rootoid admits a 
  canonical presentation by generators and relations generalizing those of Coxeter groups, Coxeter groupoids and 
  Brink-Howlett groupoids. 
    A principal rootoid also  admits a  canonical analogue of the abstract root system of a Coxeter group, but, even if finite, does not in general have a  ``realized root system''  in a real vector 
    space,  because of the existence of non-realizable oriented matroids (see \cite{DyWa}). Finite,  complete, principal rootoids admit suitable ``oriented matroid root systems''  by  
    \cite{DyWa}, but it  is an open question whether (or under what conditions) infinite principal rootoids  do. 
  
  Our first main goal in this paper is to interpret the standard abstract root system of a Brink-Howlett groupoid (which can be 
  shown to be the standard root system of the corresponding 
  principal rootoid, though we do not define this)   directly in terms of 
  the associated Coxeter system. We show that positive roots 
  correspond to certain reflection subgroups, and the standard abstract root system may be naturally extended in various ways
  (e.g. allowing analogues of  imaginary roots and non-reduced root systems) to 
  abstract root systems with positive roots corresponding to  larger and more natural classes of reflection subgroups.  Our second 
  main goal is to   study  realizations of these abstract  root systems in real vector spaces. We obtain a satisfactory realization for some 
  of these root systems, including the  standard abstract root system, 
  and those we consider with only real roots.  Finally, we extend the 
  definitions and some basic properties of Tits and imaginary cones of Coxeter groups to these realized root systems.

\subsection{Signed groupoid-sets} \label{sgspreamb} We shall not discuss rootoids  in this paper, but will make use of  the very closely  related notion of rootoidal signed groupoid-sets. These 
are not as well behaved category-theoretically as rootoids but are more concrete and better adapted for purposes here. Signed groupoid-sets
 are discussed  in \ref{signset}--\ref{realcomp1} and  Section \ref{sgs};  for more detail and proofs, see \cite{DyWa}. For basic 
 terminology and facts concerning ordered sets and lattices as used in this paper, see for example \cite{DaPr}.

\subsection{}\label{signset} The category $\mathbf{Set}_{\pm}$ of \emph{definitely signed sets} has as its objects sets $X$ with a specified free action of the  \emph{sign group} $\set{\pm 1}=\set{\pm}$ and a distinguished set $X^{+}$ of orbit 
representatives, called \emph{positive elements} of $X$. Let $X^{-}:=-X^{+}=\mset{-x\mid x\in X^{+}}$. The \emph{sign} of $x\in X$ is $\epsilon\in \set{\pm}$ where $x\in X^{\epsilon}$. 
Morphisms of definitely signed sets are
$\set{\pm }$-equivariant maps, not necessarily preserving positive elements, with composition by composition of underlying functions.   
 
 A \emph{signed groupoid-set} is by definition a pair $(G,X)$ where  $G$ is a groupoid and   $X\colon G\to 
\mathbf{Set}_{\pm}$ is a functor. We refer to $(G,X)$, or just $X$ for short,  as a \emph{signed $G$-set}.    
Since $ \mathbf{Set}_{\pm}$ is 
a concrete category, we may (and usually do) describe signed 
$G$-sets by their action maps \[(g,x)\mapsto gx\colon \lrsub{b}
{G}{a}\times X(a)\to X(b),\] which in particular satisfy $g(-x)=-
gx$.  Here $gx:=(X(g))(x)$. Elements of $X(a)$, $X(a)^{+}$ and 

$X(a)^{-}$ respectively are called \emph{roots}, \emph{positive roots} and \emph{negative 
roots} of $X$ (\emph{at} $a$ or \emph{based at} $a$). We say a root in $X(a)$ has $a$ as its 
\emph{base} (tacitly assuming in doing so  that the sets $\lsub{a}{X}:=X(a)$ for $a\in \ob(G)$ are pairwise disjoint, as we do by convention unless otherwise stated).

 \subsection{}
\label{sgsmorph}
There is a natural notion of morphism of signed  groupoid-sets (see \cite{DyGrp1}--\cite{DyGrp2} but we shall not use it in this paper. However, the notion of a morphism of signed $G$-sets 
(for fixed $G$) defined below will be used extensively.
   
Let   $(G,X)$ and $(G,Y)$ be signed $G$-sets, where $G$ is a groupoid. 
 A morphism $\nu\colon (G,X)\to (G,Y)$ is defined to be   a natural transformation 
$\nu\colon X\to Y$ of functors $G\to\Set_{\pm}$  for which the components $\nu_{a}\colon X(a)\to Y(a)$ for $a\in \ob(G)$, which are morphisms of signed sets, are  also \emph{positivity 
preserving} in the sense they restrict to maps $X(a)^+\to Y(a)^{+}$.  This gives rise to a category of signed $G$-sets, and thereby provides a natural notion of \emph{isomorphism}  of signed $G$-sets. 

 We say the above morphism $\nu$ is an \emph{embedding} of signed $G$-sets if each component $\nu_{a}$ with 
 $a\in \ob(G)$ is injective as a map of sets.   We then write  $(G,X)\hookrightarrow (G,Y)$.

We say that $(G,X)$ is a \emph{signed $G$-subset} of $(G,Y)$ 
if there is an embedding $\nu\colon (G,X)\hookrightarrow (G,Y)$ such that for each object $a$ of $G$, the  underlying function of 
the component    $\nu_{a}\colon X(a)\to Y(a)$ is an inclusion map
(so in particular, $X(a)\subseteq Y(a)$).
 By abuse of notation,  we then  write  $(G,X)\seq (G,Y)$.
 If there is an embedding $(G,X)\hookrightarrow (G,Y)$, then $(G,X)$ is isomorphic to a signed $G$-subset of $(G,Y)$.

\subsection{}\label{gpdcomp}  Let $(G, X)$ be a signed groupoid-set. For any subgroupoid $H$ of $G$, one has a signed groupoid-set $(H,X_{\vert H})$, called the restriction of 
$(G,X)$ to $H$, where the functor $X_{\vert H}\colon  H\to \Set_{\pm}$ is the restriction of $X\colon G\to \Set_{\pm}$. The restriction of $(G,X)$ to a connected component $H$ of $G$ is 
called a (connected) component of $(G, X)$.

\begin{ex}\label{Coxrootoid} Regard the Coxeter group $W$ as a one-object groupoid $W_{\bullet}$ with object $\bullet$ and 
 $W$ as the group of automorphisms of that object. Then  the abstract root system $\Phi$ of $(W,S)$ naturally gives rise to a 
 signed $W_{\bullet}$-set $(W_{\bullet},\Phi_{\bullet})$;  here, $\Phi_{\bullet}$ is the functor taking the  object $\bullet$ of 
 $W_{\bullet}$ to the signed set $\Phi$, and a morphism $w$  of $W_{\bullet}$ (i.e. an element $w$ of $W$) to the $\set{\pm}$-equivariant action by $w$  on  $\Phi$. 
\end{ex} 

\subsection{}\label{invsets} Signed groupoid-sets are  of greatest interest when they satisfy stringent additional 
conditions, discussed below,  abstracted from those satisfied in 
example \ref{Coxrootoid}. See
Section \ref{sgs}, \cite{DyGrp1} and \cite{DyWa} for   more details.  

Let $(G,X)$ be a signed groupoid-set. For an object $a$ of $G$, the \emph{left star} of $a$ is the set $\lsub{a}{G}
=\dot\bigcup_{b\in \ob(G)}\,\lrsub{a}{G}{b}$ of morphisms of $G$ with codomain $a$.  Here, we regard distinct $\hom$-sets of a 
category as pairwise disjoint, as is a common convention.  

The \emph{left inversion set} $X_{g}$ of $g\in \lsub{a}{G}$ is the set
\begin{equation}\label{leftinvset} X_{g}:=\lsub{a}{X}^{+}\cap g(\lsub{b}{X}^{-}), \quad g\in \lrsub{a}{G}{b}\end{equation} of 
positive roots based at $a$ which are made negative by $g^{-1}$.  Also define 
\begin{equation}
\wh X_{g}=\lsub{a}{X}^{+}+g(\lsub{b}{X}^{+})
\end{equation} where $A+B:=(A\sm B)\cup (B\sm A)=(A\cup B)\sm (A\cap B)$ denotes symmetric difference. Essentially by definition, the map $g\mapsto \wh X_{g}$ defines a 
$1$-coboundary for $G$, so  we have the  $1$-cocycle property \begin{equation}
\wh{X}_{gh}=\wh{X}_{g}+g(\wh X_{h}), 
\end{equation} for morphisms $g,h$ with the composite $gh$ defined.
Observe that  we have 
\begin{equation}
X_{g}=\wh{X}_{g}\cap \lsub{a}{X}^{+},\qquad \wh X_{g}=X_{g}\,\dot\cup -X_{g}. 
\end{equation}

Note also that for any morphism $\rho\colon (G,X)\to (G,Y)$ of signed groupoid sets and any $a\in \ob(G)$ and 
$g\in \lsub{a}{G}$, we have $\rho_{a}(X_{g})\seq Y_{g}$, with equality for all $a$ if $\rho_{a}(X(a))=Y(a)$ for all $a$.

\subsection{}\label{sgsterm}The \emph{weak (right) preorder} of $(G,X)$ based at $a$ is the pre-order
(that is, reflexive, transitive relation) $\lsub{a}{\leq}$ on $\lsub{a}{G}$ defined by 
\[ g \lsub{a}{\leq} h\iff X_{g}\subseteq X_{h}, \qquad  g,h\in \lsub{a}{G}.\] 

We say $(G,X)$ is \emph{faithful} if   for any  object $a$ of $G$ and any morphisms $g,g'\in \lsub{a}{G}$, one has $X_{g}=X_{g'}
$   if and only if $g=g'$. Equivalently,  $(G,X)$ is \emph{faithful} if and only if for each $a\in \ob(G)$, the weak preorder 
$(\lsub{a}{G},\lsub{a}{\leq})$  is a partial order, which is then 
called the \emph{weak order} of $(G,X)$ at $a$.   One says that $(G,\Lambda)$ is \emph{interval finite} if  it is faithful and for any 
$a\in \ob(G)$, all closed intervals in the weak order  $(\lsub{a}
{G},\lsub{a}{\leq})$ are finite.

We say that $(G,X)$ is  \emph{rootoidal} if   it is faithful and for each $a\in \mathrm{Ob}(G)$,
\begin{conds}
\item $(\lsub{a}{G}, \lsub{a}{\leq})$ is a complete meet semilattice (i.e. any non-empty subset $P$ of $\lsub{a}{G}$ has a meet  $\bigwedge P$).
\item if $P\subseteq \lsub{a}{G}$ and $q\in \lsub{a}{G}$ are such that
$X_{p}\cap X_{q}=\emptyset$ for each $p\in P$, and if the join $p':=\bigvee P$ in $(\lsub{a}{G}, \lsub{a}{\leq})$ of $P$ exists, then $X_{p'}\cap X_{q}=\emptyset$.  
\end{conds}
The condition (ii) is called the \emph{Join Orthogonality Property (JOP)} of $(\lsub{a}{G}, \lsub{a}{\leq})$.  It implies that
$(\lsub{a}{G}, \lsub{a}{\leq})$ embeds as a down-set in a complete ortholattice.

\subsection{} Rootoidal signed groupoid sets have many favorable properties, but   additional  ``discreteness'' conditions 
are required in order to restrict to more ``Coxeter-group-like'' signed groupoid sets. 

For the following definitions, assume that $(G,X)$ is a faithful signed groupoid-set. A morphism $g\in \lsub{a}{G}$ is called 
\emph{simple} if
$\vert X_{g}\vert=1$ and \emph{atomic} if $g$ is an atom of the poset
$(\lsub{a}{G}, \lsub{a}{\leq})$.  A simple morphism is clearly atomic.

For any set of generators $C$ of $G$,   the \emph{length} $\ell_{C}(g)$ of $g$ is defined to be the minimal length of an expression of $g$ as a product of elements of $C$ and their 
inverses; by convention, $\ell_{C}(g)=0$ if $g$ is an identity morphism. 

We say that $(G,X)$ is \emph{principal} if it is generated by its set $R$
of simple morphisms, and for any morphism $g$ of $G$,
$\ell_{R}(g)=\vert X_{g}\vert$. 

We say that $(G,X)$ is \emph{preprincipal} if it is generated by its set $A$ of atomic morphisms and for any $a\in\mathrm{Ob}(G)$, 
$g\in \lsub{a}{G}$ and $s\in \lsub{a}{A}:= \lsub{a}{G}\cap A$, one has either $X_{s}\subseteq X_{g}$ or $X_{s}\cap X_{g}=\emptyset$.
\subsection{} \label{realcomp1} See Section \ref{sgs} for more details of the following definitions and constructions. 

 A root in  $\lsub{a}{X}$ is said to be \emph{real} if its sign is changed by the action of $g^{-1}$ for some   $g\in \lsub{a}{G}$,
  and \emph{imaginary} otherwise.  These notions give rise to signed $G$-subsets
$(G,X^{\re})$ and $(G,X^{\im})$ of $(G,X)$. Two roots $\alpha, \beta$ based at the same object $a$ of $G$ are said to be 
\emph{parallel} if $g^{-1}\alpha$ and $g^{-1}\beta$ have the same sign as one another, for all $g\in \lsub{a}{G}$.

Some of the results involve two endofunctors of the category of  signed $G$-sets, \emph{realification}, denoted 
$(G,X)\to (G,X^{\re})$ on objects and $\nu \mapsto \nu^{\re}$ on morphisms,  and \emph{compression}, denoted 
$(G,X)\to (G,X^{\cc })$ on objects and $\nu \mapsto \nu^{\cc }$ on 
morphisms.
By definition, the positive (resp., negative) roots of $X^{\re}$ 
based at an object $a$ of $G$ are the  positive (resp., negative) real roots of $X$  which are based 
 at $a$.  The positive (resp., negative) roots of $X^{\cc }$ based at an object $a$ of $G$ are the parallelism classes of  positive 
 (resp., negative) roots of $X$  which are based at $a$.
Realification and compression commute and their composite, \emph{real compression}, is denoted by 
$(G,X)\mapsto (G,X^{\rec})$ on objects and $\nu \mapsto \nu^{\rec}$ on morphisms. We say $X$ is \emph{real} if 
$X^{\re}\cong X$, \emph{compressed} if $X^{\cc }\cong X$ and \emph{real compressed} if $X^{\rec}\cong X$.

Realification and compression don't change the underlying 
groupoid of a signed groupoid set $(G,X)$ or its weak preorders, 
so in particular, they preserve rootoidality.
  If  $(G,X)$ is preprincipal,  then $(G,X^{\rec})$ is real, compressed, principal,  and  the atomic generators of $(G,X)$ 
  are the  simple generators of $(G,X^{\rec})$.  The class of real, compressed, principal, rootoidal signed groupoid sets
 $(G,X)$ has many features in common with the class of signed groupoid sets $(W_{\bullet},\Phi_{\bullet})$ associated to 
 Coxeter systems $(W,S)$ (and also analogues for Brink-Howlett groupoids and Coxeter and Weyl groupoids). The class
  of preprincipal, rootoidal signed groupoid sets is an enlargement of that class which admits abstract analogues of 
  non-reduced root systems of finite Weyl groups and root systems (possibly with imaginary roots) of Coxeter groups of  
  Kac-Moody Lie algebras.

  \subsection{Standard abstract root systems of  Brink-Howlett groupoids} \label{standBH}Let $G$  denote a Brink-Howlett groupoid of $(W,S)$   
   and $(W_{\bullet},\Phi_{\bullet})$ denote the signed groupoid-set from Example \ref{Coxrootoid}.
There is a functor
  $F\colon G\to W_{\bullet}$ such that $F(J)=\bullet$ for all $J\in \mathrm{Ob}(G)$ and $F(J\xrightarrow{w} K)=w$. Composing the functor
  $\Phi_{\bullet}$ with $F$ gives a signed groupoid-set
  \begin{equation*}
  (G,\Lambda), \qquad \Lambda:=\Phi_{\bullet}\circ F.
  \end{equation*} 
  
  The following is a main result of this paper, proved in Section \ref{sgs}.  It is a consequence of  more general facts, some  discussed in \cite{DyGrp1}--\cite{DyGrp2}, but we provide a direct  proof. Note that part (ii) follows from (i), by the discussion in \ref{realcomp1}.
  
\begin{thm}\label{standpreprinc} Define the signed groupoid-set $(G,\Lambda)$ as in \ref{standBH}.
\begin{num}
\item $(G,\Lambda)$ is a preprincipal, rootoidal signed groupoid-set with the set of  Brink-Howlett generators of $G$ as its set of atomic morphisms.
\item $(G,\Lambda^{\rec})$  is a principal rootoidal signed groupoid-set with the set of  Brink-Howlett generators of $G$ as  its set of simple morphisms.   
\end{num} We call $(G,\Lambda^{\rec})$ the standard signed groupoid-set (attached to $G$, $(W,S)$ and $(\Phi,\Pi)$). \end{thm}

\begin{ex} Consider $(G,\Lambda)$ as in \ref{standBH}. 
Consider an object $J\seq \Pi$ of $G$. We have $\Lambda(J)^{+}=\Phi^{+}$ and $\Phi_{J}^{+}\seq \Lambda(J)^{\im   ,+}$ since for any morphism $(K,w,J)$ in $G$, one has $w(J)=K$ by definition and therefore $w(\Phi_{J}^{+})=\Phi_{K}^{+}\seq \Phi^{+}=\Lambda(K)^{+}$. Suppose now that $W$ is finite.
There is a morphism $(K,w_{S}w_{J},J)$ in $G$
where $K:=-w_{S}J\seq \Pi$. One has $
w_{S}w_{J}(\Phi^{+}\sm \Phi_{J}^{+})=\Phi^{-}\sm \Phi_{K}^{-}$, which implies that  $\Phi^{+}\sm \Phi_{J}^{+}\seq 
\Lambda(J)^{\mathrm{re},+}$. This forces 
$\Lambda(J)^{\im   ,+}=\Phi_{J}^{+}$ and  $\Lambda(J)^{\mathrm{re},+}=\Phi^{+}\sm \Phi^{+}_{J}$.

The corresponding  equalities need not hold for infinite $W$.
For example,  one readily checks that  if the Coxeter system $(W,S)$ is universal
(that is, for any distinct $r,s\in S$, the product $rs$ has infinite order), then for any non-empty subset $J$ of $S$, one has
$\Lambda(J)^{\mathrm{im},+}=\Phi^{+}$ and  $\Lambda(J)^{\mathrm{re},+}=\eset$.   \end{ex}

\subsection{Brink-Howlett generators} We shall  construct in this paper  other signed  $G$-sets  with real compression isomorphic to $(G,\Lambda^{\rec})$ as in Theorem \ref{standpreprinc}.  To begin, we recall the known  explicit  description of the Brink-Howlett generators, and related elements, of Brink-Howlett groupoids. 

Let  $(m_{r,s})_{r,s\in S}$ denote the Coxeter matrix of $(W,S)$. By definition, $m_{r,s}:=\mathrm{ord}(rs)\in \mathbb{N}_{\geq 1}\cup\set{\infty}$. Recall that the Coxeter graph is the undirected, edge-labelled (simple) graph with vertex set $S$ and an edge joining $r,s\in S$ if $m_{r,s}\geq 3$, that edge being labeled by $m_{r,s}$ if $m_{r,s}\geq 4$.

Using the natural bijection $\alpha\mapsto s_{\alpha}\colon \Pi\to S$, the  Coxeter graph of $(W,S)$ will often be regarded 
in this paper  as an edge-labeled  graph with vertex set $\Pi$. By a \emph{component} of  $\Pi$, we mean  a connected 
component  of this graph, or sometimes just the vertex set $J$ of that component. We may then refer to $\mset{s_{\alpha}\mid 
\alpha\in J}$ as a component of $S$, $W_{J}$ as a component 
of $W$, 
$\Phi_{J}$ as a component of $\Phi$, $\Phi_{J}^{+}$ as a component of $\Phi^{+}$, etc.  Similarly, we may refer to 
components of other Coxeter groups, their simple reflections, 
root systems, simple roots,  and positive systems provided the sets of simple reflections and simple roots are clear from 
context (as they are for any reflection subgroup of $(W,S)$, by 
conventions in \ref{refsg}).  We refer to a component in any of these senses as being of finite, affine, indefinite, locally finite, 
infinite type etc according to the type of $J$ (see \cite{Kac} and \cite{Dy13}).

The following  facts, and the closely related Proposition \ref{finindcond}, are essentially due to Howlett and  Deodhar. 
  See  \cite{How}, \cite{Deod}, \cite{BrHo99}, and also   \cite{Dy13}  for  $W$ of possibly infinite rank.

\begin{lem}\label{Deodgen} Let $J\seq K\seq  \Pi$. Then $\vert (W_{K}\sm W_{J})\cap T\vert$ is finite if and only if there are  
finitely many components  of  $K$ which meet $K\sm J$, and they are all of finite type. 
Assume that this is the case, and let $L\supseteq K\sm J$ be 
the union of all such components, which is a spherical subset of $\Pi$. Let $w=\nu(K\sm J,J):=w_{L}w_{L\cap J}\in W$.
Then:
\begin{num}\item $\Phi^{+}_{K}\sm \Phi^{+}_{J}=\Phi^{+}_{L}\sm \Phi^{+}_{L\cap J}=\Phi_{w^{-1}}$.
\item $(W_{K}\sm W_{J})\cap T=(W_{L}\sm W_{L\cap J})\cap T=N(w^{-1})$. 
\item $w{J}=(J\sm L)\cup w(J\cap L)=(J\sm L)\cup -w_{L}(J\cap L)\seq K$.
\item Writing $w{J}={J'}\seq K$ the element $\nu(K\sm J',J')$ of $W$ is defined similarly as $\nu(K\sm J,J)$,  and it satisfies
$\nu(K\sm J,J)^{-1}=\nu(K\sm J',J')$ in $W$.
\item $\nu(K\sm J,J)=\nu(L\sm J, L\cap J)$.
\end{num}
\end{lem}

\begin{remark} Whenever  we  use notation  of the form $\nu(M,L)$  in this paper, where $M$ and $L$ are subsets of the set of simple roots of a Coxeter system $(W,S)$, it is tacitly 
assumed that $M\cap L=\eset$ and $\vert (W_{M\cup L}\sm W_{L})\cap T\vert$ is finite, so $\nu(M,L)$  is defined by the preceding lemma.  \end{remark}
\subsection{} \label{BHgen} The triples $(J',w,J)$, with $w=\nu(K\sm J,J)$, defined as in Lemma \ref{Deodgen}, for which $J$ (or equivalently,  $J'$) are objects of  a fixed 
Brink-Howlett groupoid $G$  
are morphisms of $G$. Those for which, additionally, $K\sm J$ (or equivalently, $K\sm J'$)  is a singleton set form a set of  generators for $G$; we call them the 
\emph{Brink-Howlett generators} of $G$. See \cite{Deod}, \cite{BrHo99} and Proposition \ref{BHgenprop}. 
We frequently conflate a singleton set $\set{x}$ with $x$  when convenient for notational compactness, so if 
$K\sm J=\{\alpha\}$, may write $\nu(\alpha,J):=\nu(\{\alpha\},J)$.

In fact, $G$ is generated by its Brink-Howlett generators \cite{Deod}, subject to   relations of the following two types,  
 called \emph{rank one relations} and \emph{rank two relations} in \cite{BrHo99}.  

For any Brink-Howlett generator $(J',\nu(K\sm J,J),J)$ of $G$  as in Lemma \ref{Deodgen}(c), there is  a rank one  relation $(J',\nu(K\sm J,J),J)^{-1} =(J,\nu(K\sm J',J'),J')$.  

There is a rank two relation  for each 
morphism $(J',w,J)$ of $G$, with $w=\nu(K\sm J,J)$ where $\vert K\sm J\vert=2$ (or equivalently, $\vert K\sm J'\vert=2$).  Such a morphism has exactly two expressions as a minimal 
length product of Brink-Howlett generators,  one with $\nu(\alpha,K)$ as rightmost factor and the other with $\nu(\beta,K)$ as rightmost factor, where $K\sm J=\set{\alpha,
\beta}$;  the corresponding relation expresses the equality of those two expressions.    See \cite{BrHo99}.

\subsection{Maximal corank $k$ overgroups of  reflection subgroups} \label{refsg} Roots of the signed groupoid-set in 
Theorem \ref{standpreprinc}(b) will be seen to correspond to certain reflection
subgroups of $W$. We describe below salient properties  of reflection subgroups and their root subsystems.
 
 A \emph{reflection subgroup} of $W$ is by definition a subgroup 
  $U=\mpair{U\cap T}$ of $W$ which is generated by the 
  reflections it contains. Let $U$ be a  reflection subgroup  of $W$. Then $U$ 
  has a set of  \emph{canonical Coxeter generators} (depending on $(W,S)$) 
 \begin{equation}
 \chi(U)=\chi_{(W,S)}(U):=\{\,t\in T\mid N(t)\cap u=\{t\}\,\},
 \end{equation} for which the set of 
 reflections in $U$ is $U\cap T$.  Further, $(U,\chi(U))$ has a based root system\footnote{The technical utility of such functoriality properties of based root systems is one  reason for using them as the framework for this paper rather than just considering standard based root systems. If $(\Phi,\Pi)$ is a standard based root system, $(\Phi_{U},\Pi_{U})$ is not in general standard.}
 $(\Phi_{U},\Pi_{U})$ with positive roots $\Phi_{U}^{+}$ in $(V,B)$ where $\Phi_{U}=\mset{\alpha\in \Phi\mid s_{\alpha}\in U}$,
 $\Phi^{+}_{U}:=\Phi_{U}\cap \Phi^{+}$ and $\Pi_{U}:=\mset{\alpha\in \Phi^{+}\mid s_{\alpha}\in \chi(U)}$. In particular, if $U=W_{J}$ is the standard parabolic subgroup generated by reflections in a subset $J$ of simple roots, then $\Pi_{U}=J$.

  Define the \emph{corank} $\corank_{W}(U)$ of   $U$ in  $(W,S)$ to be the minimum 
 cardinality of  a subset $T' \seq W\cap T$ such that 
 $W=\mpair{(U\cap T)\cup T'}$. For $J\seq \Pi$, one has in particular $\Pi_{W_{J}}=J$ and   $\corank_{W}(wW_{J}w^{-1})=\vert \Pi\sm J\vert$ for any $w\in W$.  See Proposition \ref{corank}.
 
 We extend the definition of notions defined for Coxeter systems and their based root systems  to arbitrary reflection subgroups $U$ of $(W,S)$ by regarding $U$ as the Coxeter group underlying the  Coxeter system $(U,\chi(U))$, for which one 
 uses   the  based root system $(\Phi_{U},\Pi_{U})$.   In particular, this defines $\corank_{U}(U')$  for any reflection subgroups $U'$ and $U$ of $W$ with $U'\seq U$; this corank is the minimum of the cardinalities of subsets $T'$ of $U\cap T$ such that $U=\mpair{(U'\cap T)\cup T'}$.

\subsection{} For any subgroup $G$ of $W$, we say that a subgroup $H$ of $W$ with $H\sreq G$ is an \emph{overgroup} of $G$ (in $W$). We say that $H$ is an overgroup of $G$ of \emph{index}\footnote{The term co-index may be preferred for consistency with corank. }  $[H:G]$ (the index of $G$ in $H$). If $H$ and $G$ are reflection subgroups of $W$, we say that  $H$ is  a corank $n$ reflection overgroup of $G$ where $n:=\corank_{H}(G)$.

If $U$ is a reflection subgroup of $W$ and  $W_{J}\seq U$, then 
$J\seq \Pi_{U}$; in that case, $\corank_{U}(W_{J})= \vert \Pi_{U}\sm J\vert$.  If the corank is equal to $k$, we say, according to the above,  that $U$ is a \emph{corank  $k$ reflection overgroup} of $W_{J}$ (in $W$).  
For $k\in \mathbb{N}$, a \emph{maximal corank $k$  reflection overgroup} of $W_{J}$ is by definition an inclusion maximal element of the inclusion-ordered poset of corank $k$ reflection overgroups of $W_{J}$.  

The following refinement  of results in \cite{Dy21} (which concerned ranks of reflection subgroups) is important in this paper. It is proved and slightly refined in  \ref{maxrankk}--\ref{combmaxref}.
 \begin{thm}\label{maxref} Let $J\seq \Pi$  and  $U$ be a reflection overgroup of $W_{J}$.    Assume that  $\corank_{U}(W_{J})=k\in \mathbb{N}$.      \begin{num}\item If $U$ is a parabolic subgroup of $W$,  then $U$ is  a maximal   corank $k$ reflection overgroup of $W_{J}$. 
  \item  Assume  that  $\Pi_{U}$ is linearly independent.
 Then the set of all corank $k$ reflection overgroups of $W_{J}$ which contain $U$  has a maximum element. In particular,  $U$ is a reflection subgroup of a unique   maximal corank $k$ reflection overgroup of $W_{J}$. \item Assume  that  $\Pi_{U}$ is linearly independent
and that   $\Phi_{U}=\Phi\cap \Span(\Phi_{U})$.    
Then $U$ is a maximal corank $k$ reflection overgroup of $W_{J}$. Let $I$ be the set of all maximal, corank $k+1$ reflection overgroups of $W_{J}$ which contain $U$.
 Then  $\Phi\setminus \Phi_{U}=\dot\bigcup_{W'\in I}(\Phi_{W'}\setminus\Phi_{U})$.    \end{num}
     \end{thm}

\subsection{Roots and reflection subgroups} 
Note that reflections of $(W,S)$ correspond bijectively to rank one reflection overgroups in $W$ of the trivial group.
Theorem \ref{BHabrs} below   generalizes     to a  Brink-Howlett groupoid $G$ the description in \ref{abrs}  of the abstract root system of $(W,S)$ in terms of reflections, and is  a main result of this work. Sets of corank one reflection overgroups of the standard parabolic subgroup $W_{J}$ corresponding to each object $J$ of $G$ play an analogous role for $G$ as  the set of reflections  plays  for $(W,S)$ .  We first define  relevant sets of reflection subgroups. See Example \ref{gpdrootex}, and Propositions \ref{refsubgptypes} and \ref{naivers} for some basic facts concerning them.
   
\subsection{}\label{refsubtype} For $J\seq \Pi$,  define $R_{J}$ to   be the set of all corank one reflection overgroups of $W_{J}$. Let $N_{J}$ denote the set of all $U\in R_{J}$ such that $(U\sm W_{J})\cap T$ is finite. 
Let $M_{J}$ denote the set of all maximal corank one 
reflection overgroups of $W_{J}$. Let $P_{J}$ denote 
the set of 
all corank one reflection overgroups of $W_{J}$ which are parabolic subgroups of $W$.  Let $Q_{J}=P_{J}\cap N_{J}$.
 
For any $U\in R_{J}$, we have $J\seq \Pi_{U}$ and $\vert \Pi_{U}\sm J\vert =1$, so we may  write
\begin{equation}
\Pi_{U}=J\dot\cup\set{r_{U,J}}, \qquad r_{U,J}\in \Phi^{+}.
\end{equation}

 The following Theorem \ref{BHabrs}, which is our main result on abstract root systems of Brink Howlett groupoids,  is proved in Section \ref{s:abrs} using 
Lemma \ref{Deodgen}, Theorem \ref{maxref},  Proposition \ref{finindcond}  and  Theorem \ref{infTits}.

\begin{thm}\label{BHabrs} Let $G$ be a Brink-Howlett groupoid 
for $(W,S)$.  Let $X$ denote either $R$, $M$, $N$, $P$ or $Q$. 
 Then there is a signed groupoid-set $(G,\Upsilon_{X})$  as 
 follows. 
For an object $J\seq \Pi$ of $G$, 
$\Upsilon_{X}(J):=X_{J}\times\{\pm\}$ where the signed set 
$X_{J}\times\{\pm\}$ has $(X_{J}\times\{\pm\})^{+}=X_{J}\times\{+\}$ with   
 $\{\pm \}$-action $\epsilon(U,\eta)=(U,\epsilon\eta)$ for $U\in X_{J}$ and $\epsilon, \eta\in \set{\pm}$. For each morphism  $(K,w,J)$ of $G$, the  map 
$\Upsilon_{X}(K,w,J)\colon X_{J}\times\{\pm\}\to X_{K}\times\{\pm\}$ is 
\begin{equation*}
(U,\epsilon)\mapsto (wUw^{-1},(-1)^{\eta (w,U,J)}\epsilon),\qquad U\in X_{J},\quad \epsilon\in \set{\pm} 
\end{equation*}
where $\eta (w,U,J)\in \mathbb{Z}$ is equal to $1$ if $(U\setminus W_{J})\cap T\seq N(w^{-1})$ and to $0$ otherwise (in which case $(U\setminus W_{J})\cap N(w^{-1})=\eset$). Further:
\begin{num}
\item Let $(U,\epsilon)\in{\Upsilon}_{X}(J)$. Then $(U,\epsilon)\in ({\Upsilon}_{X}(J))^{\mathrm{re}}$ if and only if $\vert \Phi^{+}_{U}\sm \Phi_{J}^{+}\vert<\infty$ if and only if $\vert(U\setminus W_{J})\cap T\vert <\infty$.
\item  $(G,\Upsilon_{Q})\seq (G,\Upsilon_{P})\seq (G,\Upsilon_{M})\seq (G,\Upsilon_{R})$ and $(G,\Upsilon_{Q})\seq(G,\Upsilon_{N})\seq (G,\Upsilon_{R})$.
\item  $(G,\Upsilon_{R}^{\re})=(G,\Upsilon_{N})$ and  $(G,\Upsilon_{M}^{\re})=(G,\Upsilon_{P}^{\re})=(G,\Upsilon_{Q})$.
\item There are isomorphisms 
$ (G,\Lambda^{\rec})\cong (G,\Upsilon_{X}^{\rec})\cong (G,\Upsilon_{Q})$.  
\end{num} \end{thm}

\subsection{}  For notational compactness, we may write $\lrsub{J}{\Upsilon}{X}:=\Upsilon_{X}(J)$, and sometimes omit $X$ if $X=R$, writing for instance ${\Upsilon}:={\Upsilon}_{R}$ and   
${\lsub{J}{\Upsilon}}:=\lrsub{J}{\Upsilon}{R}$. 

We next discuss additional facts which are of interest in relation Theorem \ref{BHabrs} or are involved in its proof.
 The first of these, as already mentioned, is essentially  due to  Howlett \cite{How} and  Deodhar \cite{Deod}. It provides several equivalent  descriptions  of real roots in the root systems  in Theorem \ref{BHabrs}. \begin{prop}\label{finindcond} Let $J\seq \Pi$, and $U$ be a reflection overgroup of $W_{J}$, so $J\seq \Pi_{U}$. Then the following conditions are equivalent:
 \begin{conds}
 \item $(U\sm W_{J})\cap T$ is finite
 \item  $\Phi_{U}\sm \Phi_{J}$ is finite.
 \item   $\corank_{U}(W_{J})=\vert \Pi_{U}\sm J\vert$ is finite and for each $\alpha\in 
\Pi_{U}\sm J$, the component of  $\Pi_{U}$  
containing $\alpha$ is of finite type.
\item   the index $[U:W_{J}]$ is finite.
 \end{conds}
 \end{prop}
 The next result follows immediately from Theorem \ref{BHabrs} and Proposition \ref{finindcond}. 
 \begin{cor} The set $\Lambda^{\rec}(J)^{+}$ of positive real roots based  at an object $J$ of the standard signed groupoid-set $(G,\Lambda^{\rec})$ of a Brink-Howlett groupoid $G$ is in natural bijective correspondence  with the set of finite index, corank one, parabolic overgroups  of $W_{J}$ in $W$.   
\end{cor}
 
\subsection{}  When the conditions in Proposition 
\ref{finindcond} hold,  the element $\nu(\Pi_{U}\sm J, J)\in U$ is 
defined in the Coxeter system $(U,\chi(U))$. Natural examples 
of situations to which that proposition  applies, with $U$ 
standard parabolic in $W$, are given by  the next proposition, 
which  is  proved  as Proposition 
\ref{rootorbit}.  
 
 \begin{prop}
\label{exrootorbit} Let $\alpha\in \Phi^{+}$. Fix $K\seq \Pi$ with $B(\alpha,\beta)\geq 0$ for all $\beta\in K$ and let $J:=\mset{\beta\in K\mid B(\alpha,\beta)=0}$. Then $[W_{K}:W_{J}]<\infty$.\end{prop}
\subsection{Special standard parabolic root subsystems}\label{BHred} In the following, we regard $\Pi$ as the vertex set of the Coxeter graph of $(W,S)$. By a component of a subset $J$ of $\Pi$, we mean the vertex set of a component of the full subgraph  the Coxeter graph on vertex set $J$.   Let $J_{\fin} $ be the union of all finite type components of $J$ and  $J_{\inft}$ be the union of  all other   components (i.e. all finite components which are of affine or indefinite type, and all components which are infinite as sets) of  $J$. Thus, $J=J_{\fin} \,\dot\cup\, J_{\inft}$.

In order to study  a connected Brink-Howlett groupoid of $(W,S)$, Brink and Howlett replaced the groupoid by an isomorphic Brink-Howlett groupoid of a suitable standard parabolic subgroup $W'$ of $W$, to reduce  to the  case in which each object $J$  satisfies $J_{\inf}=\eset$ (that is,  $J=J_{\fin} $), and so that no root in $\Pi\sm J$   is joined in the Coxeter graph of $(W,S)$ to infinitely many elements of $J$.  To obtain particularly well-behaved realizations in real vector spaces  of root systems of Brink-Howlett groupoids, it is necessary  to  refine this reduction by restricting the $W'$ action from $V$ to a suitable subspace of $V$ which contains $\Phi_{W'}$.    
This reduction and its effect  on real root systems  is described in Propositions \ref{redlem} and \ref{rootembed}. The proofs of those results uses Proposition \ref{infred}, which is deduced from Proposition \ref{exrootorbit}.

\subsection{} A \emph{standard parabolic} root 
subsystem of $\Phi$ is by definition the root system $\Phi_{U}$ of a standard 
parabolic subgroup $U=W_{J}$, for some $J\seq \Pi$.

 \begin{prop}\label{infred} Let $L\seq \Pi$. \begin{num}
\item If $L=L_{\inft}$,  then $\Phi\cap L^{\perp}$ is a standard parabolic root subsystem of $\Phi$.
\item $\mset{\beta\in \Phi\mid \text{\rm  $B(\beta,\alpha)\neq 0$ for only finitely many $\alpha\in L$}}$  is a standard parabolic root subsystem of $\Phi$.  
\end{num}
\end{prop}

\subsection{Relative Tits theorem}  The following theorem, proved in Section \ref{sec:2}, is another of our main results. Its special case in which $J=\eset$ implies the theorem of Tits (\cite{Bou})  that any finite subgroup of $W$ is contained in a finite parabolic subgroup. The special case when the  corank is  one in the final statement of the theorem is particularly important for our applications.

    \begin{thm}\label{infTits} Let $J\seq \Pi$ and $W'$ be an overgroup in $W$ of $W_{J}$. Then the following conditions are equivalent:
\begin{conds}
\item The index $[W':W_{J}]$ of $W_{J}$ in $W'$ is finite.
\item There is a parabolic subgroup $U\sreq W'$ of $W$ such that $[U:W_{J}]$ is finite.
\item There exist $w\in W$ and $K\seq \Pi$ so
$L:=wJ\seq K$,  $[W_{K}:W_{L}]$ is finite and 
 $W_{L}=wW_{J}w^{-1}\seq wW'w^{-1}\seq W_{K}$.    
\end{conds}
If (i)--(iii) hold, then there is a unique inclusion-minimal element $U_{0}$ of the set of parabolic subgroups $U$ which satisfy (ii). 
 If, further, $W'$ is a reflection subgroup,
then  $\corank_{U_{0}}(W_{J})=\corank_{W'}(W_{J})$.
\end{thm}

\subsection{Realized root systems}
 By a \emph{realization} of a signed groupoid-set $(G,\Lambda)$, we mean a pair $(\mcv,\iota)$ consisting of  a representation $\mcv  $ of $G$ in the category of real vector spaces together with a natural transformation $\iota\colon \Lambda\to \mcv  $ (from the functor $G\to \mathbf{Set}$
underlying $\Lambda$ to that underlying $\mcv$) such that the components $\iota_{a}\colon \Lambda(a)\to \mcv  (a)$ of $\iota$ are injective $\set{\pm}$-equivariant maps, for all $a\in \ob(G)$.  

The  realizations of $(G,\Lambda)$ correspond naturally to families of realizations of its components. We shall often write the real vector space $\mcv(a)$ as $\lsub{a}{\mcv}$.

To emphasize the distinction between them, we call   $(G,\Lambda)$ an \emph{abstract} signed groupoid-set, and   $(G,\Lambda,\mcv  ,\iota)$ a \emph{realized} signed groupoid-set.
More informally, we call $\Lambda$  an \emph{abstract root system} of $G$ and $(\Lambda,\mcv  ,\iota)$  a \emph{realized root system} of $G$.  This extends the terminology for root systems of Coxeter groups.

\begin{ex} \label{Coxrootoid2} Consider the signed $W_{\bullet}$-set  $(W_{\bullet}, \Phi_{\bullet})$ in Example \ref{Coxrootoid}.
There is  a representation $\mcv  _{\bullet}$ of $W_{\bullet}$ in the category of real vector spaces,    
sending $\bullet$ to be ambient real vector space $V$ of $\Phi$, and each morphism $w$ of $W_{\bullet}$ to the linear map $V\to V$ induced by the action of $w$.
There is a realization $(\mcv_{\bullet},\iota)$ of $(W_{\bullet}, \Phi_{\bullet})$ where   $\iota\colon \Phi_{\bullet}\to \mcv  _{\bullet}$ is the natural transformation  whose unique component $\iota_{\bullet}\colon \Phi\to V$ is the inclusion map.  \end{ex}

\subsection{Weak realizations} Define the notion of a \emph{weak realization} $(\mcv,\iota)$ of a signed $G$-set  $(G,\Lambda)$ in  the same way as a realization, except omitting the requirement that the components $\iota_{a}$ be injective maps. Several natural constructions yield in general only weak realizations of signed groupoid-sets, and one could extend the framework of this paper  to deal with such situations more effectively by incorporating structure of ``root multiplicity'' into parts of  the development
(see \ref{mult}).

Consider  a weak realization $(\mcv,\iota)$ of $(G,\Lambda)$. Assume that for all $a\in \ob(G)$, one has \begin{equation}
\label{weakcond}
\iota_{a}(\alpha)\neq -\iota_{a}(\beta), \qquad  \alpha,\beta\in \Lambda(a)^{+}.
\end{equation}

Define the set $\wt \Lambda(a):=\iota_{a}(\Lambda(a))\seq 
\mcv(a)$. This becomes  a signed set with $\set{\pm}$ action 
induced by that on $V$, and  $\wt \Lambda(a)^{+}:=\iota_{a}
(\Lambda(a)^{+})$ as the set of positive elements.
The map $i_{a}\colon \Lambda(a)\to \wt\Lambda(a)$ given by  
 restriction of  $\iota_{a}$ is a morphism of signed sets.
 Let $j_{a}\colon \wt\Lambda(a)\to \mcv(a)$ denote the inclusion map, which is $\set{\pm}$-equivariant.

For a morphism $g\colon a\to b$ in $G$, define 
$\wt\Lambda(g)\colon \wt\Lambda(a)\to \wt\Lambda(b)$ to be 
the map obtained by restriction of $\mcv(g)\colon \mcv(a)\to 
\mcv(b)$. Then $\wt\Lambda(g)$  is also  a morphism of signed 
sets. This defines a functor $\wt\Lambda\colon G\to \Set_{\pm}
$, making $(G,\wt\Lambda)$ a signed $G$-set.  
There is a natural transformation $i\colon \Lambda\to \wt 
\Lambda$ with component $i_{a}$  at $a$ for each $a\in \ob(G)
$. It gives a morphism $i\colon (G,\Lambda)\to (G,\wt\Lambda)$ 
of signed $G$-sets, which is an isomorphism if  $(\mcv,\iota)$ is 
a realization.

There is a natural transformation $j\colon
\wt\Lambda\to \mcv$ (of set-valued functors) with component $j_{a}\colon \wt\Lambda(a)\to \mcv(a)$ at $a\in \ob(G)$.  It is easily seen that $(G,\wt\Lambda,\mcv,j)$ is a realized signed groupoid set  which we call
the \emph{associated realized signed $G$-set} of the weakly  realized $G$-set $(G,\Lambda,\mcv,\iota)$. 
  
\begin{remark*}The examples of (weak) realizations of signed groupoid-sets in which we are most interested satisfy  convexity properties (to be studied in another work)  which  imply the condition \eqref{weakcond}. We do not impose corresponding  conditions in the definition here, but simply observe that they hold in the  examples considered. \end{remark*}

\subsection{Realized root systems  of Brink-Howlett groupoids}  We now describe construction of (weak) realizations of   abstract root systems of Brink-Howlett groupoids considered  in Theorem \ref{BHabrs}.  The constructed realizations will generally depend, even up to the natural notion of isomorphism, not just on $(W,S)$ but on  the choice of  a realized root system $\Phi$ of $(W,S)$.  We fix such a choice.
 
  Recall that for  $J\seq \Pi$ and  any  corank one reflection overgroup $U\in R_{J}$ of $W_{J}$, we define $r_{U,J}\in \Phi^{+}\sm J$ by 
 \begin{equation}
\label{simpcorankone}\Pi_{U}=  J\dot\cup \{r_{U,J}\}.
 \end{equation}

 Our approach to constructing  realizations of signed groupoid sets $(G,\Upsilon_{X})$ as in Theorem \ref{BHabrs}  is to represent an abstract root $(U,\epsilon)\in \Upsilon_{X}(J)$, where $J\in \ob(G)$, by $\epsilon r_{U,J}\in V$. This element of $V$, together with $J$, determines $(U,\epsilon)$ uniquely by \eqref{simpcorankone}. However, Proposition \ref{naivers} implies that this does not in general give  a realization unless one replaces $V$ by a suitable  quotient vector space (depending on $J$ in a manner compatible with the $G$-action) 
 and  replaces $\epsilon r_{U,J}$ by its image in the quotient. 
 It is then possible that the images in the quotient space   of $r_{U,J}$ and $r_{U',J}$, where $U,U'\in X_{J}$,  could coincide even if $U\neq U'$. In general,  one   only obtains  weak realizations in this way (see Example \ref{onlyweak}),
but surprisingly (see Remark \ref{realcor}(2)),  these induce a realization of the realification $(G,\Upsilon_{X}^{\re})$.
  
 We introduce notation for Theorem \ref{weakreal}, which  records basic properties of these  weak realizations.  
 
   \subsection{} Fix notation as in Theorem \ref{BHabrs}, so
  $X$ below denotes $R$, $M$, $P$, $N$ or $Q$.

  Let $J$ be an object of $G$. Define the real vector space
   $\mcv  (J):=V/
  \Span(J_{\fin})$. Let $\pi_{J}\colon V\to V/
  \Span(J_{\fin})$ denote the canonical epimorphism $v\mapsto v+\Span(J_{\fin})$.  Define $\lsub{J}{\Delta}:=\pi_{J}(\Pi\sm J_{\fin})$ 
and a map $\iota_{J}\colon
 {\lsub{J}{\Upsilon}}\to \mcv  (J)$ by $\iota_{J}(U,\epsilon)=
 \epsilon \pi_{J}(r_{U,J})$.
  
    A morphism $(K,w,J)$ of $G$ satisfies
  $w(J)=K$ and therefore $w(J_{\fin})=K_{\fin}$.
    Let $  \mcv  (K,w,J)\colon \mcv  (J)\to \mcv  (K) $ denote the linear map given by
  $\pi_{J}(v) \mapsto \pi_{K}(wv)$ for all $v\in V$. This define a functor $\mcv  $ from $G$ to the category of real vector spaces. There is a  natural transformation
 $\iota\colon \Upsilon_{X}\to \mcv  $ (of the underlying functors functors from $G$ to the category of sets) which has component $\iota_{J}$ at $J \in \ob(G)$.  
 
We maintain the above notation and assumptions till further notice.  The following Theorem \ref{weakreal} is another main result of this paper, proved in Section \ref{sec:5}.

 \begin{thm} \label{weakreal}   \begin{num} 
  \item   $F:=(G,\Upsilon_{X},\mcv  ,\iota)$ is a  weakly realized signed groupoid-set. 
\item $F$ is a realized signed groupoid-set if $\Pi$ is linearly independent and $X=P$.
\item If $\vert J\vert =1$ for all $J\in \ob(G)$, then $F$ is a realized signed groupoid-set.
\item  If $X$ denotes  $N$ or $Q$, then  $F$ is a  realized signed groupoid-set.  
\end{num}  
 \end{thm}

 \begin{remarks*} (1)   The proof of  part (d)   depends on an extension  
 \cite{DyLe18}  to finite Coxeter groups of a 
 lemma due to Oshima \cite{Osh} for finite Weyl groups.

 (2) Theorem \ref{weakreal}(c)--(d) do not hold in general for 
 non-simply-laced finite Weyl groups $W$ (even, say, of type 
 $B_{2}$) if one uses  the standard crystallographic root system 
 of $W$ instead of the based root system as in this paper (with 
 all roots of unit square length); see Remark \ref{threeroots}.
 
 (3)  Theorem \ref{weakreal}(b) is false in general without the assumption there  that $\Pi$ is linearly independent (see Example   \ref{onlyweak}).
 
 (4) Computations based on Theorem \ref{weakreal}(c), in the special case in which $(W,S)$ is of rank three, were an important  ingredient in the (unpublished) first  proof of the main result of 
 \cite{Dy21}.  \end{remarks*}

   \begin{prop} \label{realizedroot}  Let $J\in \ob(G)$.  \begin{num}
   \item For each    $(U,\epsilon)\in \Upsilon_{X}(J)$, we have $\iota_{J}(U,\epsilon)=\epsilon \pi_{J}(r_{U,J})\neq 0$.
   \item  $\pi_{J}(\cone (\Pi))= \pi_{J}(\cone (\Pi\sm J_{\fin}))= \cone(\lsub{J}{\Delta})$.
  \item  $ \iota_{J}(\tensor*{\Upsilon}{^{+}_{X}}(J))
\seq  \cone( \lsub{J}{\Delta})$. 
\item $\cone( \lsub{J}{\Delta})\cap -\cone (\lsub{J}{\Delta})=\{0\}$.
\item Suppose that  $\alpha\in\Pi\sm J$ and $W_{J\cup\set{\alpha}}\in Q_{J}$. Then $\pi_{J}(\alpha)\in \lsub{J}{\Delta}$ and $ \pi_{J}(\alpha)\not\in \cone(\lsub{J}{\Delta}\sm\{\pi_{J}(\alpha)\})$.
\item If $\Pi$ is linearly independent, then $\lsub{J}{\Delta}$ is linearly independent.
\item If $\Pi$ is a basis of $V$,  then $\lsub{J}{\Delta}$ is a basis of $\mathcal{V}(J)$.\end{num}
\end{prop} 
 \begin{cor}\label{realcor}\begin{num}
 \item $(G,\Upsilon_{X}^{\re})$ has  a realization.
  \item  The standard signed $G$-set $(G,\Lambda^{\rec})$ of $G$ has  a realization.
 \item The functor $\mcv$ (from the groupoid $G$ to the category of real vector spaces) is faithful.
 \item  The  weakly realized signed groupoid set $(G,\Upsilon_{X},\mcv  ,\iota)$ satisfies the condition \eqref{weakcond}, so has an associated realized signed groupoid set $(G,\wt\Upsilon_{X},\mcv  ,j)$. 
\end{num} \end{cor}

 \subsection{}\label{mult} The preceding proposition and corollary show that the set $ \tensor*{\wt\Upsilon}{^{+}_{X}}(J))$
 of \emph{realized positive roots}  based at $J$
 has convexity properties which are partly analogous to those of positive root systems of Coxeter groups, with $\lsub{J}{\Delta}$  playing a  role (at the object $J$) similar to that  of $\Pi$ in $\Phi$.   We do not have    $\lsub{J}{\Delta}\seq  \tensor*{\wt\Upsilon}{^{+}_{X}}(J))$ in general. However,  $ \tensor*{\wt\Upsilon}{^{+}_{X}}(J)$ does  contain
 $\mset{\pi_{J}(\alpha)\mid \alpha\in \Pi\sm J, 
 W_{J\cup\set{\alpha}}\in X_{J}\cap P_{J}}\seq\lsub{J}{\Delta}$ 
 which we call the set of \emph{realized simple  roots} based at 
 $J$.   If $X_{J}\sreq P_{J}$, then the set of realized simple 
 roots
 based at $J$ is $\pi_{J}(\Pi\sm J)$, which is equal to $\lsub{J}
 {\Delta}$ if $J=J_{\fin}$. 
 
The Brink-Howlett generators of $G$ with codomain $J$ 
correspond to the  subset $\mset{\pi_{J}(\alpha)\mid \alpha\in 
\Pi\sm J, W_{J\cup\set{\alpha}}\in Q_{J}}\seq \lsub{J}{\Delta}$ of 
realized simple roots, which we refer to  as the set of \emph{real realized simple roots} based at $J$. The other (i.e. not real) realized simple roots at $J$ will be be called \emph{imaginary realized simple roots} at $J$.   
 
  Note that, especially in the case $X$ is $R$ or $N$,  the realized root system ${\lsub{J}{\wt\Upsilon}}$ based at $J$ may  be \emph{non-reduced}, in the sense that  multiples $c\alpha$, where $c\in \mathbb{R}_{>0}$,  $c\neq 1$, of  realized (even simple) roots $\alpha$ may also be realized roots. See Proposition \ref{prop8.5}.

\begin{remark*}  In case $F$ in Theorem \ref{weakreal} is only a weakly realized signed 
 groupoid-set, it may be useful to regard each  realized root 
 $\alpha\in\iota_{J}(\tensor*{\Upsilon}{_{X}}(J)))$ as having an 
 associated multiplicity, defined as the cardinal number 
 $\vert\{\, (U,\epsilon)\in \Upsilon_{X}(J)\mid
  \iota_{J}(U,\epsilon)=\alpha\,\}\vert$ of its fiber for the map 
  $\iota_{J}$. Our results imply that real roots for $G$ all have multiplicity equal to $1$ in this sense (provided that one uses based root systems with all
  roots of unit square length).  However, we do not develop this idea further.
  \end{remark*}

\subsection{Bilinear form} \label{bf}
Let $L\seq \Pi$ be such that all components of $L$ are of finite type. Then the restriction of $B$ to $\Span(L)$ is positive definite and  the subspaces $\Span(L)$ and $ L^{\perp}$ of $V$ are $B$-orthogonal with trivial intersection. The direct sum of these two subspaces consists of all elements $v\in V$ such that $B(v,\alpha)\neq 0$ for only finitely many $\alpha\in L$ (see Section \ref{s:6}).

We consider the following conditions on a subset  $J$ of $\Pi$.
\begin{conds}
\item For all $\alpha\in \Pi$, we have $B(\alpha,\beta)\neq 0$ for only finitely many $\beta\in J_{\fin}$.
\item $J=J_{\fin}$ (or equivalently, $J_{\inf}=\eset$).
\item  $V=\Span(J_{\fin})\oplus (J_{\fin})^{\perp}$.
\end{conds}

From the preceding remarks,  for any $J\seq \Pi$, (iii) implies (i).

\subsection{}\label{form} Assume that \ref{bf}(iii) holds for all objects $J$ of $G$. In the weakly realized signed groupoid-set   
 $(G,\Upsilon_{X},\mcv  ,\iota)$ of Theorem \ref{weakreal},
 we canonically  identify, for each object $J$ of $G$,
 \begin{equation}
 \mcv  (J)=V/\Span(J_{\fin})= ((J_{\fin})^{\perp}\oplus\Span(J_{\fin}))/\Span(J_{\fin})=(J_{\fin})^{\perp}.
 \end{equation}
 Hence there is  a symmetric bilinear form $B_{J}\colon 
  \mcv  (J)\times  \mcv  (J)\to \mathbb{R}$ arising by restriction of $B\colon V\times V\to \mathbb{R}$. 
 
 The canonical map $\pi_{J}\colon V\mapsto V/\Span(J_{\fin})$, given by
  $v\mapsto v+\Span(J_{\fin})$, identifies with the orthogonal projection $V\mapsto (J_{\fin})^{\perp}$ determined by \ref{bf}(iii). Moreover, for a  morphism
  $w\colon K\to L$ in $G$, determined by $w\in W$ with $w(K)=L$, the map $\mcv  (w\colon K\to L)\colon
  (K_{\fin})^{\perp}\to (L_{\fin})^{\perp}$ is identified with 
  the restriction of $w\colon V\to V$.  
  Denoting this map simply as $w$, it respects the bilinear forms, in the sense that
  \[B_{L}(wv,wv')=B_{K}(v,v'), \qquad v,v'\in \mcv  (K),\]
  since $W$ acts by orthogonal transformations on $(V,B)$.

\begin{assumption} \label{assume}
  For the rest of this section, we assume   that for all objects $J$ of $G$, conditions \ref{bf}(i)--(iii)  hold.  We also assume, mainly for simplicity and  uniformity in exposition,  that $X$ is equal to $R$, $M$ or $P$, so that $X_{J}\sreq P_{J}$.
 \end{assumption}
 
  \begin{remark*}   Conditions \ref{bf}(i)--(iii) slightly refine Brink-Howlett's reduction described in  \ref{BHred}.   Though it is possible to prove more technical  analogues of some of the following results assuming just \ref{bf}(iii),  we do not go into details. 
   Similarly, using Theorem \ref{BHabrs}(b) and Proposition \ref{realizedroot},   results analogous to some of  the following  can also be formulated for $X=N$ and $X=Q$. 
   \end{remark*}

\subsection{} \label{asscon}  The consequences of the assumptions \ref{assume} described in   \ref{asscon}--\ref{BHact} are proved in Section \ref{s:6}. Let $J$ in $\ob(G)$. Then $J=J_{\fin}$. We identify $\mathcal{V}(J)=J^{\perp}$ and also identify  $\pi_{J}$ with the orthogonal projection $V\to J^{\perp}$.  
Since $X_{J}\sreq P_{J}$,  $\lsub{J}{\Delta}$ is the set of realized simple roots of $\Upsilon_{X}$ at $J$.

Somewhat more generally, for any subset $L$ of $\Pi$ such that all components of $L$ are of finite type and such that
$V=\Span(L)\oplus L^{\perp}$, we let $\pi_{L}\colon V\to L^{\perp}$ denote the $B$-orthogonal projection arising  from this direct sum. We may  sometimes write $L^{\perp}$ as $\lsub{L}{V}$ or,
 if $L\in \ob(G)$,  as $\lsub{L}{\mathcal{V}}$.
 For notational compactness, we may write $v_{L}:=\pi_{L}(v)$  for $v\in V$. 

\subsection{} We next  introduce some notation to refer to the realized roots\footnote{We also  refer to these as realized roots of  the (possibly weakly) realized root system  $(G,\Upsilon_{X},\mcv  ,\iota)$.}  of the  realized root system $(G,\wt\Upsilon_{X},\mcv  ,j)$ where, we recall, we assume that $X$ denotes $P$, $M$ or $R$. In particular, by Corollary \ref{realcor}(d), $\wt\Upsilon_{X}$ has disjoint sets of positive and negative roots at each object $J$ of $G$, and disjoint sets of  real and imaginary roots at $J$, denoted as 
$\lsub{J}{\wt\Upsilon}_{X}^{+}$,  
$\lsub{J}{\wt\Upsilon}_{X}^{-}$
$\lsub{J}{\wt\Upsilon}_{X}^{\re}$ and    $\lsub{J}{\wt\Upsilon}_{X}^{\im}$,  as usual.

 Recall that for any object $J$ of $G$ and any $U\in R_{J}$, $r_{U,J}\in \Phi^{+}$ is defined by  $\Pi_{U}\sm J=\set{r_{U,J}}$.   Define $r'_{U,J}:=\pi_{J}(r_{U,J})\in J^{\perp}$.  The realized root in $\lrsub{J}{\wt\Upsilon}{X}$ corresponding to the abstract root $(U,\epsilon)\in \lrsub{J}{\Upsilon}{X}$, where $U\in X_{J}$, is $\epsilon r'_{U,J} $, for $\epsilon \in \set{\pm }$.

We have $\lsub{J}{\Delta}=\mset{\pi_{J}(\alpha)\mid\alpha\in \Pi\sm J}\seq \lsub{J}{\wt\Upsilon}_{X}$.
Define    \[\lsub{J}{\Delta}^{\mathrm{re}}:=
\mset{\pi_{J}(\alpha)\mid\alpha\in \Pi\sm J,  W_{J\cup\set{\alpha}}\in Q_{J}}=\lsub{J}{\Delta}\cap
{\lsub{J}{\wt\Upsilon}}_{R}^{\mathrm{re}}\] and  
\[\lsub{J}{\Delta}^{\mathrm{im}}:=
\mset{\pi_{J}(\alpha)\mid\alpha\in \Pi\sm J,  W_{J\cup\set{\alpha}}\in P_{J}\sm Q_{J}}= \lsub{J}{\Delta}\cap
{\lsub{J}{\wt\Upsilon}}_{R}^{\mathrm{im}}.\] We call $\lsub{J}{\Delta}$,  
$\lsub{J}{\Delta}^{\mathrm{re}}$ and $\lsub{J}{\Delta}^{\im   }$
the sets of simple roots, real simple roots and imaginary simple roots  of $\lsub{J}{\wt\Upsilon}_{X}$ respectively. Recall that if  $\Pi$ is linearly independent, then $\lsub{J}{\Delta}$ is linearly independent. 

The following proposition lists some basic
properties of  realized roots.   \begin{prop}\label{rootiprod} Let $J\in \ob(G)$ and $(U,\epsilon)\in \lrsub{J}{\Upsilon}{R}$.
\begin{num}
\item $\lsub{J}{\Delta}\seq \lsub{J}{\wt\Upsilon}^{+}_{X}\seq \cone(\lsub{J}{\Delta})$ where
$\cone(\lsub{J}{\Delta})\cap -\cone(\lsub{J}{\Delta})=\set{0}$.
\item If $\alpha,\beta\in \lsub{J}{\Delta}$, then $B_{J}(\alpha,\beta)>0$ if and only if $\alpha=\beta\in  \lsub{J}{\Delta}^{\mathrm{re}}$.
\item  $B_{J}(r'_{U,J} ,r'_{U,J} )\in \mathbb{R}$ is positive, zero or negative according as  the component of 
$\Pi_{U}$ containing $r_{U,J}$ is of 
finite, affine or indefinite type. 
\item  Let $\alpha\in {\lsub{J}{\wt\Upsilon}}_{X}$.  Then 
$\alpha\in {\lsub{J}{\wt\Upsilon}}_{X}^{\mathrm{re}}$   if and only if  
$B_{J}(\alpha,\alpha)>0$.
\item  ${\lsub{J}{\wt\Upsilon}}_{X}={\lsub{J}{\wt\Upsilon}}_{X}^{\mathrm{re}}\dot\cup 
{\lsub{J}{\wt\Upsilon}}_{X}^{\mathrm{im}}$ and $
\lsub{J}{\Delta}={\lsub{J}{\Delta}}^{\mathrm{re}}\dot\cup{\lsub{J}{\Delta}}^{\im   }$
\item If $\alpha\in \lsub{J}{\Delta}^{\mathrm{re}}$, then
$\alpha\not\in \cone(\lsub{J}{\Delta}\sm\set{\alpha})$. 
\end{num}
\end{prop}
\subsection{Longest elements}
The next proposition describes some aspects of an analogy between morphisms $\nu(K,J)\colon J\to J'$ of $G$ and longest elements of finite standard parabolic subgroups of Coxeter groups. 
 
 \begin{prop}\label{longelt} Let $J\in \ob(G)$.
 \begin{num}\item  For a  subset $K$ of $\Pi\sm J$, the element $\nu(K,J)$ of $W$ exists if and only if $K$ is finite and the restriction of $B_{J}$ to 
$\Span(\pi_{J}(K))$ is positive definite. \end{num}
 Assume in (b)--(f) that $w:=\nu(K,J)$ exists.
\begin{num} \item[(b)]  $\pi_{J}(K)=\mset{\alpha_{J}\mid \alpha\in K}\seq \lsub{J}{\Delta}^{\mathrm{re}}$ is linearly independent.
\item[(c)] 
$w^{-1}$ is the join in weak (right) order $\leq$ on $W$ of 
$\mset{\nu(\alpha,J)^{-1}\mid \alpha\in K}$.
\item[(d)]  Let $J':=wJ\seq K\cup J$  and $K':=(K\cup J)\sm J'$. Then
$ (J',w,J)$ is a morphism in $G$ with inverse
 $(J,w^{-1},J')$.  
 \item[(e)] The union $M$ of all components of $K\cup J$ which
  meet  $K$  is spherical and the   map 
  $ \alpha\mapsto -w_{M}\alpha$ is a bijection 
  $\sigma\colon K\to K'$. 
   \item[(f)] $\mathcal{V}(J',w,J)$ restricts to a bijection
 $\alpha_{J}\mapsto -(\sigma \alpha)_{J'}\colon 
 \pi_{J}(K)\to -\pi_{J'}(K')$.
 \end{num}    
\end{prop}
\begin{remark}  We say that  a  Brink-Howlett relation has  
domain $J$ if it expresses the equality of two specified products 
of Brink-Howlett generators (and inverses), where both products  
have domain $J$. For fixed $J$, the sets $K$ in the proposition
   such that $K$ has cardinality one (resp., two) and $\nu(K,J)$ 
   exists are in natural bijective correspondence with   the 
   Brink-Howlett generators with domain $J$ and also with the 
   rank one Brink-Howlett relations with domain $J$  
   (resp., with the rank two Brink-Howlett relations with domain 
   $J$).  
\end{remark} 
\subsection{Reflection-like  action of Brink-Howlett generators}    We give two formulae for the action of a Brink-Howlett generator, one showing it acts as a composite of an  orthogonal reflection in a corresponding simple root with a linear isometry, and the other  describing the action in a basis of simple roots.
\begin{prop}\label{BHact} Let $J\in \ob(G)$ and  $\alpha\in \Pi\sm J$ be such  that 
$w:=\nu(\alpha,J)$ exists. Let $J':=\nu(J,\alpha)J\seq J\cup\set{\alpha}$ and $\alpha'\in \Pi$ be such that $J\cup\set{\alpha}=J'\cup\set{\alpha'}$. Consider the Brink-Howlett generator  $g=(J',w,J)\colon J\to J'$ in $G$.  Then:
\begin{num} 
\item   $\mathcal{V}(g)=\tau_{g}s_{\alpha_{J}}=s_{\alpha'_{J'}}\tau_{g}\colon \lsub{J}{\mathcal{V}}\to \lsub{J'}{\mathcal{V}} $
 where $s_{\alpha_{J}}\colon \lsub{J}{\mathcal{V}}\to \lsub{J}{\mathcal{V}} $    denotes  the orthogonal reflection on  $(\lsub{J}{\mathcal{V}},B_{J})$ in the vector $\alpha_{J}$, $s_{\alpha'_{J'}}$ is the orthogonal reflection on $(\lsub{J'}{\mathcal{V}},B_{J'})$ in the vector $\alpha'_{J'}$ and
  $\tau_{g}\colon \lsub{J}{\mathcal{V}}\to \lsub{J'}{\mathcal{V}}$  is the linear isometry which restricts to the identity map on $\lsub{J\cup \set{\alpha}}{V}=\lsub{J'\cup \set{\alpha'}}{V}$  and maps $\alpha_{J}\mapsto \alpha'_{J'}$.
\item Suppose that $\Pi$ is a basis for $V$.  Then for $\beta\in \Pi\sm J$, we have 
\begin{equation*}
g(\beta_{J})=\kappa_{g}(\beta_{J})-\frac{B_{J}(\beta_{J},\alpha_{J})+B_{J'}(\kappa_{g}(\beta_{J}),\alpha'_{J'})}{B_{J}(\alpha_{J},\alpha_{J})}\,\alpha'_{J'}.
\end{equation*} where   $\kappa_{g}\colon\lsub{J}{\Delta}\to \lsub{J'}{\Delta}$ is the bijection given by  
 $\kappa_{g}(\alpha_{J})=\alpha'_{J'}$ and 
 $\kappa_{g}(\beta_{J})=\beta_{J'}$ if $\beta\in \Pi\sm(J\cup\set{\alpha})$.
\end{num}
\end{prop}

\subsection{Fundamental chamber and Tits cone} 
See \ref{chamber} for definitions of the (closed) fundamental chamber and the Tits cone of the  based root system $(\Phi,\Pi)$ of $(W,S)$.
We define their analogues for realized root systems of Brink-Howlett groupoids. The results we state on these are proved in Section \ref{Tits}.

 For $J\in\ob(G)$, we define the \emph{fundamental  chamber}
$\lsub{J}{\mcc}\seq\lsub{J}{\mcv}  $ of $G$ at $J$ by
\begin{equation}
\lsub{J}{\mcc}:=\mset{v\in \lsub{J}{\mcv}  \mid B_{J}(v,\lsub{J}{\Delta})\seq \mathbb{R}_{\geq 0}}.
\end{equation}

The \emph{Tits cone} $\lsub{J}{\mcx}\seq\lsub{J}{\mcv}  $ of $G$ at $J$ is defined to be the union
\begin{equation}
\lsub{J}{\mcx}:=\bigcup_{K\in \ob(G)}\,\bigcup_{w\in\lrsub{J}{G}{K}}(\mcv(w))(\lsub{K}{\mcc}).
\end{equation}
The proof of the important point (a) in the following theorem  is similar to the proof (see \cite{Bou}, \cite{Hu90}, \cite{Kac}) of  the corresponding fact for Coxeter groups.

\begin{thm}\label{Titscone} Let $J\in \ob(G)$.
\begin{num}
\item $\lsub{J}{\mcx}$ consists of all points $v\in \lsub{J}{\mcv}  $ such that $B_{J}(v,\alpha)<0$ for only a finite number of 
$\alpha\in {\lsub{J}{\wt\Upsilon}}^{\mathrm{re},+}$ and $B_{J}(v,{\lsub{J}{\wt\Upsilon}}^{\im   ,+})\seq \mathbb{R}_{\geq 0}$.  
\item In (a), one can replace $\Upsilon$ by $\Upsilon_{X}$ for $X=P,M,R$.
\item For a morphism $(J,w,K)$ in $G$, one has
$\mcv(w)( \lsub{K}{\mcx})=\lsub{J}{\mcx}$.
\item $\lsub{J}{\mcc}\seq \lsub{J}{\mcx}$ are convex cones in
$\lsub{J}{\mcv}  \seq V$.
\end{num}\end{thm}
\begin{remark*} We leave open the question of  whether $\lsub{J}{\mcx}$ consists of all points $v\in \lsub{J}{\mcv}  $ such that $B_{J}(v,\alpha)<0$ for only a finite number of 
$\alpha\in {\lsub{J}{\wt\Upsilon}}^{+}$.\end{remark*}
\subsection{}\label{stab}  It is well known that pointwise stabilizers of subsets of  the fundamental chamber of  $C$ of $(W,S)$ on $V$  are standard parabolic subgroups of $W$, which are themselves Coxeter groups. We state a  corresponding result for  Brink-Howlett groupoids. 
It is essential for the formulation  here that  the functor $\mcv$ from $G$ to the category of real vector spaces is  exactly as defined, not merely isomorphic to the one defined; the reader may  reformulate the result in more category-theoretically natural terms if desired.

Recall that for an object $J$ of $G$, $\mcv(J)=J^{\perp}\seq V$. For any morphism $g=(K,w,J)$ of $G$, $\mcv(g)\colon J^{\perp}\to K^{\perp}$ is the restriction of $w\colon V\to V$. 
 Then  $\lsub{J}{\mcx}$ may be described as follows:  the points of $\lsub{J}{\mcx}\seq \lsub{J}{\mcv}  =(J_{\fin})^{\perp}$ are  the points of $V$ of the form $wp$
where $p\in \lsub{K}{\mcc}\seq \lsub{K}{\mcv}  =(K_{\fin})^{\perp}\seq V$ for some $K\in \ob(G)$ and for some morphism $(J,w, K)$ of $G$. In this, $wp$ denotes  the image  of $p\in V$ under the action of $w\in W$ on $V$.

 Let $A\seq \lsub{J}{\mcv}  \seq V$. Define a subgroupoid $H:=\Stab_{G}(A)$ of $G$, called the \emph{pointwise stabilizer groupoid} of $A$ in $G$,
as follows. The objects of $H$ are the objects $K$ of $G$ such that $A\seq \lsub{K}{\mcv}  =K^{\perp}$.  A morphism $K\to L$ in $H$ is by definition  a morphism $(L,w,K)$ in $G$ such that $L,K\in \ob(H)$ and  $\mcv((L,w,K))$ maps $a\mapsto a$ for each $a\in A$ i.e. regarding $w\in W$, the action by $w$ on $V$ fixes $A$ pointwise. 

\begin{thm}\label{stabgpd} Let $A\seq \lsub{J}{\mcc}$, $F:=\Pi\cap A^{\perp}$ and
$G'$ denote the full Brink-Howlett groupoid of $W_{F}$. Let 
$H:=\Stab_{G}(A)$ be the pointwise stabilizer groupoid of $A$ in $G$, as defined above.
\begin{num}
\item  $A\seq C$
\item The pointwise stabilizer of $A$ in $W$ is 
$ W_{F}$.
\item  $H$ is equal to the full subgroupoid of  $G'$ on the objects of $H$. 
\item  $H$ is a  union of components of $G'$. 
\end{num}
\end{thm}
\subsection{Imaginary cone}
Let $J\in \ob(G)$. Define  cones  $\lsub{J}{\mck}$ and 
$\lsub{J}{\mcy}$  in $\lsub{J}{\mcv}  $ by
\begin{equation}
\lsub{J}{\mck}=-\lsub{J}{\mcc}\cap \cone(\lsub{J}{\Delta}).
\end{equation}
and 
\begin{equation}
\lsub{J}{\mcy}=\bigcap_{K\in \ob(G)}\,\bigcap_{w\in\lrsub{J}{G}{K}}(\mcv(w))(\cone(\lsub{K}{\Delta})).
\end{equation}
Note that for $v\in \lsub{J}{\mcv}  $, one has $v\in \lsub{J}{\mcy}$ if and only if  $w^{-1}v\in \cone(\lsub{K}{\Delta})$ for every morphism $(J,w,K)$ in $G$. 

We define the \emph{imaginary cone} $\lsub{J}{\mcz}$ of $G$ at $J$ to be 
\begin{equation}
\lsub{J}{\mcz}=\bigcup_{K\in \ob(G)}\,\bigcup_{w\in\lrsub{J}{G}{K}}(\mcv(w))(\lsub{K}{\mck})\seq \lsub{J}{\mcv}.
\end{equation} 
We call $\lsub{J}{\mck}$ the \emph{fundamental chamber}   of $\mcz$ at $J$. The following theorem is proved in Section \ref{s:8}.

\begin{thm}\label{imcone} Let $J\in \ob(G)$.
\begin{num}
\item  For any morphism $(J,w,K)$ in $G$, 
$\mcv(w)( \lsub{K}{\mcy})=\lsub{J}{\mcy}$ and 
$\mcv(w)( \lsub{K}{\mcz})=\lsub{J}{\mcz}$.
\item $\lsub{J}{\mcz}=(-\lsub{J}{\mcx})\cap \lsub{J}{\mcy}$
\item $\lsub{J}{\mcz}$ is a cone in $\lsub{J}{\mcv}$. 
\item ${\lsub{J}{\Upsilon}}^{\mathrm{im,+}} \seq \lsub{J}{\mcz}$.
\end{num}
\end{thm}

\begin{remark} One can regard  $\mcx$, $\mcy$ and $\mcz$ as cone-valued  subfunctors of $\mcv$. \end{remark} 

\section{Background and generalities on Coxeter groups} 
\label{sec:2} 
\subsection{} Throughout this paper, unless otherwise specified,  $(W,S)$ denotes a Coxeter system associated to a based root system $(\Phi,\Pi)$ in a real quadratic space $(V,B)$. 

We may assume locally  in the exposition or proofs that the Coxeter system or based root system satisfy additional conditions as convenience or necessity dictates, when there is no loss of generality in doing so.   Two constructions which produce based root systems with useful additional properties   are extension of the quadratic space, discussed in \ref{ext} and \ref{extCox}, and   lifts, discussed in  \ref{y2.4}--\ref{y2.5}.  
  \subsection{Extension of quadratic space} \label{ext} If $(V, B)$ and $(V ', B')$ are real quadratic spaces such that $V$ is a subspace of $V '$ and $B$ is the restriction of $B'$ to a bilinear form on $V$, we say $(V,B)$ is a restriction of $(V',B')$ and $(V',B')$ is an extension of $(V,B)$.
 3
 
Extensions $(V , B)$ of $(V, B)$ may be constructed as follows. Let $(U, D)$ be any real quadratic space, and $C:U\times V\to\mathbb{R}$ be any $\mathbb{R}$-bilinear form. Take $V' :=V\oplus U$ and define the symmetric bilinear form $B'$ by
\[B'(v+u,v' +u')=B(v,v')+C(u,v')+C(u',v)+D(u,u'),\] 
for $v,v' \in V$ and $u,u' \in U$.

Extending the terminology of \cite{Dy13}, a quadratic space $(V,B)$ is said to be ample for a subset $X$ of $V$ if every linear map $\Span(X)\to \mathbb{R}$ is of the form $v\mapsto B(v,v'): \Span(\Pi) \to \mathbb{R}$ for some $v' \in  V$. Then $(V,B)$ is also ample for any subset $X'$ of $\Span(X)$.

Given a real quadratic space (V,B), one may choose an extension $(V',B')$ of $(V,B)$ as above such that $(V',B')$ is ample for $V$, and is thus ample for every
subset of $V$. In fact, one may take $U := V \oplus \mathrm{Hom}_{\mathbb{R}}(V, \mathbb{R})$ in the construction in the second paragraph of this subsection, any quadratic space $(U, D)$ (for instance, $D=0$) and $C:U\times V \to\mathbb{R}$ to be the evaluation pairing $C(f,v):=f(v)$.
\subsection{} \label{extCox} Let $(V ', B'$) be an extension of the quadratic space $(V, B)$. Then any based root system $(\Phi,\Pi)$  in $(V, B)$ is also a based root system in $(V ', B')$, and the associated Coxeter systems are canonically isomorphic (since they act faithfully on $\Phi\seq V$).

Suppose given a based root system $(\Phi,\Pi)$ in the quadratic space $(V, B)$. Then  $(V, B)$ is ample for $\Pi$ if and only if it is ample for $\Phi$; in that case, we may say that $(V,B)$ is ample for $(\Phi,\Pi)$.

\subsection{Abstract based root systems} There may be uncountably many  choices of  a  realized root system $\Phi$ of a Coxeter system $(W,S)$, up to isomorphism,  even requiring the simple roots to be a basis of the ambient vector space $V$. However,  the underlying structure of $\Phi$  as ``abstract root system'' (i.e. as a  $W\times\set{\pm 1}$-set with distinguished subsets of positive roots and simple roots) depends only on $(W,S)$ as Coxeter system. 

In  \ref{abrs}--\ref{abrs2}, we give two different
(well-known) descriptions of this abstract root system directly in terms of $(W,S)$. Many of the results discussed in this section before \ref{BHred} are of a purely algebraic and combinatorial nature, in that they can be (but  are not always, here) formulated in terms of the abstract root system of $(W,S)$; however, their currently known proofs often involve realized root systems.

\label{rs}
\subsection{}\label{abrs}   We recall the definition (from \cite{Bou} and  \cite{BjBr}, for instance) of an ``abstract'' analogue $(\Xi,\Theta)$ of $(\Phi,\Pi)$, which has the advantage it  depends only on $(W,S)$ (but the disadvantage that techniques from linear algebra, convex geometry etc which usefully apply to $(\Phi,\Pi)$  are not applicable to it). 

Let $\{\pm 1\}$ be the cyclic group of order two with identity element $+1$. We sometimes write it as $\set{\pm}$ 
Define the $W\times\{\pm 1\}$-set
$\Xi:=T\times\{\pm 1\}$  with  $\{\pm 1\}$ action by multiplication on the right factor $\{\pm 1\}$ of $\Xi$ and with $W$-action defined by 
 \begin{equation} w(t,\epsilon)=(wtw^{-1},\eta_{w,t}\epsilon)\end{equation}
 for $w\in W$ and $t\in T$, where $\eta(w,t)=-1$ if $t\in N(w^{-1})$ and $\eta(w,t)=+1$ otherwise. 
 
 This $W$-action is determined by that of the simple reflections,  which is as follows:
  \begin{equation}
 s(t,\epsilon)=\begin{cases} (sts,-\epsilon),&\text{\rm if $t=s$}\\
 (sts,\epsilon)&\text{\rm if $t\neq s$}
  \end{cases},\qquad \text{\rm for $s\in S$ and $t\in T$.} 
 \end{equation}

Define sets of   \emph{positive roots} $\Xi^{+}:=T\times\{+1\}$ and \emph{simple roots}
$\Theta:=S\times \{+1\}$ of $\Xi$. Define the reflection in a root $\alpha=(t,\epsilon)\in \Xi$  by $s_{\alpha}=s_{(t,\epsilon)}:=t$.
We call the pair $(\Xi,\Theta)$ of a $W\times \set{\pm}$-set $\Xi$ with distinguished subset $\Theta$ the standard \emph{abstract based root system} of $(W,S)$.

The relation of  $(\Xi,\Theta)$ to the based root system $(\Phi,\Pi)$ is as follows: the map $\Phi\to \Xi$ given by $\epsilon\alpha\mapsto (s_{\alpha},\epsilon)$ for $\alpha\in \Phi^{+}$ and $\epsilon \in \{\pm 1\}$ is  a $W\times\{\pm 1\}$-equivariant bijection which maps $\Phi^{+}$ to $\Xi^{+}$ and $\Pi$ to $\Theta$.

\subsection{}\label{abrs2} We recall another realization of the abstract root system of $(W,S)$ which is 
used  in the study of buildings  (see 
\cite{Tits},  \cite{AbBr}, \cite{Kra}).

The $W$-action $(w,x)\mapsto wx$ by left translation on $W$  leads naturally to an action $(w,A)\mapsto wA:=\mset{wa\mid a\in A}$ of $W$ by left translation on the power set $\mathcal{P}(W)$. For each $t\in T$ and $\epsilon\in \set{\pm 1}$, let \[H^{t,\epsilon}:=\mset{w\in W\mid\epsilon(\ell(tw)-\ell(w))>0}.\] 
One has $W\sm H^{t,\epsilon}=H^{t,-\epsilon}$.
Let \[\Xi':=\mset{H^{t,\epsilon}\mid t\in T,\epsilon \in \set{\pm 1}}\seq \mathcal{P}(W),\] $\Xi^{\prime +}:=\mset{A\in \Xi'\mid 1_{W}\in A}$
 and $\Theta':=\mset{H^{s,+}\mid s\in S}\seq \Xi^{\prime+}$.
 Then $\Xi'$ is a $W\times\set{\pm 1}$-set  with $W$-action as a subset of $\mathcal{P}(W)$ and with $\set{\pm }$-action determined by $-A:=W\sm A$ for $A\in \Xi'$. The map
 $(t,\epsilon)\mapsto H^{t,\epsilon}$ is a $W\times\set{\pm}$-equivariant bijection $\Xi\xrightarrow{\cong} \Xi'$  which maps $\Xi^{+}\xrightarrow{\cong}\Xi^{\prime +}$ and $\Theta\xrightarrow{\cong} \Theta^{\prime+}$.  
 
\begin{remark} Note that  as $W$-set, $\Xi'$  is the union of the $W$-orbits on $\mathcal{P}(W)$ of elements of $\Theta'$, the action by $-1$ is $A\mapsto W\sm A$, and the positive elements of $\Xi'$ are those which contain $1_{W}$.   
  This description can be used to generalize the construction of $(\Xi', \Theta')$ to preprincipal signed groupoid-sets, for which the atomic generators are  good analogues of simple reflections. 
  
  We do not know of any   generalization of  the construction in \ref{abrs} to preprincipal signed groupoid sets, because of the absence in general of suitable analogues of conjugacy and reflections, but Theorem
 \ref{BHabrs} provides such a generalization for Brink-Howlett groupoids, using the extra structure provided by the relationship of those groupoids to  Coxeter systems.   \end{remark}

\subsection{Reflection subgroups} 
 We record some basic facts concerning reflection subgroups, 
 shortest coset representatives, and  corank. 

First, there is the following description of canonical simple systems $\Pi_{U}$ of reflection subgroups:
 \begin{thm}[\cite{DyRef},\cite{DyTh}] \label{refcan} For a subset $\Delta$ of $\Phi^{+}$, one has $\Delta=\Pi_{U}$ for some reflection subgroup $U$ of $W$ (where necessarily $U=W_{\Delta}=\mpair{s_{\beta}\mid \beta\in \Delta}$) if and only if for all distinct $\beta,\gamma\in \Delta$, one has $B(\beta,\gamma)\in -\mathrm{COS}$ where \[\mathrm{COS}:=\mset{\cos\frac{\pi}{n}\mid n\in \mathbb{N}_{\geq 2}} \cup\mset{\lambda\in \mathbb{R}\mid \lambda\geq 1}.\]\end{thm}
\begin{remark*}  A closely related result is as follows. Let $(V,B)$ be a real quadratic space and $\Pi\seq V$. Then there exists a based root system $(\Phi,\Pi)$ in $V$ if and only if $\Pi$ is positively independent,  $B(\alpha,\alpha)=1$ for all $\alpha\in \Pi$ and $B(\alpha,\beta)\in -\mathrm{COS}$ for all distinct $\alpha,\beta\in \Pi$.\end{remark*}

\begin{prop} \label{shortcos} Let $W'$ be a reflection subgroup and $x\in W$.  
\begin{num}\item If $W'=\mpair{T'}$ where $T'\seq T$, then
$W'\cap T=\mset{wtw^{-1}\mid w\in W', t\in T'}$.
\item There is  a unique element $d$ of the left coset $xW'$  with $\ell(d)$ minimal.
\item Let $y\in xW'$ and $d$ be as in (b). Then  $y=d\iff y\Pi_{W'}\seq \Phi^{+}\iff \Phi_{y^{-1}}\cap \Pi_{W'}=\eset\iff 
 \Phi_{y^{-1}}\cap \Phi_{W'}^{+}=\eset\iff N(y^{-1})\cap (W'\cap T)=\eset\iff N(y^{-1})\cap \chi(W')=\eset$.
\item For $d$ as in (b), we have $\chi(xW'x^{-1})=d\chi(W')d^{-1}$ and $\Pi_{xW'x^{-1}}=d\Pi_{W'}$.
\item  For $J\seq \Pi$,  each coset $yW_{J}$ with $y\in W$ contains a unique representative $y^{J}$ of minimal length $\ell(y^{J})$.
We have $y=y^{J}z$ for a unique element $z\in W_{J}$.
Let $W^{J}:=\{\, y^{J}\mid y\in W\,\}$.  Then  $\ell(vw)=\ell(v)+\ell(w)$ for all $v\in W^{J}$ and $w\in W_{J}$.
\end{num}\end{prop}
\begin{proof}  For (a), see \cite{DyRef}. For (b)--(c), see \cite[Lemma 1.2]{BDy}, \cite[1.7]{Dy13} or \cite[(3.4)]{DyRef}. For  (d), see \cite[Lemma 1]{DyPc}
or \cite[1.7-1.8]{Dy13}.

The assertions of (e) are well-known; see for instance \cite{Bou}, \cite{Hu90} or \cite{BjBr}. \end{proof}

 \begin{prop}\label{corank} Let $U$ be a reflection subgroup of $W$.\begin{num}
 \item One has $\corank_{W}(U)=0$ if and only if $U=W$.
  \item  If $w\in W$, then $\corank_{W}(wUw^{-1})=
 \corank_{W}(U)$.
 \item If $J\seq \Pi$, then the standard parabolic subgroup $W_{J}$ is a reflection subgroup of $W$ with $\Pi_{W_{J}}=J$ and $\corank_{W}(W_{J})=\vert \Pi\sm J\vert$.
 \item If $U_{1}\seq U_{2}\seq U_{3}$  are reflection subgroups of $W$ with $U_{i}$ parabolic in $U_{i+1}$ for $i=1,2$, then  $U_{1}$ is parabolic in $U_{3}$ and \begin{equation*}
 \corank_{U_{3}}(U_{1})=\corank_{U_{3}}(U_{2})+
 \corank_{U_{2}}(U_{1}).
 \end{equation*}  \end{num}
 \end{prop}
 \begin{proof} 
 For (a), note that if  $ \textrm{corank}_{W}(U) = 0$, then $ W =\langle U \cap T \rangle = U$ since  $U$ is a reflection subgroup. 
 
 Part (b) holds by definition of corank since for $T'\seq T$,   \[W=\mpair{(U\cap T)\cup T'}\iff W=\mpair{(wUw^{-1}\cap T)\cup wT'w^{-1}},\] because $wTw^{-1}=T$ implies 
 $w((U\cap T)\cup T')w^{-1}= (wUw^{-1}\cap T)\cup wT'w^{-1}$.
 
 We prove (c).  Let $K=\mset{s_{\alpha}\mid \alpha\in J}$. 
 To show  $\Pi_{W_{J}}=J$, it suffices to show that $\chi(W_{K})=K$.  For $s\in K$, we have $N(s)\cap W_{K}=\set{s}\cap W_{K}=\set{s}$. This shows that $K\seq \chi(W_{K})$. Equality must hold since $K$ generates $W_{K}$ and any set of Coxeter generators (such as $\chi(W_{K})$) of $W_{K}$ is an inclusion  minimal set of generators of $W_{K}$ (see \cite{Bou}).
 
 We have $W=\mpair{W_{K}\cup (S\sm K)}$. This implies that $\corank_{W}(W_{K})\leq \vert S\sm K\vert=\vert \Pi\sm J\vert$.
 To prove this inequality is in fact an equality,  it suffices by the Schr\"oder-Bernstein theorem to show that if $T'\seq T$ is such that $\mpair{(W_{K}\cap T)\cup T'}=W$, then $\vert T'\vert\geq \vert \Pi\sm J\vert$. We may  assume without loss of generality for the proof of this that $(\Phi,\Pi)$ is a standard  based root system,  so $\Pi$ is linearly independent. Let $\Gamma':=\mset{\alpha\in \Phi^{+}\mid s_{\alpha}\in T'}$.
Since  $\mpair{(W_{K}\cap T)\cup T'}=W$, Proposition \ref{shortcos}(a) implies    $\Span(J\cup \Gamma')=\Span(\Phi_{J}\cup \Gamma')=\Span(\Pi)$.
Hence the image of $\Gamma'$ under the natural surjective linear map $\Span(\Pi)\to \Span (\Pi)/\Span  (J)$ spans the quotient vector space, implying $\vert T'\vert= \vert \Gamma'\vert\geq \mathrm{dim}(\Span (\Pi)/\Span  (J))=\vert \Pi\sm J\vert$. This proves (c).  

We prove (d). Let $S_{i}:=\chi(U_{i})$ for $i=1,2,3$. 
For $i=1,2$,  $U_{i}=x_{i}\mpair{R_{i}}x_{i}^{-1}$  for some 
$x_{i}\in U_{i+1}$ and $R_{i}\seq S_{i+1}$, by definition of parabolic subgroups. By Proposition \ref{shortcos}, we may assume without loss of generality that $N(x_{i}^{-1})\cap \mpair{R_{i}}=\eset$, and this implies that $S_{i}=\chi(U_{i})=x_{i}R_{i}x_{i}^{-1}$. Hence $x_{i}^{-1}S_{i}x_{i}=R_{i}\seq S_{i+1}$  for $i=1,2$ and 
$x_{2}^{-1}x_{1}^{-1}S_{1}x_{1}x_{2}\seq S_{3}$. This shows that $U_{1}$ is parabolic in $U_{3}$, since it is generated by a conjugate $S_{1}$ of a subset of $S_{3}$. Moreover,  by (b)--(c),  
 \begin{equation*}\begin{split}
  \corank_{U_{3}}(U_{1})&=\vert S_{3}\sm x_{2}^{-1}x_{1}^{-1}S_{1}x_{1}x_{2}\vert=
 \vert (S_{3}\sm x_{2}^{-1}S_{2}x_{2})\,\dot\cup\, x_{2}^{-1}(S_{2}\sm x_{1}^{-1}S_{1}x_{1})x_{2}\vert\\
&= \vert S_{2}\sm x_{1}^{-1}S_{1}x_{1}\vert+\vert S_{3}\sm x_{2}^{-1}S_{2}x_{2}\vert=\corank_{U_{2}}(U_{1})+\corank_{U_{3}}(U_{2}),
 \end{split}
 \end{equation*}
 completing the proof.
  \end{proof}
  
\subsection{}   The following proposition is included only to suggest the complexity of behavior of ranks and coranks of  general reflection subgroups of infinite Coxeter groups, which is similar (and related) to behavior of  ranks of subgroups of non-abelian free groups.  No subsequent use is made of the result. 

Recall that a Coxeter system is of \emph{locally finite type} if each of its finite rank  standard parabolic subgroups is finite. Write $\aleph_{0}:=\vert \mathbb{N}\vert$. \begin{prop} Let $(n_{i})_{i\in \mathbb{Z}}$ be a sequence  such that $n_{2i}=3i+3$  for $i\in \mathbb{N}_{\geq 0}$,
 $n_{-2i}\in \mathbb{N}_{\geq 3}$ for $i\in\mathbb{N}_{>0}$ and 
 $n_{2i+1}=\aleph_{0}$ for $i\in \mathbb{Z}$.  Then  if  $(W,S)$ is irreducible and  is not of finite, affine or locally finite type, it contains a  chain $(W_{i})_{i\in \mathbb{Z}}$ of reflection subgroups
 such that  $W_{i}\sneq W_{j}$ if $i<j$ in $\mathbb{Z}$ and  $\rank(W_{i})=n_{i}$ for all $i\in \mathbb{Z}$.\end{prop}
 \begin{proof} 
 Recall that a Coxeter system $(U,R)$ is \emph{universal} if for any distinct $r,s\in R$, the order of $rs$ is infinite. 
 Firstly, from \cite{Dy13}, the stated assumptions on  $(W,S)$ imply that $W$ contains a universal  reflection subgroup $U$ of rank three. Secondly, in any rank three universal Coxeter group $U$ such that $\chi(U)=\set{r,s,t}$, there is an infinite rank universal  reflection  subgroup $U'$ such that  $\chi(U')=\mset{(rs)^{k}t(rs)^{-k}\mid k\in \mathbb{Z}}$, which is denumerably infinite (see \cite{DyRef} or \cite[Example 3.19]{DyTh}; note also that the subgroup $\mpair{rs,rt}$ of $W$ is free of rank two).  One may construct a chain of reflection subgroups of $(W,S)$ as described, all contained in $U$ as in the first fact, by adroit use of the second fact. Details are left to the reader. \end{proof}
 
 \subsection{Parabolic subgroups}\label{parclos} The \emph{parabolic closure} $\pc(G)$ of a subset (e.g. subgroup) $G$ of $W$ is defined to be  the intersection of all parabolic subgroups of $W$ which contain $G$. In general, it is not necessarily a parabolic subgroup of $W$, but is  a locally parabolic (reflection) subgroup as defined by Nuida \cite{Nuida}. We collect below some known facts about parabolic subgroups and parabolic closure.  \begin{prop} \label{parsub} Let $U$ and $U'$ be parabolic subgroups of $W$, and $J\seq \Pi$.
\begin{num}\item $U\cap U'$ is a parabolic subgroup of both $U$ and $U'$.
\item $U\cap U'$ is  a parabolic subgroup of $W$.
\item If $W'$ is a finite reflection subgroup of $W$,  and $P:=\pc(W')$, then $P$ is finite, $\rank(P)=\rank(W')=\dim(\Span(\Phi_{W'}))=\vert \Pi_{W'}\vert$ and $\Phi_{P}=\Span(\Phi_{W'})\cap \Phi$.
\item If $W'=\mpair{J\cup X}$ for some finite $X\seq W$, then $\pc(W')=P$ where $P$ is the unique parabolic subgroup $P\sreq W'$ of minimum (finite) corank $\corank_{P}(W_{J})$.
\end{num}
 \end{prop}
 \begin{proof}  Parts (a)-(c) are well-known; see for example
 \cite[2.4]{DyPc}, \cite[2.7, 2.9, 2.12]{Dy13}  and \cite[1.2.6, 3.2]{Kra}. 
 
 We prove (d). Since $X$ is finite, there is a finite subset $Q$ of $\Pi$ such that $X\seq W_{Q}$. For instance, take $Q$ to be the set of all simple roots corresponding to simple reflections $s\in S$ such that there is some $x\in X$ such that $s$ appears as a factor in some (equivalently, every) reduced expression of $x$. Let $H:=W_{J\cup Q}$. Clearly, $W'\seq H$. By definition of corank, $\corank_{H}(W_{J})\leq \vert Q\vert<\infty$. 
 Thus, $H $ is a (standard) parabolic subgroup of $W$ such that $ W' \seq H$ and $ \corank_{H}(W_{J})$ is finite.

Amongst all  parabolic subgroups $P$ of $W$ such that  $P\sreq W'$, choose one for  which $ \corank_{P}(W_{J})$ is minimal. 
Denote it as $P$, so   
$ \corank_{P}(W_{J})\leq \corank_{H}(W_{J})<\infty $. 

Suppose that $ P'$ is a parabolic subgroup of $W$ such that $ W' \seq P'$. Note that $P' \cap P$ is parabolic in $P$ by (b), and $P'\cap P\sreq W'\sreq W_{J}$. Since $W_{J}$ is standard parabolic in $W$, it is (standard) parabolic in any reflection subgroup (such as $P'\cap P$) which contains it. By Proposition
\ref{corank}(d), 
\[  \corank_{P}(P'\cap P) + \corank_{P' \cap P}(W_{J})=\corank_{P}(W_{J}) <\infty. \] By choice of $P$,  $\corank_{P' \cap P}(W_{J})\geq\corank_{P}(W_{J})$. This forces $\corank_{P' \cap P}(W_{J})=\corank_{P}(W_{J})$ and
$\corank_{P}(P'\cap P)=0$  and so $P=P'\cap P$ by  Proposition
\ref{corank}(a).   Thus, $P\seq P'$. Hence $P$ is the inclusion-minimum parabolic subgroup containing $W'$, which implies
$ \pc(W') = P$.  This proves (d).
 \end{proof}

\begin{lem}\label{conjpair} Let $W'$ be a parabolic subgroup of $W$, say $wW'w^{-1}=W_{K}$ where $K\seq \Pi$ and $w\in W$. Assume that  $W'\sreq W_{J}$ where $J\seq \Pi$.   Let $d$ be the element of minimal length in $wW'$. Then $dW'd^{-1}=W_{K}$ and $dW_{J}d^{-1}=W_{L}$ where  $L:=dJ\seq K$. 
\end{lem}
\begin{proof}
This holds since $d \Pi_{W'}=K$, by Proposition \ref{shortcos}(d), and $J\seq \Pi_{W'}$.
\end{proof}

\subsection{Fundamental chamber, Tits cone and facial subgroups} 
\label{chamber}
Recall that the \emph{fundamental chamber}\footnote{For  other notions of reflection representations of Coxeter groups, as in \cite{Bou} or \cite{Fu},  the fundamental chamber and Tits cone of a root system are naturally defined as subsets of the dual space of the ambient vector space of the root system, or subsets of the ambient vector space of an associated ``dual'' root system. 
The root systems we consider in this paper are 
 ``self-dual,'' and for simplicity, we regard the fundamental  chamber and Tits cone as   subsets of the ambient vector space of the root system, in the natural way.} 
   $C$ and \emph{Tits cone} $X$ of $(W,S)$ 
on $(V,B)$ are defined by \[\,C:=\{\,v\in V\mid B(v,\Pi)\subseteq \mathbb{R}_{\geq 0}\,\},\qquad X:=\bigcup_{w\in W} wC.  \] 

A subset $J$ of $\Pi$ is said to be  \emph{facial} if $J=\Pi\cap v^{\perp}$ for some $v\in C$, where $C$ is the fundamental chamber of $(W,S)$ on $(V,B)$. 
A  \emph{standard facial subgroup} is defined to be   a standard parabolic subgroup $W_{J}$ of $W$ such that $J\seq \Pi$ is facial. 
A  \emph{facial subgroup} of $W$ is  a (necessarily parabolic) subgroup of $W$ which is  $W$-conjugate to a standard facial subgroup.

\begin{prop}\label{CoxfundTits}\begin{num}
\item Every $W$-orbit on $X$ contains a unique point of $C$.
\item If $v\in C$, the stabilizer in $W$ of $v$ is the standard 
facial subgroup $W_{J}$ where 
$J:=\{\,\alpha\in \Pi\mid B(\alpha,v)=0\,\}=\Pi\cap v^{\perp}$. \item The standard facial subgroups of $W$ are the stabilizers of points of the fundamental chamber $C$.
\item The facial subgroups of $W$ are  the stabilizers in $W$ of points of the Tits cone $X$.
\item  The Tits cone $X$ consists of all points $v$ of $V$ which satisfy  $B(v,\alpha)<0$ for only finitely many $\alpha\in \Phi^{+}$. In particular, $X$ is a convex cone. 
 \item Let $ w = s_{1}s_{2} \cdots s_{n}\in W$,  where $s_{i}=s_{\alpha_{i}}$ withs  $ \alpha_{i}\in \Pi$. Let $ \beta_{i} : = s_{1}s_{2} \cdots s_{i-1}(\alpha_{i})$
    for  $i=1,\ldots, n$.  Then $ v- w(v) = \sum_{i=1}^{n}B(v, \alpha_{i}^{\vee})\beta_{i}$  for all $ v \in V$. Further, if $\ell(w)=n$, then  $\set{\beta_{1},\ldots,\beta_{n}}=\Phi_{w}\seq \Phi^{+}$.  
   \item If $v\in V$ and $w\in W_{J}$ where $J\seq \Pi$, then
$v-wv\in \Span(J)$.  
\item If $v\in C$ and $w\in W_{J}$ where $J\seq \Pi$, then
$v-wv\in \cone(J)$. 
\end{num}\end{prop}

\begin{proof} See \cite[Lemma 1.13 and  Lemma 1.16]{Dy13}. \end{proof}

\subsection{} Recall that a poset $I$ is said to be directed if $I\neq \eset$ and for any $x,y\in I$, there is $z\in I$ such that $x\leq z$ and $y\leq z$. A  $I$-directed family of subsets  of a set $X$ is a family $(A_{i})_{i\in I}$ of subsets of $X$ such that 
$i\leq j$ in $I$ implies $A_{i}\seq A_{j}$.

We shall use below the following fact. If $I$ is a directed set and 
$(A_{i})_{i\in I}$ and $(B_{i})_{i\in I}$ are $I$-directed families  of subsets of $X$, then 
\begin{equation}\label{dirunion}(\bigcup_{i\in I}A_{i})\cap (\bigcup_{i\in I}B_{i})=\bigcup_{i\in I}(A_{i}\cap B_{i}).\end{equation}

We shall also require some elementary convex geometry.
Given a cone $C=\cone(X)$ in $V$, a subset $F$ of $C$ is called a face of $C$ if $F=\cone(F)$ is itself a cone, and for all $x,y\in C$, if $x+y\in F$ then $x,y\in F$. For any $f\in V^{*}:=\mathrm{Hom}(V,\mathbb{R})$ such that $f(C)\seq \mathbb{R}_{\geq 0}$, $C\cap \ker f$ is a face of $C$, called an \emph{exposed face}.   We use below the well-known fact that 
all faces of a \emph{polyhedral cone} in $V$ (i.e. a cone of the form $\cone(X)$ where $X\seq V$ is finite) are exposed faces (see for instance \cite{Web}).

 \begin{prop}\label{faces} Let $J\seq \Pi$.
\begin{num}\item $\Pi\cap \cone(J)=J$.  \item  If  $J$ is a facial subset of $\Pi$, then  $\cone(J)$ is an exposed  face of $\cone(\Pi)$.
The converse holds if $(V,B)$ is ample for $\Phi$.
\item Suppose that $\Pi$ is  linearly independent. Then $\cone(J)$ is a face of $\cone(\Pi)$. If $(V,B)$ is ample for $\Phi$, then $J$ is a facial subset of $\Pi$. 
\item If $(V,B)$ is ample for $\Phi$ and $\Pi$ is finite, then 
$J$ is a facial subset of $\Pi$ if and only if $\cone(J)$ is a face of $\cone(\Pi)$. 
\end{num}
\end{prop}
\begin{proof} We prove (a). Suppose that $\alpha\in \Pi\cap \cone(J)$. Write $\alpha=\sum_{\beta\in J}c_{\beta}\beta$ where all $c_{\beta}\geq 0$.  Then
\[1=B(\alpha,\alpha)=\sum_{\beta\in J}c_{\beta}B(\alpha,\beta).\] Since $B(\alpha,\beta)\leq 0$ if $\beta\neq \alpha$, we must have  $\alpha\in J$. Hence $\Pi\cap \cone(J)\seq J$. The reverse inclusion is trivial, so we have equality, proving (a). 

Parts (b)--(c) follow readily from the definitions. Details are omitted. 

We prove (d). The ``only if'' direction follows from (b). The  ``if'' direction follows from (b) since the exposed faces and faces of $\cone(\Pi)$ coincide, when $\Pi$ is finite.
\end{proof}

\begin{prop}\label{facialcomp}\begin{num}
\item If a connected Brink-Howlett groupoid has an object which is a facial  subset of $\Pi$, all its objects are facial subsets of $\Pi$.
\item Suppose that  $J\seq \Pi$ is such that either (1) $J$ is a facial subset of $\Pi$ (2) $J\cap \Gamma$ is a facial subset of $\Gamma$ for each finite subset $\Gamma$ of $\Pi$  or (3)  each component of $J$ is of finite type. Then $\cone(J)$ is a face of $\cone(\Pi)$.
\item Suppose that  $J\seq \Pi$ is such that $\cone(J)$ is a face of $\cone(\Pi)$. Then for any $G\seq \Pi$,  $G\cap \cone(J)=G\cap J$,  $\Phi^{+}_{J}=\Phi^{+}\cap\Span(J)=\Phi\cap \cone (J)$ and
\[\cone(G)\cap \cone(J)= \cone(G)\cap \Span(J)=\cone(G\cap J).\]  
 \end{num}
\end{prop}
\begin{proof} Part (a) follows directly from the fact (\cite[Lemma 2.10(d)]{Dy13}) that if $J$ is a  facial subset of  $\Pi$ and $w\in W$ satisfies $wJ=K$, then $K$ is a facial subset of $\Pi$.

For the proofs of (b)--(c),  assume without loss of generality, by extending the quadratic space $(V,B)$ if necessary,  that $(V,B)$ is ample for $\Pi$. Let $I$ denote  the inclusion-ordered poset of finite subsets of $\Pi$. 

We prove (b). If (1) holds, then $\cone(J)$ is a face of $\cone(\Pi)$ by Proposition \ref{faces}. Suppose (2) holds. Then for $\Gamma\in I$, $\cone(J\cap \Gamma)$ is a face of $\cone(\Gamma)$. It is readily checked from the definition of faces of cones that, since $I$ is directed,  this implies that 
$\cone(J)=\bigcup_{\Gamma\in I}\cone(\Gamma \cap J)$ is a face of $\cone(\Pi)=\bigcup_{\Gamma\in I}\cone(\Gamma )$.  
Finally, suppose (3) holds.  For each $\Gamma\in I$, $\cone(\eset)=\set{0}$ is a face of $\cone(\Pi)$, by positive independence of $\Pi$. Hence $\eset$ is a facial subset of $\Gamma$, by Proposition \ref{faces}(d). By \cite[Lemma 2.10(b)]{Dy13}, $J\cap \Gamma$ is a facial subset of $\Gamma$, since $J\cap \Gamma$ is finite with all its components of finite type. By (2), it follows that $\cone(J)$ is a face of $\cone(\Pi)$. This proves (b). 

We now prove (c).  Let $G\seq \Pi$. By Proposition \ref{faces}(a), we have \[G\cap \cone(J)=G\cap \Pi\cap \cone(J)=G\cap J.\]
Suppose first that $J$ is a facial subset of $\Pi$. Then the assertion in (c) that 
$\cone(G)\cap \cone(J)= \cone(G)\cap \Span(J)=\cone(G\cap J)$ is easily checked using the definition of facial subset,  and the remaining  
assertion of (c) (concerning $\Phi_{J}^{+}$) is contained  in \cite[Lemma 2.4]{Dy13}.

Now let $J$ be any subset of $\Pi$ such that $\cone(J)$ is a face of $\cone(\Pi)$. We claim  that for each $\Gamma\in I$,
$\cone(\Gamma\cap J)$ is a face of $\cone(\Gamma)$. Certainly,  $\cone(\Gamma\cap J)$ is a cone. Suppose that
$x,y\in \cone(\Gamma)$ satisfy $x+y\in \cone(\Gamma\cap J)$.
Write $x=\sum_{\alpha\in \Pi}c_{\alpha}\alpha$ and 
$y=\sum_{\alpha\in \Pi}d_{\alpha}\alpha$ where $c_{\alpha}=0=d_{\alpha}$ unless $\alpha\in \Gamma$. 
For $\alpha\in \Pi$, we have $(c_{\alpha}+d_{\alpha})\alpha+
\sum_{\beta\in \Pi,\beta\neq \alpha} (c_{\beta}+d_{\beta})\beta=x+y\in \cone(\Gamma\cap J)\seq \cone(J)$.  Since $\cone(J)$ is a face of $\cone(\Pi)$, it follows that 
$(c_{\alpha}+d_{\alpha})\alpha\in \cone(J)$. This implies that if  $c_{\alpha}\neq 0$ or $d_{\alpha}\neq 0$, then  $\alpha\in \Gamma$ and $\alpha\in \cone(J)\cap \Pi=J$ by Proposition \ref{faces}(a), so $\alpha\in J\cap \Gamma$. We conclude that $x,y\in  \cone(\Gamma\cap J)$, proving the above claim. 

The claim  implies that for each  $\Gamma\in I$,
$\Gamma\cap J$ is a facial subset of $\cone(\Gamma)$. 
By the first paragraph of the proof of (c), all assertions of (c) hold with $\Pi$ replaced by $\Gamma$, $J$ replaced by $J\cap \Gamma$ and $G$ replaced by $G\cap \Gamma$.
Note that  $\Pi$ (resp., $J$, $\cone(J)$, $\Span(J)$, $G$,   $G\cap J$, $\cone(G\cap J)$,
 $\Phi=\Phi_{\Pi}$, $\Phi^{+}=\Phi^{+}_{\Pi}$,  $\Phi_{J}^{+}$) is the directed union over $\Gamma\in I$ of  the sets $\Pi\cap \Gamma$ (resp., $J\cap \Gamma$, $\cone(J\cap \Gamma)$, $\Span(J\cap \Gamma)$, $G\cap \Gamma$,  $G\cap J\cap \Gamma$,  $\cone(G\cap J\cap \Gamma)$,  
 $\Phi_{\Gamma}$, $\Phi^{+}_{\Gamma}$,  $\Phi_{J\cap \Gamma}^{+}$).  The assertions of (c) now  follow in general from the  special case considered above,  when $J$ is facial in $\Pi$,  by taking directed unions and using \eqref{dirunion}.  
\end{proof}
\subsection{} \label{pres} We digress to observe an obvious consequence of the preceding proposition which was not explicitly highlighted  in  \cite{Vin} (where results essentially equivalent to some of those above were obtained in finite rank) or   \cite{Dy13}.

Define the \emph{presentation graph} of a Coxeter 
system $(W,S)$ to be the  edge-labeled (undirected, simple) 
graph  with vertex set $S$ and an edge $\set{r,s}$  joining 
$r,s\in S$ if $r\neq s$ and $rs$ has finite order $m_{r,s}<\infty$, 
 that edge being labeled by the positive integer $m_{r,s}$ if $m_{r,s}>2$.  The presentation graph should not be confused 
 with the Coxeter graph, which may have edges labeled by $\infty$  and does not have edges joining commuting simple 
 reflections.
%  By means of the bijection $\alpha\mapsto s_{\alpha}\colon\Pi\to S$, we may identify the vertex set of the 
% presentation graph with the standard simple system $\Pi$. 

The definition of based root system implies that the map $\alpha\mapsto \cone(\alpha)$ is a bijection between $\Pi$ and the set of one-dimensional faces (i.e. extreme rays) of $\cone(\Pi)$. The following result, which is an immediate consequence of Proposition \ref{facialcomp}(b)  shows that the   edges of the presentation graph correspond to certain (not necessarily all) two-dimensional faces of $\cone(\Pi)$.
\begin{cor} \label{presgraph} Let $\alpha,\beta\in \Pi$ be distinct.  If there is an edge  of the presentation graph joining $s_{\alpha}$ and $s_{\beta}$, then $\cone(\set{\alpha,\beta})$ is a face of $\cone(\Pi)$.  \end{cor}
\begin{remark*} Parts of Proposition \ref{facialcomp} can be  
 strengthened and generalized, leading in particular to more general constructions of realized or weakly realized root systems of Brink-Howlett groupoids based on the more general root systems of Coxeter groups described in \cite{Fu}, and   extensions of  Corollary \ref{presgraph} and parts of Proposition \ref{facialcomp}  to Brink-Howlett groupoids. We do not discuss this further in this paper.\end{remark*}

\subsection{} The corollary above can be reformulated as follows if $(W,S)$ has finite rank.  

Assume that $\Pi$ is finite, so   $\cone(\Pi)$ is a polyhedral cone. One may choose an affine hyperplane $H$ in $V$
such that for each $\alpha\in \Pi$, $\cone(\alpha)\cap H=\set{p_{\alpha}}$ is a singleton set with $p_{\alpha}\neq 0$. Then $P:=H\cap \cone(\Pi)$ is a convex polytope (a polytopal cross-section of $\cone(\Pi)$) and the map $\alpha\mapsto p_{\alpha}$ induces a bijection  from $\Pi$ to the set of vertices of $P$. 
In this situation, the  corollary implies that if   $\alpha,\beta\in \Pi$ are distinct and  there is an edge  of the presentation graph joining $s_{\alpha}$ and $s_{\beta}$, then  there is an edge of $P$ with $p_{\alpha}$ and $p_{\beta}$ as its vertices. Hence (the unlabelled graph underlying) the presentation graph identifies with  a subgraph (containing all vertices) of the edge-graph of the convex polytope $P$. 
More generally, the abstract simplicial complex of spherical subsets of $\Pi$ has a geometric realization which is a subcomplex of the geometric simplicial complex of simplicial faces of $P$.   See Example \ref{onlyweak}.

  \subsection{Finiteness criteria} 
We collect  below some variants, which are useful in this paper, of known finiteness criteria for irreducible Coxeter systems. (There are  many other important finiteness criteria we do not list.)    Some of the  cited proofs  involve reduction of the results in general to their  special cases for   finite rank Coxeter systems in  \cite[\S 4]{Deod} (see also \cite[Proposition 1.2.6]{Kra}). 

\begin{prop}\label{fincond} Assume that $(W,S)$ is irreducible (so $S\neq \emptyset$). Then the following conditions are equivalent.
\begin{conds}
\item $W$ is finite.
\item $\Phi$ is finite.
\item For some $J\subsetneq \Pi$, $W^{J}$ is finite.
\item For some $J\subsetneq \Pi$, $\Phi\sm \Phi_{J}$ is finite.
\item For some $\alpha\in \Pi$, the  $W$-orbit  $W\alpha$ on $\Phi$ is finite. 
\item  For some $J\subsetneq \Pi$, the index $[W:W_{J}]$ is finite.
\item For some $J\subsetneq \Pi$, $(W\sm W_{J})\cap T$ is finite.
\item For some $s\in S$, the  conjugacy class in $W$ of $s$    is finite. 
 \end{conds}
In conditions (iii)--(viii), ``some'' can be replaced by ``all''.
\end{prop} 

 \begin{proof}   It is well known that (i) and (ii) are equivalent.
 Indeed, we have  (ii)$\implies$(i) since  the natural $W$-action on $\Phi$ is faithful. The converse implication  (i)$\implies$(ii). holds since the map $\beta\mapsto s_{\beta}\colon \Phi\to W$
 has the $\set{\pm}$-orbits on $\Phi$ as its non-empty fibers. This map  $\beta\mapsto s_{\beta}$ restricts to  surjections $\Phi\sm \Phi_{J}\twoheadrightarrow (W\sm W_{J})\cap T$ for  $J\seq \Pi$, and $W\alpha\twoheadrightarrow \mset{ws_{\alpha}w^{-1}\mid w\in W}$ for $\alpha\in \Pi$, from which it follows that (iv)$\iff$(vii) and (v)$\iff$(viii). We  also have (iii)$\iff$(vi) since for $J\seq \Pi$, $W^{J}$ is a set of left coset representatives for $W_{J}$ in $W$.   
  
  By the equivalence of (i) and (ii), condition (i) implies the variants of (iii)-(viii) with ``all'' instead of ``some.''  Also, each such variant with ``all'' implies the corresponding condition with ``some,'' since $S$ and $\Pi$ are non-empty.
  Finally, by \cite[Proposition 1.21 and Lemma 1.23]{Dy13}, each condition  (iii), (iv) or (v) implies (i). \end{proof}
  \begin{proof}[Proof of Proposition \ref{finindcond}]
We assume without loss of generality that $\Pi_{U}\sm J$ is finite, since that finiteness is readily implied by each  condition \ref{finindcond}(i)--(iv). There are only finitely many components of $\Pi_{U}$ which intersect $\Pi_{U}\sm J$ non-trivially. Denote these components as $\Pi_{1},\ldots, \Pi_{n}$
The other components of $\Pi_{U}$ are those components of $J$ which have no vertex joined (in the Coxeter graph of $U$) to a vertex in $\Pi_{U}\sm J$. For $i=1,\ldots, n$, 
 let $U_{i}:=W_{\Pi_{i}}$. Then $U_{i}$ is a component of $U$
 (i.e. a non-trivial,  inclusion-maximal, irreducible standard parabolic subgroup of $U$).  Define $J_{i}:=J\cap \Pi_{i}\seq \Pi$. Then $W_{U_{i}}$ is a reflection overgroup of $W_{J_{i}}$, and the above implies  
 \begin{align*}
 (U\sm W_{J})\cap T&=\dot\bigcup_{i=1}^{n}((U_{i}\sm W_{J_{i}})\cap T),& 
 \Phi_{U}\sm \Phi_{J}&=\dot\bigcup_{i=1}^{n}(\Phi_{U_{i}}\sm \Phi_{J_{i}}),\\
 \corank_{U}(W_{J})&=\sum_{i=1}^{n}
 \corank_{U_{i}}(W_{J_{i}}),&
 [U:W_{J}]&=\prod_{i=1}^{n}[U_{i}
 \colon W_{J_{i}}].
 \end{align*}
Using these formulae,  the proof of Proposition \ref{finindcond} readily reduces to that of its special case in which $U$ is irreducible, in which case it follows by applying    Proposition \ref{fincond} to $(U,\chi(U))$ instead of $(W,S)$.
 \end{proof}
 
\subsection{Root poset} Define a relation $\lessdot$ on $\Phi$ such that $\rho \lessdot \tau$ if  $\rho,\tau\in \Phi$ and there exists $\alpha\in \Pi$ such that $\rho=s_{\alpha}\tau$ and $0\neq \tau-\rho\in \cone(\alpha)$. Let  $\leq$ denote  the preorder on $\Phi$  which is defined as the reflexive, transitive closure of the relation $\lessdot$. We call $(\Phi,\leq)$ the \emph{root poset} 
of $(\Phi,\Pi)$, as justified by the following proposition. This differs slightly from  \cite[Chapter 4]{BjBr},  where the restriction of $\leq$ to a partial order on $\Phi^{+}$
 is called the root poset. 
\begin{prop}\label{rootposet} 
\begin{num}\item The preorder $\leq $ on $\Phi$ is a partial order. 
\item The map $\rho\mapsto-\rho$ is an antitone (i.e. order reversing) bijection of the root poset with itself.
\item Define a function $\ell'\colon \Phi\to \mathbb{Z}$ by $\ell'(\epsilon\alpha)=\epsilon \ell(s_{\alpha})$ for $\alpha\in \Phi^{+}$ and $\epsilon\in \set{\pm 1}$. Then  $\ell'(\tau)$  is odd
and $\ell'(-\tau)=-\ell'(\tau)$ for all $\tau\in \Phi$.  
\item If $\rho,\tau\in \Phi$ with $\rho\lessdot \tau$, then $\ell'(\tau)=\ell'(\rho)+2$.  
\item  $\lessdot$ is the covering relation for $(\Phi,\leq)$. 
\item If $\tau\in \Phi^{+}$, then $\{\rho\in \Phi^{+}\mid \rho\leq \tau\}$ is finite.

\item For any $\rho\leq \tau$ in $\Phi$, the closed interval
$[\rho,\tau]:=\{\,\sigma\in \Phi\mid \rho\leq \sigma\leq \tau\,\}$ in  the root poset is finite.
\end{num}\end{prop}
\begin{proof}  For (a), we have to show $\leq$ is anti-symmetric. The definitions of $\lessdot$ and $\leq$ imply that  if $\rho\leq \tau$ in $\Phi$, then $\tau-\rho\in \cone(\Pi)$. Since
$\cone(\Pi)\cap -\cone(\Pi)=\{0\}$, (a) follows. Part (b) follows from the easily checked fact that if $\rho\lessdot \tau$ in $\Phi$, then $-\tau\lessdot -\rho$. Part (c) s trivial. 

We prove (d).  Let $\rho,\tau\in \Phi$ with $\rho\lessdot \tau$.
If $\rho$ and $\tau$ are both positive roots,  or both negative roots, the result follows from the definitions using  the notion of the depth of a positive root and its relation to length of the corresponding reflection (see \cite[Equation (1) and Proposition 2.4]{DFHM}). Otherwise, we have $\tau\in \Phi^{+}$ and 
$\rho=s_{\alpha}(\tau)\in \Phi_{-}$ for some $\alpha\in \Pi$, which implies $\tau=\alpha=-\rho\in \Pi$, $\ell'(\tau)=1$ and so $\ell'(\rho)=-1$. Part (e) follows easily from (d) and the definitions.

We prove (f). Fix $\tau\in \Phi^{+}$. Suppose that $\rho \lessdot \tau$ in $\Phi$ with $\rho\in \Phi^{+}$. Let  $\alpha\in \Pi$ with $\rho=s_{\alpha }\tau$ and $0\neq \tau-\rho\in \cone(\alpha)$.  Then $s_{\tau}=s_{\alpha}s_{\rho}s_{\alpha}$ with $\ell(s_{\tau})=\ell(s_{\rho})+2$ by (d). Hence $ s_{\rho}\leq' s_{\tau}$ where  $\leq'$ is Bruhat order of $(W,S)$.  It now follows that for any $\tau\in \Phi^{+}$, any $\rho\in \Phi^{+}$ with $\rho\leq \tau$ satisfies $1\leq' s_{\rho}\leq' s_{\tau}$. Since Bruhat intervals are finite and $\beta\mapsto s_{\beta}\colon \Phi^{+}\to T$ is injective, (f) follows.

 Finally,  (g) follows
from (b)--(c) on noting that
\begin{equation*}
[\rho,\tau]\seq\{\,\sigma\in \Phi^{+}\mid \sigma\leq \tau\,\}\cup
\{\,\sigma\in \Phi^{-}\mid \rho\leq\sigma\,\}.\qedhere\end{equation*}
\end{proof}
\begin{prop}  \label{rootorbit} Let $K\seq \Pi$ and let   $C_{K}:=\{\, v\in V\mid B(v,K)\seq \mathbb{R}_{\geq 0}\,\}$ denote the fundamental chamber for $W_{K}$ on $V$.  Let $\gamma\in \Phi^{+}\cap C_{K}$. Then:
\begin{num}
\item The stabilizer of $\gamma$ in $W_{K}$ is $W_{J}$ where $J:=\mset{\beta\in K\mid B(\gamma,\beta)=0}$.
\item If $\rho\in W_{K}\gamma$, then $\rho\leq \gamma$ in the root poset $(\Phi,\leq)$.
\item $\vert W_{K}\gamma\vert=[W_{K}\colon W_{J}]<\infty$.
\item $\gamma\in \Phi^{+}_{K}\iff W_{K}\gamma\cap \Phi^{-}\neq \eset\iff \gamma\in \Phi^{+}_{K}\sm \Phi^{+}_{J}$.
\item Suppose that $\gamma\in \Phi_{K}^{+}\cap C_{K}$.  Let $L$ denote the component of 
 $K$ such that $\gamma\in \Phi_{L}$. Then $L$ is of finite type, 
$\gamma\in \Phi^{+}_{L}\cap C_{L}$ and  $K=J\cup L$.  
\end{num} \end{prop} 
\begin{proof} Part (a) follows from Proposition \ref{CoxfundTits}.

To prove (b),  write $\rho=w\gamma$ where $w\in W_{K}^{J}:=W_{K}\cap W^{J}$. Choose a reduced expression
$w=s_{\alpha_{m}}\cdots s_{\alpha_{1}}$ with each $\alpha_{i}\in K$. We show  that 
\begin{equation*}
\rho=s_{\alpha_{m}}\cdots s_{\alpha_{1}}\gamma\lessdot\ldots \lessdot s_{\alpha_{1}}\gamma\lessdot \gamma.
\end{equation*} 
To prove this, it suffices to show that  for $i=1,\ldots, m$, 
\[0<B(\alpha_{i},s_{\alpha_{i-1}}\cdots s_{\alpha_{1}}(\gamma))=B(s_{\alpha_{1}}\cdots s_{\alpha_{i-1}}(\alpha_{i}),
\gamma).\] Let $\beta_{i}:=s_{\alpha_{1}}\ldots s_{\alpha_{i-1}}(\alpha_{i})$. We have to show $B(\beta_{i},\gamma)>0$. It is well known (see for instance \cite{Hu90}, \cite{BjBr}, \cite[Lemma 1.16]{Dy13}) that
$\Phi_{w^{-1}}=\mset{\beta_{1},\ldots, \beta_{m}}$. Since $w\in W^{J}_{K}$, we have $\Phi_{w^{-1}}\seq \Phi_{K}^{+}\sm \Phi_{J}$.  From $\beta_{i}\in \Phi_{K}^{+}$, we get $B(\beta_{i},\gamma)\geq 0$. Since  $\beta_{i}\not\in \Phi_{J}$, we have 
$s_{\beta_{i}}\in W_{K}\sm W_{J}$, so $s_{\beta_{i}}(\gamma)\neq \gamma$ and $B(\beta_{i},\gamma)\neq 0$.  Hence $B(\beta_{i},\gamma)> 0$, completing the proof of (b).

Note that (a) implies $\vert W_{K}\gamma\vert=[W_{K}\colon W_{J}]$ in (c). To complete the proof of (c), we shall  show that the orbit $W_{K}\gamma$ is finite. 

If the orbit  $W_{K}\gamma$ is contained in $\Phi^{+}$, it is finite by Proposition \ref{rootposet}(f). 
Otherwise, let $\rho\in W_{K}\gamma\cap \Phi^{-}$. We may write $\rho=w\gamma$ with $w\in W_{K}^{J}$ as in the proof of (b). In that argument,   there must exist
$i>0$ such that $\delta:=s_{\alpha_{i-1}}\cdots s_{\alpha_{1}}\gamma\in \Phi^{+}$ and $s_{\alpha_{i}}\cdots s_{\alpha_{1}}\gamma\in \Phi^{-}$. This implies that $\delta=\alpha_{i}\in K$.
Since $\delta\in W_{K}\gamma\cap K$, we get  \[W_{K}\gamma=W_{K}\delta=(W_{K}s_{\delta})\delta=W_{K}(s_{\delta}\delta)=-W_{K}\delta=-W_{K}\gamma.\]
Since $W_{K}\gamma$ has maximum element $\gamma$ in
 $\leq$,  it also has minimum element $-\gamma$ by
  Proposition \ref{rootposet}(c).   Thus, 
$W_{K}\gamma\seq \{\,\tau\in \Phi\mid -\gamma\leq \tau\leq 
\gamma\,\}$, which is finite by Proposition \ref{rootposet}(g). 
 This proves (c). 
 
 We prove (d). If $\gamma\in \Phi_{K}^{+}$, then 
 $-\gamma=s_{\gamma}\gamma\in W_{K}\gamma\cap \Phi_{-}$.
Next, suppose $W_{K}\gamma\cap \Phi_{-}\neq \eset$. From 
the proof of (c), there exists $w\in W_{K}^{J}$ with $w\gamma\in \Phi^{-}$. Then $\gamma\in  \Phi_{w^{-1}}\seq \Phi_{K}^{+}\sm \Phi_{J}^{+}$.  Trivially, the rightmost condition in (d) implies the leftmost one.  This completes the proof of (d).

We prove (e).   By (c), $[W_{K}\colon W_{J}]$ is finite, so Proposition \ref{finindcond} implies that  the component $\Phi_{L}$ of $\Phi_{K}$ containing $\gamma$ is of finite type. Every component of $K$ other than $L$  is orthogonal to $L$, hence contained in $J$. This shows that $K=J\cup L$. Since $L\seq K$, we have $\gamma\in C_{K}\seq C_{L}$.     \end{proof}

\begin{proof}[Proof of Proposition \ref{infred}]
We prove (a). Let $M:=\Pi\cap L^{\perp}$. 
We will show that $\Phi\cap L^{\perp}=\Phi_{M}$. Clearly
$\Phi_{M}\seq \Phi\cap L^{\perp}$, since $M\seq \Phi\cap L^{\perp}$. For the proof of the reverse inclusion,  let  $\alpha\in \Phi^{+}\cap L^{\perp}$. We show
$\alpha\in \Phi_{M}$ by induction on $\ell(s_{\alpha})$. 
If $\ell(s_{\alpha})=1$, then $\alpha\in \Pi\cap L^{\perp}=M\seq \Phi_{M}$. 

Suppose $\ell(s_{\alpha})>1$. Since $\alpha\in \cone(\Pi)$ and $B(\alpha,\alpha)=1$, there exists some  $\beta\in \Pi$ with $B(\alpha,\beta)>0$.  Let $K:=L\cup\set{\beta}$  Then $\alpha\in 
\Phi^{+}\cap C_{K}$. Note $L=\mset{\gamma\in K\mid B(\alpha,\gamma)=0}$ and $\beta\in  K\sm L$.  By Proposition \ref{rootorbit}, the index $[W_{K}:W_{L}]$ is finite, and therefore by Proposition \ref{finindcond}, the component of $K$ containing
$\beta$ is of finite type. Since each component of $L$ is of infinite type,
$\beta$ cannot be joined  to any vertex of  $L$  in the Coxeter graph of  $K$. Hence $\beta\in \Pi\cap L^{\perp}=M$.

Now from $\beta\in \Pi\cap L^{\perp} $ and $\alpha\in 
\Phi\cap L^{\perp}$, follows that $s_{\beta}(\alpha)\in 
\Phi\cap L^{\perp}$.  Since $s_{\beta}(\alpha)\lessdot \alpha$ where $\ell(s_{\alpha})\geq 3$, we also have $\ell(s_{s_{\beta}\alpha})=\ell(s_{\alpha})-2\geq 1$ and $s_{\beta}(\alpha)\in \Phi^{+}$, by Proposition \ref{rootposet}. Induction gives $s_{\beta}\alpha\in \Phi^{+}_{M}$ and so $\alpha\in \Phi_{M}$, proving (a)

The proof of (b) follows the same general lines as  that of (a). 
Let $U:=\mset{v\in V\mid \text{\rm  $B(v,\gamma)\neq 0$ for only finitely many $\gamma\in L$}}$.
 Let $M:=\Pi\cap U$.  We show $\Phi\cap U=\Phi_{M}$. Clearly
$\Phi_{M}\seq \Phi\cap U$, since $M\seq \Phi\cap U$. For the proof of the reverse inclusion,  let  $\alpha\in \Phi^{+}\cap U$. We show
$\alpha\in \Phi_{M}$ by induction on $\ell(s_{\alpha})$. 
If $\ell(s_{\alpha})=1$, then $\alpha\in \Pi\cap U=M\seq \Phi_{M}$. 

Suppose $\ell(s_{\alpha})>1$. Since $B(\alpha,\alpha)>0$, we may  choose  $\beta\in \Pi$ with $c:=B(\alpha,\beta^{\vee})>0$. 
Then  $s_{\beta}(\alpha)=\alpha-c\beta\in \Phi^{+}$.  We claim that $\beta\in M$. Suppose to the contrary that $\beta\not \in M$.
For $\gamma\in L$, we have 
\begin{equation}\label{ipequation}
 B(s_{\beta}(\alpha),\gamma)=B(\alpha-c
\beta,\gamma)=B(\alpha,\gamma)-cB(\beta,\gamma).
\end{equation} Note that  $B(\alpha,\gamma)\neq 0$ for only 
finitely many $\gamma\in L$ since $\alpha\in U$. Since 
$\beta\in \Pi$ but  $\beta\not \in M$, we have $\beta\not \in U$. 
That is,  $B(\beta,\gamma)\neq 0$ for infinitely many 
$\gamma\in L\seq \Pi$. But $B(\beta,\gamma)\leq 0$ for all 
$\gamma\in \Pi\sm\{\beta\}$. Hence 
$B(\beta,\gamma)< 0$ for infinitely many $\gamma\in L$. By 
\eqref{ipequation}, it follows that 
$K:=\mset{\gamma\in L\mid B(s_\beta(\alpha),\gamma)>0}$ is 
infinite. Then $s_{\beta}(\alpha)\in \Phi^{+}\cap C_{K}$. We 
have $J:=\mset{\gamma\in K\mid 
B(s_\beta(\alpha),\gamma)=0}=\eset$. Proposition \ref{rootorbit} 
implies that $\vert W_{K}\vert =[W_{K}\colon W_{J}]$ is finite, 
contrary to the fact that $K\seq W_{K}$ is infinite.  This 
contradiction shows that $\beta\in M$ as claimed.

Since $\alpha\in U\cap \Phi^{+}$ and $\beta\in M= U\cap \Pi$ with $\ell(s_{\alpha})\geq 3$, it follows that $s_{\beta}(\alpha)\in \Phi^{+}\cap U$. By induction, $s_{\beta}(\alpha)\in \Phi_{M}$. Then since $\beta\in M$, we get $\alpha\in \Phi_{M}$ as required to complete the proof of (b). \end{proof}

\begin{prop}\label{posdef} Let $J\seq \Pi$ be such that the components
$(J_{i})_{i\in I}$ of $J$ are all of finite type.
\begin{num}
\item The restriction of the form  $B$ to $\Span(J)$ is positive definite. More precisely,   $\Span(J)=\bigoplus_{i\in I}\Span(J_{i})$ is an orthogonal direct sum of finite-dimensional subspaces 
$\Span(J_{i})$, on each of which the restriction of $B$ is positive definite.
\item $J$ is linearly independent.
\item $\cone(\Pi)\cap \Span(J)=\cone(J)$, $\Phi_{J}= \Phi\cap \Span(J)$ and $\Phi^{+}_{J}= \Phi\cap \cone(J)$.
\item $\cone(\Pi\sm J)\cap \Span(J)=\set{0}$.
\end{num}
\end{prop}
\begin{proof} Note that $\Span(J_{i})$ and $\Span(J_{i'})$ are orthogonal for $i\neq i'$ in $I$, since $J_{i}$ and $J_{i'}$ are orthogonal. Parts (a)--(b)  then follow from the well-known facts that
for any $i\in I$,  finiteness of $W_{J_{i}}$ implies that  $J_{i}$ is linearly independent
and the restriction of $B$ to $\Span(J_{i})$ is positive definite
(see for instance \cite{Bou}, \cite{Kra} or \cite{Dy13}).

Parts (c)--(d) follow from Proposition \ref{facialcomp}.
 \end{proof}
 \begin{proof}[Proof of relative  Tits theorem (Theorem \ref{infTits})] 
\label{infTitsproof} The assertions of the theorem are independent of the choice of root system, so we may assume $\Pi$ is linearly independent.     By \ref{ext}, we may also assume without loss of generality that $(V,B)$ is ample for $(\Phi,\Pi)$. That is,   every linear function $\Span(\Pi)\to \mathbb{R}$ is of the form $\alpha\mapsto B(\alpha,v)$ for some $v\in V$.
We prove that (i) implies (ii). The argument is essentially the same as that of Tits proof in the case $J=\eset$.  For the proof, we may  replace $W$ by any parabolic subgroup $U$  of $W$ which contains $W'$. Fix a set of  left coset representatives  $\set{w_{1},\ldots, w_{n}}$, where $n:=[W':W_{J}]$ and $w_{1}=1_{W}$,  for $W_{J}$ in $W'$. Thus,\begin{equation*}
W'=w_{1}W_{J}\dot\cup\ldots\dot\cup w_{n}W_{J}. 
\end{equation*}     
We replace $W$ if necessary by its standard parabolic subgroup generated by $J$ and the (finitely many) simple reflections $s$ in $S$ for which there exists some $i=1,\ldots, n$ such that $s$ appears  in  a reduced expression for  $w_{i}$.  Thus, without loss of generality, we assume $\corank_{W}(W_{J})<\infty$. We prove the result by induction on this corank.

Choose $\rho_{J}\in C$ so $B(\rho_{J},\alpha)=0$ for all
$\alpha\in J$ and $B(\rho_{J},\alpha)>0$ for all $\alpha\in \Pi\sm J$. Then the stabilizer in $W$ of $\rho_{J}$ is $W_{J}$.
Define
\begin{equation*}
\rho_{J}':=w_{1}\rho_{J}+\dotsb +w_{n}\rho_{J}\in V.
\end{equation*}
Now  $\rho'_{J}$ is fixed by $W'$. For if $w\in W'$, there exists a
permutation $\sigma$ in the symmetric group  $S_{n}$ and elements $u_{i}\in W_{J}$ such that
\begin{equation*}
ww_{i}=w_{\sigma i}u_{i}, \qquad i=1,\ldots,n
\end{equation*} 
Then
\begin{equation*}
w\rho_{J}'=w\sum_{i=1}^{n}w_{i}\rho_{J}=
\sum_{i=1}^{n}ww_{i}\rho_{J}=\sum_{i=1}^{n}w_{\sigma i}u_{i}\rho_{J}= \sum_{i=1}^{n}w_{\sigma i}\rho_{J}=\rho'_{J}
\end{equation*} as required. We also have $\rho_{J}'\in X$, the Tits cone of $W$ on $V$, so the stabilizer $U$ of $\rho_{J}'$ is a parabolic subgroup of $W$ with $U\sreq W'\sreq W_{J}$.
If the corank of $W_{J}$ in $U$ is strictly less than the corank of $W_{J}$ in $W$,   the desired conclusion follows by induction.
Suppose it is not strictly less. Then it  is equal, so the corank of $U$ in $W$ is $0$ by Lemma \ref{corank}, which forces $U=W$.
This means $\rho_{J}'\in \Pi^{\perp}=C\cap -C\seq -X$.
Since $X\sreq C$ is a $W$-stable cone with $\rho_{J}\in X$.   
\begin{equation*}
\rho_{J}=w_{1}\rho_{J}=\rho'_{J}-(w_{2}\rho_{J}+\dotsb+ w_{n}\rho_{J})\in  -X
\end{equation*}
Proposition \ref{CoxfundTits} implies $B(\rho_{J},\alpha)>0$ for only finitely many $\alpha\in \Phi^{+}$. But  since $\Pi$ is linearly independent,  we have  $\Phi^{+}\sm\Phi_{J}^{+}\seq \cone(\Pi)\sm\Span(\Pi_{J})$. By definition of $\rho_{J}$,    $B(\rho_{J},\Phi^{+}\sm\Phi_{J}^{+})\seq \mathbb{R}_{>0}$.  Therefore $\Phi^{+}\sm \Phi_{J}^{+}$ is finite.  
By Proposition \ref{finindcond},  this  proves that 
$[W:W_{J}]<\infty$ and therefore that (i) implies (ii). 

Now we prove (ii) implies (iii). Assume (ii) holds. 
We may write $U=dW_{K}d^{-1}$ for some $K\seq \Pi$ and $d\in W$.
Without loss of generality, we may assume that 
$d\in W^{K}$. Then $\Pi_{U}=d\Pi_{W_{K}}=dK$.
We have $U\sreq W'\sreq W_{J}$. Hence 
$J\seq \Pi_{U}=dK$, 
$L:=d^{-1}J\seq K$ and $d^{-1}W_{J}d\seq d^{-1}W'd\seq d^{-1}Ud=W_{K}$. Taking $w=d^{-1}$,  the conditions of (iii) are satisfied.

One has   (iii) implies (i), since given (iii),
one has \[[W':W_{J}]=[wW'w^{-1} : wW_{J}w^{-1}]=[wW'w^{-1} : W_{L}]\] and 
\[[W_{K}:wW'w^{-1}][wW'w^{-1}: W_{L}]=[W_{K}:W_{L}]<\infty.\]

The statement in the theorem concerning existence of  $U_{0}$ follows directly from Proposition \ref{parsub}(a).  The final statement of the theorem   is  part of (c) of  the following proposition.
\end{proof}

\begin{prop}\label{refinfTits} Let $J\seq \Pi$ and $W'\sreq W_{J}$ be a reflection subgroup of $W$ such that $[W':W_{J}]$ is finite. Let  $U_{0}:=\pc(W')$ (see \ref{parsub}).  Then:\begin{num}
\item  $[U_{0}:W_{J}]<\infty$,
 \item  $U_{0}$ is a  parabolic subgroup of $W$.  
 \item $\corank_{U_{0}}(W_{J})=\corank_{W'}(W_{J})<\infty$. 
 \end{num} 
    \end{prop}
\begin{proof} By Theorem \ref{infTits},  there is a finite index overgroup of $W_{J}$ which contains $W'$ and is a parabolic 
subgroup  of $W$.  By replacing $W$ by that overgroup, we may assume  without loss of generality for the rest of the proof    that $[W:W_{J}]<\infty$. This immediately implies (a). Part (b)  holds by Proposition \ref{parsub}(c).

Let $E$ denote the subspace of $V$ consisting of all 
$v\in(J_{\inft})^{\perp}$ such that $B(v,\alpha)\neq 0$ for only 
finitely many $\alpha\in J_{\fin}$. By Proposition \ref{infred}, 
$E\cap \Phi=\Phi_{M}$ is a standard parabolic subsystem
of $\Phi$ where $M:=E\cap \Pi$.
 We have $J_{\fin}\seq M$, $J_{\inft}\cap M=\eset$ and
 $\Pi\sm J\seq M$ by Proposition \ref{finindcond}. Hence $M=
 J_{\fin}\,\dot\cup\, (\Pi\sm J)=\Pi\sm J_{\inf}$. 
 Proposition \ref{finindcond}
 shows further now that $M=\Pi_{\fin}$ and $\Pi_{\inft}=J_{\inft}$.
 
 Let $K$ denote the union of those components of $\Pi_{W'}$ which contain a root in $\Pi_{W'}\sm J$.
Note that  $K$ is of finite type and $K\seq \Phi_{M}$, by Proposition \ref{finindcond}. Hence $\Phi_{K}\seq \Phi_{M}$. Choose  a finite union of components $N$ of $M$ so that $\Phi_{N}\sreq \Phi_{K}\cup (\Pi\sm J)$.  We have $N\perp (M\sm N)$ and $N\seq M\perp (J\sm M)$ so $N\perp (J\sm N)$.

 Let $L:=K\cup (N\cap J)\seq \Pi_{W'}$, so $W_{L}$ is a finite reflection subgroup of $W_{N}$ with $\Pi_{W_{L}}=L$. Consider the parabolic closure $\pc(L)\seq W_{N}$ of $L$.
It is a finite parabolic subgroup of $W_{N}$, and we write $P:=\Pi_{\pc(L)}\sreq N\cap J$ for its canonical simple system.
Let $Q=P\cup (J\sm N)$. We have $Q\sreq (N\cap J)\cup (J\sm N)=J$ and $W_{Q}\sreq W_{P}\sreq  W_{L}\sreq W_{K}\sreq K\sreq \Pi_{W'}\sm J$,
so $W_{Q}\sreq \Pi_{W'}$. We claim that $U_{0}:=W_{Q}$ is a parabolic subgroup
and that $\corank_{U_{0}}(W_{J})=\corank_{W'}(W_{J})$.

Since $P=\Pi_{\pc(L)}$ is parabolic in $W_{N}$, there exists  $d\in W_{N}$ with $dP\seq N$.  We therefore  have
$d(N\cap J)\seq N$.  Since $d\in W_{N}$ where $ N\perp(J\sm N)$, we have $d(J\sm N)=J\sm N\seq \Pi$. 
It follows that $dQ\seq \Pi$, which implies that $W_{Q}$ is parabolic and $\Pi_{W_{Q}}=Q$.

By definition of $K$, we have  $\corank_{W'}(W_{J})=\vert \Pi_{W'}\sm J \vert=\vert K\sm (K\cap J)\vert$. Note that since $N\sreq P\sreq J\cap N$, we have  $J\cap N\sreq J\cap P\sreq J\cap N$ and so $J\cap N=J\cap P$. Using $Q=P\cup(J\sm N)$,
$\vert P\vert =\rank(\pc(L))=\vert \Pi_{W_{L}}\vert =\vert L\vert$,  $L:=K\cup (N\cap J)$, we find \begin{equation*}
\begin{split}&\corank_{U_{0}}(W_{J})=\vert \Pi_{U_{0}}\sm J\vert =
\vert Q\sm J\vert=\vert((P\cup (J\sm N))\sm J)\vert=\vert  P\sm (J\cap P)\vert \\
=& \vert P\vert -\vert J\cap P\vert =\vert L\vert -\vert J\cap N\vert =(\vert K\vert +\vert N\cap J\vert -\vert K\cap (N\cap J)\vert) -
\vert J\cap N\vert\\ =&\vert K\vert  -\vert (K\cap J)\cap N)\vert
=\vert K\vert-\vert K\cap J\vert =\corank_{W'}(W_{J})  \end{split}
\end{equation*}
since $W_{K}\seq W_{N}$ where $N\seq \Pi$ implies  $ K\cap J\seq N$. 

To complete the proof, it suffices to show that $U_{0}=\pc(W')$.
But since $W_{L}\seq W'$, we have $\pc(W')\sreq \pc(W_{L})=W_{P}$. Also, $\pc(W')\sreq J\sm N$, so $\pc(W')\sreq W_{Q}\sreq W'$. We have $\pc(W')= W_{Q}$  since $W_{Q}$ is parabolic.
\end{proof}

  \subsection{Maximal corank $k$ reflection overgroups} 
 \label{maxrankk}  We next prove Theorem \ref{maxref}, which generalizes facts about rank $k$ reflection subgroups in \cite{Dy21}. For convenience, subsections \ref{y2.4}--\ref{y2.10} recall without proof the definitions and technical  facts  from \cite{Dy21} concerning  ``lifts'' of root systems, which are used  in the proofs here.  
 \subsection{} \label{y2.4} 
 The category $C$ of lifts of root systems has as objects  the  quadruples $D=(V_{0},B_{0},\Phi_{0},\Pi_{0})$ where $(\Phi_{0},\Pi_{0})$ is a based 
root system for the  real quadratic space $(V_{0},B_{0})$. A morphism 
$L\colon D_{1}\to D_{2}$ in $C$, where 
$D_{i}=(V_{i},B_{i},\Phi_{i},\Pi_{i})$, is by definition a morphism 
$L\colon (V_{1},B_{1})\to (V_{2},B_{2})$ of quadratic spaces (that is,  a  $\mathbb{R}$-linear map $L\colon V_{1}\to V_{2}$ which satisfies  $B_{2}(L(u),L(v))=B_{1}(u,v)$ for all $u$ and $v$ in $V_{1}$) such that the function $L\colon V_{1}\to V_{2}$ restricts to a bijection $\Pi_{1}\to \Pi_{2}$. Composition of morphisms is  given by composition of the underlying linear maps. We call $D_{1}$ above (or more precisely, the pair $(D_{1},L)$) a lift of $D_{2}$.

Consider a morphism $L\colon D_{1}\to D_{2}$ in $C$ as above.
Let  $(W_{i},S_{i})$ and $\Phi_{i}^{+}$ 
denote respectively  the Coxeter system and positive roots attached to 
$D_{i}$, for $i=1,2$. The map defined by  $s_{\alpha}\mapsto s_{L(\alpha)}\colon S_{1}\to S_{2}$, 
for $\alpha\in \Pi_{1}$, is a bijection which extends uniquely to a group 
isomorphism $\theta\colon W_{1}\to W_{2}$. That is,  $\theta\colon 
 (W_{1},S_{1})\to (W_{2},S_{2})$  is an isomorphism of Coxeter systems.
 Unless otherwise specified,  we identify $(W_{2},S_{2})=(W_{1},S_{1})$ by means of $\theta$, so $\theta=\mathrm{Id}_{W_{1}}$.
One then has $L(w(u))=w(L(u))$ for all $w\in W_{1}$ and 
 $u\in V_{1}$. Further, $L$ restricts to bijections $\Phi_{1}\to \Phi_{2}$ and 
 $\Phi_{1}^{+}\to \Phi_{2}^{+}$ satisfying $s_{L(\alpha)}=s_{\alpha}$ for 
 all $\alpha\in \Phi_{1}$. 

\subsection{} \label{y2.5} For any fixed object $D=(V,B,\Phi,\Pi)$ of $C$, there exists  an object
 $\wt D=(\wt V,\wt B,\wt \Phi,\wt \Pi)$ of $C$ and a morphism
 $L\colon \wt D\to D$ such that $\wt \Pi$ is a $\mathbb{R}$-basis of $\wt V$.  The pair $(\wt D,L)$ is uniquely determined by $D$ up to isomorphism as an object of the category of objects  of $C$ over $D$; we call  $\wt D$ (with $L$ understood), the universal lift of $D$, and call $\wt \Phi$ a universal lift of $\Phi$. A universal lift of a lift  of $D$ is a universal lift of $D$.
 
 \subsection{} \label{y2.6} In the following, consider a based root system
 $(\Phi,\Pi)$, with associated Coxeter system $(W,S)$, in a quadratic space $(V,B)$.
  Note that $D=(V,B,\Phi,\Pi)$ is an object of the category $C$. 
 For any reflection subgroup $W'$ of $W$, there is an object
 $D_{W'}:=(V,B,\Phi_{W'}, \Pi_{W'})$ of $C$.  
 
 Consider a lift $(D',L)$ of $D_{W'}$ where 
 $D'=(V',B',\Psi,\Delta)$. 
 A reflection subgroup $U$ of $W'$
 determines from $D$ an object $D_{U}=(V,B,\Phi_{U},\Pi_{U})$ of $C$, in the same way as $D_{W'}$ is determined from $D$ by  $W'$. 
 Also, $U$   determines from $D'$ an object   $D'_{U}=(V',B',\Psi_{U},\Delta_{U})$ of $C$ in the same way as $D_{W'}$ is determined from $D$ by  $W'$,  We have $\Psi_{U}\seq \Psi$ and $\Delta_{U}\seq \Psi^{+}$,   
 and  $(D'_{U},L)$ is a lift of $D_{U}$

  To lighten notation and terminology, we shall say  simply that  the root system $\Psi$ is a lift
 of $\Phi_{W'}$, and that  $\Psi_{U}$ is the \emph{induced lift} of $\Phi_{U}$.  If the  simple roots of $\Psi$ are denoted $\Delta$, 
 we denote the simple roots of $\Psi_{U}$ as $\Delta_{U}$, as above.
 
 \begin{lem}\label{y2.7} Let $\Psi$ be a lift, with simple roots $\Delta$,  of $\Phi$, and let $U\subseteq U'$ be reflection subgroups of $W$. 
 \begin{num}
 \item If $\Pi_{U}$ is linearly independent, then $\Delta_{U}$ is linearly independent. 
  \item If $\Phi_{U}=\Phi\cap \Span(\Phi_{U})$, then 
  $\Psi_{U}=\Psi\cap \Span(\Psi_{U})$.
 \item If $\Psi_{U'}\subseteq \Span(\Psi_{U})$, then $\Phi_{U'}\subseteq \Span(\Phi_{U})$.
 \end{num}
 \end{lem}
 
\subsection{}\label{y2.9} 
 Let $\mathcal{W}$ denote the (inclusion-ordered)  complete lattice of reflection subgroups of $W$. The join  of a subset $X=\mset{W_{i}\mid i\in I}$ of $\mathcal{W}$ is denoted as $\bigvee X$ or $\bigvee_{i\in I} W_{i}$.  
 For any subset $P$ of $\mathcal{W}$ such that $P$ has  a maximum element (in the induced order by inclusion), we  denote that maximum element   by $\max(P)$. 
 
 For any reflection subgroup $U$ of $W$, say that $W$ is \emph{$\Phi$-spanned by $U$} if 
 $\Phi\subseteq \Span(\Phi_{U})$. Since $\Phi=\Phi_{W}$,  $W$   is trivially  $\Phi$-spanned by $W$. 
 Introduce a relation $\leq$ on $\mathcal{W}$ be letting $U\leq W'$, for $U,W'\in \mathcal{W}$, if $U\subseteq W'$ and  for every lift  $\Psi$ of  $\Phi_{W'}$, $W'$ is $\Psi$-spanned by $U$ (that is, $\Psi\subseteq \Span(\Psi_{U})$). Equivalently by Lemma \ref{y2.7}, $U\leq W' $ if  $U\subseteq W'$ and $W'$ is $\Psi$-spanned by $U$ where    $\Psi$ is a universal lift of  $\Phi_{W'}$.
 
  Define a map $c\colon \mathcal{W}\to \mathcal{W}$ as follows.
 For each reflection subgroup $U$ of $W$, let $P_{U}:=\{W'\in \mathcal{W}\mid U\leq W'\}$ and  $c(U):=\bigvee_{W'\in P_{U}}W'$. 

 \begin{lem}\label{y2.10}\begin{num}\item The relation  $\leq$ is a partial order on $\mathcal{W}$.
 \item   $c$ is a closure operator on $\mathcal{W}$. 
 \item  If $U\in \mathcal{W}$, then  $c(U)=\max(P_{U})$. 
\end{num}\end{lem}
 
\begin{prop}\label{y2.13}  Let $J\seq \Pi$ and  $U\in \mathcal{W}$ with $U\sreq W_{J}$. Assume $\corank_{U}(W_{J})=k\in \mathbb{N}$  and    that $\Pi_{U}$ is linearly independent.  Define  \[P_{U}':=\{W'\in \mathcal{W}\mid W'\supseteq U,\corank_{W'}(W_{J})=k\}.\]  Then
\begin{num}
\item  For any $W'\in \mathcal{W}$ with $W'\supseteq U$, one has $\corank_{W'}(W_{J})\geq \corank_{U}(W_{J})$.
\item $P_{U}=P'_{U}$, and $c(U)=\max(P_{U}')$.
\item If $W'\in P_{U}$, then $\Pi_{W'}$ is linearly independent. 
  \end{num}   \end{prop}

\begin{proof}
We prove (a). Let $W'\in \mathcal{W}$ with $W'\supseteq U$.
 Then $\Span(\Pi_{U})\subseteq \Span(\Phi_{W'})=\Span(\Pi_{W'})$. Since $\Pi_{U}$ is linearly independent by assumption, $\Pi_{U}\sreq J$ and
 $\Pi_{W'}\sreq J$ imply
 \begin{equation*}\begin{split}
 \corank_{U}(W_{J})&=\vert \Pi_{U}\sm J\vert =\dim(\Span(\Pi_{U})/\Span(J))\\&\leq
 \dim(\Span(\Pi_{W'})/\Span(J))\leq  \vert \Pi_{W'}\sm J\vert =\corank_{W'}(W_{J}).\end{split}
 \end{equation*}
 
 In the rest of the proof, for a reflection subgroup denoted $W'$, we   denote a universal lift of $\Phi_{W'}$ as $\Psi$, and the (linearly independent) set of   simple roots of $\Psi$ as  $\Delta$.
 
We prove  (b).  We first   show that $P_{U}\subseteq P'_{U}$.  Let $W'\in P_{U}$. Then $U\leq W'$. 
  We have  $U\subseteq W'$ and  $\Psi_{U}\subseteq \Psi\subseteq  \Span(\Psi_{U})$, so 
$\Span(\Psi)= \Span(\Psi_{U})$.
We may identify the induced lift $\Psi_{W_{J}}$ of $\Phi_{J}$
with $\Phi_{J}$ since $J\seq \Pi_{U}$ is linearly independent.
 In particular, we make the identification  $J=\Pi_{W_{J}}=\Delta_{W_{J}}$, so
 $J\seq \Delta_{U}$ also.
Since $\Delta$ and $\Delta_{U}$ are linearly independent and both contain $J$, we have \begin{equation*}\begin{split}
\corank_{W'}(W_{J})&=\vert \Delta\sm J\vert=\dim(\Span(\Delta)/\Span(J)) =\dim(\Span(\Psi)/\Span(J))\\&=\dim(\Span(\Psi_{U})/\Span(J))=
\dim(\Span(\Delta_{U})/\Span(J))\\
&= \vert \Delta_{U}\sm J\vert=\corank_{U}(W_{J})=k.
\end{split}
\end{equation*}
Hence  $W'\in P'_{U}$, which shows that
$P_{U}\subseteq P'_{U}$.

To prove the reverse inclusion $P'_{U}\subseteq P_{U}$, let 
$W'\in P'_{U}$. That is, $W'\in \mathcal{W}$,  $U\subseteq W'$
and $\corank_{W'}(W_{J})=k$.  To prove that $W'\in P_{U}$ (that is, $U\leq W'$) it will suffice to show that $W'$ is $\Psi$-spanned by ${U}$, 
since   $\Psi$ is  a universal lift of $\Phi_{W'}$. 
Since $\Pi_{U}$ is linearly independent, it follows that $\Delta_{U}$ is linearly independent. Let $\pi\colon \Span(\Delta)\to \Span(\Delta)/\Span(J)$ be the canonical map.  Note that 
\[\Span(\Delta)/\Span(J)=\Span(\pi(\Delta\sm J))\] has basis $\pi(\Delta\sm J)$ and  dimension $k$, since $\Delta$ is linearly independent, $\Delta\sreq J$ and $\vert \Delta\sm J\vert= \corank_{W'}(W_{J})=k$.
Since $\Delta_{U}\sreq J$ is linearly independent,  and 
$\vert \Delta_{U}\sm J\vert=\corank_{U}(W_{J})=k$, it follows that
$\pi(\Delta_{U}\sm J)$ is another basis for $\Span(\Delta)/\Span(J)$.
In particular, $\pi(\Psi)\seq \Span(\pi(\Delta_{U}\sm J))$, which implies \[\Psi\seq \Span(\Delta_{U}\sm J)+\ker \pi=\Span(\Delta_{U}\sm J)+\Span(J)=\Span(\Delta_{U})\seq \Span(\Psi_{U}).\]
 Thus,  $W'$ is $\Psi$-spanned by $U$ as claimed.  
This shows that $P_{U}=P_{U}'$. Then $c(U)=\max(P'_{U})$ follows  from Proposition \ref{y2.10}(c), completing the proof of (b). 

Finally, we prove (c). Let $W'\in P_{U}$.
Since $\Phi_{W'}$ is (trivially) a lift of $\Phi_{W'}$, it follows by definition of $P_{U}$ that
$\Span(\Phi_{U})=\Span(\Phi_{W'})$. That is, $\Span(\Pi_{U})=\Span(\Pi_{W'})$.  Let $\pi\colon \Span(\Pi)\to \Span(\Pi)/\Span(J)$ be the projection.
We have \begin{equation*}
\Span(\pi(\Pi_{U}\sm J))=\Span(\Pi_{U})/\Span (J)=\Span(\Pi_{W'})/\Span(J)=\Span(\pi(\Pi_{W'}\sm J))
\end{equation*} But since  $\Pi_{U}\sreq J$ is linearly independent, $\pi(\Pi_{U}\sm J)$ is linearly independent, of cardinality $\vert \Pi_{U}\sm J\vert =\corank_{U}(W_{J})=k$. Also, $\Pi_{W'}\sm J$ has cardinality $\vert \Pi_{W'}\sm J\vert=\corank_{W'}(W_{J})=k$
since $W'\in P'_{U}$ by (b). It follows that $\pi(\Pi_{W'}\sm J)$ is linearly independent and hence 
$(\Pi_{W'}\sm J)\cup J=\Pi_{W'}$ is linearly independent since $J$, as a subset of the linearly independent set $\Pi_{U}$, is a basis of $\Span(J)=\ker \pi$. 
   \end{proof} 
    \begin{thm}\label{y1.8} Let $J\seq \Delta$  and  $U\sreq W_{J}$ be a reflection subgroup of $W$.    Assume that  $\corank_{U}(W_{J})=k\in \mathbb{N}$ and that  $\Pi_{U}$ is linearly independent.
  \begin{num}\item 
  The set $P'_{U}$ of all reflection subgroups $W'$ of $W$ which contain $U$ and satisfy $\corank_{W'}(W_{J})=k$ has a maximum element
  $c(U)$ under inclusion.
\item Suppose further that $\Phi_{U}=\Phi\cap \Span(\Phi_{U})$.    
Then $c(U)=U$. Let $I$ be the set of all reflection subgroups which are inclusion maximal elements of   the set of all reflection subgroups $W'$  of $W$  which contain $U$ and satisfy $\corank_{W'}(W_{J})=k+1$. Then  $\Phi\setminus \Phi_{U}=\dot\bigcup_{W'\in I}(\Phi_{W'}\setminus\Phi_{U})$.    \end{num}
     \end{thm}

 \begin{proof}\label{y2.14}
 Theorem \ref{y1.8}(a) follows directly from Propositions \ref{y2.13}(b) and \ref{y2.10}(c).
 
 We prove Theorem \ref{y1.8}(b). Fix $U\in \mathcal{W}$
 such that $\corank_{U}(W_{J})=k\in \mathbb{N}$, $\Pi_{U}$ is linearly independent and $\Phi_{U}=\Phi\cap \Span(\Phi_{U})$.
 We show first that $c(U)=U$. To see this, let $W'\in P_{U}$.  Then $U\leq W'$, so in particular,   $U\subseteq W'$ and  $U$ $\Phi$-spans $W$. By assumption, this implies  $\Phi_{W'}\subseteq\Span(\Phi_{U})\cap \Phi=\Phi_{U}$ and so 
 $ W'\subseteq U\subseteq  W'$, showing $W'=U$.
 Thus, $P_{U}=\{U\}$ and $c(U)=\max(P_{U})=U$ as required.

  We claim that for any $\alpha\in \Phi\setminus \Phi_{U}$,
  the reflection subgroup $W':=\langle U\cup \{s_{\alpha}\}\rangle$ satisfies $\corank_{W'}(W_{J})=k+1$.
  By Proposition \ref{y2.13}(a), $\corank_{W'}(W_{J})\geq k$.
  The inequality must be strict, since $W'\supsetneq U$ and  
  Proposition \ref{y2.13}(b) guarantees that $U$ is an inclusion maximal refelction subgroup with $\corank_{U}(W_{J})=k$.  Also,  \[W'=\langle \chi(U)\cup \{s_{\alpha}\}\rangle=
  \langle \mset{s_{\beta}\mid \beta\in J}\cup(\chi(U)\sm \mset{s_{\beta}\mid \beta\in J}) 
  \cup\{s_{\alpha}\}\rangle,\] so 
  \begin{equation*}
  \corank_{W'}(W_{J})\leq \vert  (\chi(U)\sm \mset{s_{\beta}\mid \beta\in J}) \vert+1 =\corank_{U}(W_{J}) +1=k+1,
  \end{equation*} proving the claim.  
  
  We claim that for any reflection subgroup  $W'\supseteq U$ such that  $\corank_{W'}(W_{J})=k+1$,  the set $\Pi_{W'} $ is linearly independent.
  For we have \begin{equation*}
  \Span(\Pi_{W'})=\Span(\Phi_{W'})\supseteq 
  \Span(\Phi_{U})=\Span(\Pi_{U}).
  \end{equation*}  Moreover, the containment $\Span(\Phi_{W'})\supseteq 
  \Span(\Phi_{U})$ is strict. For otherwise, we would have 
  $\Phi_{W'}\subseteq \Phi\cap \Span(\Phi_{U})=\Phi_{U}$
  and so $W'\subseteq U\subseteq W'$, which would imply that $W'=U$ contrary to $\corank_{W'}(W_{J})=k+1>k=\corank_{U}(W_{J})$.  
   Since $J\seq \Pi_{U}$ and $J\seq \Pi_{W'}$, it follows that 
  \begin{equation*}
  \Span(\Pi_{W'})/\Span(J)\supsetneq 
  \Span(\Pi_{U})/\Span(J).
  \end{equation*}
   Noting that $\Pi_{U}\sreq J$ is linearly independent,  we have \begin{equation*}
 \dim( \Span(\Pi_{U})/\Span(J))=\vert \Pi_{U}\sm J\vert =\corank_{U}(W_{J})=k.
 \end{equation*}
 
  Hence  $\dim \Span(\Pi_{W'})/\Span(J)>k$. On the other hand,  \begin{equation*}
  k+1=\corank_{W'}(W_{J})=\vert \Pi_{W'}\sm J \vert \geq \dim (\Span(\Pi_{W'})/\Span(J))>k.
  \end{equation*} It follows that  \[\dim (\Span(\Pi_{W'})/\Span( J))=\vert \Pi_{W'}\sm J\vert=k+1\in \mathbb{N}.\] Now since $J\seq \Pi_{W'}$ is linearly independent, it follows that $\Pi_{W'}$ is linearly independent as claimed.   
   
From the  two claims above, the definition of $I$ and Proposition \ref{y2.13}(b), it follows directly that for each $\alpha\in \Phi\setminus \Phi_{U}$, there is a unique element $W'\in I$
such that $\langle U\cup \{s_{\alpha}\}\rangle \subseteq W' $, or equivalently by definition of $I$, such that $\alpha\in \Phi_{W'}$. This is  what we were required to prove.   
\end{proof}  
\begin{proof}[Proof of Theorem \ref{maxref}]
Theorem \ref{maxref}(b)--(c) follow from Theorem \ref{y1.8}(a)--(b) respectively. To prove Theorem \ref{maxref}(a), there is no loss of generality in assuming that $\Pi$ is linearly independent, by taking a universal lift of $\Phi$, for instance. Then $\Phi_{U}$ is linearly independent and $\Phi_{U}=\Phi\cap \Span(\Phi_{U})$ by 
Proposition   \ref{facialcomp}, so  Theorem \ref{maxref}(a) follows from Theorem \ref{y1.8}(b).\end{proof}

\subsection{} The following proposition  constructs maximal rank $k$ reflection overgroups under purely group-theoretic assumptions which are more general than the conditions in Theorem \ref{maxref}(b). However, the resulting maximal rank $k$ reflection overgroups are not known to be unique, as in Theorem \ref{maxref}(b). 

\begin{prop}\label{combmaxref} Assume that $J\seq \Pi$ and that  $U$ is a reflection overgroup of $W_{J}$ such that  $\corank_{U}(W_{J})=k\in \mathbb{N}$. Assume also that there is no reflection overgroup $U'$ of $U$ with $\corank_{U'}(W_{J})<k$ (this holds for instance if $\Pi_{U}$ is linearly independent or if $k=1$). Then there are only finitely many reflection overgroups $U'$ of $U$ such that  $\corank_{U'}(W_{J})=k$. 
In particular, $U$ is contained in some inclusion maximal corank $k$ reflection overgroup of $W_{J}$ 
\end{prop} 
\begin{proof} Let $U$ be any subgroup of $W$. Let $Q_{U}$ denote the set of all reflection subgroups $V$ of $W$ such
that $V$ contains $U$ and $U$ is not contained in any proper standard parabolic subgroup
$\mpair{K}$, where $K\sneq \chi(V)$, of $V$.  
 An argument similar to that in the proof of \cite[Proposition 1]{DyPc} shows that if $U=\mpair{W_{J}\cup X}$ for some finite subset $X$ of $W$, then $Q_{U}$ is finite.  Now suppose that $U$ is a reflection overgroup of $W_{J}$ satisfying the assumptions of the proposition. Then $U=\mpair{W_{J}\cup X}$ for some finite subset $X$ of $T$, so $Q_{U}$ is finite.  We claim that any reflection overgroup $U'$ of $U$ such that $\corank_{U'}(W_{J})=k$ is an element of $Q_{U}$. For otherwise, $U'$  would have a proper standard parabolic subgroup $U''$, where $\Pi_{U''}\sneq \Pi_{U'}$, such that $U\seq U''$ and we would have
 \[k=\corank_{U'}(W_{J})=\vert \Pi_{U'}\sm J\vert =
 \vert \Pi_{U'}\sm \Pi_{U''}\vert +\vert \Pi_{U''}\sm J\vert > 
 \vert \Pi_{U''}\sm J\vert=\corank_{U''}(W_{J}),\]
 contrary to the assumptions.
\end{proof}
\subsection{Root coefficients} We end this section with some facts we shall have need of concerning root coefficients. \begin{prop} \label{threeroots}
\begin{num}
 \item Suppose that $\alpha$, $\beta$ and $\gamma$ are in $\Phi$ and $\gamma-\beta\in \Span(\alpha)$. Then either   
$\gamma=\beta$ or $\gamma=s_{\alpha}(\beta)$. Hence  
$\Phi\cap (\beta+\Span(\set{\alpha})=W_{\set{\alpha}}\beta$.
\item If $W$ is finite, $J\seq \Pi$ and $\beta\in \Phi\sm \Phi_{J}$,  then $W\beta\cap (\beta+\Span(J))=W_{J}\beta$.
\item Assume $W$ is finite and irreducible.  Let $\alpha\in \Pi$ and $J:=\Pi\sm\set{\alpha}$. Then for any $\beta\in \Phi\sm \Phi_{J}$, we have  $\Phi\cap  (\beta+
\Span(J))=W_{J}\beta$.   
\end{num} 
\end{prop}
\begin{remark*}
 Proposition \ref{threeroots}(b) also holds if $\Phi$ 
is a  reduced, crystallographic root system of a finite Weyl 
group, which is not in general a based root system in the sense 
of this paper since roots may have unequal lengths; see  
\cite{Bou}, \cite{Osh} and \cite{DyLe18}. However, (a) (and 
therefore (c))   does \emph{not} hold in general for non-simply-laced 
(reduced) crystallographic root systems of finite Weyl groups,
because such root systems contain root strings of length at 
least $3$ (see \cite{Bou} or \cite{HuLA}). 
\end{remark*} 
\begin{proof} The second assertion of (a) follows immediately 
from the first, the  
proof of which uses only the facts that $\alpha,\beta,\gamma\in 
V$, $B(\beta,\beta)=B(\gamma,\gamma)$,
$B(\alpha,\alpha)\neq 0$ and $\gamma- \beta\in \Span(\alpha)$.
 Write $\gamma=\beta+c\alpha$, where $c\in \mathbb{R}$.
Then 
\begin{equation*}\begin{split}
0&=-B(\beta,\beta)+B(\gamma,\gamma)=-B(\beta,\beta)
+B(\beta+c\alpha,\beta+c\alpha)\\&=
-B(\beta,\beta)+B(\beta,\beta)+2cB(\beta,\alpha)+c^{2}B(\alpha,
\alpha)
=c(cB(\alpha,\alpha)+2B(\beta,\alpha)).
\end{split}\end{equation*}
Hence  either $c=0$, in which case $ \gamma=\beta$, or $c=-\frac{2B(\beta,\alpha)}{B(\alpha,\alpha)}$,
in which case, $\gamma=s_{\alpha}(\beta)$.

Part (b) is a special case of \cite[Proposition 1.4(a)]{DyLe18}.

We prove (c). Let assumptions be as there. If $\vert \Pi\vert \leq 1$, the result is trivial  and if $\vert \Pi\vert =2$, the desired conclusion follows from (a). 
Assume $\vert \Pi\vert>2$.

The inclusion $W_{J}\beta\seq \Phi\cap(\beta+\Span(J))$ holds since for any $v\in V$ and $\gamma\in J$, $s_{\gamma}v-v\in \Span(J)$. To prove the reverse inclusion, let
$\gamma\in \Phi\cap(\beta+\Span(J))$. If $\gamma\in W\beta$, the result follows from (b). Suppose $\gamma\not\in W\beta$.
It will suffice to show that $W\gamma\cap (\beta+\Span(J))=\eset$. 
Note that $\Phi$ contains at least two $W$-orbits of roots, so
in virtue of our assumptions, $\Phi$ must be of type $B_{n}$ 
(equivalently,  $C_{n}$) for $n>2$, or of type $F_{4}$, and there 
are exactly two $W$-orbits of roots, namely $W\beta$ and 
$W\gamma$.

Define $\dot\rho:=\rho$ if $\rho\in W\beta$ and 
$\dot\rho:=\sqrt{2}\rho$ if $\rho\in W\gamma$, so in either case
$\dot \rho=\Vert \dot \rho\Vert \rho$. Then  $\Psi:=\mset{\dot\rho\mid \rho\in \Phi}$ is a crystallographic root system  of $W$ in $\Span(\Pi)$, with simple roots 
$\Delta:=\mset{\dot\rho\mid \rho\in \Pi}$. 
Since $\Psi$  is crystallographic,  for any $\rho\in \Psi$, there are integers $(n_{\tau,\rho})_{\tau\in \Pi}$ such that  $\dot\rho=\sum_{\tau\in \Pi }n_{\tau,\rho}\dot\tau$. 
Hence $\rho=\sum_{\tau\in \Pi }c_{\tau,\rho}\tau$ where
$c_{\tau,\rho}=\frac{\Vert\dot \tau\Vert}{\Vert \dot \rho\Vert}n_{\tau,\rho}\in \frac{\Vert\dot \tau\Vert}{\Vert \dot \rho\Vert}\mathbb{Z}$. 

Suppose
 that   $\delta:=w\gamma\in W\gamma\cap(\beta+\Span(J))$
 where $w\in W$.
 Note $c_{\alpha,\beta}\neq 0$  since $\beta\in \Span(J\cup\set{\alpha})\sm \Span(J)$.  Also,  $ c_{\alpha,\delta}=c_{\alpha,\beta}$ since $\delta\in \beta+\Span(J)$. 
 Hence $c_{\alpha,\delta}\in 
 \frac{\Vert \dot \alpha \Vert}{\Vert \dot \beta\Vert}\mathbb{\mathbb{Z}}\sm\set{0}={\Vert \dot \alpha \Vert}\mathbb{\mathbb{Z}}\sm\set{0}$. But also, 
 $c_{\alpha, \delta}\in \frac{\Vert \dot \alpha \Vert}{\Vert \dot \delta\Vert}\mathbb{\mathbb{Z}}=
 \frac{\Vert \dot \alpha \Vert}{\Vert \dot \gamma\Vert}\mathbb{\mathbb{Z}}=\frac{\Vert \dot \alpha \Vert}{\sqrt{2}}\mathbb{\mathbb{Z}}$.  This would imply that $\frac{\sqrt{2}}
 {\Vert \dot \alpha \Vert}c_{\alpha,\delta}\in \mathbb{Z}\cap \sqrt{2}(\mathbb{Z}\sm\set{0})=\eset$, a contradiction  which completes the proof of (c).
\end{proof}
\subsection{} We call a linear map $\height$ as in (a) in the following proposition a \emph{height function}  on $V$  for $(\Phi,\Pi)$.
\begin{prop}\label{height} Assume that $(W,S)$ is of finite rank.
 \begin{num}
\item There exists a linear function $\height\colon V\to \mathbb{R}$ such that $\height(\alpha)>0$ for all $\alpha\in \Pi$. 
\item For any function $\height$ as in (a), there exists a positive constant $\epsilon$ such that
$\height(\alpha)>\epsilon \ell(s_{\alpha})$ for all $\alpha\in \Phi^{+}$.
\end{num}\end{prop}
\begin{proof}  We sketch a proof.  Part (a) follows from basic facts 
about duals of polyhedral cones; see \cite[Lemma 4.2(d)]{Dy13} for 
an equivalent result.
 Part (b) holds on taking  $\epsilon:=\epsilon_{1}\epsilon _{2}$ 
 where $\epsilon_{1}>0$ is chosen so that $0<\epsilon_{1}
 <\height(\alpha)$ for all $\alpha\in \Pi$, and $\epsilon_{2}$ is 
 chosen   so $0<\epsilon_{2}<1$ and 
$\epsilon_{2}<\vert B(\alpha,\beta)\vert $ and  for all pairs $(\alpha,\beta)$ of non-orthogonal simple roots such that there is a (rank 
two) finite standard  parabolic root subsystem of $\Phi$ containing 
both $\alpha$ and $\beta$.  (This implies that $\epsilon_{2}<\vert 
B(\alpha,\beta)\vert $ for any non-orthogonal roots $\alpha$ and 
$\beta$; see \cite[Lemma 1.22(1)]{DyQuo}.) The above more 
precise version of (b) can be proved  by induction on 
$\ell(s_{\alpha})$, in a similar manner to the  proof of its special 
case \cite[Lemma 1.22(2)]{DyQuo}. \end{proof}   

\section{Brink-Howlett groupoids as signed groupoid-sets}
\label{sgs}  Subsections \ref{realifcn}--\ref{reduced} provide more details concerning general signed groupoid sets, and especially the notions of  real and imaginary roots, realification and compression.  The rest of Section \ref{sgs} provides a proof of Theorem \ref{standpreprinc}.% 

\subsection{} \label{realifcn}
Let $(G,X)$ be a signed $G$-set. For any $a\in \ob(G)$, 
each root $x\in X(a)$ has a sign $\sgn(x)\in \set{\pm}=\set{\pm 1}$, with $\sgn(x)=+$ if  $x\in X(a)^{+}$ or $\sgn(x)=-$ if  $x\in X(a)^{-}$. We say that $x$ is an \emph{imaginary root} if $\sgn(x)=\sgn(gx)$  for all $g\in G_{a}:=\bigcup_{b\in \mathrm{Ob}(G)}\lrsub{b}{G}{a}$, and that $x$ is a  \emph{real root} otherwise. We say $(G,X)$ is a  \emph{real} signed groupoid set if  every root of $X$ is real.

It is easily seen there is a signed $G$-subset $X^{\re}$ of $X$ 
such that, for all $a\in \ob(G)$,  $X^{\re}(a)$ is equal to the set 
of all real roots of $X(a)$. One necessarily has  
 $X^{\re,\epsilon}(a)=X^{\re}(a)\cap X^{\epsilon}(a)$ for
  $\epsilon\in \set{\pm}$.   We call  the signed $G$-set $
  (G,X^{\mathrm{re}})$ the \emph{realification} of $(G,X)$. Note
  $(G,(X^{\re})^{\re})=(G,X^{\re})$. Also,  $(G,X)$ is real if and only if it is isomorphic to $(G,X^{\re})$.    We have  \begin{equation}\label{realinvset}X^{\re,+}(a)=\bigcup_{g\in \lsub{a}{G}}X_{g},\qquad g\in \lsub{a}{G}\implies X^{\re}_{g}=X_{g}.\end{equation} 
Similarly, there is a signed $G$-set $(G,X^{\im})$ such that for each  $a\in \ob(G)$, $X^{\im}(a)=X(a)\sm X^{\re}(a)$ is the set of all imaginary roots in $X(a)$.

\subsection{}  Consider a morphism $\nu\colon (G,X)\mapsto (G,Y)$ of signed $G$-sets. For all $a\in \ob(G)$, the component
$\nu_{a}\colon X(a)\to Y(a)$ is positivity preserving. Hence 
\begin{equation}\label{sgn1}\sgn(x)=\sgn(\nu_{a}(x)),\qquad \text{\rm if $x\in X(a)$}
\end{equation}   and 
\begin{equation}\label{sgn2} \sgn(gx)=\sgn(\nu_{b}(gx))=\sgn(g\nu_{a}(x))\qquad\text{\rm if $x\in X(a)$, $g\in \lrsub{b}{G}{a}$.}
\end{equation} 
These equations imply that 
\begin{equation}\label{realmorph}
x\in X^{\re}(a)\iff  \nu_{a}(x)\in Y^{\re}(a).
\end{equation}
In particular, $\nu_{a}$
restricts to a  positivity preserving morphism $\nu^{\re}_{a}\colon 
X^{\re}(a)\to Y^{\re}(a)$ of signed sets. The maps $(\nu^{\re}_{a})_{a\in \ob(G)}$ so defined
are the components of  a morphism $\nu^\re\colon (G,X^{\re})\to (G,Y^{\re})$ of signed $G$-sets. The maps $(G,X)\mapsto (G,X^{\re})$ on objects and $\nu\mapsto \nu^{\re}$ on morphisms  defines an endofunctor, which we call \emph{realification}, of the category of signed $G$-sets.  

Using imaginary roots instead of real roots, there  is a  similarly defined endofunctor, denoted  $(G,X)\mapsto (G,X^{\im})$ on objects, and $\nu\mapsto \nu^{\im}$ on morphisms, of the category of signed $G$-sets. The only  essential use we make of it is notational:   $X^{\im}(a)$, $X^{\im,+}(a)$ and $X^{\im,-}(a)$ denote,  respectively,   the sets of   imaginary roots, positive imaginary roots and negative imaginary roots of $X$ based at $a$. We note however that  $(G,X)\cong(G,X^{\re})\coprod (G,X^{\im})$ where $\coprod$ denotes coproduct in the category of signed $G$-sets.

\subsection{} \label{compn}For each object $a$ of $G$,  the \emph{dominance preorder}\footnote{This notion abstracts the dominance order on the root system of a Coxeter group, as defined initially on positive roots in \cite{BHAut} and extended to  the whole root system in \cite{HTR}.} $\lsub{a}{\preceq}$ of $(G,X)$ at $a$ is the preorder on $X(a)$ which is defined  as follows: for $x,y\in X(a)$, let  $x \lsub{a}{\preceq}y$  if   for all $b\in \ob(G)$ and  $g\in \lrsub{b}{G}{a}$, we have $gy\in X^{-}(b) \implies gx\in X^{-}(b)$.  There is then an associated equivalence relation $\lsub{a}{\sim}$ on $X(a)$, which we call \emph{parallelism} (at $a$), in which 
\[ x \lsub{a}{\sim}y\iff  x \lsub{a}{\preceq}y \text{ \rm and }
y \lsub{a}{\preceq}x, \qquad x,y\in X(a).\]  All roots in the parallelism class
$\lsub{a}{[x]}:=\{\,y\in X(a)\mid y\lsub{a}{\sim}x\,\}$  (i.e. $\lsub{a}{\sim}\,$-equivalence class)  of a root $x$ based at $a$ have the sign as $x$, and we define the sign (positive or negative) of  $\lsub{a}{[x]}$ to be equal to that of $x$.  

Note that if a root is real, then so is every root in its parallelism 
class. The parallelism class of an imaginary root is the set of all 
imaginary roots with the same base and sign as that root.

We say that $(G,X)$ is compressed if all of the dominance 
preorders $(\,\lsub{a}{X}, \lsub{a}{\preceq})$ with $a\in \ob(G)$ 
are partial orders, in which case we call them \emph{dominance 
orders}. Note that $(G,X)$ is compressed if and only if any  two 
parallel roots of $X$ (which necessarily have the same base) 
are equal.
 
\subsection{}\label{comp1}    Let $X^{\cc}(a)$ denote the set of all parallelism 
classes in $\lsub{a}{X}$. That is, 
\[ X^{\cc}(a):=X(a)/\negthinspace
 \negthinspace\negthinspace\lsub{a}{\sim}\,\,=\{\,\lsub{a}{[x]}\mid x\in \lsub{a}{X}\,\}. \]  This is a signed set with   $-\lsub{a}[x]:=\lsub{a}[-x]$ and positive and negative elements as previously defined: 
 $X^{c,\epsilon}(a)=\mset{\lsub{a}{[x]}\mid x\in X^{\epsilon }(a)}$ for $\epsilon \in \set{\pm}$.  
 
 The \emph{compression} of $(G,X)$ is defined to be the  signed $G$-set $(G,X^{\cc  })$ where $X^{\cc  }$, as a functor  from $G$ to the category of signed sets, has value 
  $X^{\cc  }(a)$ at each object $a$ of $G$  equal to the signed set just defined, and is given on morphisms of $G$ by
   \[\text{\rm $g\lsub{a}[x]=(X^{\cc  }(g))(\lsub{a}[x]):=    \lsub{b}{[gx]}$, \qquad  for $g\in \lrsub{b}{G}{a}$ and $x\in X(a)$.}\]
   It follows that  $(G,X)$ is compressed if and only if it is isomorphic to $(G,X^{\cc })$, and $(G,(X^{\cc})^{\cc})\cong(G,X^{c})$.   
   
Every inversion set for $(G,X)$ is a union of parallelism classes.  Hence, for $g\in \lsub{a}{G}$, the inversion sets of $g$ in $X$ and $X^{\cc}$ are related by
\begin{equation}\label{compinvset}
X_{g}=\bigcup_{y\in (X^{\cc})_{g}}y, \qquad (X^{\cc})_{g}=\mset{\lsub{a}{[x]}\mid x\in X_{g}}
\end{equation}

\subsection{}\label{comp2} 
The  construction $(G,X)\mapsto (G,X^{\cc })$ on signed $G$-sets extends  to morphisms of  signed $G$-sets as 
follows. Let $\nu\colon (G,X)\to (G,Y)$ be a  morphism of signed sets. It follows from \eqref{sgn1}--\eqref{sgn2} that \begin{equation}\label{dompres1}
x\lrsub{a}{\preceq}{X}x'\iff \nu_{a}(x)\lrsub{a}{\preceq}{Y}\nu_{a}(x'), \qquad x,x'\in X(a).
\end{equation}
where $\lrsub{a}{\preceq}{X}$ and $\lrsub{a}{\preceq}{Y}$ are the dominance preorders of $X$ and $Y$ based at $a$, respectively.  That is, $\nu_{a}$ induces an embedding of preorders $(X(a), \lrsub{a}{\preceq}{X})\to (Y(a), \lrsub{a}{\preceq}{Y})$.    This in turn implies that 
\begin{equation}\label{dompres2}
x\lrsub{a}{\sim}{X}x'\iff \nu_{a}(x)\lrsub{a}{\sim}{Y}\nu_{a}(x'), \qquad x,x'\in X(a).
\end{equation} 
where $\lrsub{a}{\sim}{X}$ and $\lrsub{a}{\sim}{Y}$ denote parallelism  on $X(a)$ and $Y(a)$, respectively. 

Denote the parallelism classes of $\lrsub{a}{\sim}{X}$ and $\lrsub{a}{\sim}{Y}$ as $\lrsub{a}{[x]}{X}$ for $x\in X(a)$ and $\nts\lrsub{a}{\sim}{Y}$ for $y\in Y(a)$ respectively. 
It follows from above  that there is a (well-defined) positivity preserving morphism of signed sets
$\nu^{\cc }_{a}\colon X^{\cc }(a)\to Y^{\cc }(a)$ defined by 
\begin{equation*}
\nu^{\cc }_{a}(\lrsub{a}{[x]}{X})=\lrsub{a}{[\nu_{a}(x)]}{Y},
\qquad x\in X(a).
\end{equation*} Moreover, the underlying function (of sets) of 
$\nu^{\cc }_{a}$ is injective, and it is therfore bijective if  (the underlying function of) $\nu_{a}$ is surjective.

The definitions now show that  $\nu^{\cc }:=(\nu^{\cc }_{a})_{a\in \ob(G)}$ is a morphism  $X^{\cc }\to Y^{\cc }$ of signed $G$-sets. Moreover, $\nu^{\cc}$ is an isomorphism if the (underlying functions of) the components of $\nu$ are surjective. The maps $(G,X)\mapsto (G,X^{\cc })$ on objects and $\nu\mapsto \nu^{\cc }$ on morphisms, evidently  define an endofunctor, which we call \emph{compression}, of the category of signed $G$-sets.
  
 \subsection{}  \label{realcomp}
 The \emph{real compression} of the signed $G$-set $(G,X)$ is the signed $G$-set $(G, X^{\rec})$ where
   \[ X^{\rec}:=(X^{\mathrm{re}})^{\cc  }=
 (X^{\cc  })^{\mathrm{re}}.\] 
 For an object $a$ of $G$, $X^{\rec}(a)$ is the set of  parallelism classes of real roots in $X(a)$.  Similarly, for a morphism
 $\nu\colon (G,X)\to (G,Y)$ of signed $G$-sets, , one has $(\nu^{\re})^{\cc}=(\nu^{\cc })^{\re}\colon (G, X^{\rec})\to (G,Y^{\rec})$. It follows that  $\mathrm{Id}$, $\re$, $\cc $,
 $\rec=\re\circ\cc =\cc \circ\re$ are four endofunctors of the category of signed $G$-sets.
 
There is a morphism
$\iota_{X}\colon (G,X^{\re})\to (G,X)$  of signed groupoid sets 
with components given by the inclusions ${X}^{\re}(a)\to X(a)$ for $a\in \ob(G)$.  These morphisms define a  natural transformations $\iota\colon \re\to \mathrm{Id}$. Similarly, there is a  morphism
$\eta_{X}\colon (G,X)\to (G,X^{\cc})$  of signed groupoid sets 
with component at $a\in \ob(G)$ given by the natural projection
$X(a)\to X(a)/\nts\nts\lsub{a}{\sim}\,=X^{\cc}(a)$. These endofunctors and natural transformations appear in a commutative diagram as follows: 
\begin{equation*}
\xymatrix{
{\re}\ar[r]^{\iota}\ar[d]_{\re\,\eta}& {\mathrm{Id}}\ar[d]^{\eta}\\
{\rec}\ar[r]_{\iota\cc}&{\cc}
}
\end{equation*}

   \begin{prop}\label{preprinc} Let $(G,X)$ be a  signed groupoid set. 
   \begin{num}
   \item  For  $a\in \ob(G)$, the weak preorders $(\lsub{a}{G},\lsub{a}{\leq})$ of 
   $(G,X)$, $(G,X^{\re})$, $(G,X^{\cc })$ and  $(G,X^{\rec})$ at $a$ are equal.
   \item If any of the four signed $G$-sets in (a) is faithful (resp.,  preprincipal), then they all are.
   \end{num}
   \end{prop}
  \begin{proof}
 Using  \eqref{realinvset},  the  weak preorders at $a$ of $(G,X)$ and $(G,X^{\re})$ are equal, and, by the definitions,  $(G,X)$ is faithful (resp., preprincipal) if and only if $(G,X^{\re})$ is faithful (resp., preprincipal). Similarly, using \eqref{compinvset},  the weak preorders  of $(G,X)$ and $(G,X^{\cc})$ are  equal, and by the definitions,  $(G,X)$ is faithful (resp., preprincipal) if and only if $(G,X^{\cc})$ is faithful (resp., preprincipal). All the assertions now follow on recalling that  $(G,(X^{\re})^{\cc})=(G, X^{\rec})$.
  \end{proof}
   We shall also use the following general facts (see  \cite{DyWa} or \cite{DyGrp1}--\cite{DyGrp2}).
 \begin{prop} \label{princpreprinc}
  Let  $(G,X)$ is   preprincipal signed groupoid set, and let $A$ denote the set of atomic generators of $(G,X)$. 
 \begin{num} \item For any $a\in \ob(G)$, the set of parallelism classes of real roots of $(G,X)$ at $a$  is
$X^{\rec}(a)=\mset{g(X_{s})\mid b\in \ob(G), g\in \lrsub{a}{G}{b}, s\in A\cap \lsub{b}{G}}$. 
\item $(G,X^{\rec})$ is a principal signed groupoid set with $A$ as its set of simple generators.   It is  real and compressed. 
   \end{num}\end{prop}
   \begin{remark*} Under the hypotheses of (a), if  
   $a,b_{i}\in \ob(G)$, $g_{i}\in \lrsub{a}{G}{b_{i}}$, $s_{i}\in A\cap \lsub{b_{i}}{G}$ for $i=1,2$, then the subsets $g_{i}(X_{s_{i}})$  of $\lsub{a}{X}$ for $i=1,2$ are either disjoint or equal.   
     \end{remark*}
    
   \begin{ex} \label{reduced}The signed $W$-set $(W_{\bullet},\Phi_{\bullet})$ of Example \ref{Coxrootoid} is real, compressed, principal and rootoidal.

 If $W$ is, further, a  finite Weyl group, it has also a (possibly non-reduced) crystallographic root system $\Psi$  in the sense of \cite{Bou}; this  has a set of simple roots $\Delta$, but $(\Psi,\Delta)$ is  not a based root system in the sense of this paper. The natural $W$-action on  $\Psi$  gives rise, in a similar manner to the construction of $(W_{\bullet}, \Phi_{\bullet})$, to a signed groupoid set $(W_{\bullet},\Psi_{\bullet})$. Each parallelism classes of roots for $(W_{\bullet},\Psi_{\bullet})$ is either  a singleton set $\set{\alpha}$ or of the form $\set{\alpha,2\alpha}$ for some  root $\alpha\in \Psi$.  The set of parallelism classes is in bijection with the set of indivisible roots $\alpha$ of $\Phi$. Thus, $(W_{\bullet}, \Phi_{\bullet})$ is compressed if and only if  $\Psi$ is reduced in the usual sense (that is, $\alpha\in \Psi$ implies $2\alpha\not \in \Psi$); in that case,
$(W_{\bullet}, \Phi_{\bullet})$ and $(W_{\bullet}, \Psi_{\bullet})$ are isomorphic as signed $W_{\bullet}$-sets. It follows also that  $(W_{\bullet}, \Psi_{\bullet})$ is principal if and only if $\Psi$ is reduced. Even if  $\Psi$ is non-reduced,   
$(W_{\bullet}, \Psi_{\bullet})$ is real, preprincipal and rootoidal, and its (real) compression is isomorphic to $(W_{\bullet}, \Phi_{\bullet})$.    \end{ex}

\subsection{Abstract root systems of Brink-Howlett groupoid sets}
In the remainder of this section, $(G,\Lambda)$ denotes the signed groupoid set attached to a  Brink-Howlett groupoid $G$  associated to the Coxeter system $(W,S)$ and its based root system $(\Phi,\Pi)$ as in \ref{standBH}.  
Recall that the weak order of $(W,S)$ is denoted as  $(W,\leq)$.

\begin{prop} \label{weakprop} Let  $J\in \ob(G)$ and $g,g'\in \lsub{J}{G}$. Write  $g=(J,w,K)$ and $g'=(J,w',K)$ where $K,K'\seq \Pi$, $w,w'\in W$ and $wK=w'K'=J$. \begin{num}\item $\vert \Lambda_{g}\vert$ is finite.
\item $g\lsub{J}{\leq}g'$ if and only if $ w\leq w'$.
\item If $\Lambda_{g}= \Lambda_{g'}$, then $g=g'$. That is,  $(G,\Lambda)$ is faithful.  
\item The map $F_{J}\colon (\lsub{J}{G},\nts\nts\lsub{J}{\leq})\to (W,\leq)$, induced by  restriction  to $\lsub{J}{G}$ of the map of morphisms for the functor $F\colon G\to W_{\bullet}$ in \ref{standBH}, is an order-embedding. 
\item  For  each object $K$ of $G$, each closed interval  in the poset $(\lsub{K}{G},\lsub{K}{\leq})$ is finite.  That is, $(G,\Lambda)$ is interval finite. 
\end{num} \end{prop}
\begin{proof}

We have $\Lambda(J)^{+}=\Phi^{+}$and $ \Lambda(J)^{-}=\Phi^{-}$, so
\begin{equation}
\Lambda_{g}=\Lambda(J)^{+}\cap w(\Lambda(K)^{-})=\Phi^{+}\cap w(\Phi^{-})=\Phi_{w}.
\end{equation}   Part (a) follows since $\vert \Phi_{w}\vert =\ell(w)<\infty$. 
The definitions imply from above  that 
\begin{equation}
g\lsub{J}{\leq}g'\iff  \Lambda_{g}\seq \Lambda_{g'}\iff \Phi_{w}\seq \Phi_{w'}\iff w\leq w',\end{equation} proving (b).
If $\Lambda_{g}=\Lambda_{g'}$, then from above,  $g\lsub{J}{\leq}g'$ and $g'\lsub{J}{\leq}g$ and hence $w\leq w'$ and $w'\leq w$, so $w=w'$. Then $K=w^{-1}J=w^{\prime-1}J=K'$ so $g=g'$, proving (c).

In (d), we have $F_{J}((J,w,K))=w$  for $(J,w,K)\in \lsub{J}{G}$. Therefore, $F_{J}$ is  an order embedding by (b).  Part (e) follows from (a), (c) and the definition of the weak order $\lsub{J}{\leq}$, or from (d) and the fact intervals in weak order on $W$ are finite.\end{proof}  

\begin{lem} \label{BHeltlem} Let $w\colon J\to K$ be a morphism in $G$.
\begin{num}
\item If $L\seq \Pi\cap \Phi_{w^{-1}}$, 
then $ \nu(L,J)^{-1}\leq w^{-1}$. 
\item If $\beta\in \Pi$ and $\alpha\in \Phi_{w^{-1}}\cap \Phi_{J\cup\set{\beta}}$, then $ \nu(\beta,J)^{-1}\leq w^{-1}$.
\item   If  $\alpha\in \Phi_{w^{-1}}\cap \Pi$, then $ \nu(\alpha,J)^{-1}\leq w^{-1}$
\end{num}\end{lem}
\begin{proof}
Assume without loss of generality  that $\Pi$ is linearly independent. For (a), note first that  $w(J)=K\seq \Phi^{+}$ and $w(L)\seq \Phi^{-}$, so $J\cap L=\eset$.  Hence  \begin{equation*}
\Phi^{+}_{L\cup J}\sm \Phi^{+}_{J}\seq
\left((\cone(L)+\Span(J))\sm \Span(J)\right)\cap \Phi. 
\end{equation*} Using $w(J)=K$ and  $w(L)\seq \Phi^{-}$ again shows that \begin{equation*}
w(\Phi^{+}_{L\cup J}\sm \Phi^{+}_{J})\seq 
\left((\cone(w(L))+\Span(K))\sm \Span(K)\right)\cap \Phi\seq \Phi^{-}. 
\end{equation*}
Thus, $\Phi^{+}_{L\cup J}\sm \Phi^{+}_{J}\seq \Phi_{w^{-1}}$, which is finite. Therefore, $\nu(L,J)$ is defined, and since  $\Phi_{\nu(L,J)^{-1}}=\Phi^{+}_{L\cup J}\sm \Phi^{+}_{J}$,  (a) follows. 

To prove (b), note $\alpha\in \cone(J\cup\set{\beta})$.
Since $w(J)=K\seq \Phi^{+}$ and $w(\alpha)\in \Phi^{-}$ we must have $w(\beta)\in \Phi^{-}$ and so $L:=\set{\beta}\seq \Pi\cap \Phi_{w^{-1}}$. Now (b) follows from (a). Part (c) is the special case  of (b) in which $\beta=\alpha$.
\end{proof}

\begin{prop}[{\cite[Proposition 5.5]{Deod}}, {\cite[Proposition 2.3]{BrHo99}}]  \label{BHgenprop} Let $g=(K,w,J)$ be a morphism in $G$. Then there are morphism $g_{i}=(J_{i},w_{i},J_{i-1})$ of $G$,  for $i=1,\ldots, n$, where $J_{0}=J$, $J_{n}=K$, 
$\alpha_{i}\in \Pi$, $J_{i}\seq \Pi$  and  $w_{i}=\nu(\alpha_{i},J_{i-1})$    for $i=1,\ldots, n$,  such that $g=g_{n}\cdots g_{1}$ in $G$, and $w=w_{n}\cdots w_{1}$ and  $\ell(w)=\sum_{i=1}^{n}\ell(w_{i})$ in $W$. Moreover, if $\alpha\in \Pi\cap \Phi_{w^{-1}}$, one may choose these morphisms so that  $\alpha_{1}=\alpha$. 
\end{prop}
\begin{proof}
 We prove this by induction on $\ell(w)$. If $\ell(w)=0$, then $w=1_{W}$, $K=J$ and one takes $n=0$ and $J_{0}=J$.
 Otherwise, we have $w\neq 1$ and may choose $\alpha\in \Pi\cap \Phi_{w^{-1}}$. By Lemma \ref{BHeltlem},  we have $\nu(\alpha,J)^{-1}\leq w^{-1}$. That is,
 we have $w=w'\nu(\alpha,J)$ where $\ell(w)=\ell(w')+\ell(\nu(\alpha,J)>\ell(w')$.  Let $J':=\nu(\alpha,J)J\seq \Pi$. 
 One has morphisms $g':=(K,w',J')$ and  $g_{1}=(J',\nu(\alpha,J),J)$ in $G$ with $g=g'g_{1}$. The desired conclusion follows by applying the inductive hypothesis to the morphism $g'$ of $G$.   
\end{proof}

\begin{prop}\label{prepatom}
\begin{num}
\item  The set of atomic morphisms of $ (G,\Lambda)$ coincides with the set of Brink-Howlett generators of $G$.
\item  $(G,\Lambda)$ is preprincipal.
\end{num}
\end{prop}
\begin{proof}
We prove (a).  Consider a non-identity morphism $g$ of $G$ in  $\lsub{K}{G}$.
Write it as a composite $g=g_{n}\cdots g_{1}$ of Brink-Howlett generators $g_{i}=(J_{i},\nu(\alpha_{i},J_{i-1}),J_{i-1})$ as in  Proposition
\ref{BHgenprop}, where necessarily $n>0$.  From $w=w_{n}\cdots w_{1}$ with $l(w)=\sum_{i=1}^{n} \ell(w_{i})$, it follows that $w_{n}\leq w$ in $W$ and hence $g_{n}\lsub{K}{\leq} g$ in $G$ by Proposition \ref{weakprop}. If $g$ is an atom of  $(\lsub{K}{G},\nts\lsub{K}{\leq})$, this forces $g=g_{n}$, since $g_{n}\neq 1_{K}$.  Hence every atomic morphism of $G$ is a Brink-Howlett generator. 

Conversely, let $g=(K,\nu(\alpha,J),J)$ be a Brink-Howlett generator of $G$. Since $\hskip-.2352cm\lsub{K}{\leq}$ has finite intervals and a minimum element, there is an atom $a$ of  $\nts\lsub{K}{\leq}$ such that $a\nts\lsub{K}{\leq}g$.   From the previous paragraph, $a$ is a Brink-Howlett generator $a=(K, \nu(\alpha',J'),J')$, and we therefore  have $\nu(\alpha',J')\leq \nu(\alpha,J)$. Write $g^{-1}=(J,\nu(\beta, K),K)$ and $a^{-1}=(J',\nu(\beta', K),K)$. 
Then
\[\Phi_{\nu(\beta',K)^{-1}}=\Phi_{\nu(\alpha',J')}\seq \Phi_{\nu(\alpha,J)}=\Phi_{\nu(\beta,K)^{-1}}.\]
Hence \begin{equation*}\begin{split}
\set{\beta'}&=\Pi\cap (\Phi_{K\cup\set{\beta'}}^{+}\sm \Phi_{K}^+)=\Pi\cap \Phi_{\nu(\beta',K)^{-1}}\\ &\seq \Pi\cap \Phi_{\nu(\beta,K)^{-1}}=\Pi\cap (\Phi_{K\cup\set{\beta}}^{+}\sm \Phi_{K}^+)=\set{\beta}.\end{split}\end{equation*} 
From this, we conclude $\beta'=\beta$, $J'=J$,  $g^{-1}=a^{-1}$ and $g=a$. Hence $g$ is  an atom of $\lsub{K}{\leq}$ as required to complete the proof of (a). 

Now we prove (b). Let $s,g\in \lsub{J}{G}$ be such that $s$ is atomic. Then $s$ is a Brink-Howlett generator $s=(J,\nu(\alpha,K),K)$  and $g$ is of the form $g=(J,w,K')$. We are to show that if $\Lambda_{s}\cap \Lambda_{g}\neq \eset$, then $\Lambda_{s}\seq \Lambda _{g}$. Equivalently, we have to show that if 
$\Phi_{\nu(\alpha,K)}\cap \Phi_{w}\neq \eset$, then
$\Phi_{\nu(\alpha,K)}\seq \Phi_{w}$.  Suppose that $\gamma\in \Phi_{\nu(\alpha,K)}\cap \Phi_{w}$.
Write $s^{-1}
=(K,\nu(\beta,J),J)$.  Then $\gamma\in \Phi_{\nu(\alpha,K)}=\Phi_{\nu(\beta,J)^{-1}}=\Phi^{+}_{J\cup\set{\beta}}
\sm\Phi_{J}^{+}$. Hence $\gamma\in \Phi_{w}\cap \Phi^{+}
_{J\cup\set{\beta}}$. By Lemma \ref{BHeltlem}, we have 
$\nu(\alpha,K)=\nu(\beta,J)^{-1}\leq w$. That is, $\Phi_{\nu(\alpha,K)}\seq \Phi_{w}$ as required, proving (b).
 \end{proof}
 
 \subsection{} We now discuss squares in the signed groupoid set $R=(W_{\bullet},\Phi_{\bullet})$, stating only the minimum needed for proofs in this section. For more extensive discussions of squares in various  levels of generality, see \cite{DyGrp2} and  \cite{DyWa}.  In the following, we regard any  subset of $W$ as a poset in the order induced by weak order  $\leq$ on $W$.

\begin{defi}\label{squaredef}
A \emph{square} (of $W$) is a quadruple $(w,x,y,z)$ of elements of $W$   such that $wx=yz$ and $w(\Phi_x)=\Phi_y$:
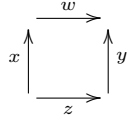
\begin{equation*}
\xymatrix{{}\ar[r]^{{w}}&{}\\
{}\ar[u]^{x}\ar[r]_{{z}}&{}\ar[u]_{y}
}
\end{equation*}
\end{defi}

\begin{prop}\label{squaresym} Let  $x$, $y$, $z$ and $w$ be elements  of $W$. Then the following conditions are equivalent:
\begin{conds}
\item $(w,x,y,z)$ is a square. That is, $wx=yz$ and $w(\Phi_{x})=\Phi_{y}$.
\item $wx=yz$ and $\Phi_{wx}=\Phi_{w}\dotcup\Phi_{y}$.
\item $(y,z,w,x)$ is a square. That is, $yz=wx$ and $y(\Phi_{z})=\Phi_{w}$.
\item $(w^{-1},y,x,z^{-1})$ is a square. That is, $w^{-1}y=xz^{-1}$ and $w^{-1}(\Phi_{y})=\Phi_{x}$.
\item $wx=yz$  and in weak order on $W$,  $w\vee y=wx$ and
$w^{-1}\vee x=w^{-1}y$.
\item $wx=yz$  and in weak order on $W$,  $w\vee y \leq wx$ and
$w^{-1}\vee x\leq w^{-1}y$.
\end{conds}
\end{prop}
\begin{remarks*} (1) There is an action of the dihedral group of order $8$ on $W^{4}$, in which the two simple reflections act by the maps $(w,x,y,z)\mapsto (y,z,w,x)$ and 
$(w,x,y,z)\mapsto (w^{-1},y,x,z^{-1})$. Since  (i)$\iff$(iii)$\iff$(iv), the set of squares of $W$ is invariant under this action. 

(2) By (1), the dihedral group of order $8$  also acts naturally on the ``characterizations'' (of which (i)--(vi) provide examples) of   squares $(w,x,y,z)$. 

(3) A square $(w,x,y,z)$ is uniquely determined by the ordered pair $(w,x)$ (resp., $(w,y)$, $(y,z)$, $(x,z)$).   \end{remarks*}
\begin{proof}   We show  (i)$\implies$(ii). Suppose that  (i) holds. Since $w(\Phi_{x})=\Phi_{y}\seq \Phi^{+}$, we have $\Phi_{x}\cap \Phi_{w^{-1}}=\eset$ and hence $\Phi_{wx}=\Phi_{w}\dotcup w(\Phi_{x})=\Phi_{w}\dotcup \Phi_{y}$, which implies (ii).

We prove  (ii)$\implies$(i). Assume (ii) holds. Since $\Phi_{w}\seq \Phi_{wx}$, we have $w\leq wx$ and so $\Phi_{wx}=\Phi_{w}\dotcup w(\Phi_{x})$. By (ii). we get $\Phi_{y}=w(\Phi_{x})$, proving (i).

Condition (ii) is equivalent to $yz=wx$ and $\Phi_{yz}=\Phi_{y}\dotcup \Phi_{w}$, which by symmetry and (ii)$\iff$(i), is equivalent to (iii). 

Condition (i) is  equivalent to $y^{-1}w=zx^{-1}$ and 
$w^{-1}(\Phi_{y})=\Phi_{x}$, which is (iv). Hence (i)$\iff$(iv).

We show (i)$\implies$(v). Suppose (i) holds. Then (ii) holds, and it implies $w\vee y=wx$. Similarly,  since (iv) holds, we have  
$w^{-1}\vee x=w^{-1}y$. Hence (v) holds.

It is trivial that (v) implies (vi). Finally, we prove that (vi)$\implies$(i). Suppose (vi) holds.
Since $w\leq w\vee y\leq wx$, we have $\Phi_{w^{-1}}\cap \Phi_{x}=\eset$. Similarly, $w^{-1}\leq w^{-1}\vee  x\leq  w^{-1}y$ implies $\Phi_{w}\cap \Phi_{y}=\eset$. These imply \[\Phi_{w}\dotcup \Phi_{y}\seq \Phi_{w\vee y}\seq  \Phi_{wx}=\Phi_{w}\dotcup w(\Phi_{x})\] and so
$\Phi_{y}\seq w(\Phi_{x})$. Similarly,
 \[\Phi_{w^{-1}}\dotcup\Phi_{x}\seq \Phi_{w^{-1}\vee x}\seq \Phi_{w^{-1}y}=\Phi_{w^{-1}}\dotcup w^{-1}(\Phi_{y})\] and thus  $\Phi_{x}\seq w^{-1}(\Phi_{y})$ i.e. $w(\Phi_{x})\seq \Phi_{y}$.
 So we have $wx=yz$ and $w(\Phi_{x})= \Phi_{y}$, proving (i).
\end{proof}

\subsection{} Recall that a \emph{complete meet semi-lattice} is a non-empty  poset $(X,\preceq)$ such that  every non-empty subset of $X$ has a meet  (also called its glb or inf) in $X$.

Let $X$ be a complete meet semilattice.   Then  any subset $A$ of $X$ which has an upper bound in $X$  has a join (also called its lub or sup) in $X$; 
that join is the meet of the set of  upper bounds of $A$. We say that a non-empty subset $Y$ of $X$ is a \emph{join-closed meet sub-semilattice} of $X$ if for any non-empty subset $A$ of $Y$, its meet $\bigwedge A$ in $X$ satisfies  $\bigwedge A\in Y$, and further, if $A$ has a join $\bigvee A \in X$, then $\bigvee A\in Y$. 

We say a morphism $f\colon X_{1}\to X_{2}$ of posets preserves existing joins (resp., existing meets) if for any subset 
$A$ of $X_{1}$ such that $\bigvee A$ (resp., $\bigwedge A$) exists  in $X_{1}$,  the join $\bigvee f(A)$ (resp., the meet $\bigwedge f(A)$) exists  in $X_{2}$ and one has  $f(\bigvee A)=\bigvee f(A)$ (resp., 
$f(\bigwedge A)=\bigvee f(A)$).

\begin{prop}\label{ordiso} Let $w\in W$.  Define $A_{w}:=\mset{x\in W\mid \Phi_{x}\cap \Phi_{w^{-1}}=\eset}$ and $B_{w}:=\mset{z\in W\mid \Phi_{w}\seq \Phi_{z}}$, so $1_{W}\in A_{w}$ and $w\in B_{w}$.
\begin{num}
\item $A_{w}$ and $B_{w}$ are join-closed meet sub-semilattices of $(W,\leq)$.
\item The map   $\theta_{w}\colon A_{w}\to B_{w}$ defined by $\theta_{w}(x):=wx$ is an order isomorphism.
\item $\theta_{w}$ and $\theta_{w}^{-1}$ each  preserve all existing meets and joins of subsets of their domain.
\end{num}   \end{prop}
\begin{proof} We prove (a).  The fact $A_{w}$ is closed under formation of  meets  in $W$ of non-empty subsets of $A_{w}$ is clear since $A_{w}$ is a down-set of $W$.  Also, if $A\seq A_{w}$ and $\bigvee A$ exists in $(W,\leq)$, then $\bigvee A\in A_{w}$ by the JOP for $(W,\leq)$. The assertion for $B_{w}$ is clear since $B_{w}$ is the principal up-set of $W$ generated by $w$.

Part (b) follows from basic properties of weak order (see \ref{weakorder}), and (c) is a direct consequence of (b).
\end{proof} 

\begin{prop}\label{lemma:squaremeetjoin}
 Suppose  that $(w,x_{i},y_i,z_i)$ are squares, for $i\in I$.
 \begin{num}
 \item If $I\neq\eset$,  there is a unique square $(w,x,y,z)$ with $x=\bigwedge_ix_i$ and $y=\bigwedge_iy_i$.
 \item The join $x=\bigvee_ix_i$ exists in $(W,\leq)$ if and only if the join  $y=\bigvee_iy_i$ exists. In that case, there exists a unique square $(w,x,y,z)$.
  \end{num}  
\end{prop}

\begin{proof} We  make extensive use of Propositions \ref{ordiso} and \ref{squaresym} in the proof, without further explicit reference to them. In both (a) and (b), there is at most one square  $(w,x,y,z)$ satisfying  the indicated conditions,  by Remark \ref{squaresym}(3).

We prove (a). We have $x_{i}\in A_{w}$  for all $i\in I$, so $x=\bigwedge_{i}x_{i}\in A_{w}$.
For each $i\in I$,  $y\leq y_{i}\leq wx_{i}=\theta_{w}(x_{i})$, so 
$y \leq \bigwedge_{i\in I}  \theta_{w}(x_{i})=\theta_{w}(\bigwedge_{i\in I}x_{i})=\theta_{w}(x)=wx$. Also, $w\leq wx$ since $x\in A_{w}$. Hence $w\vee y\leq wx$.   By symmetry, we also have $w^{-1}\vee x\leq w^{-1}y$. Choose $z$ so that $wx=yz$. Then $(w,x,y,z)$ is a square, proving (a). 

Now we prove (b).  Suppose that $x=\bigvee_{i} x_{i}$ exists in $W$.  Since $x_{i}\in A_{w}$ for each $i\in I$, we have $x\in A_{w}$. 
For each $i\in I$, we have $y_{i}\leq w\vee y_{i}=wx_{i}=\theta_{w}(x_{i})\leq\theta_{w}(x)=wx$, so $y=\bigvee_{i}y_{i}$ exists and $y\leq wx$. By symmetry, if  $y=\bigvee_{i}y_{i}$ exists, then $x=\bigvee_{i} x_{i}$ exists and $x\leq w^{-1}y$.   

Assume henceforward that both joins $x$ and $y$ exist. 
Then from above, we have $y\leq wx$ and $x\leq w^{-1}y$.
Since $x\in A_{w}$,  we have $w\leq wx$, hence $w\vee y\leq wx$. By symmetry, we also have  $w^{-1}\vee x\leq w^{-1}x$.
 Choose $z$ so that $wx=yz$. Then $(w,x,y,z)$ is a square, proving (b). 
\end{proof}

\begin{prop}\label{BHsquare} Let $J$ be an object of $G$ and $w\in W$. Then there is a morphism $(J,w,K)$ in $G$ for some $K\seq \Pi$ if and only if for each $\alpha\in J$, there exists a square
$(w,x_{\alpha},s_{\alpha},z_{\alpha})$ for some $x_{\alpha},z_{\alpha}\in W$. If so, then  necessarily
$K=w^{-1}J$ and for all $\alpha\in J$,  $z_{\alpha}=w$ and 
$x_{\alpha}=s_{w^{-1}\alpha}$ where $w^{-1}\alpha\in K$. 
\end{prop}
\begin{proof}
If $(J,w,K)$ is a morphisms in $J$, then $K=w^{-1}J\seq \Pi$ and for each $\alpha\in J$, there exists a square
$R:=(w,s_{w^{-1}\alpha}, s_{\alpha}, w)$ since $ws_{w^{-1}\alpha}=s_{\alpha}w$ and $w(\Phi_{s_{w^{-1}\alpha}})=w\set{w^{-1}\alpha}=\set{\alpha}=
\Phi_{s_{\alpha}}$ because  $\alpha\in J\seq \Pi$ and $w^{-1}\alpha\in K\seq \Pi$. By Remark \ref{squaresym}(3),
$R$ is the only square of the form $(w,x_{\alpha},s_{\alpha},z_{\alpha})$ where $x_{\alpha},z_{\alpha}\in W$.

Conversely, suppose that for each $\alpha\in J$, there exists a square  $(w,x_{\alpha},s_{\alpha},z_{\alpha})$ where $x_{\alpha},z_{\alpha}\in W$.  Let $\alpha\in J$. Then  
$w(\Phi_{x_{\alpha}})=\Phi_{s_{\alpha}}=\set{\alpha}$.
This implies that $\Phi_{x_{\alpha}}$ is a singleton set,
so $x_{\alpha}\in S$, say $x_{\alpha}=s_{\alpha'}$ where $\alpha'\in \Pi$.  But then $\set{\alpha}=w(\Phi_{x_{\alpha}})=
w(\Phi_{s_{\alpha'}}) =w\set{\alpha'}=\set{w\alpha'}$ forcing
$w^{-1}\alpha=\alpha'\in \Pi$. This shows that $K:=w^{-1}J=\mset{w^{-1}\alpha\mid \alpha\in J}\seq \Pi$ and so $(J,w,K)$ is a morphism of $G$.  
 \end{proof}
 
 \begin{prop}\label{subsemi} Let $J\in \ob(G)$. \begin{num}
 \item The map $F_{J}\colon (\lsub{J}{G}, \lsub{J}{\leq})\to (W,\leq)$ in Proposition \ref{weakprop} preserves all meets of non-empty subsets of  $\lsub{J}{G}$ and all 
existing  joins of subsets of $\lsub{J}{G}$. 
 \item The image of $F_{J}$ is a join-closed meet sub-semilattice of $(W,\leq)$.  
 \end{num}  \end{prop}
 \begin{proof} Consider an index set $I$ and a family of morphisms $(g_{i})_{i\in I}$ in $\lsub{a}{G}$. Write $w_{i}:=F_{J}(g_{i})\in W$. Since $F_{J}$ is an order embedding, it will suffice to prove the following two claims:
 \begin{num}
 \item[(1)]  $g:=\bigvee_{i} g_{i}$ exists if and only if $w:=\bigvee_{i} w_{i}$ exists, and then $w=F_{J}(g)$.
 \item[(2)]  If $I\neq \eset$,  then $g':=\bigwedge_{i} g_{i}$ exists,
 $w':=\bigwedge_{i}w_{i}$ exists and  $w'=F_{J}(g'_{i})$.
 \end{num} 
 
 Write $g_{i}=(J,w_{i},K_{i})$ for $i\in I$. Then by Proposition \ref{BHsquare}, for each $i\in I$ and $\alpha\in J$, there is a square $(w_{i},x_{i,\alpha},s_{\alpha},z_{i,\alpha})$ or equivalently, a square $(s_{\alpha},z_{i,\alpha},w_{i},x_{i,\alpha})$.
 
 We prove (1). Suppose first that $g:=\bigvee_{i} g_{i}$ exists.
 Then $w_{i}=F_{J}(g_{i})\leq F_{J}(g)$, so $F_{J}(g) $ is an upper bound for $\set{w_{i}\mid i\in I}$ and $w:=\bigvee_{i\in I}w_{i}$ exists. Hence to prove (1), we may assume that 
 $w:=\bigvee_{i\in I}w_{i}$ exists,  and it will suffice to show show that $g=\bigvee_{i}g_{i}$ exists and $F_{J}(g)=w$.  By Proposition \ref{lemma:squaremeetjoin}, for each $\alpha\in J$, there is a square
 $(s_{\alpha},z_{\alpha},w,x_{\alpha})$. Hence there is a square $(w,x_{\alpha},s_{\alpha},z_{\alpha})$. By Proposition \ref{BHsquare}, there is a morphism $g=(J,w,K)$ in $G$, where
 $K:=w^{-1}J$.  We check that $g=\bigvee_{i}g_{i}$.  First, $g_{i}\lsub{J}{\leq}g$ for all $i$ since $F_{J}(g_{i})=w_{i}\leq w=F_{J}(g)$. 
 Second, if $h\in \lsub{J}{G}$ with $g_{i} \lsub{J}{\leq}h$ for all $i$, then $w_{i}=F_{J}(g_{i})\leq F_{J}(h)$. This implies $F_{G}(g)=w=\bigvee_{i}w_{i}\leq F_{J}(h)$ and so $g\lsub{J}{\leq}h$. 
 This proves that $g=\bigvee_{i}g_{i}$. We have $F_{J}(g)=w$,  proving (1). 
 
 Now we prove (2).  Let $w':=\bigvee_{i}w_{i}$, which exists since $W$ is a complete meet semilattice and $I\neq \eset$.
 By Proposition \ref{lemma:squaremeetjoin}, for each $\alpha\in J$, there is a square
 $(s_{\alpha},z'_{\alpha},w',x'_{\alpha})$. Hence there is a square $(w',x'_{\alpha},s_{\alpha},z'_{\alpha})$. By Proposition \ref{BHsquare}, there is a morphism $g'=(J,w',K')$ in $G$, where
 $K':=w^{\prime-1}J$.  We check that $g'=\bigwedge_{i}g_{i}$. 
 First, $g'\lsub{J}{\leq}g_{i}$ for all $i$ since $F_{J}(g')=w'\leq w_{i}=F_{J}(g_{i})$. 
 Second, if $h'\in \lsub{J}{G}$ with $h'\lsub{J}{\leq}g_{i}$ for all $i$, then $F_{J}(h')\leq F_{J}(g_{i})=w_{i}$. This implies $F_{J}(h')\leq \bigwedge_{i}w_{i}=w'=F_{J}(g')$ and so $h'\lsub{J}{\leq}g'$. 
 This completes the proof that $g'=\bigwedge_{i}g_{i}$.
 Since $F_{J}(g')=w'$, this proves  (2).  \end{proof}
 \begin{proof}[Proof of Theorem \ref{standpreprinc}]
 
 We prove Theorem \ref{standpreprinc}(a).  
Propositions \ref{weakorder} and  \ref{prepatom} prove that $(G, \Lambda)$ is faithful and  preprincipal and that its atomic morphisms are the Brink-Howlett generators. It remains to show  that $ (G, \Lambda)$ is rootoidal. Let  $J$ be an object  of $ (G, \Lambda)$. Since $(W,\leq)$ is a complete meet semilattice, the weak order $ (\lsub{J}{G} , \lsub{J}{\leq})$ at $J$ is a complete meet semilattice by Proposition \ref{subsemi}(a). Finally, we prove the JOP, which is the assertion that if $ (g_{i})_{i\in I}$ and 
$h$ are in $\lsub{J}{G}$ are such that $\Lambda_{g_{i}}\cap \Lambda_{h}=\eset$ for all $i\in I$ and  $g:=\bigvee_{i\in I}g_{i}$ exists,   then $\Lambda_{g}\cap \Lambda_{h}=\eset$.
Write $g_{i}=(J,w_{i},K_{i})$, $g=(J,w,K)$ and $h=(J,v,L)$ where $w_{i},w,v\in W$ and $K_{i},K,L\in \ob(G)$. Then $w_{i}=F_{J}(g)$, $v=F_{J}(h)$ and by Proposition \ref{subsemi}, 
$w=F_{J}(g)=\bigvee_{i}F_{J}(g_{i})=\bigvee_{i}w_{i}$.
We have 
$\Phi_{w_{i}}\cap \Phi_{v}= \Lambda_{g_{i}}\cap \Lambda_{h}
=\eset$. 
Using the JOP for $(W_{\bullet}, \Phi_{\bullet})$, we have
$\Lambda_{g}\cap \Lambda_{h}=\Phi_{w}\cap\Phi_{v}=\eset$, 
proving the JOP for $(G,\Lambda)$. 

Theorem \ref{standpreprinc}(b)  is an immediate consequence of  Theorem \ref{standpreprinc}(a) and \cite[Lemma 2.26(b)]{DyWa}.
\end{proof}

\section{Abstract root systems of Brink-Howlett groupoids} 
\label{s:abrs}
\begin{ex}\label{gpdrootex} Suppose that $W$ is finite.  Let $J\seq S$. The unique $U'\in M_{J}$ containing
$U\in R_{J}$ is the parabolic closure of $U$, which is defined to be  the intersection of all parabolic subgroups of $W$ which contain $U$. One has $U'\in P_{J}$ and   $\Phi_{U'}=\Span(\Phi_{U})\cap \Phi$.  It follows that $M_{J}=P_{J}$ (for finite $W$). In general, $P_{J}\neq R_{J}$; for instance, this is well known for $(W,S)$  of type $B_{2}$.

Suppose instead that $W$ is irreducible affine and $J=\{\alpha\}$ where $\alpha\in \Pi$. Then $M_{J}$ consists of the rank two, finite parabolic subgroups containing $s=s_{\alpha}\in S$, together with the unique
infinite maximal dihedral reflection subgroup $W'$ containing $s$. In the standard realization of $W$ as Euclidean reflection group, $W'$ is generated by the reflections in the  reflecting hyperplanes parallel (in the usual sense of Euclidean geometry) to the reflecting hyperplane for $s$. In particular, $M_{J}\supsetneq P_{J}$ in this case unless $W$ is of rank two (that is,  type $\widetilde A_{1}$). 
 \end{ex}
 \begin{prop} \label{refsubgptypes}
\begin{num}
\item  $P_{J}\seq M_{J}$ and  $Q_{J}=M_{J}\cap N_{J}$.
\item Inclusions as shown in the following Hasse  diagram hold  amongst the sets $Q_{J}$, $P_{J}$, $M_{J}$,
$N_{J}$ and $R_{J}$ of corank one reflection overgroups of $W_{J}$:  
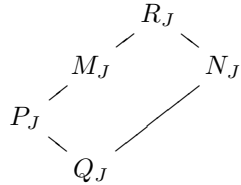
\begin{equation*}
\xymatrix@R=5pt@C=5pt{&&{R_{J}}&\\
&{M_{J}}\ar@{-}[ur]&&{N_{J}}\ar@{-}[ul]\\
{P_{J}}\ar@{-}[ur]&&&\\
&{Q_{J}}\ar@{-}[uurr]\ar@{-}[ul]&&&
}
\end{equation*} \end{num} 
\end{prop}
\begin{remarks*} Additional inclusions may hold in particular cases. For instance, if $W$ is finite, then $M_{J}=P_{J}$ and $R_{J}=N_{J}$, while if $(W,S)$ is of type $A$, then $R_{J}=P_{J}$. \end{remarks*}
\begin{proof}[Proof of Proposition \ref{refsubgptypes}]  We prove (a). We have $P_{J}\seq M_{J}$ by Theorem \ref{maxref}(a). Hence $Q_{J}=P_{J}\cap N_{J}\seq
M_{J}\cap N_{J}$.  For the reverse inclusion, let $U\in M_{J}\cap N_{J}$. Since $(U\sm W_{J})\cap T$ is finite, Proposition \ref{finindcond} implies  $[U:W_{J}]<\infty$.  By Theorem \ref{infTits},  there exists $U'\in P_{J}\seq R_{J}$ such that  $U\seq U'$.  Since $U\in M_{J}$, we have $U'=U$. Hence $U\in P_{J} \cap N_{J}$ as  required.

The inclusions  diagrammed in (b)  follow from (a) and the definitions. 
\end{proof}
\begin{prop}\label{naivers}  Let $(K,w,J)$ be a morphism in $G$ and 
$U\in R_{J}$. Then: \begin{num}
\item $r_{U,J} \not \in \Phi_{w^{-1}}$ holds if and only if  $((U\sm W_{J})\cap T)\cap N(w^{-1})=\eset$, and also if and only if $(\Phi^{+}_{U}\sm\Phi^{+}_{J})\cap \Phi_{w^{-1}}=\eset$. 
\item  $r_{U,J} \in \Phi_{w^{-1}}$  holds if and only if  $((U\sm W_{J})\cap T)\seq N(w^{-1})$, and also  if and only if
$\Phi^{+}_{U}\sm\Phi^{+}_{J}\seq \Phi_{w^{-1}}$. \item   If $r_{U,J} \not \in \Phi_{w^{-1}}$, then $\Pi_{wUw^{-1}}=w\Pi_{U}$ and  
 $r_{wUw^{-1},K}=wr_{U,J}$. 
 \item    Suppose $\alpha:=r_{U,J} \in \Phi_{w^{-1}}$. Let 
 $L\dot\cup\set{\alpha}$ be the component of $\Pi_{U} $ 
 containing $\alpha$, where $L\seq \Pi_{U}$. Then  $u:=\nu(\alpha,J)=w_{L\cup\set{\alpha}}w_{L}\in U$ is defined, 
 $\Pi_{wUw^{-1}}=wu^{-1}\Pi_{U}$ and  
 $r_{wUw^{-1},K}=-ww_{L}r_{U,J}$. 
      \end{num}
  \end{prop}
  \begin{proof}
Make assumptions as in Proposition \ref{naivers}.  For the proof, there is no loss of generality in assuming that $\Pi$ is linearly dependent. Then for any subset $I$ of $\Pi$, $\Phi_{I}=\Span(I)\cap \Phi$.

The map $\alpha\mapsto s_{\alpha}\colon \Phi^{+}\to T$ is a bijection which maps $\Phi_{w^{-1}}$ onto $N(w^{-1})$ and 
$\Phi^{+}_{U}\sm\Phi_{J}^{+}$ onto $(U\sm W_{J})\cap T$.
Hence  $((U\sm W_{J})\cap T)\cap N(w^{-1})=\eset$ if and only if $(\Phi^{+}_{U}\sm\Phi^{+}_{J})\cap \Phi_{w^{-1}}=\eset$.
 Similarly,   $((U\sm W_{J})\cap T)\seq N(w^{-1})$ if and only if
$(\Phi^{+}_{U}\sm\Phi^{+}_{J})\seq \Phi_{w^{-1}}$.

To prove (a)--(b), it will suffice, by the previous paragraph, to show the following three conditions (i)--(iii) are equivalent:
\begin{conds}
\item $r_{U,J}\in \Phi_{w}^{-1}$
\item $(\Phi^{+}_{U}\sm\Phi^{+}_{J})\cap \Phi_{w^{-1}}\neq \eset$ 
\item $\Phi_{U}^{+}\sm \Phi_{J}^{+} \seq  \Phi_{w^{-1}}$
\end{conds}
Since $r_{U,J}\in \Phi^{+}_{U}\sm \Phi_{J}^{+}$, we have  (i)$\implies$(ii) and also  (iii)$\implies$(i) .

We prove (ii)$\implies$(i).   Let $\beta\in (\Phi_{U}^{+}\sm \Phi_{J}^{+}) \cap \Phi_{w^{-1}}$.  Then we may write  
$\beta=\sum_{\gamma\in \Pi_{U} }c_{\gamma}\gamma$ where $c_{\gamma}\geq 0$ for all $\gamma\in \Pi_{U}=J\cup\set{r_{U,J}}$ and $c_{r_{U,J}}>0$. Since $w(\beta)=\sum_{\gamma\in \Pi_{U} }c_{\gamma}w(\gamma)\in \Phi^{-}$ and $K:=w(J)\seq \Pi$, this implies that $w(r_{u,J})\in \Phi^{-}$. That is, $r_{U,J}\in \Phi_{w^{-1}}$ as required.

Finally, we show that (i)$\implies$(iii). Assume that $r_{U,J}\in \Phi_{w^{-1}}$.   For any $\beta'\in (\Phi_{U}^{+}\sm \Phi_{J}^{+})$, write  
$\beta'=\sum_{\gamma\in \Pi_{U} }c'_{\gamma}\gamma$ 
where $c'_{\gamma}\geq 0$ for all $\gamma\in \Pi_{U}=J\cup\set{r_{U,J}}$ and $c'_{r_{U,J}}>0$. We have 
$w(J)\seq K$ and $w(r_{U,J})\in \Phi^{-}\sm w(\Phi_{J})=\Phi^{-}\sm \Phi_{K}$, which implies that  $w(\beta') =
\sum_{\gamma\in \Pi_{U} }c'_{\gamma}\gamma\in \Phi^{-}$. 
Hence $\beta'\in \Phi_{w^{-1}}$, proving (iii).

 This completes the proof of (a)--(b). We next prove (c).
Suppose that  $r_{U,J}\not\in \Phi_{w^{-1}}$. It follows that 
\[w\Pi_{U}=w(J\cup\set{r_{J,K}})=w(J)\cup\set{w(r_{U,J})}=K\cup\set{w(r_{U,J})}\seq \Phi^{+}.\] This implies that $\Pi_{wUw^{-1}}=w\Pi_{U}=K\dot\cup\set{wr_{U,J}}$. Since $ W_{K}=wW_{J}w^{-1}\seq wUw^{-1}$, it follows that  $wUw^{-1}\in R_{K}$ and so $\Pi_{wUw^{-1}}=  K\dot\cup\set{r_{wUw^{-1},K}}$. Hence 
$r_{wUw^{-1},K}=w(r_{U,J})$, proving (c).

Finally, we prove (d).  Assume $\alpha=r_{U,J}\in \Phi_{w^{-1}}$. 
Since  $\vert \Phi_{w^{-1}}\vert =\ell(w)<\infty$, (b) implies that $\Phi^{+}_{U}\sm \Phi_{J}^{+}$ is finite.  By Propositions \ref{Deodgen} and \ref{finindcond}, the component $L\cup\set{\alpha}$ of $\Pi_{U}$ is of finite type. Hence $u:=\nu(\alpha, L)=w_{L\cup\set{\alpha}}
w_{L}$ is defined in $(U,\chi(U))$. We have $\Pi_{U}=J\dotcup\set{\alpha}=(J\sm L)\cup (L\cup\set{\alpha})$.
where $(J\sm L)\perp  (L\cup\set{\alpha})$. Since $u\in W_{L\cup\set{\alpha}}$, we have  $u^{-1}(J\sm L)=J\sm L$, and  so  
\begin{equation*}\begin{split}
u^{-1}\Pi_{U}=&u^{-1}((J\sm L)\cup(L\cup\set\alpha))=u^{-1}(J\sm L)\cup u^{-1}(L\cup\set{\alpha})\\=&(J\sm L)\cup w_{L}w_{L\cup\set{\alpha}}(L\cup\set{\alpha})=
(J\sm L)\cup -w_{L}(L\cup\set{\alpha})\\=&(J\sm L)\cup (L\cup\set{-w_{L}\alpha})=J\cup\set{-w_{L}\alpha}.\end{split}
\end{equation*}
Now $-w_{L}\alpha\in -\alpha+\Span(L)$. 
Since $\alpha\not\in w_{-1}\Phi_{K}=\Phi_{J}$, it follows that  $-w \alpha\in \Phi^{+}\sm \Phi_{K}=\Phi^{+}\sm \Span(K)$. Since  $w(\Span(L))\seq w\Span(J)=\Span K$, this implies  that $-ww_{L}\alpha\in \Phi^{+}$. Hence
$wu^{-1}\Pi_{U}=wJ\cup\set{-ww_{L}\alpha}=K\cup\set{-ww_{L}\alpha}\seq \Phi^{+}$. This shows that $wu^{-1}\Pi_{U}=\Pi_{wu^{-1}U(wu^{-1})^{-1}}=\Pi_{wUw^{-1}}$.  Since $U\in R_{J}$, we have  $wUw^{-1}\in R_{wUw^{-1}}=R_{K}$ and so
$ \Pi_{wUw^{-1}}=K\dot\cup\set{r_{wUw^{-1},K}}=K\cup\set{-ww_{L}\alpha}$. Therefore 
$r_{wUw^{-1},K}=-ww_{L}\alpha=-ww_{L}r_{U,J}$, which completes the proof of (d). 
 \end{proof}
 \begin{proof}[Proof of Theorem \ref{BHabrs}]
For each possible choice of  $X$ (i.e. $R$, $M$, $N$, $P$ or $Q$)  and each morphism $(K,w,J)$ in $G$,  and all $U\in X_{J}$, one has $wUw^{-1}\in X_{K}$.
 Hence the formula for $\Upsilon_{X}(K,w,J)$  in the statement of the theorem defines a function $X_{J}\times\set{\pm}\to X_{K}\times\set{\pm}$.
 Straightforward calculation  shows that the specified  maps $\Upsilon_{X}$ on objects and morphisms define a functor  giving
   a   signed $G$-set  $\Upsilon_{X}$ provided  that for  any morphisms $(K,x,J)$ and $(L,y,K)$ of $G$ and all   $U\in X_{J}$, we have 
 \begin{equation*}
 \eta (yx,U,J)\equiv \eta (y,xUx^{-1},K)+\eta (x,U,J)\pmod{2} 
\end{equation*} 

Recall that $\eta (x,U,J)$ is equal to $1$ or $0$ according as whether $(U\sm W_{J})\cap T$ is contained in, or disjoint from,
$N(x^{-1})$ (containment or disjointness are the only possibilities, by Proposition \ref{naivers}). This implies similar descriptions of $\eta (y,xUx^{-1},K)$ and $\eta (yx,U,J)$.   Since  $W_{K}=xW_{J}x^{-1}$,  that description of $\eta (y,xUx^{-1},K)$ implies that $\eta (y,xUx^{-1},K)$ is equal to $1$ or $0$ according as whether  $(U\sm W_{J})\cap T$ is contained in, or disjoint from, $x^{-1}N(y^{-1})x$.  We relate $\eta (yx,U,J)$ to 
$\eta (x,U,J)$ and $\eta (y,xUx^{-1},K)$, by  considering cases, using the $1$-cocycle formula 
\[N((yx)^{-1})=N(x^{-1}y^{-1})=
 N(x^{-1})+x^{-1}N(y^{-1})x\] (in which, recall, $+$ denotes symmetric difference of sets: $A+B:=(A\cup B)\sm (A\cap B)$).
 
 \begin{case} $\eta (x,U,J)=0$ and $\eta (y, xUx^{-1},K)=0$.
 
 We have $(U\sm W_{J})\cap N(x^{-1})=\eset$ and 
 $(U\sm W_{J})\cap x^{-1}N(y^{-1})x=\eset$. The $1$-cocycle formula implies $(U\sm W_{J})\cap N((yx)^{-1})=\eset$, so 
 \[\eta (x,U,J)+\eta (y, xUx^{-1},K)=0=\eta (xy,U,J).\] 
   \end{case}
   
   \begin{case} $\eta (x,U,J)=0$ and $\eta (y, xUx^{-1},K)=1$.
 
 We have $(U\sm W_{J})\cap N(x^{-1})=\eset$ and 
 $(U\sm W_{J})\cap T\seq  x^{-1}N(y^{-1})x$. The $1$-cocycle formula implies $(U\sm W_{J})\cap T\seq  N((yx)^{-1})$, so 
 \[\eta (x,U,J)+\eta (y, xUx^{-1},K)=1=\eta (xy,U,J).\] 
   \end{case}

\begin{case} $\eta (x,U,J)=1$ and $\eta (y, xUx^{-1},K)=0$.
 
 We have $(U\sm W_{J})\cap T\seq  N(x^{-1})$ and 
 $(U\sm W_{J})\cap x^{-1}N(y^{-1})x=\eset$. The $1$-cocycle formula implies $(U\sm W_{J})\cap T\seq  N((yx)^{-1})$, so 
 \[\eta (x,U,J)+\eta (y, xUx^{-1},K)=1=\eta (xy,U,J).\] 
   \end{case}

\begin{case} $\eta (x,U,J)=1$ and $\eta (y, xUx^{-1},K)=1$.
 
 We have $(U\sm W_{J})\cap T\seq  N(x^{-1})$ and 
 $(U\sm W_{J})\cap T\seq  x^{-1}N(y^{-1})x$. The $1$-cocycle formula implies $(U\sm W_{J})\cap N((yx)^{-1})=\eset$, so 
 \[\eta (x,U,J)+\eta (y, xUx^{-1},K) =2\equiv 0=\eta (xy,U,J) \pmod{2}.\] 
   \end{case}
   
   This completes the proof that $\Upsilon_{X}$ as defined  in Theorem \ref{BHabrs} is a signed $G$-set .

   For the proof of \ref{BHabrs}(a), let $J\in \ob(G)$ and  $(U,\epsilon)\in \Upsilon_{X}(J)$. Suppose first that $(U,\epsilon)$ is a real root.    This means that there is a morphism $(K,w,J)$ in $G$ such that 
   $(\Upsilon_{X}(K,w,J))(U,\epsilon)=(wUw^{-1},-\epsilon)$. By definition of $\Upsilon_{X}$, we have  $(U\sm W_{J})\cap T\seq N(w^{-1})$. Since $N(w^{-1})$ is finite,  $(U\sm W_{J})\cap T $ is also finite.
   
   Conversely, suppose $(U\sm W_{J})\cap T $ is finite.   By Proposition \ref{finindcond}, the reflection subgroup $U$ of $W$ is a finite index overgroup of $W_{J}$. By Theorem \ref{infTits}, there is a  parabolic subgroup $U'\seq U$ of $W$ such that  $[U':W_{J}]$ is finite and, moreover,
   $\corank_{U'}(W_{J})=\corank_{U}(W_{J})=1$.  We have $J\seq \Pi_{U'}$ since $W_{J}\seq U'$. Since $U'$ is parabolic, there exists $w\in W$ with $w(\Pi_{U'})\seq \Pi$. Let $K=w(J)\seq \Pi$.    Then $\Pi\sreq \Pi_{wU'w^{-1}}=w\Pi_{U'}\sreq wJ=K$.
 Since $\corank_{wU'w^{-1}}(W_{K})=\corank_{U'}(W_{J})=1$, we may  write $w(\Pi_{U'})\sm K=\set{\alpha}$ where $\alpha\in \Pi$ (so by definition,  $\alpha=r_{wU'w^{-1},K}$). 
 
  Note that  $(K,w,J)$ is a morphism of $G$. Now since $w(\Pi_{U'})\seq \Pi$, we have
 $w\Phi^{+}_{U}\seq 
 w\Phi^{+}_{U'}=\Phi^{+}_{wU'w^{-1}}\seq \Phi^{+}$.
 This implies $(U\sm W_{J})\cap N(w^{-1})=\eset$ and so
 $(\Upsilon_{X}(K,w,J))(U,\epsilon)=(wUw^{-1},\epsilon)\in X_{K}\times\set{\pm}$.
 
 We have  $[wU'w^{-1}:W_{K}]=[U':W_{J}]<\infty$, so $\Phi^{+}_{K\cup\set{\alpha}}\sm \Phi^{+}_{K}=\Phi^{+}_{wU'w^{-1}}\sm \Phi^{+}_{K}$ is finite by Proposition \ref{finindcond}. Hence
 $y:=\nu(\alpha,K)\in W$ is defined and there is a morphism
 $(L,y,K)$ in $G$ where $L:=yK\seq \Pi$.
 Note that \[N(y^{-1})=(W_{K\cup\set{\alpha}}\sm W_{K})\cap T=(wU'w^{-1}\sm W_{K})\cap T\sreq (wUw^{-1}\sm W_{K})\cap T.\] Hence $(\Upsilon_{X}(L,y,K))(wUw^{-1},\epsilon)=((yw)U(yw)^{-1},-\epsilon)\in X_{L}\times\set{\pm}$. 
 The composite morphism $(L,yw,J)=(L,y,K)(K,w,J)$ of $G$  satisfies  $(\Upsilon_{X}(L,yw,J))(U,\epsilon)=(ywU(yw)^{-1},-\epsilon)$.  By the definitions,  $(U,\epsilon)$ is a real root, proving (a). 
 
 Note that  the  functors $\Upsilon_{X}$ are all subfunctors of $\Upsilon_{R}$ (as set-valued functors). Theorem \ref{BHabrs}(b) then follows  from Proposition \ref{refsubgptypes}(b).

    We prove (c).  The real roots of $(G,\Upsilon_{X})$ based at $J$ are, by part (a), the roots of the form $(U,\epsilon)$ where $U\in X_{J}$, $\epsilon\in \set{\pm}$ and $(U\sm W_{J})\cap T$ is finite.  By definition of $N$, the only condition on $U$ is just that $U\in N_{J}\cap X_{J}$. The assertions of (c) follow since $N_{J}\cap R_{J}=N_{J}$ (recall that $N_{J}\seq R_{J}$),  $N_{J}\cap M_{J}=Q_{J}$ by    Proposition \ref{refsubgptypes}(a),  and $N_{J}\cap P_{J}=Q_{J}$ by definition.

For the proof of (d), we shall define  a signed $G$-subset $(G,\Xi)$ of $(G,\Lambda)$ and morphisms of signed sets $\tau\colon (G,\Upsilon_{R})\to (G,\Upsilon_{M})$ and $\rho\colon (G,\Xi)\to (G,\Upsilon_{M})$.  

For a morphism $(K,w,J)$ in $G$, we have $\Lambda(J)=\Phi\sreq \Phi_{J}$ and  $(\Lambda(K,w,J))(\Phi_{J}^{\pm})=\Phi_{K}^{\pm}$. 
This implies that there is a signed $G$-subset $(G,\Xi)$ 
of $(G,\Lambda)$  such that for any  $J\in \ob(G)$, the 
underlying set of  $\Xi(J)$ is $\Xi(J):=\Phi\sm \Phi_{J}\seq \Phi=\Lambda(J)$. The above also shows that
$\Phi_{J}^{\pm}\seq \Lambda^{\im}_{J}$ for $J\in \ob(G)$ and hence
\begin{equation}\label{LamXirec}
(G,\Lambda^{\re})=(G,\Xi^{\re}), \qquad (G,\Lambda^{\rec})=(G,\Xi^{\rec}). 
\end{equation}

Define a morphism of signed $G$-sets  $\rho\colon (G,\Xi)\to (G,\Upsilon_{M})$ as follows.   For any $J\in \ob(G)$
 and $\alpha\in \Xi(J)=\Phi\sm \Phi_{J}$,  Theorem 
 \ref{maxref} implies that there exist unique  $U\in M_{J}$ and 
 $\epsilon\in \set{\pm}$ such that  $\alpha\in \epsilon(\Phi^{+}_{U}\sm \Phi_{J}^{+})$.
 Define  $\rho_{J}\colon \Xi(J)\to \Upsilon_{M}(J)$  by 
 $\rho_{J}(\alpha)=(U,\epsilon)$.  This defines  a 
 morphism of signed $G$-sets, since,  for a morphism $(K,w,J)$ in $G$, and $\alpha, U,\epsilon$
  as above,  we have 
 \[w\alpha\in  w \epsilon(\Phi^{+}_{U}\sm \Phi_{J}^{+}) =
 (-1)^{\nu(w,U,J)}\epsilon (\Phi^{+}_{wUw^{-1}}\sm \Phi_{K}^{+}),\] where the  equation holds as $\nu(w,U,J)=1$ if  $ \Phi^{+}_{U}\sm \Phi_{J}^{+}\seq \Phi_{w^{-1}}$ and $\nu(w,U,J)=0$ otherwise, in which case $(\Phi^{+}_{U}\sm \Phi_{J}^{+})\cap \Phi_{w^{-1}}=\eset$.  
 
 For each object $J$ of $G$, the component $\rho_{J}$ is surjective. From \eqref{comp2}, we see that $\rho^{\cc}\colon \Xi^{\cc}\to \Upsilon_{M}^{\cc}$ is an isomorphism. Hence so is $\rho^{\rec}\colon \Xi^{\rec}\to \Upsilon_{M}^{\rec}$.   By  \eqref{LamXirec} and (c), \begin{equation}
 \label{LambdaUpsM}(G, \Lambda^{\rec})=(G,\Xi^{\rec})\cong(G,\Upsilon_{M}^{\rec})=(G,\Upsilon_{P}^{\rec})=(G,\Upsilon_{Q}^{\cc}).
 \end{equation}
 
 Define a morphism $\tau\colon (G,\Upsilon_{R})\to (G,\Upsilon_{M})$  as follows.  For any  object $J$ of $G$ and  element $U$ of $R_{J}$, let $U_{[J]}\in M_{J}$ denote the maximum corank-one reflection overgroup of $W_{J}$ such that $U\seq U_{[J]}$, which exists by Theorem \ref{maxref}. 
 Let $\tau_{J}\colon
 \Upsilon_{R}(J)\to \Upsilon_{M}(J)$  denote  the map $(U,\epsilon)\mapsto (U_{[J]},\epsilon)\colon R_{J}\times\set{\pm}\to (M_{J}\times\set{\pm})$.  There is    a natural transformation $\tau=(\tau_{J})_{J\in \ob(G)}\colon
 \Upsilon_{R}\to \Upsilon_{M}$ (of  underlying functors from $G$ to the category of sets). To see this, one notes  that for a morphism $(K,w,G)$ in $G$, one has   $(wUw^{-1})_{[K]}=w(U_{[J]})w^{-1}$ and  $\eta(w,U,J)=\eta(w,U_{[J]},J)$ for $J\in \ob(G)$ and $U\in R_{J}$.   The components $\tau_{J}$ of $\tau$ are positivity preserving morphisms of signed sets, so 
 $\tau$ defines a morphism $(G,\Upsilon_{R})\to (G,\Upsilon_{M})$.  The composite
 $(G,\Upsilon_{M})\hookrightarrow(G,\Upsilon_{R})\xrightarrow{\tau}(G,\Upsilon_{M})$ is the identity morphism since for $J\in \ob(G)$ and $U\in M_{J}$, one has $U_{[J]}=U$.
 In particular, the (underlying maps of sets) of the components of (the natural transformation) $\tau$ are surjective.  From \ref{comp2}, it follows that $\tau^{\cc}\colon \Upsilon_{R}^{\cc}\to \Upsilon_{M}^{\cc}$ is an isomorphism and hence there are morphisms
 \begin{equation}\label{taueq}
 \tau^{\rec}\colon (G,\Upsilon_{R}^{\rec})\xrightarrow{\cong} (G,\Upsilon_{M}^{\rec}), \quad \tau^{\re}\colon (G,\Upsilon_{N})\to (G,\Upsilon_{Q}).
 \end{equation}  where in writing the domain and codomain of $\tau^{\re}$, we use (c). 
 
 By \eqref{LambdaUpsM}, we  have $(G,\Lambda^{\rec})\cong (G,\Upsilon_{M}^{\rec})\cong (G,\Upsilon_{Q}^{\cc})$. Since $(G,\Lambda)$ is preprincipal,   Proposition \ref{preprinc} implies that $(G,\Upsilon_{Q})$ is real and  preprincipal, with the same atomic generators (namely, the Brink-Howlett generators) as $(G,\Lambda)$.   We show that  $(G,\Upsilon_{Q})$ is compressed.  By Proposition \ref{princpreprinc},  this holds if and only if the inversion set 
 $(\Upsilon_{Q})_{s}$ of each Brink Howlett generator $s$ of $G$ is a singleton set.  
 
 Write $s=(J,w,K)$ where $J\seq \Pi$, $w\in W$  and $K=wJ$.  There exists   $\alpha\in \Pi\sm J$ such that $w=(\nu(\alpha,J))^{-1}$, and then $\Phi_{w}=\Phi_{J\cup\set{\alpha}}^{+}\sm \Phi^{+}_{J}$. 
 Now a root $(U,\epsilon)\in (\Upsilon_{Q})(J)=Q_{J}\times\set{\pm}$ lies  in the inversion set $(\Upsilon_{Q})_{s}$   if and only if it is positive  and 
 $(\Upsilon(J,w,K))(U,\epsilon)=(wUw^{-1},-\epsilon)$.
 This holds if and only if $\epsilon =+$ and $\Phi_{U}^{+}\sm \Phi_{J}^{+}\seq \Phi_{w^{-1}}=\Phi_{J\cup\set{\alpha}}^{+}\sm \Phi^{+}_{J}$, where the last condition  implies   $U\seq W_{J\cup\set{\alpha}}$.  Since $U$ is a parabolic subgroup  which is a  corank one reflection overgroup of $W_{J}$  Theorem \ref{maxref}(a) (for instance) implies   $U=W_{J\cup\set{\alpha}}$. Hence
 $(\Upsilon_{Q})_{s}=\set{(W_{J\cup\set{\alpha}},+)}$ is a singleton set, as required, and $(G, \Upsilon_{N}) \cong (G, \Upsilon_{N}^{\cc})$ is compressed.  
 
 Using  the fact that $(G,\Upsilon_{N})$ is compressed, we obtain  from \eqref{LambdaUpsM} and \eqref{taueq} all the assertions of (d) except that $(G,\Upsilon_{N}^{\cc})\cong (G,\Upsilon_{Q})$. But since the components of $\tau$ are surjective, so are those of $\tau^{\re}$ (because the components preserve  both real and  imaginary roots).    From \eqref{comp2}, we see that $\tau^{\rec}$ induces an isomorphism $(G,\Upsilon_{N}^{\cc})\to (G,\Upsilon_{Q}^{\cc})\cong (G,\Upsilon_{Q})$.   This completes the proof of the theorem. \end{proof}
\begin{remark}  The fibers of the components of $\tau^{\re}$ are the parallelism classes of $(G,\Upsilon_{N})$.
Concretely,  for each object $J$  of $G$, the parallelism classes
of $(G,\Upsilon_{N})$ at the object $J$ of $G$ are the sets
$\mset{(U',\epsilon)\mid U'\in N_{J}, U'\seq U}$ for $\epsilon \in \set{\pm}$ and $U'\in Q_{J}$.
\end{remark} 
\begin{cor}\label{rootcor}  Any inversion set for a signed groupoid set  $(G,\Upsilon_{X})$ as in Theorem \ref{BHabrs} is a finite set.
\end{cor}
\begin{proof}
It is sufficient to prove that  the inversion set 
$\Upsilon_{X,g}=\Upsilon_{X}(K)^{+}\cap g(\Upsilon_{X}(J)^{-})$ is finite, where  $g=(K,w,J)$ is a morphism in $G$.  By Theorem \ref{BHabrs}, this inversion set consists of all roots $(U,+)$ where $U\in X_{J}$ satisfies $(U\sm W_{J})\cap T\seq N(w^{-1})$. Recall that $N(w^{-1})$ is a finite set. Since $U=\mpair{U\cap T}$ and $U\cap T=(W_{J\cap T})\cup
(U\sm W_{J})\cap T$, it follows that there are only finitely many possibilities for $U$. 
\end{proof}
 \subsection{}\label{redlemdef}  Proposition \ref{redlem} below will be used in studying realizations  (of root systems of Brink-Howlett groupoids) with particularly favorable properties.  In formulating it, we use the notation of \ref{BHred}. 

 Let $G$ be a connected Brink-Howlett groupoid of $(W,S)$. For any object $J\seq \Pi$  of $G$,    let $\Sigma_{J}$ denote the union of $J_{\inft}$ with the set of  all vertices of $\Pi\sm J$ which are joined (by an edge of the Coxeter graph of $(W,S)$) either to a vertex of $J_{\inft}$ or to infinitely many vertices of $J_{\fin}$.     Let $E_{J}$ be the subspace of $V$ which consists of all vectors  $v$ in $V$ such that  $B(v,\beta)=0$ for all $\beta\in J_{\inft}$ and  $B(v,\beta)\neq 0$ for only finitely many $\beta\in J_{\fin}$.  Thus, $\Pi\sm \Sigma_{J}=\Pi\cap E_{J}$.

\begin{prop}\label{redlem} Let  notation be as in \ref{redlemdef}.     
 \begin{num}
\item Let $J,K\in \ob(G)$. Then $J_{\inft}=K_{\inft}$,  the symmetric difference $J_{\fin}+K_{\fin}$ is finite,
 $\Sigma_{J}=\Sigma_{K}$ and $E_{J}=E_{K}$.  Then $\Sigma':=J_{\inf}$,  $\Sigma:=\Sigma_{J}$ and $E:=E_{J}$ are well-defined (independent of $J$). Further, $\Sigma'\seq \Sigma$ and $E\perp \Sigma'$.
 \item $\Phi':=E\cap \Phi$ is a standard parabolic root subsystem of $\Phi$ with canonical simple system $\Pi'=E\cap \Pi=\Pi\setminus \Sigma$. Denote the corresponding  Coxeter system as $(W',S')$. Then $W'=W_{\Phi'}=W_{\Pi'}$, $\Phi'=\Phi_{W'}$, $\Pi'=\Pi_{W'}$ and
 $S'=\chi(W')=\mset{s_{\beta}\mid\beta\in \Pi'}$.   
 \item  There is a component $G'$, with  $\ob(G')=\mset{J_{\fin}\mid J\in \ob(G)}$, of  the full  Brink-Howlett groupoid  of $(W',S')$, and a groupoid isomorphism 
 $\theta\colon G\to G'$ given by $J\mapsto J_{\fin}$ on objects  and 
 $(K,w,J)\mapsto (K_{\fin},w,J_{\fin})$ on morphisms,   
 which  induces  a bijection between the set of Brink-Howlett generators of $G$  and that of  $G'$.
 \item Let $(G,\Lambda)$  be as defined from $(W,S)$, $\Phi$ and $G$ as in \ref{standBH} and let $(G',\Lambda')$ be defined analogously from $(W',S')$, $G'$ and $\Phi'$. Consider the real compressions $(G,\Lambda^{\rec})$ and $(G',(\Lambda')^{\rec})$.   Then  there  is an equality of signed $G$-sets  $(G,\Lambda^{\rec})=(G,(\Lambda^{\prime})^{\rec}\circ\theta)$.
\end{num} \end{prop}
\begin{remark*} For  $K\in \ob(G')$,  no vertex of the  Coxeter graph of $\Pi'$ is joined  to infinitely many elements of   $K$, and all components of the Coxeter graph of $K$ are spherical.   \end{remark*}
\begin{proof} Since $G$ is connected and  generated by its 
Brink-Howlett generators, induction  reduces the proof of  the assertions of (a) involving both $J$ and $K$ to the  case in which there 
exists a Brink-Howlett generator $(K,w,J)$ of $G$, say with  
$w=\nu(\alpha,J)$. Then $w^{-1}=\nu(K,\beta)$ where $J\dotcup\set{\alpha}=K\dotcup\set{\beta}$. Denote the 
component of $J\cup\set{\alpha}$   containing $\alpha$ as $L\dotcup\set{\alpha}$ where $L\seq J$.  Then 
$L\cup\set{\alpha}$ is of finite type, so $L$  is a union of finitely 
many components of $J_{\fin}$. It follows that $\alpha\not\in \Sigma_{J}$ and  $L\seq J_{\fin}\seq \Pi\sm \Sigma_{J}$.
Hence   \begin{equation}\label{Sigmaeq}
w=w_{L\cup\set{\alpha}}w_{L}\in W_{L\cup\set{\alpha}}\seq W_{\Pi\sm \Sigma_{J}}.
\end{equation} Since   $J\sm L\perp L\cup\set{\alpha}$, we have
$K=w(J)=w(J\sm L)\dotcup w(L)=(J\sm L)\dotcup w(L)$ where 
$w(L)\seq L\cup\set{\alpha}$ is of finite type and $w(L)\perp (J\sm L)$. This implies that $K_{\inft}=J_{\inft}$.
Using $+$ for symmetric difference, we therefore have
\begin{equation*}
\begin{split}
  J_{\fin}+K_{\fin}&=(J_{\fin}+J_{\inft})+(K+K_{\inft})
  =J+K\\&=((J\dotcup\set{\alpha})+\set{\alpha})+((K\dotcup\set{\beta})+\set{\beta})=\set{\alpha}+\set{\beta}
\end{split}
\end{equation*} which is finite.  This proves the assertions of (a) involving both $J$ and $K$. The other assertions of (a) follow    directly using the definitions. 

 The first assertion of (b) is proved by noting that   
$\Phi'':=\Phi\cap (J_{\inft})^{\perp}$ is a standard parabolic root subsystem of $\Phi$ by Proposition \ref{infred}(a), and that,
since $J_{\fin}\seq \Phi''\cap \Pi$,   \[\Phi'=\mset{\beta\in \Phi''\mid B(\beta,\gamma)\neq 0 \text{ \rm for only finitely many $\gamma\in J_{\fin}$}}\] is a standard parabolic root subsystem of $\Phi''$ by Proposition \ref{infred}(b).  The rest of (b) follows using basic facts about reflection subgroups.
 
 Part (c) is from \cite{BrHo99}, but we give a proof for completeness. We claim first  that for any morphism
 $(K,w,J)$ of $G$, we have $w\in W'$. It suffices to prove the 
claim  in case $(K,w,J)$ is a Brink-Howlett generator of $G$,  say with $w=\nu(\alpha,J)$ as in the proof of (a).  In that proof, the component  $L\cup\set{\alpha}$ of  $J\cup\set{\alpha}$  is contained in 
$J_{\fin}\cup\set{\alpha}$, so 
it  is the component of $J_{\fin}\dotcup\set{\alpha}$ containing $\alpha$. Hence $\nu(\alpha, J\sm \Sigma)$ is defined in $(W',S')$ and is equal to $w_{L\cup\set{\alpha}}w_{L}=w$,
 by \eqref{Sigmaeq}. This proves the claim. Note also that $w(J_{\fin})=K_{\fin}$ since $w(J)=K$ and $w=\nu(\alpha,J)=\nu(\alpha,L)=\nu(\alpha, J_{\fin})$, so $(K_{\fin},w,J_{\fin})$ is a Brink-Howlett morphism of  the full Brink-Howlett groupoid  $H$ of $(W',S')$. 
 
In fact, for any morphism $(K,w,J)$ of $G$, we  have $w(J)=K$ and therefore $w(J_{\fin})=K_{\fin}$ and $w(J_{\inf})=K_{\inf}$ (i.e. $w(\Sigma')=\Sigma'$). It follows from the claim  that the maps
$J\mapsto J_{\fin}$ on objects and $(K,w,J)\mapsto  (K_{\fin},w,J_{\fin})$ on morphisms define a homomorphism $\theta'$ of groupoids from $G$ to $H$. It is an embedding (that is, $\theta'$ is injective both on objects and on morphisms) since for $J\in \ob(G)$, we have $J=J_{\fin}\dotcup J_{\inf}=J_{\fin}\dotcup \Sigma'$. Moreover, the proof of the claim shows that $\theta'$ sends Brink-Howlett generators of $G$ to Brink-Howlett generators of $H$. 

To complete the proof of (c), it will suffice to show that if 
$(K',w,J')$ is a Brink-Howlett generator of $H$, where 
$J'=\theta(J)=J_{\fin}$ with  $J\in \ob(G)$, then there exists 
$K$ in $\ob(G)$ such that $\theta'(K)=K'$, $(K,w,J)$ 
is a Brink-Howlett generator of  $G$ and  $\theta'(K,w,J)=(K',w,J')$.  
Write $w=\nu(\alpha,J')$ with $\alpha\in \Pi'\sm J'$ and let 
$L\dotcup\set{\alpha}$, with $L\seq J'$, denote the (spherical) component of   $J'\dotcup\set{\alpha}$ which contains $\alpha$.  Now 
$J=J'\dotcup J_{\inft}$ where $J'=J_{\fin}\perp J_{\inft}$ and $\alpha\in \Pi'=\Pi\sm \Sigma$ so $\alpha\perp J_{\inft}$ (in particular, $\alpha\not\in J_{\inft}$). Hence $J_{\inft}\perp  L\dotcup\set{\alpha}$. Thus, 
$L\dotcup\set{\alpha}$   is the component of $J\dotcup\set{\alpha}$ containing $\alpha$, and $\nu(\alpha,J)=\nu(\alpha,L)=\nu(\alpha,J')=w$ are all defined.
Let $K:=w(J)$, so $(K,w,J)$ is a Brink-Howlett generator of $G$ and $K\in \ob(G)$.
Then $\theta(K,w,J)=(\theta(K),w, J_{\fin})=(\theta(K),w, J')$ is a morphism  of $H$. Since $(K',w,J')$ is also a morphism of $H$, we get $K'=\theta(K)$, completing the proof of (c).

 The proof of (d) uses Proposition \ref{princpreprinc}(a).
 For an object $J$ of $G$, the result implies that the roots in $\Lambda^{\rec}(J)$ are the subsets of $\Phi$ of  the form  $w(\Phi_{s})$, where $(J,w,K)$ is a morphism of $G$ and $(K,s,L)$ is an atomic morphism (i.e. Brink-Howlett generator) of $(G,\Lambda)$. Such a  root is positive if $w(\Phi_{s})\seq \Phi^{+}$ and negative otherwise i.e.  if   $w(\Phi_{s})\seq \Phi^{-}$. For a morphism $(L,u,J)$ of $G$, one has 
 $(\Lambda^{\rec}(L,u,J))(w(\Phi_{s}))=(uw)(\Phi_{s})$. The description for the signed $G$-set $(\Lambda')^{\rec}\circ \theta$ is essentially the same,
 using inversion sets in the  standard parabolic subsystem  $\Phi'$ of $\Phi$ instead of in  $\Phi$. This gives the equality as desired since, for any morphism $(K,s,L)$ of $G$, atomic or not, we have $s\in W'$,  and for any $s\in W'$, we have  $\Phi'_{s}=\Phi_{s}$. 
   \end{proof}

\begin{prop}\label{rootembed} Let assumptions and notation be as in Proposition \ref{redlem}. Let $X$ denote either $R$, $M$, $N$, $P$ or $Q$ as in Theorem \ref{BHabrs}. 
For $J\seq \Pi$, let $X_{J}$ denote the set of rank one reflection overgroups of $W_{J}$ in $(W,S)$  as in \ref{redlem}, and for $J\seq \Pi'$,   let 
$X'_{J}$ denote the corresponding set of reflection overgroups of  $W_{J}'$ in $(W',S')$. Let $(G,\Upsilon_{X})$ and $(G',\Upsilon'_{X'})$  denote the signed $G$-set as in Theorem \ref{BHabrs} and the analogously defined signed $G'$-set, respectively. This gives two signed $G'$-sets
$(G',\Upsilon_{X}\circ\theta^{-1})$ and $(G',\Upsilon'_{X'})$ where $\theta\colon G\xrightarrow{\cong} G'$ is the isomorphism in Proposition \ref{redlem}. 
\begin{num}
\item For $K$ in $\ob(G')$ and $U'\in X'_{K}$,  there is a 
reflection subgroup $U=\iota_{K}(U')\in X_{K\cup \Sigma}$ such that  $\Pi_{U}=\Pi'_{U'}\dot\cup\Sigma'$. 
 One has 
$\Phi_{U}=\Phi'_{U'}\dotcup \Phi_{W_{\Sigma'}}$. 
\item There is an embedding $\rho=\rho_{X}\colon (G',\Upsilon'_{X'})
\hookrightarrow (G',\Upsilon_{X}\circ \theta^{-1})$ of signed $G'$-sets  
where for $K\in \ob(G')$ the component 
$\rho_{K}\colon \Upsilon'_{X'}(K)\to 
\Upsilon_{X}(K\cup\Sigma')$ of $\rho$  is 
defined by  $\rho_{K}(U',\epsilon)=(\iota_{K}(U'),\epsilon)$
 for all $U'\in X'_{K}$ and $\epsilon\in \set{\pm}$. 
 \item If $X=N$ or $X=Q$, then 
 $\rho_{X}\colon (G',\Upsilon^{\prime}_{X'})
\xrightarrow{\cong} (G',\Upsilon_{X}\circ\theta^{-1})$ is an isomorphism of signed $G'$-sets.
\end{num}
\end{prop}
\begin{proof} We prove (a). Let $K$ in $\ob(G')$ and 
$U'\in R'_{K}$ be any corank one 
reflection overgroup of  $W'_{K}$ in 
$W'$. Then $\Pi'_{U'}\seq \Phi_{W'}=\Phi'$. 
We have $\Pi'_{U'}\perp
\Sigma'$ since $\Phi'\seq E\perp \Sigma'$. Since $\Pi'\cup \Sigma'\seq \Phi^{+}$,  Theorem \eqref{refcan} implies that 
there is a reflection subgroup $U$ of  $W$ with $\Pi_{U}=\Pi'_{U'}\dot\cup \Sigma'$.  Since $\Pi'_{U'}\perp \Sigma'$,
both $\Pi'_{U'}$ and $\Sigma'$ are unions of components of 
$\Pi_{U}$, so $\Phi_{U}=\Phi'_{U'}\dotcup \Phi_{W_{\Sigma'}}$
 and  $U=U'W_{\Sigma'}\cong U\times W_{\Sigma'}$.
Also, $\Pi_{U}=\Pi'_{U'}\dot\cup \Sigma'\sreq K\dotcup\Sigma'$
with $\vert \Pi_{U}\sm (K\dotcup \Sigma')\vert=\vert \Pi'_{U'}\sm K\vert =1$ since $U'\in  R'_{K}$ is a corank one reflection overgroup of $W'_{K}$. Hence $U$ is a corank one reflection overgroup of $W_{K\cup \Sigma'}$. That is, $U\in R_{K\cup \Sigma'}$. To complete the proof of (a), it  remains to show that if $U'\in  X'_{K}\seq R'_{K}$
where $X$ denotes $M$, $N$, $P$ or $Q$, then  
$U\in  X_{K\cup \Sigma'}$ (recall \ref{refsubtype}).

Suppose first that $X=N$. Let $U'\in N'_{K}$. Then  $\Phi^{\prime+}_{U'}\sm \Phi^{'+}_{K}$ is finite. By Proposition \ref{finindcond},
we have $[U':W_{K}]<\infty$. Hence $[U:W_{K\cup\Sigma'}]=[U'\times W_{\Sigma'}: W_{K}\times W_{\Sigma'} ]=[U':W_{K}]$ is finite, and  so
$\Phi^{+}_{U}\sm \Phi^{+}_{K\cup \Sigma'}$ is finite by Proposition \ref{finindcond} again. Hence $U\in N_{K\cup \Sigma'}$.

Next, suppose that $X=P$. Let $U'\in P'_{K}$ Then there is $w\in W'$ and  $L\seq \Pi'$ such that  $w(L)=\Pi'_{U'}\sreq K$.
Let $J:=w^{-1}K\seq L$. Then $(K,w,J)$ is a morphism in $G'$.
Hence $(K\dot\cup \Sigma',w,J\dot \cup\Sigma')$ is a morphism in $G$. Note that $W_{L\dot\cup \Sigma'}$ is a rank one reflection overgroup of $W_{J}$ and it is a standard parabolic   subgroup of $W$.  Hence $W_{w(L\dot\cup\Sigma)}=W_{wL\dot\cup \Sigma}=W_{\Pi_{U'}\dot\cup \Sigma}=U$ is a parabolic subgroup of $W$, showing that $U\in P_{K\cup\Sigma}$.

The next case we consider is that when $X=Q$. Let $U'\in Q'_{K}$. That is, $U'\in N'_{K}$ and $U'\in P'_{K}$. By the previous two paragraphs, we have $U\in N_{K}$ and $U\in P_{K}$, so $U\in Q_{K}$ as required.

The final case in the proof of (a) is that in which $X=M$.
Let $U'\in M'_{K}$. That is, $U'$ is a maximal corank one reflection overgroup of $W'_{K}$ in $W'$. Write $\Pi'_{U'}=K\dot\cup\set{\alpha}$ where $\alpha\in \Phi_{W'}^{+}$. Then $\Pi_{U} =K\dot\cup\set{\alpha}\dot\cup \Sigma'$, so
$\alpha=r_{U,K\dotcup\Sigma'}$. Let 
$U''$ be a corank one reflection overgroup of $W_{K\cup \Sigma'}$ in $W$ such that $U''\sreq U$.    We have to show that $U''=U$. For this argument, we assume for simplicity that $\Pi$ is linearly independent (this entails no loss of generality).  Write $\Pi_{U''}=K\dot\cup\set{\beta}\dot\cup \Sigma'$
where $\beta=r_{U'', K\dot\cup \Sigma'}\in \Phi^{+}$. 
Since $\alpha\in \Pi_{U}\seq \Phi_{U''}\seq \cone(\Pi_{U''})$, we may write
\begin{equation}\label{alphasum}
\alpha=c_{\beta}\beta+\sum_{\gamma\in K\cup\Sigma'}c_{\gamma}\gamma,\qquad \text{\rm   where  $c_{\delta}\geq 0$ for all $\delta\in \Pi_{U''}$.}\end{equation} 
We must have $c_{\beta}>0$ since $s_{\alpha}\not\in W_{K\cup \Sigma'}$. Hence 
$
\beta=b_{\alpha}\alpha-\sum_{\gamma\in K\cup \Sigma'}b_{\gamma}\gamma $
where $b_{\alpha}=c_{\beta}^{-1}>0$ and 
$b_{\gamma}=c_{\beta}^{-1}c_{\gamma}\geq 0$ for all
$\gamma\in K\cup \Sigma'$. Define the finite subset $\Delta:=\mset{\gamma\in \Sigma'\mid b_{\gamma}>0}$ of $\Sigma'$. 
For each $\delta\in\Sigma'$, \eqref{refcan} implies that  
$
0\leq -B(\beta,\delta)=\sum_{\gamma\in \Delta}b_{\gamma}B(\gamma,\delta)
$ since $B(K\cup\set{\alpha},\delta)\seq B(E,\Sigma')\seq\set{0}$. Now $B(\gamma,\delta)\leq 0$ for all distinct $\gamma,\delta\in \Sigma'$. Hence if $\delta\in \Sigma'\sm \Delta$, we must have $B(\gamma,\delta)=0$ for all $\gamma\in \Delta$. This implies that $\Delta$ is a union of  components of $\Sigma'$.  Moreover, for each component $\Delta'$ of $\Delta$, we have 
$\sum_{\gamma\in \Delta'}b_{\gamma}B(\gamma,\delta)\geq 0
$ for all $\delta\in \Delta'$, where all $b_{\gamma}>0$. Hence 
the component $\Delta'$ must be of finite or affine type, by \cite[Ch 4]{Kac}.
But by construction, $\Sigma'$ has no components of finite type, 
so $\Delta'$ is of affine type,  and 
$\sum_{\gamma\in \Delta'}b_{\gamma}B(\gamma,\delta)\geq 0
$ for all $\delta\in \Delta'$ with all $b_{\gamma}>0$ implies that 
$\sum_{\gamma\in \Delta'}b_{\gamma}B(\gamma,\delta)=0
$ for all $\delta\in \Delta'$.

For each $\delta\in \Sigma'$,  the above proves that
$B(\beta,\delta)=-\sum_{\Delta'}\sum_{\gamma\in \Delta'}b_{\gamma}B(\gamma,\delta)=0$, where $\Delta'$ runs over all components of $\Delta$. Thus, $\beta\in \Phi\cap \Sigma^{\prime\perp}$, which by Proposition \ref{infred}, is a standard parabolic root subsystem of $\Phi$ with simple roots 
$\Pi\cap \Sigma^{\prime\perp}$.
We also  have $\alpha\in \Phi_{W'}=\Phi\cap E\seq \Phi\cap \Sigma^{\prime\perp}$.
Hence in \eqref{alphasum}, $\alpha-c_{\beta}\beta\in \Span(
\Pi\cap \Sigma^{\prime\perp})$. Linear independence of $\Pi$ ensures that for $\gamma\in K\cup\Sigma'$, we have $c_{\gamma}=0$  unless $\gamma\in \Pi\cap \Sigma^{\prime\perp}$.  In particular, $c_{\gamma}=0=b_{\gamma}$ for $\gamma\in \Sigma'$. 
Hence $\beta=b_{\alpha}\alpha-\sum_{\gamma\in K}b_{\gamma}\gamma$.
Since $B(\alpha,\gamma)\neq 0$ for only finitely many $\gamma\in K$, and all components of $K$ are of finite type, this
implies that    $B(\beta,\gamma)\neq 0$ for only finitely many $\gamma\in K$.  Since we have already seen $\beta\in \Sigma^{\prime\perp}$, it follows that $\beta\in \Phi^{+}\cap E=\Phi_{W'}^{+}$. Let  $U''':=W'_{K\cup\set{\beta}}$. Since $K\cup\set{\beta}\seq K\cup\set{\beta}\cup \Sigma'=\Pi_{U''}$, we have $\Pi_{U'''}=K\cup\set{\beta}$, so $U'''$ is a corank one overgroup of $W'_{K}$.

Note that since $\Pi_{U''}=\Pi_{U'''}\cup\Sigma'$ where
$\Pi_{U'''}=K\cup\set{\beta}\perp \Sigma'$, we have
$\Phi_{U''}=\Phi_{U'''}\dot\cup \Phi_{\Sigma'}$.
Similarly,  $\Pi_{U}=\Pi_{U'}\cup\Sigma'$ where $\Pi_{U'}=K\cup\set{\alpha}\perp\Sigma'$ and so 
$\Phi_{U}=\Phi_{U'}\dot\cup\Phi_{\Sigma'}$.
Since $U\seq U''$, we have $\Phi_{U}\seq\Phi_{U''}$ and hence
$\Phi_{U'}\seq \Phi_{U'''}$. This implies $U'\seq U'''$.
Since $U'''$ is a corank one overgroup of $W'_{K}$ in $W'$
and $U'$ is a maximal corank one overgroup of $W'_{K}$ in $W'$, this forces  $U'''=U'$, $\Pi_{U'''}=\Pi_{U'}$, $\Pi_{U''}=\Pi_{U}$ and hence $U''=U$, as required to complete the proof of (a). 

We now prove (b).  Let $K\in \ob(G)$. The function $\iota_{K}\colon X_{K}'\to X_{\theta(K)}= X_{K\cup \Sigma'}$ is injective, 
since if $U=\iota_{K}(U')$ with $U'\in X'_{K}$, then $\Pi_{U}=\Pi'_{U'}\dotcup \Sigma'$.  Directly from its definition, $\rho_{K}$ is positivity-preserving. It remains to show that
$\rho$ is a natural transformation (of the underlying set-valued functors).  Consider a morphism
$g=(J,w,K)$ in $G'$. Then $w'\in W'$. Let $(U',\epsilon)\in \Upsilon'_{X'}(K)$.
We are required to show \begin{equation}\label{dispeq}\rho_{J}(((\Upsilon'_{X'}(g)) (U',\epsilon))=(\Upsilon_{X}(\theta^{-1}(g)))(\rho_{K}(U,\epsilon)).\end{equation}
In the left hand side of \eqref{dispeq}, $((\Upsilon'_{X'}(g)) (U',\epsilon)=(wU'w^{-1},\eta\epsilon)$ where $\eta=-1$ if 
$\Phi^{\prime +}_{U'}\sm \Phi^{\prime+}_{W_{K'}}\seq\Phi'_{w}$
and $\eta=1$ otherwise. Hence the left hand side of \eqref{dispeq} is equal to
$(U'', \eta\epsilon)$ where $\Pi_{U''}=\Pi_{wU'w^{-1}}\dotcup \Sigma'$. 
 In the right hand side of \eqref{dispeq}, we have $\theta^{-1}(g)=(J\dot\cup \Sigma',w,K\dot\cup\Sigma')$ and $\rho_{K}(U',\epsilon)=(U,\epsilon)$ where $\Pi_{U}=\Pi_{U'}\dot\cup\Sigma'$. Hence the right hand side of \eqref{dispeq}
 is $(wUw^{-1},\eta'\epsilon)$ where 
 $\eta'=-1$ if 
$\Phi^{+}_{U}\sm \Phi^{+}_{W_{K\dotcup \Sigma'}}\seq\Phi_{w}$
and $\eta'=1$ otherwise. 

To complete the proof of (b), we therefore need only show that $wUw^{-1}=U''$ and $\eta=\eta'$.
Since $w\in W'$ where $\Pi_{W'}\perp \Sigma'$, we have \[wUw^{-1}=w(U'W_{\Sigma'})w^{-1}=(wU'w^{-1})W_{\Sigma'}=U''.\]  
Further,  $\Phi'_{w}:=\Phi^{\prime +}\cap w\Phi^{\prime +}=\Phi_{w}$ since $W'$ is a standard parabolic subgroup of $W$ and $w\in W'$. 
Now $\eta'=\eta$ follows because \[\Phi^{+}_{U}\sm \Phi^{+}_{W_{K\dotcup \Sigma'}}=(\Phi^{+}_{U'}\dotcup \Phi^{+}_{W_{\Sigma'}})
\sm (\Phi^{+}_{W_{K}}\dotcup \Phi^{+}_{W_{\Sigma'}})
= \Phi^{+}_{U'}
\sm \Phi^{+}_{W_{K}}= \Phi^{\prime+}_{U'}
\sm \Phi^{\prime+}_{W_{K}}.
\]

We now turn to the proof of (c). First, let $X$ denote either $N$ or $Q$, and write
$\rho:=\rho_{X}$. We have to show that for $K\in \ob(G')$,
the component $\rho_{K}$ of $\rho$  is bijective as a function.
By (b),  so we need only show this component is surjective.
By definition of $\rho$, it suffices to show  that if  $K\in \ob(G')$ and $U\in X_{K\cup \Sigma'}$, there exists $U'\in X'_{K'}$ such that  $U=\iota_{K}(U')$. In either case $X=N$ or $X=Q$, we have $U\in N_{K\cup \Sigma'}$. That is,   $U$ is a rank one, finite index reflection overgroup  of $W_{K\cup\Sigma'}$ in $W$.
Hence $\Pi_{U}=K\dot\cup\set{\beta}\dot\cup \Sigma'$ for some $\beta\in \Phi^{+}$. Since $W_{K\cup \Sigma'}$ has finite index in $U$, the component of $\Pi_{U}$ containing $\beta$ is of  finite type, by Proposition \ref{finindcond}. Hence $\beta\perp \Sigma'$ (since each component of $\Sigma'$ is of infinite type)
and $B(\beta,\alpha)$ is non-zero for only finitely many elements of $K$. By definition, we have  $\beta\in \Phi^{+}\cap E=\Phi^{\prime+}$. Hence $U':=W'_{K\cup\set{\beta}}$ is a corank one reflection overgroup of $W'_{K}$ in $W'$, with $\Pi'_{U'}=\Pi_{U'}=K\cup\set{\beta}\seq \Pi_{U}$.  In particular, $\beta\in \Phi_{W'}$. Since $K\cup\set{\beta}\perp \Sigma'$, we have $\Phi_{U}=\Phi'_{U'}\dotcup  \Phi_{\Sigma'}$. Hence \[\Phi'_{U'}\sm \Phi'_{K}=
(\Phi'_{U'}\dotcup\Phi_{\Sigma'})\sm (\Phi'_{K}\dotcup \Phi_{\Sigma'})=\Phi_{U}\sm \Phi_{K\cup \Sigma'}\] is finite, and so $U'$ is a finite index, corank one reflection overgroup of $W'_{K}$ in $W'$. That is, $U'\in N'_{K'}$.  By the definitions,  $\iota_{K}(U')=U$, which completes the proof of (c) in case $X=N$.

If, instead, $X=Q$, then, additionally, $U$ is a parabolic subgroup of $W$, and we have to show that $U'$ is a parabolic subgroup of $W'$. Since $U$ is parabolic in $W$, there exists $w\in W$ such that $w\Pi_{U}\seq \Pi$. That is,
$w(K\dot\cup\set{\beta}\dot\cup \Sigma') \seq \Pi$.
Then $L:=w(K\dot\cup \Sigma')\seq \Pi$,  and 
 $g:=(L,w,K\cup\Sigma')$ is a morphism  in $G$ since $K\cup\Sigma'\in \ob(G)$. 
 Then  $g':=\theta(g)=(J,w,K)\in \mor(G')$ where $J=L\sm \Sigma'$. We have  $w\in W'$,
 $w\Sigma'=\Sigma'$,  $J=wK\seq \Pi'$ and  
 $L=J\dotcup \Sigma'$. We have 
 \[w\Pi_{U'}=w(K\dot\cup\set{\beta})\seq 
 w\Pi_{U}\cap \Phi'\seq \Pi\cap \Phi'=\Pi'=\Pi_{W'}\] and 
  $w\in W'$. Hence $U'$ is a parabolic subgroup of $W'$, as 
  required.
 \end{proof}
 \begin{prop}\label{realrootred}  Let notation and assumptions be as in Proposition \ref{rootembed}. Then $\rho_{X}$ induces an isomorphism
$(G',(\Upsilon'_{X'})^{\re})
\hookrightarrow (G',\Upsilon_{X}^{\re}\circ \theta^{-1})$ of signed $G'$-sets 
 \end{prop}
 \begin{proof}
 This follows using Theorem \ref{BHabrs}(c)  and
 Proposition \ref{rootembed}(c). Details are omitted. \end{proof}

\section{Realized root systems of Brink-Howlett groupoids}
\label{sec:5}
\begin{proof}[Proof of Theorem \ref{weakreal}(a)] 
We show that $ \iota \colon \Upsilon_{X} \rightarrow \mcv$ is  a natural transformation of set-valued functors on $G$, by checking  that for any morphism $ (K,w,J) $ of $G$, \[\mcv(K,w,J)  \circ \iota_{J}  =  \iota_{K} \circ  \Upsilon_{X}(K,w,J).\]
If $ (U, \epsilon) \in \lrsub{J}{\Upsilon}{X}$, then 
\[ (\mcv(K,w,J)  \circ \iota_{J})(U , \epsilon) = \mcv(K,w,J)(\epsilon \pi_{J}(r_{U,J}))= \epsilon \pi_{K} (w r_{U,J})\]
and, since $w(J)=K$ implies that $w\pi_{J}(v)=\pi_{K}(wv)$,
 \[(\iota_{K} \circ  \Upsilon_{X}(K,w,J))(U , \epsilon) = \iota_{K}(wUw^{-1} , \epsilon \nu)
= (-1)^{\eta(w,U,J)} \epsilon  \pi_{K}(r_{wUw^{-1},K})\]
where $\eta(w,U,J)$ is as defined in Theorem \ref{BHabrs}.
Thus, we need to show that \[ \pi_{K} (w r_{U,J}) = (-1)^{\eta(w,U,J)} \pi_{K}(r_{wUw^{-1},K}).\]

If $\eta(w,U,J)=0$,  then $(U\sm W_{J})\cap N(w^{-1})=\eset$
and Proposition \ref{naivers} gives $wr_{U,J}=r_{wUw^{-1},K}$ as required. Otherwise,  we have $\eta(w,U,J)=1$ and $(U\sm W_{J})\cap T\seq  N(w^{-1})$.  Then
Proposition \ref{naivers} gives $ww_{L}r_{U,J}=-r_{wUw^{-1},K}$ where $L\dotcup\set{r_{U,J}}$ is the (spherical) component of $\Pi_{U}$ containing $r_{U,J}$, with $L\seq J$. So to complete the proof of the above displayed equation, we have to show that
$\pi_{K}(wr_{U,J})=\pi_{K}(ww_{L}r_{U,J})$.
This holds since $w_{L}r_{U,J}\in r_{U,J}+\Span(L)$  where 
$\pi_{K}(w\Span(L))\seq \pi_{K}(w\Span(J))=\pi_{K}(\Span(K))=\set{0}$. 

This competes the proof that $\iota$ is a natural transformation.  
It is obvious from the definitions that the components of $\iota$ are $\{ \pm \}$-equivariant maps. This proves that $F$ is a weakly realized signed groupoid-set, as required.
\end{proof}
\begin{proof}[Proof of Theorem \ref{weakreal}(b)] 
Suppose now that $ \Pi$ is linearly independent and that $ X = P$. It suffices to show that each $ \iota_{J}: \lrsub{J}{\Upsilon}{P} \rightarrow \mcv(J)$ is an injection for each object $ J$ of $G$. Let us suppose that $ (U_{i} , \epsilon_{i}) \in \lrsub{J}{\Upsilon}{P}$ for $i=1,2$ are  such that 
$\iota_{J}(U_{1} , \epsilon_{1}) = \iota_{J}(U_{2} , \epsilon_{2})$.
This means that 
$ \epsilon_{1} \pi_{J}(r_{U_{1},J}) = \epsilon_{2} \pi_{J}(r_{ U_{2},J})$.
Hence  $ r_{U_{1},J} \in \Span(\{ r_{U_{2},J} \} \cup J_{\fin})$. Thus, we get that $ \Span(\Pi_{U_{1}} ) = \Span(\Pi_{U_{2}} )$ where $ \Pi_{U_{i}} = \{ r_{U_{i},J} \} \cup J$  for $i=1,2$. Since we are assuming that $ U_{i}$ is parabolic and that $\Pi$ is linearly independent, we have that $ \Phi_{U_{i}} = \Span(\Pi_{U_{i}}) \cap \Phi$ for $i = 1,2$. Thus,
\[ \Phi_{U_{1}} = \Span(\Pi_{U_{1}}) \cap \Phi = \Span(\Pi_{U_{2}}) \cap \Phi = \Phi_{U_{2}}\]
and thus $ U_{1} =U_{2}$. Note that if $\epsilon_{1}$ and $ \epsilon_{2}$ had opposite signs, then one would get $ 2\pi_{J}(r_{U_{1},J}) = 0$ and thus $ \pi_{J}(r_{U_{1},J}) = 0$. But since $ W_{J} \sneq U_{1}$, and because $ \Phi_{J} = \Span(J) \cap \Phi$, we must have that $ \pi_{J}(r_{U_{1},J}) \neq 0$, a contradiction. Therefore, $ \epsilon_{1} = \epsilon_{2}$, and thus $\iota_{J}$ is injective for any object $J$ of $G$. This proves part (b).
\end{proof}
\begin{proof}[Proof of Theorem \ref{weakreal}(c)] 
Suppose now that $\vert J\vert = 1$. Let $ J = \{ \alpha \}$. Let $ (U_{i}, \epsilon_{i}) \in \lrsub{J}{\Upsilon}{X}$ for $i=1,2$. Suppose that 
$ \iota_{J}(U_{1} , \epsilon_{1})  = \iota_{J}(U_{2}, \epsilon_{2})$.
That is, 
$ \epsilon_{1} \pi_{J}(r_{U_{1},J}) = \epsilon_{2} \pi_{J}(r_{U_{2},J})$.
If $ \epsilon_{1} \neq \epsilon_{2}$, then $ \pi_{J}(r_{U_{1},J}) = -\pi_{J}(r_{U_{2},J})$, so
$\pi_{J}(r_{U_{1},J}+r_{U_{2},J}) = 0$,
meaning that $ r_{U_{1},J} + r_{U_{2},J} = c \alpha$
for some $ c \in \mathbb{R}$. Note that if $c \leq 0$, then 
\[ r_{U_{1},J} + r_{U_{2},J}+ (-c)\alpha = 0\]
where the left hand side is a non-negative linear combination of positive roots, with at least one coefficient strictly positive. This contradicts the positive independence of $\Pi$. If $ c> 0$, then
\[c = cB(\alpha, \alpha) = B(\alpha, c\alpha)= B(\alpha, r_{U_{1},J} + r_{U_{2},J}) = B(\alpha , r_{U_{1},J}) + B(\alpha, r_{U_{2},J}) \]
where the two terms on the right hand side above are non-positive since $ \{ \alpha, r_{U_{1},J} \} $ and $ \{ \alpha, r_{U_{2},J} \} $ are root bases. This contradicts $c>0$. 
If $c=0$, then $r_{U_{1},J}=-r_{U_{2},J}$ contrary to the fact $r_{U_{i},J}$ is a  positive root for $i=1,2$.

Thus, we must conclude that $ \epsilon_{1} = \epsilon_{2}$ and therefore $\pi_{J}(r_{U_{1},J}) = \pi_{J}(r_{U_{2},J})$.
That is, $ r_{U_{1},J} = r_{U_{2},J} + c\alpha$
for some $ c \in \mathbb{R}$.  By Proposition \ref{threeroots}(a), 
either $ r_{U_{1},J} = r_{U_{2},J}$ or $r_{U_{2},J}= s_{\alpha}(r_{U_{1},J})$.  We show that $ r_{U_{1},J} = r_{U_{2},J}$  in either case.  To see this, suppose that  
 $ s_{\alpha}(r_{U_{1},J}) = r_{U_{2},J}$. Because $ \{ \alpha, r_{U_{i},J} \}=\Pi_{U_{i}}$, we have  $ B(\alpha, r_{U_{i},J}) \leq 0$ for $i=1,2$. Hence  
\[ 0\geq B(\alpha, r_{U_{2},J} ) = B(\alpha , s_{\alpha}r_{U_{1},J}) = B(s_{\alpha}(\alpha), r_{U_{1},J})
  = -B(\alpha , r_{U_{1},J}) \geq 0.\]
Thus,  $ B(\alpha, r_{U_{2},J}) = 0$, and so $ r_{U_{1},J} = s_{\alpha}(r_{U_{2},J}) = r_{U_{2},J}$.  Hence  $ r_{U_{1},J} = r_{U_{2},J}$, which implies that $(U_{1},\epsilon) = (U_{2},\epsilon)$. Thus, we have that $\iota_{J}$ is injective for each object $J$ with $ |J|=1$. This proves  
Theorem \ref{weakreal}(c).
\end{proof}
\begin{proof}[Proof of Theorem \ref{weakreal}(d)]
 Assume $X=N$ or $X=Q$.
We  need to establish that $\iota \colon \lrsub{J}{\Upsilon}{X} \rightarrow \mcv (J) $ is injective for each object $ J$ of $G$. Suppose that $ (U_{i} , \epsilon_{i}) \in \lrsub{J}{\Upsilon}{X}$ for $i=1,2$ with  $ \iota_{J}(U_{1}, \epsilon_{1}) = \iota_{J}(U_{2} , \epsilon_{2}) $. Then  $ \epsilon_{1} \pi_{J}(r_{U_{1},J}) = \epsilon_{2} \pi_{J}(r_{U_{2},J})$.

Consider first the case in which  $ \epsilon_{1}=- \epsilon_{2}$. Then 
$\pi_{J}(r_{U_{1},J}+r_{U_{2},J}) = 0$.
That is, $r_{U_{1},J} + r_{U_{2},J} \in \Span(J_{\fin})$.
Note that $ B(r_{U_{i},J} , \alpha) \leq 0$ for all $ \alpha \in J_{\fin}$ and for $i=1,2$. Hence $ B(r_{U_{1},J} + r_{U_{2},J} , \alpha) \leq 0$ for all  $\alpha \in J_{\fin}$.
Since $ r_{U_{1},J} + r_{U_{2},J} \in \Span(J_{\fin})$, this gives 
$ B(r_{U_{1},J} + r_{U_{2},J}, r_{U_{1},J} + r_{U_{2},J}) \leq 0$.
But  the bilinear form $B(-,-)$ is positive definite on $ \Span(J_{\fin})$, so this implies that $ r_{U_{1},J} + r_{U_{2},J} = 0$, meaning that  $ r_{U_{1},J} =- r_{U_{2},J}$. This gives a contradiction since $ r_{U_{i},J} \in \Phi^{+}$ for $i=1,2$, so  this case cannot occur.

We must  therefore have $ \epsilon_{1} = \epsilon_{2}$, and so $ \pi_{J}(r_{U_{1},J}) =  \pi_{J}(r_{U_{2},J})$.
Thus, $ r_{U_{2},J} \in  r_{U_{1},J} + \Span(J_{\fin})$. 
Since $U_{1}$ is a finite index overgroup of $W_{J}$,
the component $ L_{1}\dotcup\set{r_{U_{1},J}}$  of $ \{ r_{U_{1},J} \} \cup J$ containing $r_{U_{1},J}$ is of finite type, and $L_{1}\seq J_{\fin}$.  Hence $L_{1}\dotcup\set{r_{U_{1},J}}$ is also the component of $J_{\fin}\dotcup\set{r_{U_{1},J}}$
containing $r_{U_{1},J}$. This shows that   $W_{J_{\fin}\dotcup\set{r_{U_{1},J}}}$ is a finite index, corank one reflection overgroup of $W_{J_{\fin}}$. 
 By Theorem \ref{infTits}, there is a corank one, finite index overgroup $U_{0}$ of  $W_{J_{\fin}}$ which is parabolic in $W$  and contains $W_{J_{\fin}\dotcup\set{r_{U_{1},J}}}$.  By Proposition \ref{posdef},  $\Phi_{J_{\fin}}=\Span(J_{\fin})\cap \Phi$ and $J_{\fin}$ is linearly independent. 
 It follows  that $r_{U_{0},J_{\fin}}\not\in \Span(J_{\fin})$ and $\Pi_{U_{0}}=J_{\fin}\dotcup\set{r_{U_{0},J_{\fin}}}$ is linearly independent. Since $ W_{J_{\fin}\dotcup\set{r_{U_{1},J}}}\seq U_{0}$ with $r_{U_{1},J}\not\in \Phi_{J_{\fin}}$, we see that
 $r_{U_{1},J}\in \Span(\Pi_{U_{0}})\sm \Span(J_{\fin})$.
 Since $r_{U_{2},J}\in r_{U_{1},J}+\Span(J_{\fin})$, it follows that
 $r_{U_{2},J}\in \Span(\Pi_{U_{0}})\sm\Span(J_{\fin})$, too. 
 
Let $\alpha:=r_{U_{0},J_{\fin}}$ and  $M:=\set{\alpha}\dotcup L_{0}$, where $L_{0}\seq J_{\fin}$, denote the (spherical) component of $\Pi_{U_{0}}$ containing $\alpha$. Write $U:=W_{M}$, so $\Phi_{U}$ is a component of $\Phi_{U_{0}}$  with canonical simple system $\Pi_{U}=M$, and all other components of $\Phi_{U_{0}}$ are contained in $\Phi_{J_{\fin}}$. We therefore  have $r_{U_{i},J}\in \Phi_{U}\sm \Phi_{L_{0}}$ for $i=1,2$. Hence $r_{U_{2},J}-r_{U_{1},J}\in \Span(M)\cap \Span J_{\fin}=\Span(L_{0})$, since $M\cup J_{\fin}=\Pi_{U_{0}}$, which is linearly independent, and
 $M\cap J_{\fin}=L_{0}$. 
 
 To summarize, we have a finite irreducible based root system $\Phi_{U}$ with simple roots $\Pi_{U}=M$, a simple  root  $\alpha\in M$ and roots $r_{U_{i},J}\in \Phi_{U}\sm \Phi_{W_{L_{0}}}$, where $L_{0}:=M\sm\set{\alpha}$,  such that  $r_{U_{2},J}\in r_{U_{1},J}+\Span(L_{0})$. By Proposition \ref{threeroots}(c), it follows that 
 $r_{U_{2},J}\in W_{L_{0}}r_{U_{1},J}$. But for $i=1,2$,  we have  $r_{U_{i},J}\dotcup L_{0}\seq \Pi_{U_{i}}$ and hence, by Theorem \ref{refcan}, $r_{U_{i},J}\in -\mathcal{C}_{W_{L_{0}}}$. 
By Proposition \ref{CoxfundTits}, it follows that
$r_{U_{1},J}=r_{U_{2},J}$. This implies that $U_{1}=U_{2}$ and hence $(U_{1},\epsilon_{1})=(U_{2},\epsilon_{2})$ as required.
\end{proof}

\begin{proof}[Proof of Proposition \ref{realizedroot}] 
We prove (a). Let $ J$ be an object of $G$. By Proposition \ref{posdef},    we have   $\Span(J_{\fin})\cap \Phi=\Phi_{J_{\fin}}$. If $(U,\epsilon)\in \Upsilon_{X}(J)$, then $r_{U,J}\in \Phi_{U}\sm \Phi_{J}\seq \Phi\sm\Phi_{J_{\fin}}$ and therefore  $r_{U,J}\not \in \Span(J_{\fin})=\ker(\pi_{J})$.   Hence $\iota_{J}(U,\epsilon)=\pi_{J}(\epsilon r_{U,J})\neq 0$. This proves (a).

We prove  (b). We have  $ \Pi = J_{\fin} \cup (\Pi \sm J_{\fin})$ where $ J_{\fin} \seq \ker (\pi_{J})$. Since $\pi_{J}$ is linear,
\begin{equation*}
\begin{split}
 \pi_{J}( \cone (\Pi)) &= \cone (\pi_{J}(\Pi)) = \cone(\pi_{J}(J_{\fin} \cup (\Pi \sm J_{\fin})  ))\\&= \cone(\pi_{J}(\Pi \sm J_{\fin})) = \cone (\lsub{J}{\Delta}),
\end{split}
\end{equation*}
proving (b).

We prove (c).  Let $ (U, +) \in {\Upsilon}_{X}^{+}(J)$.
Then $ \iota_{J}(U ,+) = \pi_{J}(r_{U,J})\seq \pi_{J}(\cone \Pi)=\cone(\lsub{J}{\Delta}) $ by (b).   We conclude that $ \iota_{J}({\Upsilon}_{X}^{+}(J)) \seq \cone(\lsub{J}{\Delta})$, proving (c).

We prove (d). 
 Let $ x \in \cone(\lsub{J}{\Delta}) \cap -\cone(\lsub{J}{\Delta})$. Note that,  by (b), we have $ \cone(\lsub{J}{\Delta}) = \pi_{J}(\cone(\Pi\sm J_{\fin}))$.  Hence $x=\pi_{J}(y)=-\pi_{J}(z)$ for some elements
$y,z\in \cone(\Pi\sm J_{\fin})$.  Hence $y+z\in  \cone(\Pi\sm J_{\fin})$. We also have $y+z\in \ker(\pi_{J})= \Span(J_{\fin})$ since $\pi_{J}(y+z)=x-x=0$. Hence
$y+z\in \cone(\Pi\sm J_{\fin})\cap \Span(J_{\fin})=\set{0}$, by 
Proposition \ref{posdef}(e). Since $y,z\in \cone(\Pi)$, this implies that $y=z=0$, by positive independence of $\Pi$. Hence $x=0$, completing the proof of (d).  

Parts (e)--(f) follow from the definitions by elementary linear algebra.
\end{proof}
\begin{proof}[Proof of Corollary \ref{realcor}] By Theorem \ref{BHabrs},  $(G,\Upsilon_{X}^{\re})$
is equal to either  $(G,\Upsilon_{N})$ or $(G,\Upsilon_{Q})$, and
$(G,\Upsilon_{\Lambda}^{\rec})$  is isomorphic to $(G,\Upsilon_{Q})$. Hence parts (a)--(b) follow from Theorem \ref{weakreal}.  The functor $\Upsilon_{\Lambda}^{\rec}$ is faithful (since it determines the weak orders of $G$), and it is a subfunctor of $\mcv$ (regarding both as set valued  functors), so (c) follows also. Part (d) follows directly from Proposition \ref{realizedroot}(a),(c),(d).
\end{proof}

\section{Bilinear forms of realized  Brink-Howlett groupoids}
\label{s:6}

\subsection{}\label{BHred1} We say that the  form $B$ is positive definite on a subspace $U$ of $V$ if the restriction of $B$ to $U$ is positive definite i.e. if  $B(v,v)> 0$ for all non-zero $v\in U$.

For any subset $L$ of $\Pi$, define the 
subspace $V_{L}$ of $V$ by 
\[V_{L}:=\mset{v\in V\mid B(v,\alpha)\neq \set{0}\text{ \rm for only finitely many $\alpha\in L$}}. \]

\begin{prop} \label{BHred2} Let $L \seq \Pi$. \begin{num}
\item If $L $ is finite, then $V_{L }=V$.
\item If $\alpha\in \Phi\cap V_{L }$, then $B(\alpha,\beta)\neq 0$ for only finitely many $\beta\in L $.  
\item If every component of  $L $ is of finite type, then
$V_{L }=\Span(L )\oplus L ^{\perp}$ where $\Span(L )$ and $L ^{\perp}$ are $B$-orthogonal.
\item If  any one of the conditions \ref{bf}(i)--(iii) holds for some object $J$ of a connected Brink-Howlett groupoid $G$, then it holds for every object of $G$.
\end{num}
\end{prop}
\begin{proof}
Parts (a)--(b) are trivial. 

We prove (c).  Clearly, $\Span(L)$ and $L^{\perp}$ are $B$-orthogonal for any subset $L$ of $\Pi$. Assume every component of $L$ is of finite type.  Then in particular, $B$ is positive definite on $\Span(L)$, which implies that
the sum $\Span(L )+ L ^{\perp}$ is direct, since for $v\in \Span(L)\cap L^{\perp}$, we have \[0\leq B(v,v)\in  B(\Span(L),L^{\perp})\seq\set{0}\] and hence $v=0$.   Clearly, $L^{\perp}\seq V_{L}$, and we have $\Span(L)\seq V_{L}$ since every component of $L$ is finite. Hence $\Span(L )+ L ^{\perp}\seq V_{L}$. 

To complete the proof of (c), it remains to show that if $v\in V_{L}$, then $v\in  
\Span(L )+ L ^{\perp}$. Let $J$ be the union of all components of $L$ which contain a root which is not orthogonal to $v$.
By the assumptions, $J$ is a finite union of components of $L$, so $\Span(J)\perp\Span(L\sm J)$.  We also have $v\perp \Span (L\sm J)$ by the definition of $J$.
Since all components of $L$ are of finite type finite, $J$ is a  spherical subset of $\Pi$. Since $B$ is positive definite on $\Span(J)$,  we may choose a  basis
$v_{1},\ldots, v_{n}$ of $\Span(J)$ such that $B(v_{i},v_{j})=\delta_{i,j}$ (the Kronecker delta). Let $v'=\sum_{i=1}^{n}B(v,v_{i})v_{i}\in \Span(J)\seq \Span(L)$ and $v'':=v-v'$. We have $B(v,v_{i})=B(v',v_{i})$ for $i=1,\ldots, n$, so $v''=v-v'\in J^{\perp}$. Also, $v,v'\in (L\sm J)^{\perp}$ so $v''\in (L\sm J)^{\perp}$.  Hence $v''\in L^{\perp}$, $v'\in \Span(L)$ and $v=v'+v''$, as required.

To prove (d), let $(K,w,J)$ be a morphism of a Brink-Howlett groupoid $G$. We are to show that if $J$ satisfies any one of the conditions \ref{bf}(i)--(iii),  then that same condition holds with $J$ replaced by $K$. For (i) and (ii), this follows since, by Proposition \ref{redlem}(a),  the symmetric difference  $J_{\fin}+K_{\fin}$ is finite, and $J_{\inf}=K_{\inf}$.  For (iii), it suffices  by (c)  to show that $V_{J_{\fin}}=V_{K_{\fin}}$; this holds since the definitions imply that $V_{L}=V_{L'}$ if $L,L'\seq \Pi$ with $L+L'$ finite.
\end{proof}

\subsection{}  As shown in \cite{BrHo99},  the study of  connected Brink-Howlett 
 groupoids $G$ in general can be reduced  to the case when  each object $J$ of $G$
satisfies the conditions \ref{bf}(i)--(ii).  Propositions \ref{redlem}--\ref{realrootred}  show that the study of such a groupoid $G$ and its (abstract) real roots can be reduced to the situation when  each object $J$ of $G$
satisfies \ref{bf}(i)--(iii); Theorem \ref{weakreal}(d) then implies that the corresponding systems of real roots can be faithfully realized in real vector spaces.   These conditions lead to useful extra  structure, such as an analogue of the $W$-invariant  bilinear form $B$ on $V$, and some gains in simplicity of statements, with little loss for many purposes, and  we impose them henceforward.

\begin{assumption}\label{genass} For the remainder of this paper, we assume unless otherwise stated that  each object $J$ of $G$ satisfies conditions \ref{bf}(i)--(iii). We consider the weakly realized
signed groupoid set $(G,\Upsilon_{X},\mcv,\iota)$ where $X$ denotes either $P$, $M$ or $R$, and adopt the associated notation and identifications described in \ref{assume}--\ref{form}. In particular, for an object $J$ of $J$, $\pi_{J}$ denotes the orthogonal projection $V\to \mathcal{V}(J)=J^{\perp}$.
 \end{assumption}

   \subsection{} \label{asscons} 
 Let  $F $ be a  subset of $\Pi$ such that $F =F _{\fin}$ and  $V=\Span(F )+F ^{\perp}$. 
 Let $\pi_{F }\colon V\to F ^{\perp}$ denote the orthogonal projection determined by this orthogonal direct sum. This extends the notation already introduced in case $F \in \ob(G)$.
 
  The projection $\pi_{F  }$ may be alternatively described as follows. Given $\beta\in V$, choose a finite subset $M$ of $F $ 
 such that $M$ is  a union of components of $F  $ and   $\mset{\alpha\in F  \mid B(\alpha,\beta)\neq 0}\seq M$.  Then $M$ is a spherical subset of $\Pi$, and 
 \begin{equation}
 \pi_{F  }(\beta)=\frac{1}{\vert W_{M}\vert}\sum_{w\in W_{M}}w\beta.
 \end{equation}
 We leave the verification of this formula to the reader.
 Note that it implies $\pi_{F}(\beta)=\pi_{M}(\beta)$, as can easily be checked by other means (e.g. using Proposition \ref{finfacproj}(a) below).
 
 We may  sometimes write $F ^{\perp}$ as $\lsub{F }{V}$ or,
 if $F \in \ob(G)$,  as $\lsub{F }{\mathcal{V}}$.
 For notational compactness, we shall  often write $v_{F }:=\pi_{F }(v)$  for $v\in V$. 
 
 Subsets $F \seq \Pi$ such that $F =F _{\fin}$ and $V=\Span(F )\oplus F ^{\perp}$  include any $F\seq \Pi$ such that $W_{F}$ is a finite index overgroup of $W_{J}$ where $J$ is an object of $G$, and any spherical subset $F$ of $\Pi$.
% , and any finite intersection of such sets $K$ arising for $J$ ranging over  a finite set of objects of a fixed component of $G$.
 In particular, they include the sets $K$ and  $L=J\cup K$ where $J\in \ob(G)$ and $K\seq \Pi$ is such that $\nu(K,J)$ is defined. 
 
We shall require below some  properties of  the projection $\pi_{L}$, for suitable $L\seq \Pi$, which we prove by checking the existence of a suitable analogue of ``fundamental weights'' for $L$ in $\Span(L)$. The proof reduces to the case where $L$ is finite and irreducible, in which case the relevant properties are known from Vinberg's lemma (\cite[Ch 4]{Kac}), Perron-Frobenius theory
(see for example \cite{Sen}) or \cite[Ch 5, \S3, no. 5]{Bou}.
If $\beta\in L$,  we write $L_{\beta}$ for the component of $L$ containing $\beta$.

\begin{prop}\label{fincomp} Suppose that $L \seq \Pi  $ has only finite-type components. \begin{num}\item There is a unique basis $(\omega^{L }_{\alpha})_{\alpha\in L }$ of $\Span(L )$ such that for all $\alpha,\beta\in L $, one has $B(\omega^{L }_{\alpha},\beta)=\delta_{\alpha,\beta}$ (where $\delta_{\alpha,\beta}\in \mathbb{R}$ is the Kronecker delta).
\item Let $\beta\in L $. Write $\omega_{\beta}^{L}=\sum_{\alpha\in L }a^{L }_{\alpha,\beta}\alpha$ with each $a_{\alpha,\beta}^{L }\in \mathbb{R}$. Then $a^{L }_{\alpha,\beta}\neq 0\iff \alpha\in L _{\beta}\iff  a^{L }_{\alpha,\beta}> 0$.  
\item If $\alpha,\beta\in L $, then $B(\omega_{\alpha}^{L },\omega_{\beta}^L )=a^{L }_{\alpha,\beta}$. Hence $B(\omega_{\alpha}^{L },\omega_{\beta}^L )>0$  if $L _{\alpha}=L _{\beta}$ and 
$B(\omega_{\alpha}^{L },\omega_{\beta}^L )=0$ otherwise.
\item For any $\beta\in L$, $\omega_{\beta}^{L}$ is an element of   $\cone(L_{\beta})$.  
\end{num}\end{prop}
\begin{remark*} (1) We often abbreviate $\omega^{L }_{\alpha}=\omega_{\alpha}$ and $a^{L }_{\alpha,\beta}=a_{\alpha,\beta}$  if $L $ is understood.

(2) In (d),  $\omega_{\beta}^{L}$ is, more precisely,  in the  relative interior of $\cone(L_{\beta})$. 
\end{remark*}
\begin{proof}  We first prove  (a)--(b) in the special  case in which $L $ is irreducible and hence finite. By \cite[Ch 5, \S3, no. 5, Lemme 3]{Bou}, $L$ is linearly independent. Part (a) is then clear since   the restriction of $B$ to $\Span(L )$ is positive definite and hence non-singular.

For  $\beta\in L $, write $\omega_{\beta}=\sum_{\alpha\in L }a_{\alpha,\beta}^{L }\alpha$  where each $a_{\alpha,\beta}^{L }\in \mathbb{R}$. We have to show  that  $ a^{L }_{\alpha,\beta}>0$ for all $\alpha,\beta\in L$. Fix $\beta\in L $. Since $B(\omega^{L }_{\beta},\alpha)\geq 0$ for all $\alpha \in L $, \cite[Ch 5, \S3, no. 5, Lemme 6(ii)]{Bou} shows that $ a^{L }_{\alpha,\beta}\geq 0$ for all $\alpha\in L $.  Suppose 
that $a^{L }_{\alpha,\beta}=0$ for some $\alpha\in L $. Since $L \neq \eset$, one has $\omega^{L }_{\beta}\neq 0$ and so $a_{\alpha',\beta}\neq 0$ for some $\alpha'\in L $. Since $L $ is connected, there must  exist $\alpha,\alpha'\in L $ such that $\alpha$ is joined (by an edge) to $\alpha'$  in the Coxeter graph of $L $, 
$a^{L }_{\alpha,\beta}=0$ and $a^{L }_{\alpha',\beta}>0$.
Then  \[B(\omega^{L }_{\beta},\alpha)=\sum_{\gamma\in L }a^{L }_{\gamma,\beta}B(\gamma,\alpha)<0\] since $a^{L }_{\gamma,\beta}B(\gamma,\alpha)$ is zero for $\gamma=\alpha$, negative for $\gamma=\alpha'$ and non-positive for all $\gamma\in L \sm \set{\alpha,\alpha'}$.   This contradiction to $B(\omega^{L }_{\beta},\alpha)=\delta_{\beta,\alpha}$ completes the proof of (b) if $L $ is irreducible. 

Now we consider the general situation in which all   components of $L $ are of finite type. Then $\omega^{F }_{\alpha}$ is defined for each component $F $ of $L $ and each $\alpha\in F $. For $\alpha\in L $, define $\omega_{\alpha}^{L }:=\omega_{\alpha}^{F }$ where $F :=L _{\alpha}$. Since the components of $L $ are pairwise orthogonal, one has $B(\omega_{\alpha}^{L },\beta)=\delta_{\alpha,\beta}$ for all $\beta\in L $.
 For fixed $\alpha$, these equations uniquely determine $\alpha^{L }$, since the restriction of $B$ to $\Span(L )$ is non-singular. This proves (a) in general.
 
 From the proof of (a), it follows  that in (b), we have $a^{L }_{\alpha,\beta}=0$  if $L _{\alpha}\neq L _{\beta}$, and 
 $a^{L }_{\alpha,\beta}=a^{F }_{\alpha,\beta}$ if  $L _{\alpha}=L _{\beta}=F $. The proof of (b) in general then immediately reduces to the special case already treated,  in which $L $ is irreducible.
 
 We prove (c). For $\alpha,\beta\in L $, we have
 \begin{equation*}B(\omega_{\alpha}^{L },\omega_{\beta}^{L })=B(\omega_{\alpha}^{L },\sum_{\gamma\in L }a_{\gamma,\beta}^{L }\gamma)=\sum_{\gamma\in L }a_{\gamma,\beta}^{L }B(\omega_{\alpha}^{L },\gamma)=a_{\alpha,\beta}^{L }, \end{equation*} by (a). Hence (c) follows from (b).
 Part (d) follows directly from (b).
\end{proof}
\begin{prop}\label{finfacproj} Suppose that $L \seq \Pi  $ has only finite-type components and that $V_{L}=V$.  Then 
\begin{num}
\item For all $v\in V$,  
$\pi_{L }(v)=v-\sum_{\alpha\in L }B(v,\alpha)\omega^{L }_{\alpha}$  
\item If $\alpha\in V$, then $\pi_{L  }(\alpha)\in \alpha+\Span(L )$.
\item  If $\alpha,\beta\in V$, then $B(\pi_{L  }(\alpha),\pi_{L  }(\beta))=B(\alpha,\pi_{L }(\beta))$.
\item Let $C_{L }:=\mset{v\in V\mid B(v,L )\seq \mathbb{R}_{\geq 0}}$ denote the fundamental chamber for $W_{L}$ on $V$.  Then  $\cone(\Pi  \sm L)\seq -C_{L }$ and $\cone(\mset{\omega_{\alpha}^{L}\mid \alpha\in L})=C_{L}\cap \cone(L)$.
\item  If $\alpha\in -C_{L  }$, then $\pi_{L  }(\alpha)\in \alpha+
(C_{L}\cap\cone(L  ))$.
\item If 
 $\alpha,\beta\in -C_{L }$, then 
 $B(\pi_{L }(\alpha),\pi_{L }(\beta))\leq B(\alpha,\beta)$.
 \end{num}
  \end{prop}
  \begin{proof}  Let $v\in V$. Then  $v':=\sum_{\alpha\in L }B(v,\alpha)\omega^{L }_{\alpha}\in \Span(L )$ is defined
  and satisfies $B(v',\alpha)=B(v,\alpha)$ for all $\alpha\in L $.
  Hence $v'':=v-v'\in L ^{\perp}$ and  $v=v'+v''$, so $\pi_{L}(V)=v''$, proving (a).  Part (b) follows from (a).
  For (c), note that, by (b),  $B(\alpha-\pi_{L}(\alpha),\pi_{L}(\beta))\in B(\Span(L), L^{\perp})\seq\set{0}$.
  
  In (d), the assertion that $\cone(\Pi  \sm L)\seq -C_{L }$ follows from Theorem \ref{refcan} and the other assertion follows from
  Proposition \ref{fincomp}(a).  Part (e) follows from (a) and (d).
  
  To prove (f), note that by (c) and (e), 
  \[ B(\pi_{L}(\alpha),\pi_{L}(\beta))-B(\alpha,\beta)=
  B(\alpha,\pi_{L}(\beta)-\beta)\in B(-C_{L},\cone(L))\seq\mathbb{R}_{\leq 0}.\qedhere\]  \end{proof}
 \subsection{} For  a function $f\colon M\to U$ where $M$ is a finite set and $U$ is a subspace of $V$, we define the Gram determinant
\begin{equation}
\gr(f):=\det((B(f(\alpha),f(\beta)))_{\alpha,\beta\in M})\in \mathbb{R}
\end{equation} By convention, the  $0\times 0$ matrix has determinant $1$,  so $\gr(f):=1$ if $M=\eset$.
For any $N\seq M$, let $f_{N}\colon N\to U$ denote the restriction of $f$ to $N$. Thus,  $f=f_{M}$.

Suppose  that $N$ is a subset of $M$ such that the restriction of $B$ to a bilinear form on $\Span(f(N))$ is non-singular
(that is, $\Span(f(N))\cap (f(N))^{\perp}=\set{0}$).
Then there is a $B$-orthogonal direct sum decomposition
$V=\Span(f(N))\oplus (f(N))^{\perp}$. We let $p_{N}\colon V\to (f(N))^{\perp} $ denote the projection on $(f(N))^{\perp}$ corresponding to this decomposition.

\begin{prop}\label{Gram} With  notation and assumptions  as in the previous subsection,  
\[  \gr(f_{M})=\gr(f_{N})\gr(p_{N}\circ f_{M\setminus N}).\] 
\end{prop}
\begin{proof}  For the proof, suppose without loss of generality that $M=\set{1,\ldots, m}$ and $N=\set{1,\ldots, n}$ where $n\leq m$. 
We may assume that $1\leq n<m$, since otherwise the result holds trivially by the conventions. Write $f(i)=f_{i}$ for $i\in M$.
Let $A$ be the $m\times m$ real matrix with $(i,j)$-entry $a_{i,j}:=B(f_{i},f_{j})$ if $1\leq i,j\leq m$. Let $A'$ be the $m\times m$ matrix with $(i,j)$-entry 
\[a'_{i,j}:=\begin{cases}B(f_{i},f_{j}), &\text{\rm  if $1\leq i,j\leq n$}\\
B(p_{N}(f_{i}), p_{N}(f_{j})), &  
\text{\rm  if $n+1\leq i,j\leq m$}\\
0, &\text{\rm otherwise.}\end{cases}\] 
For $n+1\leq j\leq m$ and $1\leq i\leq n$, write
$p_{N}(f_{j})=\sum_{i=1}^{m}c_{i,j}f_{i}$ where $c_{i,j}\in \mathbb{R}$, and define $c_{i,j}:=\delta_{i,j}$ if $n+1\leq i\leq m$.  For $1\leq j\leq m$ and $1\leq i\leq n$, let $c_{i,j}:=\delta_{i,j}$.
Let $C$ denote the $m\times m$ matrix with $(i,j)$-entry $c_{i,j}$.
A simple computation suing bilinearity of $B$  shows that
$A'=CAC^{\mathrm{t}}$ where $C^{\mathrm{t}}$ denotes the transpose of $C$. We have $\det(C)=1$ since $C$ is upper unitriangular, so \begin{equation*}
\gr(f_{M})=\det(A)=\det(A')=\gr(f_{N})\gr(p_{N}f_{M\sm N})
\end{equation*} 
since $\det(a'_{ij})_{1\leq i,j\leq n}=\gr(f_{N}$, 
$\det(a'_{ij})_{n+i\leq i,j\leq m}=\gr(f_{N})$,  and $a'_{ij}=0$ if $i\leq n$ and $j>n$,  or if  $i>n$ and $j\leq n$. 
\end{proof}

\begin{prop} \label{propbform} Let $J\in \ob(G)$ and $U\sreq W_{J}$ be a reflection subgroup of $W$. Then $\Pi_{U}\sreq  J$.
\begin{num}
\item Let $\alpha,\beta\in\Pi_{U}\sm J$. Then  $B_{J}(\pi_{J}(\alpha),\pi_{J}(\beta))\leq B(\alpha,\beta)$ where if $\alpha\neq \beta$, then  $B(\alpha,\beta)\leq 0$.
\item Let $\alpha\in\Pi_{U}\sm J$ and $L'=L\cup\set{\alpha}$ denote the component of $\Pi_{U}$ containing $\alpha$, where $L\seq J$. Then $L$ is of finite type and, letting $\alpha':=\pi_{J}(\alpha)$, we have  
\begin{equation*}
\begin{cases}
B_{J}(\alpha',\alpha')>0, &\text{\rm if $L'$ is of finite type}\\ 
B_{J}(\alpha',\alpha')=0, &\text{\rm if $L'$ is of affine type}\\
B_{J}(\alpha',\alpha')<0, &\text{\rm  otherwise i.e. if $L'$ is of indefinite type.}\\
\end{cases}
\end{equation*} 
\item If $\Pi_{U}\sm J$ is finite and $L$ is a finite  union of  components of $J$ which contains all roots joined to a root in $\Pi_{U}\sm J$, then $\det(B_{J}(\pi_{J}(\alpha),\pi_{J}(\beta)))_{\alpha,\beta\in \Pi_{U}\sm J}$ has the same sign ($+$, $-$ or $0$) as $\det(B(\alpha,\beta))_{\alpha,\beta\in (\Pi_{U}\sm J)\cup L}$. 
\end{num}
\end{prop}
\begin{proof}
Part (a) holds by Proposition \ref{finfacproj}(f) since   $ \Pi_{U} \sm J\seq -C_{J}$ by Theorem \ref{refcan}.

We prove (b). Recall \ref{genass}. Note  $K:=\mset{\beta\in J\mid B(\alpha,\beta)\neq 0}$ is finite. Since $L$ is the union of the components of $J$ which contain an element of $K$, $L$ is a spherical subset of $\Pi$.  Note that $\alpha'=\pi_{J}(\alpha)=\pi_{L}(\alpha)\in \Span(L')$ since $\alpha\in L'$ and $L'\perp (J\sm L)$.  
Let $f:L'\to V$ be the inclusion.
By Proposition \ref{Gram}, we have
$\gr(f_{L'})=\gr(f_{L})B_{J}(\alpha',\alpha')$. Since $L$ is spherical,  $\gr(f_{L})>0$.  Now by \cite[Ch 4]{Kac},  $L'$ is of finite (resp., affine) type if and only if the restriction of $B$ to $\Span(L')$ is positive definite (resp., positive semidefinite of corank one). Since  the restriction of 
$B$ to $\Span(L)$ is positive definite, it follows that  $L'$ is of finite (resp., affine) type if and only if  $B_{J}(\alpha',\alpha')$ is positive (resp., zero). Part (b) follows.

The proof of (c) is similar to that of (b). We have 
$\pi_{J}(\alpha)=\pi_{L}(\alpha)$ for all $\alpha\in \Pi_{U}\sm J$, 
and $L$ is spherical. 
Let $f\colon  (\Pi_{U}\sm J)\cup L\to V$ be the inclusion. Then 
Proposition \ref{Gram} implies that 
$\gr(f)=\gr(f_{L})\gr(\pi_{L}\circ f_{\Pi_{U}\sm J})$. Since $L$ is 
spherical, $\gr(f_{L})>0$. Then (c) follows since 
$\gr(f)=\det((B(\alpha,\beta))_{\alpha,\beta\in (\Pi_{U}\sm J)\cup 
L})$ and $\gr(\pi_{L}\circ f_{\Pi_{U}\sm J})=\det((B(\pi_{J}
(\alpha),\pi_{J}(\beta)))_{\alpha,\beta\in \Pi_{U}\sm J})$. 
\end{proof}

  \begin{proof}[Proof of Proposition \ref{rootiprod}] 
   Part (a) is clear from the discussion above and Proposition \ref{realizedroot}.
   
   We prove (b).     Let $\alpha,\beta\in \lsub{J}{\Delta}$.
 Write $\alpha=\pi_{J}(\alpha')$ and $\beta=\pi_{J}(\beta')$ where $\alpha',\beta'\in \Pi\sm J$.   Proposition \ref{propbform}(a)--(b) show  that $B_{J}(\alpha,\beta)>0$ if and only if  $\alpha'=\beta'$ and 
 the component $L'$ of $J\cup\set{\alpha'}$ containing $\alpha'$ is of finite type. But $L'$ is of finite type if and only if $\alpha\in \lsub{J}{\Delta}^{\re}$, by Theorem \ref{BHabrs} and Proposition \ref{finindcond}.  Part (b) follows. Part (c) also follows using Proposition \ref{propbform}(c), and (d) follows from (c)  using
 Theorem \ref{BHabrs} and Proposition \ref{finindcond}.
 
 In  (e), it is sufficient to prove the two unions there  are of disjoint sets. Disjointness in the first union holds for any signed groupoid set, and that in the second union follows  since $\lsub{J}{\Delta}\seq {\lsub{J}{\wt\Upsilon}}_{X}$.
 Part (f) follows from Proposition \ref{realizedroot}(e).  \end{proof}
 
 \begin{cor}\label{notrealize} Let $J\in \ob(G)$ and  let $\alpha',\beta'\in \Pi\sm J$
 with $\pi_{J}(\alpha')=\pi_{J}(\beta')\in \lsub{J}{\Delta}^{\re}$. Then $\alpha'=\beta'$.\end{cor}

\begin{proof}  If instead $\alpha'\neq \beta'$,  the   proof of  Proposition \ref{rootiprod}(b) above would show that \[0\geq B_{J}(\pi_{J}(\alpha'),\pi_{J}(\beta'))>0.\qedhere\]\end{proof} 

  \begin{ex}\label{onlyweak} We give an example to show that, if  $\Pi$ is linearly dependent, the above corollary does not hold in general  if $\lsub{J}{\Delta}^{\re}$ is  replaced by $\lsub{J}{\Delta}$, and 
   $(\mcv,\iota)$  does not in general afford a realization (only a weak realization) of $(G,\Upsilon_{P})$.    
      
   Let $V$ be a three-dimensional  real vector space  with basis 
$\alpha_{1},\alpha_{2},\alpha_{3}$ and let 
$\alpha_{4}:=\alpha_{2}+\alpha_{3}-\alpha_{1}\in V$.  Let $B$ denote the symmetric 
bilinear form on $V$ such that the matrix 
$(2B(\alpha_{i},\alpha_{j}))_{1\leq i,j\leq 4}$ is the symmetric generalized Cartan matrix
   \begin{equation*}
   \begin{bmatrix*}[r] 2 & -1 & 0&-3\\
   -1& 2 & -3& 0\\
   0&-3&2&-1\\
   -3&0&-1&2
   \end{bmatrix*}.
   \end{equation*} 
 By bilinearity of $B$, the last row and column of this matrix are determined by the upper left $3\times 3$ submatrix and the definition of $\alpha_{4}$.
 
 Let $\Pi=\mset{\alpha_{i}\mid 1\leq i\leq 4}$.  Let $s_{i}$ denote the reflection $s_{i}:=s_{\alpha_{i}}$ on $V$,
 $S:=\mset{s_{i}\mid 1\leq i\leq4}$, $W=\mpair{S}\seq \mathrm{O}(V)$ and $\Phi:=\mset{w\alpha\mid w\in W,\alpha\in \Pi}$. Then $(W,S)$ is a Coxeter system and $(\Phi,\Pi)$ is a based root system of $(W,S)$ in the quadratic space $(V,B)$. 
 The Coxeter graph and presentation graph (see \ref{presgraph}) of $(W,S)$, regarded as graphs with vertex set $\Pi$, are respectively 
 
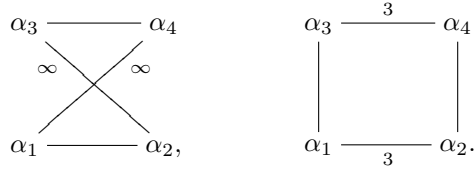
\begin{equation*}
 \xymatrix@-10pt{
 {\alpha_{3}}\ar@{-}[rr]\ar@{-}[ddrr]_<<<<{\infty}&&{\alpha_{4}}\ar@{-}[ddll]^<<<<{\infty}\\
 &&\\
 {\alpha_{1}}\ar@{-}[rr]&&{\alpha_{2},}
 }\qquad\qquad 
 \xymatrix@-10pt{
 {\alpha_{3}}\ar@{-}[rr]^{3}\ar@{-}[dd]&&{\alpha_{4}}\ar@{-}[dd]\\
 &&\\
 {\alpha_{1}}\ar@{-}[rr]_{3}&&{\alpha_{2}.}
 }
 \end{equation*}
 Note that $\cone(\Pi)$ is a polyhedral cone, with the elements of $\Pi$  as representatives of the extreme rays of the cone. In fact, the convex polygon $P$ with vertices $\alpha_{i}$ for $i=1,\ldots, 4$ is a quadrilateral cross-section of the cone. Note that $P$ must be as diagrammed  at the right above since
 $\frac{1}{2}(\alpha_{1}+\alpha_{4})=\frac{1}{2}(\alpha_{2}+\alpha_{3})$.

 Corollary \ref{presgraph} explains why  the presentation graph is a subgraph of the edge graph of the quadrilateral $P$, or equivalently, why the edge of the Coxeter graph joining $\alpha_{1}$ and $\alpha_{4}$ (respectively,  
 $\alpha_{2}$ and $\alpha_{3}$) has  label $\infty$.
 
 Let $J:=\set{\alpha_{1},\alpha_{2}}$, which is a spherical subset of  type $A_{2}$ of $\Pi$, and let  $G$ be the component  containing  $J$ of the full Brink-Howlett groupoid of $(W,S)$.  Then $G$ has only one object, $J$, and one  morphism. 
The conditions \ref{genass}(i)--(iii) hold. It is straightforward to check $\pi_{J}(\alpha_{3})=\alpha_{3}+2\alpha_{2}+\alpha_{1}$.
By symmetry, we have $\pi_{J}(\alpha_{4})=\alpha_{4}+2\alpha_{1}+\alpha_{2}$. Hence $\pi_{J}(\alpha_{4})=\pi_{J}(\alpha_{3})$. By  Corollary \ref{notrealize},
$\pi_{J}(\alpha_{3})=\pi_{J}(\alpha_{4})\in \lsub{J}{\Delta}^{\im}$. Since
$U_{i}:=W_{J\cup\set{\alpha_{i}}}\in P_{J}$ for $i=3,4$ are distinct
with $r_{U_{i},J}=\alpha_{i}$, it follows that  $(\nu,\iota)$ is not a realization of $(G,\Upsilon_{P})$ in this example.
      \end{ex}

\begin{proof}[Proof of Proposition \ref{longelt}] 
We prove (a). Let $K\seq \Pi\sm J$. 
Write the union of all components of $K\cup J$ which intersect $K$ as $L\dot\cup K$.  Note  $L\seq J$ is the union of all 
components of $J$ which contain a vertex which is joined to a vertex of $K$ in the Coxeter graph of $\Pi$. Recall that 
$\nu(K,J)$ exists if and only $L\cup K$ is spherical, in which case $K\cup L$ is linearly independent. 

If $\nu(K,J)$ exists, then $K$ is  certainly finite. So there is no loss of generality in  assuming for the proof of (a) that $K$ is finite. Since $K\perp J\sm L$, we have 
$\pi_{J}(K)=\pi_{L}(K)\seq \Span(K\cup L)$. Since $L$ is spherical,  the restriction of 
$B$ to $\Span(L)$ is positive definite and  there is a $B$-orthogonal direct sum $\Span(K\cup L)=\Span(L)\oplus
\Span(\pi_{L}(K))$.   It follows that the 
restriction of $B$ to  $\Span(K\cup L)$ is positive definite (i.e. $K\cup L$ is spherical) if and only if the restriction of $B$ (or 
equivalently, of $B_{J}$) to $\Span(\pi_{L}(K))=\Span(\pi_{J}(K))$ is positive definite. This proves (a).

For the  proof of (b)--(f), assume  that $w:=\nu(K,J)$ exists, and let $L$ be as in the proof of (a).  Since $K\cup L$ is spherical, it is linearly independent.   From the proof of (a),
\begin{equation*} \begin{split}
\vert L\vert +\vert K\vert &=\dim(\Span(L\cup K))=\dim(\Span(L))+\dim(\Span(\pi_{L}(K)))\\&=\vert L\vert +\dim(\Span(\pi_{L}(K))).
\end{split}\end{equation*}
Hence $\dim(\Span(\pi_{J}(K)))=\dim(\Span(\pi_{J}(K)))=\vert K\vert$, so $\pi_{J}(K)$ is linearly independent, proving (b). 

We prove (c). We have $w=\nu(K,J)=\nu(K,L)=w_{K\cup L}w_{L}$ and $\nu(\alpha,J)=\nu(\alpha,L)=w_{L\cup\set{\alpha}}w_{L}$ for $\alpha\in K$. We have to show that
$
\bigvee_{\alpha\in K}w_{L}w_{L\cup\set{\alpha}}=w_{L}w_{K\cup L}$ in weak right order of $W$. This follows from 
$
\bigvee_{\alpha\in K}w_{L\cup\set{\alpha}}=w_{K\cup L}$
since each element of $W$ appearing in this formula is greater than or equal to $w_{L}$ in weak right order.

For (d)--(f), refer to  Lemma \ref{Deodgen}.  First,  (d) here is just a reformulation of Lemma \ref{Deodgen}(d).
To prove (e), note that 
 $M=L\dotcup K$ is spherical from above, $w=w_{M}w_{L}$, $J'=w(J)=(J\sm L)\dot\cup-w_{M}(L)$ and  \[K'=((J\sm L)\cup M)\sm ((J\sm L)\cup -w_{M}(L))=M\sm -w_{M}(L)=-w_{M}(K),\] since $M=L\dotcup K$ and $w_{M}(M)=-M$.  Hence the map $\sigma$ defined in (e) is a bijection as stated.
 Finally, for $\alpha\in K$, we have
 $w(\alpha_{J})=w(\pi_{J}(\alpha))=\pi_{J'}(w\alpha)$.
 We have \[w\alpha=w_{M}w_{L}(\alpha)\in w_{M}(\alpha+\Span(L))=-\sigma(\alpha)+\Span(w_{M}(L))\]
 where, since $w_{L}L=-L$, \[\Span(w_{M}(L)=\Span(w_{M}w_{L}(L))\seq \Span(w(J))=\Span(J')\seq \ker(\pi_{J'}).\] 
 Hence $w(\alpha_{J})=(-\sigma\alpha)_{J'}$ and (f) follows from (e) since the map $\beta\mapsto \pi_{J'}(\beta)\colon K'\to \lsub{J'}{\Delta}^{\re}$ is injective by Corollary \ref{notrealize}.
 Alternatively, one could finish  by noting that by symmetry, there is a map
 $-\beta_{J'}\mapsto (\sigma^{-1}\beta)_{J}\colon -\Pi_{J'}(K')\to \Pi_{J}(K)$ and that it is an  inverse to the map in (f).
\end{proof}

\subsection{} The following proposition records some basic facts about  ``non-reduced'' (weakly) realized root systems.
\begin{prop} \label{prop8.5} Let $J\in \ob(G)$, $U,U'\in R_{J}$ and 
$U''\in M_{J}$.
\begin{num}
\item Assume  $U\seq U'$. Then $r'_{U,J} =cr'_{U',J}$ for some $c\in \mathbb{R}_{\geq 1}$.  
\item Suppose that $r'_{U',J} =cr'_{U,J} $ with $c\in \mathbb{R}_{>0}$ and that  $(U,+)\in \lsub{J}{\Upsilon}^{\re}_{R}$.  Then  $(U',+)\in \lsub{J}{\Upsilon}^{\re}_{R}$.
\item Assume that  $r'_{U,J} =cr'_{U'',J} $ where $c>0$ and both  
$(U,+)$ and  $(U'',+)$ are elements of $   \lsub{J}{\Upsilon}^{\re}_{R}$. Then $U\seq U''$.
\item  Suppose that $(U,+)\in \lsub{J}{\Upsilon}^{\re}_{R}$, and  that
 $U\seq U''$. Write, by (a),  $r'_{U,J} =cr'_{U'',J} $ where $c\in \mathbb{R}_{\geq 1}$. Then $c=1$ if and only if $U''=U$. 
 \item For any  root $\alpha\in \lsub{J}{\wt\Upsilon}_{R}^{\re,+}$, there are only finitely many $c\in \mathbb{R}_{>0}$ such that 
  $c\alpha\in \lsub{J}{\wt\Upsilon}_{R}^{\re,+}$.
\end{num}
\end{prop}
\begin{proof}  We prove (a).  Since $ U \seq U'$, we have $ \Phi_{U} \seq \Phi_{U'}$, and thus each element of $ \Pi_{U} = J \cup \{ r_{U,J} \}$ can be expressed as a non-negative linear combination of $ \Pi_{U'} = J \cup \set{r_{U',J} }$. Hence
$r_{U,J} \in cr_{U',J}  + \cone(J)$
where $ c \geq 0$. We must have  $ c > 0$. For suppose that $ c = 0$. Then $ \Pi_{U} \subseteq \Span(J)$,  which implies that $\Pi_{U}\seq \Phi_{J}$ by Proposition \ref{posdef}  and so $U\seq W_{J}$, contrary to  $U \in R_{J} $. By  \cite[Proposition 2.26]{Br94}, we actually have $ c \geq 1$. Applying $ \pi_{J}$ to  $ r_{U,J} \in cr_{U',J}  + \cone(J)$ shows
$ r'_{U,J}  \in cr'_{U',J}  + \pi_{J}(\cone(J))$.
Since $ J\seq \ker(\pi_{J})$, (a) follows.

For the proof of (b), suppose that $r'_{U,J} =cr'_{U',J} $ where $c>0$ and $(U,+)\in \lsub{J}{\Upsilon}^{\re,+}_{R}$. There is a morphism $g=(K,w,J)$ in $G$  such that  $(\Upsilon_{R}(g))(U,+)\in \Upsilon^{-}_{R}(K)$. Using Proposition   \ref{rootiprod}(a) shows in turn that    $wr'_{U,J} \in 
\lsub{K}{\wt \Upsilon}_{R}^{-}\seq -\cone(\lsub{K}{\Delta})$,  that 
$wr'_{U',J} =cwr'_{U,J} \in  -\cone(\lsub{K}{\Delta})$,
that $wr'_{U',J} \not \in  \cone(\lsub{K}{\Delta})$, 
 that $wr'_{U',J} \in 
\lsub{K}{\wt \Upsilon}_{R}^{-}$  and hence that   $(\Upsilon_{R}(g))(U',+)\in \Upsilon^{-}_{R}(K)$ and $(U',+)\in \lsub{J}{\Upsilon}^{\re,+}_{R}$.

We prove (c). Note that $U\in N_{J}$ and $U''\in Q_{J}$.
We have $r_{U,J} \in cr_{U'',J} + \Span(J)\seq \Span(\Pi_{U''})$ since $\Pi_{U''}=J\cup\set{r_{U'',J} }$. Note $U''$ is a parabolic subgroup of $W$ with all its components of finite type.  By applying  Proposition \ref{posdef} to a standard parabolic subgroup which is   $W$-conjugate
to $U''$ , it follows that $r_{U,J} \in \Phi_{U''}$. Hence $\Pi_{U}=J\cup\set{r_{U,J} }\seq \Pi_{U''}$ and so $U\seq U''$ as required. 

Now we prove (d). The assumptions imply that $U\in N_{J}$ and $U''\in Q_{J}\seq N_{J}$.
Clearly, $c=1$ if $U''=U$. On the other hand, if  $c=1$, then 
$r'_{U,J}=r'_{U'',J}$, which implies $U=U''$ by Theorem 
\ref{weakreal}(d). 

Finally, we prove (e). We may assume without loss of generality
by (a)--(d) that  $\alpha=r'_{U'',J} $ where  $U''\in Q_{J}$, so $[U'':W_{J}]$ is finite. By (c), it suffices to prove there are only finitely many $U\in R_{J}$ such that $U\seq U''$. Any such $U$ satisfies  $W_{J}\seq U\seq U''$, and    is therefore the union of   some subset of the finitely many
left (say) cosets of $W_{J}$ in $U''$, so  there are only finitely many possibilities for $U$. \end{proof}

\begin{cor} \label{parclass} Let $J\in \ob(G)$. Then
 the parallelism class of a root $\alpha\in\lsub{J}{\wt\Upsilon}^{\re}_{X}$ is $\mathbb{R}_{> 0}\alpha\cap \lsub{J}{\wt\Upsilon}_{X}$, which  is a finite set. 
 \end{cor}
\begin{proof} Clearly, any positive multiple of $\alpha$ is in the same parallelism class as $\alpha$.  Consider any roots $\alpha_{i}\in\lsub{J}{\wt\Upsilon}^{\re}_{X}$ for $i=1,2$. We have $\alpha_{i}=\epsilon_{i} r_{U_{i},J}$ for some $U_{i}\in X_{J}$ and $\epsilon_{i} \in \set{\pm 1}$. Since $\alpha_{i}$ is a real root, we also have  $U_{i}\in N_{J}$.  By  Theorem \ref{infTits}, there exists unique
$U_{i}''\in Q_{J}\seq P_{J}\seq X_{J}$ such that  $U_{i}\seq U_{i}''$. Then $U_{i}''\in  P_{J}\seq  M_{J}$. Proposition \ref{prop8.5} implies $r_{U_{i},J}=c_{i}r_{U_{i}'',J}$ for some scalar $c_{i}\geq 1$. Let 
$\beta_{i}:= (r_{U_{i}'',J},\epsilon_{i})\in \lsub{J}{\Upsilon}_{Q}\seq  \lsub{J}{\Upsilon}_{X}^{\re}$.
By Theorem \ref{weakreal}, $\alpha_{1}$ and $ \alpha_{2}$ are in the same parallelism class  
in $\lsub{J}{\wt\Upsilon}^{\re}_{X}$ if and only if 
$\beta_{1}, \beta_{2}$ are in the same parallelism class in
$\lsub{J}{\Upsilon}_{X}^{\re}$, and hence if and only if $\beta_{1}$ and $\beta_{2}$ are  in the same parallelism class in $\lsub{J}{\Upsilon}_{Q}$  (recall from Theorem \ref{BHabrs} that $(G,\lsub{J}{\Upsilon}_{Q})\seq (G, \lsub{J}{\Upsilon}_{X})$).

By Theorem \ref{standpreprinc}, the parallelism classes of roots
of $(G,\Lambda^{\rec})$ are all singleton sets. By Theorem
\ref{BHabrs}, so are those of  (the isomorphic signed $G$-set) $ (G,\Upsilon_{Q})$.
Hence if $\alpha_{1}$ and $\alpha_{2}$ are in the same parallelism class of $\lsub{J}{\wt\Upsilon}^{\re}_{X}$, then $\beta_{1}=\beta_{2}$
and so $\alpha_{2}=c_{2}c_{1}^{-1}\alpha_{1}\in \mathbb{R}_{>0}\alpha_{1}$ as required. The final assertion, on finiteness of  parallelism classes, holds by Proposition \ref{prop8.5}(e).
\end{proof}
\begin{proof}[Proof of Proposition \ref{BHact}]
From  Proposition \ref{longelt}, we have 
\begin{equation}
\nu(\alpha,J)\alpha_{J}=-\alpha'_{J'}
\end{equation} where
\begin{equation}\label{simprootiprod}
B_{J}(\alpha_{J},\alpha_{J})=B_{J'}(\alpha'_{J'},\alpha'_{J'})\neq 0.
\end{equation}

Write $L:=J\cup\set{\alpha}=J'\cup\set{\alpha'}$.
Let $M\seq J$ be such that $M\cup\set{\alpha}$ is  the component of $L=J\cup \set{\alpha}$ which contains $\alpha$. 
Observe for subsequent use that  the conditions on $F$  in \ref{asscons} all hold if $F$ is any one of $J$, $L$, $M$ or  $M\cup\set{\alpha}$.

 There is  is a $B$-orthogonal direct sum decomposition
 $V=\lsub{L}{V}\oplus\Span(L)$ where $\lsub{L}{V}=L^{\perp}$, as described in \ref{asscons}.    We therefore have an orthogonal direct sum decomposition
 \begin{equation}
\lsub{J}{V}=\lsub{L}{V}\oplus\Span(\alpha_{J}).
\end{equation}  
In fact, for $v\in V$, we have $v=v_{J}$ if  $v \in \lsub{J}{V}$, and  in general,
 \begin{equation}\label{projform}
v_{J}= v_{L}+c\alpha_{J}\qquad c:=\frac{B_{J}(v_{J},\alpha_{J})}{B_{J}(\alpha_{J},\alpha_{J})}
 \end{equation} From the analogous formula for $v_{J'}$, it follows that for all $v\in V$,
 \begin{equation}\label{projequal}
 v_{J}-\frac{B_{J}(v_{J},\alpha_{J})}{B_{J}(\alpha_{J},\alpha_{J})}\alpha_{J}=v_{L}=v_{J'}-\frac{B_{J'}(v_{J'},\alpha'_{J'})}{B_{J'}(\alpha'_{J'},\alpha'_{J'})}\alpha'_{J'}.
 \end{equation}
 Using this, it follows there is  linear isometry $\tau_{g}\colon \lsub{J}{\mathcal{V}}\to \lsub{J'}{\mathcal{V}}$   as stated  in Proposition \ref{BHact}(a),   given explicitly by
 \begin{equation}
 \tau_{g}(v)=v-\frac{B_{J}(v,\alpha_{J})}{B_{J}(\alpha_{J},\alpha_{J})}\alpha_{J}+\frac{B_{J}(v,\alpha_{J})}{B_{J}(\alpha_{J},\alpha_{J})}\alpha'_{J'},\qquad v\in \lsub{J}{\mathcal{V}}.
 \end{equation}
 Saying $\tau_{g}$ is a linear isometry means that it is a $\mathbb{R}$-linear map satisfying
 \begin{equation}
 B_{J'}(\tau_{g}(v),\tau_{g}(v'))=B_{J}(v,v'),\qquad v,v'\in \lsub{J}{\mcv}.
 \end{equation}
The orthogonal reflection $s_{\alpha_{J}}\colon \lsub{J}{\mathcal{V}}\to \lsub{J}{\mathcal{V}} $   in $\alpha_{J}$   on the quadratic space $(\lsub{J}{\mathcal{V}},B_{J})$, is  given by the standard formula
 \begin{equation}
 s_{\alpha_{J}}(v)=v-\frac{2B_{J}(v,\alpha_{J})}{B_{J}(\alpha_{J},\alpha_{J})}\alpha_{J},\qquad v\in \lsub{J}{\mathcal{V}}.
 \end{equation}
 There is a similar formula  for  the orthogonal reflection
 $s_{\alpha'_{J'}}$ on $(\lsub{J'}{\mathcal{V}},B_{J'})$ in the vector $\alpha'_{J'}$. 
 
 Note that $\alpha_{J}=\pi_{M}(\alpha)$. There is  a $B$-orthogonal direct sum decomposition 
 \begin{equation}
 V=\lsub{L}{V}\oplus \Span(\alpha_{J})\oplus \Span(J\sm M)\oplus\Span(M)
 \end{equation} in which the sum of the first two (resp., last two,  first three, last three, second and fourth, first and third) factors is $\lsub{J}{\mcv}$ (resp., $\Span(J)$, $\lsub{M}{V}$, $\Span(L)$, $\Span(M\cup\set{\alpha})$, $\lsub{M\cup\set{\alpha}}{V}$).
 
 Recall that $\mathcal{V}(g)$ is  a linear map obtained by 
 restricting the domain and codomain of $w\colon V\to V$ 
 appropriately, where $w:=w_{M\cup\set{\alpha}}w_{M}$. It 
 follows that $\mathcal{V}(g)$ acts as the identity on $\lsub{M\cup\set{\alpha}}{V}$ and maps $\alpha_{J}$ to $-\alpha'_{J'}$. This determines the action of $w$ on $\lsub{M}{V}\sreq \lsub{J}{\mcv}$.   It is easily checked 
 that  $w$ and $\tau_{g}s_{\alpha_{J}}$  both act in the same way  on $\lsub{J}{\mcv}$. 
 This proves the  formula $\mathcal{V}(g)=\tau_{g}s_{\alpha_{J}}$.  The corresponding formula for $g^{-1}$ instead of $g$ is $\mathcal{V}(g^{-1})=\tau_{g^{-1}}s'_{\alpha'_{J'}}$. Taking inverses shows that 
 $\mathcal{V}(g)=s'_{\alpha'_{J'}}\tau_{g}$ since  $\tau_{g^{-1}}=(\tau_{g})^{-1}$.
 Hence Proposition \ref{BHact}(a) holds.

 We may give an explicit formula for the linear map $\mathcal{V}(g)$ as follows.  For any $v\in V$, we have   $v_{J}\in \lsub{J}{\mathcal{V}}$ and  \eqref{projform} shows  
 \begin{equation*}
 v_{J}=v_{L}+\frac{B_{J}(v_{J},\alpha_{J})}{B_{J}(\alpha_{J},\alpha_{J})}\alpha_{J}=\left(v_{J}-\frac{B_{J}(v_{J},\alpha_{J})}{B_{J}(\alpha_{J},\alpha_{J})}\alpha_{J}\right )+\frac{B_{J}(v_{J},\alpha_{J})}{B_{J}(\alpha_{J},\alpha_{J})}\alpha_{J} .\end{equation*} Therefore,  
  \begin{equation}\label{BHformc}
  gv_{J}:=(\mathcal{V}(g))(v_{J})=v_{J}-\frac{B_{J}(v_{J},\alpha_{J})}{B_{J}(\alpha_{J},\alpha_{J})}(\alpha_{J}+\alpha'_{J'})\end{equation}
  If $v\in \lsub{J}{\mathcal{V}}$, then $v=v_{J}$ and this simplifies to 
  \begin{equation} (\mathcal{V}(g))(v)=v-\frac{B_{J}(v,\alpha_{J})}{B_{J}(\alpha_{J},\alpha_{J})}(\alpha_{J}+\alpha'_{J'}).\end{equation}
  Also, for any $v\in V$, 
  \eqref{projequal} implies the following alternative form of 
  \eqref{BHformc}:   
  \begin{equation} gv_{J}=v_{J'}-\frac{B_{J}(v_{J},\alpha_{J})+B_{J'}(v_{J'},\alpha'_{J'})}{B_{J}(\alpha_{J},\alpha_{J})}\,\alpha'_{J'}.
\end{equation}
It is straightforward to check that, if $\Pi$ is linearly independent, then for  $v=\beta$, where $\beta\in \Pi\sm J$,  this last formula reduces to 
the assertion of Proposition \ref{BHact}(b); in accordance with the case-wise definition of $\kappa_{g}(\beta_{J})$, one distinguishes the cases $\beta=\alpha$ and $\beta\neq \alpha$  in the argument.  
\end{proof}

\section{Tits cones of realized   Brink-Howlett groupoids}
\label{Tits}
\subsection{} We begin this section by discussing alternative descriptions   of $ \lsub{J}{\mathcal{C}}$ and $ \lsub{J}{\mathcal{X}}$, and their dependence on the realized root system used to define them. 

Let $X$ denote either $R$, $M$ or $P$. Note that  by Proposition \ref{rootiprod}, 
\begin{equation}\begin{split}
\lrsub{J}{\mcc}{X}:=&\mset{v\in \lsub{J}{\mcv}  \mid B_{J}(v,{\lsub{J}{\wt\Upsilon}}^{+}_{X})\seq \mathbb{R}_{\geq 0}}\\
=&\mset{v\in \lsub{J}{\mcv}  \mid B_{J}(v,\lsub{J}{\Delta})\seq \mathbb{R}_{\geq 0}}
=\lsub{J}{\mcc}\end{split}
\end{equation}
which is independent of $X$.   Hence
\begin{equation}\lsub{J}{\mcx}_{X}:=\bigcup_{K\in \ob(G)}\,\bigcup_{w\in\lrsub{J}{G}{K}}(\mcv(w))(\lsub{K}{\mcc}_{X})=\lsub{J}{\mcx}\end{equation}  is also independent of $X$.
(This is consistent with our convention of (possibly)  omitting $X$ from notation if $X=R$.)
Define 
\begin{equation}
\lsub{J}{\mcc}^{\re}_{X}:=
\mset{v\in \lsub{J}{\mcv}  \mid B_{J}(v,{\lsub{J}{\wt\Upsilon}}^{\re,+}_{X})\seq \mathbb{R}_{\geq 0}}\end{equation} and
\begin{equation}\lsub{J}{\mcx}^{\re}_{X}=\bigcup_{K\in \ob(G)}\,\bigcup_{w\in\lrsub{J}{G}{K}}(\mcv(w))(\lsub{K}{\mcc}^{\re}_{X}),\end{equation} and also make  analogous definitions of $\lsub{J}{\mcc}^{\im}_{X}$ and $\lsub{J}{\mcx}^{\im}_{X}$ by replacing ``$\re$'' with ``$\im$'' above. Then 
$\lsub{J}{\mcc}^{\re}_{X}$ is independent of $X$, by Theorem \ref{BHabrs}(c) and Proposition \ref{prop8.5}, and will be denoted  as $\lsub{J}{\mcc}^{\re}$. Hence
$\lsub{J}{\mcx}^{\re}_{X}$ is also independent of $X$, and we may denote it as $\lsub{J}{\mcx}^{\re}$ .
Proposition \ref{prop8.5}  implies further that
$\lsub{J}{\mcc}^{\im}_{R}=\lsub{J}{\mcc}^{\im}_{M}$, so 
$\lsub{J}{\mcx}^{\im}_{R}=\lsub{J}{\mcx}^{\im}_{M}$.
Since  ${\lsub{J}{\wt\Upsilon}}^{+}_{X}={\lsub{J}{\wt\Upsilon}}^{\re,+}_{X}\cup {\lsub{J}{\wt\Upsilon}}^{\re,+}_{X}$, we have 
\begin{equation}
\lsub{J}{\mcc}^{\re}_{X}\cap \lsub{J}{\mcc}^{\im}_{X}=\lsub{J}{\mcc}_{X}=\lsub{J}{\mcc}, 
\end{equation}
independent of $X$. 
For any morphism $(J,w,K)$ in $G$, we have $\mcv(w) ({\lsub{K}{\wt\Upsilon}}^{\im,+}_{X})=
{\lsub{J}{\wt\Upsilon}}^{\im,+}_{X}$, by definition of imaginary roots,  and hence 
$(\mcv(w))(\lsub{K}{\mcc}^{\im}_{X})\seq  \lsub{J}{\mcc}^{\im}_{X}
$. By definition of $\lsub{J}{\mcx}^{\im}_{X}$, this implies that
\begin{equation}
\lsub{J}{\mcx}^{\im}_{X}=  \lsub{J}{\mcc}^{\im}_{X}, \qquad J\in \ob(G).
\end{equation}
We now compute
\begin{equation*}\begin{split}
\lsub{J}{\mcx}_{X}&=\bigcup_{K\in \ob(G)}\,\bigcup_{w\in\lrsub{J}{G}{K}}(\mcv(w))(\lrsub{K}{\mcc}{X})=
\bigcup_{K\in \ob(G)}\,\bigcup_{w\in\lrsub{J}{G}{K}}(\mcv(w))({\lsub{K}{\mcc}}^{\re}_{X}\cap {\lsub{K}{\mcc}}^{\im}_{X})\\
&=
\bigcup_{K\in \ob(G)}\,\bigcup_{w\in\lrsub{J}{G}{K}}\bigl((\mcv(w))({\lsub{K}{\mcc}}^{\re}_{X})\cap (\mcv(w))({\lsub{K}{\mcc}}^{\im}_{X})\bigr)\\
&=
\bigcup_{K\in \ob(G)}\,\bigcup_{w\in\lrsub{J}{G}{K}}\bigl((\mcv(w))({\lsub{K}{\mcc}}^{\re}_{X})\cap{\lsub{J}{\mcc}}^{\im}_{X}\bigr)\\
&=
{\lsub{J}{\mcc}}^{\im}_{X}\cap \bigcup_{K\in \ob(G)}\,\bigcup_{w\in\lrsub{J}{G}{K}}(\mcv(w))({\lsub{K}{\mcc}}^{\re}_{X})
=
{\lsub{J}{\mcx}}^{\im}_{X}\cap \lsub{J}{\mcx}^{\re}_{X}. \end{split}
\end{equation*}
Thus, we have 
\begin{equation}\label{coneint}
\lsub{J}{\mcx}={\lsub{J}{\mcx}}^{\im}_{X}\cap \lsub{J}{\mcx}^{\re}.
\end{equation}
We leave open  the questions of whether   $\lsub{J}{\mcc}^{\im}_{X}=\lsub{J}{\mcx}^{\im}_{X}$ is independent of $X$ and whether 
$\lsub{J}{\mcx}^{\re}$ is a cone.

  \begin{prop} \label{propc}
    For any object $ J$ of the groupoid $G$, we have $ \lsub{J}{\mathcal{C}} = \mathcal{C} \cap J^{\perp}$ where
   $\mathcal{C} : = \mset{ v\in V \mid B(v, \alpha) \geq 0  \text{ \rm for all $\alpha \in \Pi$} }$.
\end{prop}
\begin{proof}
    Let $ x \in \lsub{J}{\mathcal{C}}$. Then  $x\in  \lsub{J}{\mcv}  = J^{\perp}$. Hence $ B(x, J) \seq{\set{0}} $. Also by definition of $ \lsub{J}{\mathcal{C}}$, we have 
   $ B(x, \pi_{J}(\alpha)) \geq 0$ for all   $\alpha \in \Pi \sm J $.
   Since  $ \pi_{J}(\alpha) \in  \alpha + \Span(J )$ where $B(x,\Span(J))=\set{0}$, this gives $B(x,\alpha)\geq 0$ for all $\alpha\in \Pi\sm J$. Hence $B(x,\Pi)\seq \mathbb{R}_{\geq 0}$ and $x\in \mathcal{C}\cap J^{\perp}$.  This proves 
   $\lsub{J}{\mathcal{C}}\seq \mathcal{C}\cap J^{\perp}$.

 To prove the reverse inclusion, let $ x \in \mathcal{C} \cap J ^{\perp}$. Then $ x \in J ^{\perp} = \lsub{J}{\mcv}$. Further, since $ x \in \mathcal{C}$, we have 
$ B(x , \Pi\sm J)\seq B(x,\Pi)\seq\mathbb{R}_{ \geq 0}$. Hence
  $ B(x , \alpha) \geq 0$ for all  $\alpha \in \Pi \sm J$.
    But since $ \pi_{J}(\alpha) \in \alpha + \Span(J )$ and $ x\in J ^{\perp}$, it follows that $ B(x, \alpha ) = B(x, \pi_{J}(\alpha))$ for all $ \alpha \in \Pi \sm J $. Thus, we get
$ B(x , \pi_{J}(\alpha)) \geq 0$  for all $\alpha \in \Pi \sm J $,
    so we deduce that $ x \in \lsub{J}{\mathcal{C}}$. This proves that $ \lsub{J}{\mathcal{C}} = \mathcal{C} \cap J ^{\perp}$.
\end{proof}
\subsection{}  Let $ J$ be an object of the groupoid $G$. Let $ v \in \lsub{J}{\mcv} = J ^{\perp}$. Define 
\[ A_{X,J}(v) :=  \mset{ \beta \in {\lsub{J}{\wt\Upsilon}}_{X}^{\mathrm{re}, +} \mid B_{J}(v, \beta) < 0 }. \]
\begin{lem}\label{Titslem} Let $J\in \ob(G)$.
 \begin{num}
\item $\lsub{J}{\mcx}^{\re}\seq\mset{v\in \lsub{J}{\mcv}\mid  \vert A_{X,J}(v)\vert <\infty}$.
\item $\mset{v\in \lsub{J}{\mcx}^{\im}_{X}\mid  \vert A_{X,J}(v)\vert <\infty} \seq \lsub{J}{\mcx}$. 
\end{num}
\end{lem}
\begin{proof} We prove (a). 
Let $v\in \lsub{J}{\mcx}^{\re}$.  By definition, there exists a morphism $g=(J,w,K)$ in $G$ and an element $v'\in \lsub{K}{\mcc}^{\re}$ such that
$v=(\mcv(g))(v')$. We shall show that 
$A_{X,J}(v)\seq \wt\Upsilon_{X,g}$, where the right hand side is the 
$\wt\Upsilon_{X}$-inversion set of $g$. Suppose that 
$\alpha\in A_{X,J}(v)$.
Then \[0>B_{J}(\alpha,v)=
B_{K}((\mcv(g^{-1}))(\alpha),(\mcv(g^{-1}))(v))=B_{J}
((\wt\Upsilon_{X}(g^{-1}))(\alpha),v').\]  Since $v'\in \lsub{J}{\mcc}^{\re}$, 
we must have $(\wt\Upsilon_{X}(g^{-1}))(\alpha)\in \lsub{K}{\wt\Upsilon}^{-}$. Since  $\alpha\in \lsub{J}{\wt\Upsilon}^{+}_{X}$ by 
definition of  $A_{X,J}(v)$,
this shows that $\alpha\in  \wt\Upsilon_{X,g}$ and hence 
$A_{X,J}(v)\seq \wt\Upsilon_{X,g}$. 
Now we note that 
 \[\wt\Upsilon_{X,g}=\wt\Upsilon^{\re}_{X,g}\cong 
\wt\Upsilon^{\re}_{X,g}=\wt\Upsilon^{\re}_{X,g}\] where the bijection is from Theorem \ref{weakreal}. Using Corollary \ref{rootcor}, we see that  $\wt\Upsilon_{X,g}$ is finite, and hence so is $A_{X,J}(v)$.
This competes the proof of (a).

To prove (b), we show by induction on $n$ that if $v\in \lsub{J}{\mcx}_{X}^{\im}$ and $A_{X,J}(v)$ is finite of cardinality $n$, then $v\in \lsub{J}{\mcx}$.
If $n=0$, then  $A_{X,J}(v)=\eset$ and  \[v\in \lsub{J}{\mcc}^{\re}\cap  \lsub{J}{\mcx}^{\im}=\lsub{J}{\mcc}\seq\lsub{J}{\mcx}\] by the definitions. Suppose $n>0$, say $\alpha\in A_{X,J}$. We may write
$\alpha=\sum_{\beta\in \Gamma}c_{\beta}\beta$ where
$\Gamma$ is a finite subset of $\lsub{J}{\Delta}$ and   $c_{\beta}\in \mathbb{R}_{>0}$ for all $\beta\in \Gamma$. Since 
$B_{J}(\alpha,v)=\sum_{\beta\in \Gamma}c_{\beta}B_{J}(\beta,v)<0$, we have $B_{J}(\beta,v)<0$ for some $\beta\in \Gamma$. We can't have $\beta\in\lsub{J}{\Delta}^{\im}$, since  
$v\in \lsub{J}{\mcx}^{\im}$. Hence $\beta\in \lsub{J}{\Delta}^{\re}\cap A_{X,J}(v)$.

Write $\beta=\pi_{J}(\gamma)=\gamma_{J}$ where $\gamma\in \Pi$. Let $w=\nu(\gamma,J)$, $K:=wJ$ and define $\delta\in \Pi$ by $J\cup\set{\gamma}=K\cup\set{\delta}$, so $w^{-1}=\nu(\delta,K)$. We have a morphism $g:=(K,w,J)$ of $G$, 
with $(\wt\Upsilon_{X}(g))(\beta)=(\wt\Upsilon_{X}(g))(\gamma_{J})=-\delta_{K}$. 

Write $\wt\Upsilon_{X}(g)$ as $w$ for notational simplicity.  
Let $v':=w(v)\in \lsub{K}{\mcv}$. 
We claim that
\begin{equation}
A_{X,K}(v')=w(A_{X,J}(v)\sm \Span(\beta))
\end{equation} To prove this, note first that   $B_{K}(\delta_{K},v')=B_{K}(-w\beta,wv)=-B_{J}(\beta,v)>0$ so $\Span(\delta_{K})=w(\Span(\beta))$ is disjoint from $A_{X,K}(v')$. 
There is a bijection 
\[\rho\mapsto w^{-1}\rho\colon  \lsub{K}{\wt \Upsilon}^{+}\sm \Span(\delta_{K})\xrightarrow{\cong}
\lsub{J}{\wt \Upsilon}^{+}\sm \Span(\beta).\]
This holds, for instance, because the inversion set for $\wt \Upsilon_{X}$ of the Brink-Howlett generator $g$ is the  parallelism class of roots containing $\delta_{K}$ (by  Theorem \ref{standpreprinc} and Proposition \ref{princpreprinc})
and that parallelism classes is as described in Proposition \ref{parclass}.
The claim follows, since for $\rho\in \lsub{K}{\wt \Upsilon}^{+}\sm \Span(\delta_{K})$, we have
\begin{equation*}
\begin{split}
 &B_{K}(\rho, v')<0\iff B_{K}(\rho,wv)<0\\ \iff &B_{J}(w^{-1}\rho, w^{-1}v')<0\iff B_{J}(w^{-1}\rho, v)<0. 
\end{split}
\end{equation*}

Now in the claim, $\beta\in A_{X,J}$, so $\vert A_{X,K}(v')\vert <\vert A_{X,J}\vert$. Also, $v'=w(v)\in \lsub{K}{\mcx}^{\im}$ since  $v\in \lsub{J}{\mcx}^{\im}$. It follows by the inductive hypothesis that $v'\in \lsub{K}{\mcx}$, from which $v=wv'\in \lsub{J}{\mcx}$ follows using the definitions.
  \end{proof}

\begin{proof}[Proof of Theorem \ref{Titscone}]
By \eqref{coneint} and  Lemma \ref{Titslem}, we have \[\lsub{J}{\mcx}=\lsub{J}{\mcx}^{\re}\cap\lsub{J}{\mcx}_{X}^{\im}\seq \mset{v\in \lsub{J}{\mcx}_{X}^{\im}\mid \vert A_{X,J}(v)\vert <\infty}\seq \lsub{J}{\mcx},\] which implies Theorem \ref{Titscone}(a)--(b). Part (c)  follows readily using  the fact that $\mcv$ is a functor. It remains to prove (d). It is clear  that  $\lsub{J}{\mcc}$ is a cone, by its definition. The fact that $\lsub{J}{\mcx}$ is a cone follows from its description in (a).
\end{proof}

\begin{proof}[Proof of Theorem \ref{stabgpd}]
Part (a) follows from Proposition \ref{propc}.

Since $ F : = \Pi \cap A^{\perp}$ and $ A \seq C$, part (b) follows from Proposition \ref{CoxfundTits}. 

For part (c), let $ K, L$ be objects of $H$. Let $(L, w ,K)$ be a morphism of $G'$. Then  $w \in W_{F}$, and since $W_{F}$ stabilizes $ A$ pointwise, we conclude that $\mcv(L,w,K)$ stabilizes $ A$ pointwise. Thus, $ (L,w,K) $ is a morphism of $H$. Conversely, if $ (L,w,K)$ is a morphism of $H$, then $ \mcv(L,w,K)$ fixes $ A$ pointwise, so by part (b), $ w \in W_{F}$, and thus $ w$ is a morphism of $G'$. This proves (c).

For part (d), let $K$ be an object of $H$. By definition of $H$, we have  $ A \seq K^{\perp}$. Note that $ K$ is also an object of the full subgroupoid $G'$. Let $ L$ be another object of $G'$ that lies in the same component as $ K$. Since $ L$ and $ K$ lie in the same component of $G'$, it follows that there exists a morphism $ (L,w,K)$ of $G'$ where $ w\in W_{F}$. 

We claim that $L\seq A^{\perp}$. Let $\beta \in L$. There exists  $ \alpha \in K$ such that
$  \beta = w(\alpha ) $.
Since $ F = \Pi \cap A^{\perp}$, we have $\Span(F)\seq A^{\perp}$.  Also,  $ \alpha \in K\seq A^{\perp}$ since $ A \seq K^{\perp}$.
Hence $\beta=w(\alpha)\in  \alpha + \Span(F)\seq A^{\perp}$,  for all $ \beta \in L$. Thus, $ L \seq A^{\perp}$ as claimed. 

This claim implies that $A \seq L^{\perp}$, so $ L$ is an object of $ H$. We have shown that if $ L$ is an object of $G'$ that lies in the same component of $G'$ as some object $ K$ (where $K$ is an object of $H$), then $L$ is also an object of $H$. This fact along with what we proved in part (c) proves that $ H$ is a union of components of $G'$.
\end{proof}

\section{Imaginary cones of realized  Brink-Howlett groupoids}
\label{s:8}
\subsection{} Define  $ \mck : = (-\mcc) \cap \cone(\Pi) \seq V$ and $ \mcz := \bigcup_{w\in W}w(\mck) \seq V$.

\begin{prop} \label{propkcone}
    If   $J$ is an object of $G$,  then
$ \lsub{J}{\mck} \seq \mck$.
\end{prop}

\begin{proof}
     Let $ v \in \lsub{J}{\mck}=(-\lsub{J}{\mcc}) \cap \cone(\lsub{J}{\Delta})$. Then  $v \in -\lsub{J}{\mcc} = -(\mcc \cap J ^{\perp})\seq -\mcc$, by Proposition \ref{propc}. 
We also have  $ v \in \cone(\lsub{J}{\Delta})$ where$\lsub{J}{\Delta}=\pi_{J}(\Pi\sm J)$. For $\alpha\in \Pi\sm J$, we have
    $\pi_{J}(\alpha)=\alpha-\sum_{\beta\in J}B(\alpha,\beta)\omega_{\beta}^{J}$ by Proposition \ref{fincomp}. Now   by Theorem \ref{refcan},  $B(\alpha,\beta)\leq 0$ for $\beta\in J$, since $\alpha\in \Pi\sm J$, and $\omega_{\beta}^{J}\in \cone(J)\seq \cone(\Pi)$ by \ref{finfacproj}.   Hence 
    $\pi_{J}(\alpha)\in \cone(\Pi)$. It follows that $v\in \cone(\Pi)$. Hence $v\in -\mcc\cap \cone(\Pi)=\mck$. 
        Thus, we have proven that $ \lsub{J}{\mck} \seq \mck$.
\end{proof}

\begin{prop} \label{propzcone}
    Let $ J $ be an object of $G$. Then
$ \lsub{J}{\mcz} \seq \mcz $.
\end{prop}
\begin{proof} 
    By definition of $ \lsub{J}{\mcz}$ and Proposition \ref{propkcone}, we have  
    \[ \lsub{J}{\mcz}=\bigcup_{K\in \ob(G)}\,\bigcup_{w\in\lrsub{J}{G}{K}}(\mcv(w))(\lsub{K}{\mck})\seq \bigcup_{K\in \ob(G)}\,\bigcup_{w\in\lrsub{J}{G}{K}}w\mck\seq\bigcup_{w\in W}w\mck=\mcz. \]
   \end{proof}
\begin{proof}[Proof of Theorem \ref{imcone}(a)--(c)] 
Part (a) follows readily from the definitions  using the fact that $\mcv$ is a functor. 

We prove (b).Since $ \lsub{K}{\mck} \seq -\lsub{K}{\mcc}$, the definition of $ \lsub{J}{\mcz}$ gives
\begin{equation*}
\lsub{J}{\mcz} = \bigcup_{K\in \ob(G)}\,\bigcup_{w\in\lrsub{J}{G}{K}}(\mcv(w))(\lsub{K}{\mck})\seq  \bigcup_{K\in \ob(G)}\,\bigcup_{w\in\lrsub{J}{G}{K}}(\mcv(w))(-\lsub{K}{\mcc})
  = - \lsub{J}{\mcx}.\end{equation*}
Thus, $ \lsub{J}{\mcz} \seq - \lsub{J}{\mcx}$.

Now by Proposition \ref{propzcone} and \cite[Proposition 3.2(a)]{Dy13}, we have  $\lsub{J}{Z}\seq \mcz$. By  \cite[Proposition 3.2(a)]{Dy13}, we have 
$ \mcz = (-\mcx) \cap \big( \bigcap_{w\in W}w(\cone(\Pi)) \big)$.
Hence, 
\[ \lsub{J}{\mcz} \seq \mcz \seq \bigcap_{w\in W}w(\cone(\Pi)) \seq \bigcap_{K \in \ob(G)} \bigcap_{w\in \lrsub{J}{G}{K}}w(\cone(\Pi))\]
Thus,
\[ \lsub{J}{\mcz} \seq (-\lsub{J}{\mcx}) \cap \big( \bigcap_{K \in \ob(G)} \bigcap_{w\in \lrsub{J}{G}{K}}w(\cone(\Pi)) \big) \]
But since $ -\lsub{J}{\mcx} \seq J ^{\perp}=w(K^{\perp})$ for $w$ and $K$ as in the intersection, this gives
\[ \lsub{J}{\mcz} \seq (-\lsub{J}{\mcx}) \cap \big( \bigcap_{K \in \ob(G)} \bigcap_{w\in \lrsub{J}{G}{K}}(w(\cone(\Pi))\cap w(K^{\perp}) \big).
 \] 
By  Proposition \ref{realizedroot}(b), we have
\begin{equation*}
w(\cone(\Pi))\cap w(K^{\perp})=w(\cone(\Pi)\cap K^{\perp})\seq 
w(\pi_{K}(\cone(\Pi)))=w(\cone(\lsub{K}{\Delta})).
\end{equation*} Therefore, 
\begin{equation*}
\lsub{J}{\mcz}\seq (-\lsub{J}{\mcx}) \cap \big( \bigcap_{K \in \ob(G)} 
\bigcap_{w\in \lrsub{J}{G}{K}}w(\cone(\lsub{K}{\Delta})) \big)=
 (-\lsub{J}{\mcx}) \cap \lsub{J}{\mcy}.
\end{equation*}

Now suppose that $ v \in (-\lsub{J}{\mcx}) \cap \lsub{J}{\mcy}$. Since $ v \in - \lsub{J}{\mcx}$, it follows that $ v \in g(-\lsub{K}{\mcc})$ for some $ g \in \lrsub{J}{G}{K}$. Hence $ g^{-1}(v) \in - \lsub{K}{\mcc}$. Since $ v \in \lsub{J}{\mcy} $, part (a) implies  that $ g^{-1}(v) \in g^{-1}( \lsub{J}{\mcy}) =  \lsub{K}{\mcy}$. By the definition of $ \lsub{K}{\mcy}$, it is obvious that $ \lsub{K}{\mcy} \seq \cone(\lsub{K}{\Delta})$. Thus, $ g^{-1}(v) \in (-\lsub{K}{\mcc}) \cap \cone(\lsub{K}{\Delta}) = \lsub{K}{\mck}$. Therefore,
$ v \in g(\lsub{K}{\mck}) \seq \lsub{J}{\mcz}$.
Thus, we have established that $(-\lsub{J}{\mcx}) \cap \lsub{J}{\mcy} \seq \lsub{J}{\mcz} $, and (b) is proved.

We prove (c). By (b), we have  $\lsub{J}{\mcz}=(-\lsub{J}{\mcx})\cap \lsub{J}{\mcy}$. Since $ \lsub{J}{\mcy}$ is an intersection of convex cones, it is itself  a convex cone. By  Theorem \ref{Titscone}(d), we know that $ \lsub{J}{\mcx}$ is a  cone.   Thus, $ \lsub{J}{\mcz}$ is a  cone, since it is an intersection of two cones. 
\end{proof}

\subsection{} For any  $ \gamma \in {\lsub{J}{\wt\Upsilon}}^{\im,+}$, define $ \Omega_{J}(\gamma) : = \mset{ \alpha \in \Pi \sm J  \mid B_{J}(\gamma , \pi_{J}(\alpha)) > 0   }$.

\begin{lem} \label{propimfinite}
    Let $ \gamma \in {\lsub{J}{\wt\Upsilon}}^{\im,+}$.  Write $\gamma=\sum_{\beta\in \Pi\sm J}c_{\beta}\pi_{J}(\beta)$ where  $c_{\beta}\geq 0$ for all $\beta\in \Pi$. Then:
    \begin{num}
    \item $\Omega_{J}(\gamma)\seq \mset{\alpha\in \Pi\sm J\mid c_{\alpha}> 0, \pi_{J}(\alpha)\in\lsub{J}{\Delta}^{\re} }$.
           \item  $\vert\Omega_{J}(\gamma)\vert < \infty$.
\item   $ \nu(\alpha ,J)$ is defined for all $\alpha\in \Omega_{J}(\gamma)$.
     \item  $ \Omega_{J}(\gamma) = \emptyset\iff  \gamma \in  \lsub{J}{\mck}$.
   \end{num}
    \end{lem}

\begin{proof} Suppose that $\alpha\in \Omega_{J}(\gamma)$.
Then
\begin{equation}
0<B_{J}(\gamma,\pi_{J}(\alpha))=\sum_{\beta\in \Pi\sm J}c_{\beta}B(\pi_{J}(\beta),\pi_{J}(\alpha))
\end{equation} where $c_{\beta}\geq 0$ for all $\beta\in \Pi\sm J$
Hence there exists $\beta\in \Pi\sm J$ such that $c_{\beta}>0$ and
$B(\pi_{J}(\beta),\pi_{J}(\alpha))>0$. By Proposition 
\ref{propbform}, this implies that $\beta= \alpha$ and the 
component of $J\cup\set{\alpha}$ containing $\alpha$ is of finite 
type, so $\pi_{J}(\alpha)\in\lsub{J}{\Delta}^{\re}$.  Since 
$c_{\alpha}=c_{\beta}>0$,  this proves (a).  Part (b) follows since in 
(a), $c_{\alpha}\neq 0$ for only finitely many $\alpha\in \Pi\sm J$.
Part (c) holds since  $\nu(\alpha,J)$ is defined for any $\alpha\in \lsub{J}{\Delta}^{\re}$. 

We prove (d).  Suppose that $\Omega_{J}(\gamma)=\eset$.
Proposition \ref{rootiprod} implies that  $\gamma\in \cone(\lsub{J}{\Delta})$, 
where, by definition,  $\lsub{J}{\Delta}=\pi_{J}(\Pi\sm J)$. By definition of $\Omega_{J}(\gamma)$, we have $B_{J}(\gamma,\lsub{J}{\Delta})\seq \mathbb{R}_{\leq 0}$. Hence $\gamma\in -\lsub{J}{\mcc}\cap \cone(\lsub{J}{\Delta})=\lsub{J}{\mck}$ as required. Conversely, suppose that  $\gamma\in\lsub{J}{\mck}$.
Then  $\gamma\in -\lsub{J}{\mcc}$, so $B_{J}(\gamma, \lsub{J}{\Delta})\seq \mathbb{R}_{\leq0}$ and $\Omega_{J}(\gamma)=\eset$.
\end{proof}

\begin{lem} \label{propimrootheight} Assume that $(W,S)$ is of finite rank. Fix a height function $\height\colon V\to \mathbb{R}$   as in Proposition \ref{height}(a).
 Let $ \gamma \in {\lsub{J}{\wt\Upsilon}}^{\im,+}$. 
 If $ |\Omega_{J}(\gamma)| > 0$, then there exists some $ \alpha \in \Omega_{J}(\gamma)$ such that $ \nu(\alpha, J)$ exists. For any such $\alpha$, we have   $  \height(\nu(\alpha, J) \gamma) < \height(\gamma)$.
\end{lem}

\begin{proof}
    By Proposition \ref{propimfinite}, there exists some $ \alpha \in \Omega_{J}(\gamma)$ such that $ \nu(\alpha, J)$ is defined. Let $ \nu(\alpha, J) = s_{1}s_{2} \dots s_{n}$ be a reduced expression in $W$.   For $i=1,\ldots, n$, let $\alpha_{i}\in \Pi$ with $s_{\alpha_{i}}= s_{i}$, and define $ \beta_{i} = s_{1}s_{2} \dots s_{i-1}(\alpha_{i})$. By Proposition \ref{CoxfundTits}, we have
\begin{equation}\label{gammaeq} \gamma - \nu(\alpha,J)\gamma = B(\gamma, \alpha_{1}^{\vee})\beta_{1} + B(\gamma, \alpha_{2}^{\vee})\beta_{2} + \ldots + B(\gamma, \alpha_{n}^{\vee})\beta_{n}.\end{equation}
We have $\alpha_{i}\in J\cup\set{\alpha}$ for all $i$. 
If $\alpha_{i}\in J$, then $B(\gamma,\alpha_{i}^{\vee})=0$ since $\gamma\in J^{\perp}$. If $\alpha_{i}=\alpha$, then
$ B(\gamma , \alpha_{i}) = B_{J}(\gamma, \pi_{J}(\alpha)) > 0$ since $\gamma\in J^{\perp}$ and $\alpha\in \Omega_{J}(\alpha)$.
  Since $ \nu(\alpha, J) \in W_{J \cup \{ \alpha \} } \sm W_{J}$, 
  there must exist at least one $j$ such that $ \alpha_{j} = \alpha$.
   We conclude that $ B(\gamma , \alpha_{i}) \geq 0$ for all $i$ and 
   that there exists at least one $j$ such that $ B(\gamma , \alpha_{j}) = B(\gamma, \alpha) > 0$. Hence, the right hand side of
  \eqref{gammaeq} 
   is a non-zero,  non-negative linear combination of positive roots. This implies that  $ \height(\gamma - \nu(\alpha,J)\gamma) > 0$
and so $ \height(\gamma) > \height(\nu(\alpha, J)\gamma)$ as required.
\end{proof}

\begin{lem} \label{propimemptyset} Assume that $(W,S)$ is of finite rank.
   Let $\gamma \in {\lsub{J}{\wt\Upsilon}}^{\im,+} $. Then there exists a morphism $g=(K,w, J)$ of $G$ such that
   $(\mcv(g))(\gamma)\in \lsub{K}{\mck}$.   
\end{lem}
\begin{proof} We first   assume that $\Pi$ is linearly independent. Fix a   height function $\height\colon V\to \mathbb{R}$ such that $\height(\alpha)=1$ for all $\alpha\in \Pi$.
    If $ \Omega_{J}(\gamma) = \emptyset$, then take $g=(J,1_{W},J)\colon J \to J$. Now assume that $\vert \Omega_{J}(\gamma)\vert >0 $. 
    
    Let us assume for the sake of contradiction that no such $ g$ existed.  Then  $ \Omega_{K}(g(\gamma)) \neq \emptyset$  for all $K\in \ob(G)$ and all morphisms $g\colon J\to K$ in $G$.
Let $J_{0}:=J$. Applying Proposition \ref{propimrootheight},  there is a 
   a Brink-Howlett generator $ h_{1}\colon J_{0}\to J_{1}$ such that $ \height(\gamma) > \height( h_{1}(\gamma) )$. We have  $h_{1}(\gamma) \in {\lsub{h_{1}(J)}{\wt\Upsilon}}^{\im,+}$. We must have  $ \vert \Omega_{h_{1}(J)}(h_{1}(\gamma))\vert > 0$ by the assumption at the start of this paragraph. Repeating this argument gives  an infinite  sequence
   $h_{i}=(J_{i},w_{i},J_{i-1})$ of Brink-Howlett generators for $i=1,2,3,\ldots$ such that
\[ \height(\gamma) > \height(w_{1}(\gamma)) > \height(w_{2}w_{1}(\gamma)) > \height(w_{3}w_{2}w_{1}(\gamma)) > \dots .\]
    Since $ \gamma \in {\lsub{J}{\wt\Upsilon}}^{\im,+}$, we can write $\gamma = \pi_{J}(r_{U,J})$ where $ r_{U,J} \in \Phi^{+}$ for some
    (infinite index) corank $1$ reflection overgroup $U$ of $ W_{J}$. 
    We  may write 
\[ \gamma = \pi_{J}(r_{U,J}) = r_{U,J} + c_{1} \alpha_{1} + \dots + c_{n} \alpha_{n}\]
    where $ c_{i}\in \mathbb{R}$ and $\alpha_{i} \in J  \seq \Pi$ for each $i$. Consider the action of  a Brink-Howlett morphism $h=(K,w,J)$ on the above equation:
\[ w(\gamma) = w(r_{U,J}) + c_{1}w(\alpha_{1}) + \dots + c_{n} w(\alpha_{n})\]
    where $ w(\alpha_{i}) \in K \seq \Pi$ for each $i$. Thus, \[\height(c_{1}w(\alpha_{1}) + \dots + c_{n} w(\alpha_{n})=
    c_{1}+\ldots+c_{n}=
    \height(c_{1} \alpha_{1} + \dots + c_{n} \alpha_{n} ).\] This tells us that $ \height(\gamma) > \height(w(\gamma))$ if and only if $ \height(r_{U,J}) > \height(w(r_{U,J}))$. Thus, the sequence 
\[ \height(\gamma) > \height(w_{1}(\gamma)) > \height(w_{2}w_{1}(\gamma)) > \height(w_{3}w_{2}w_{1}(\gamma)) > \dots\]
    implies 
\[ \height(r_{U,J}) > \height(w_{1}(r_{U,J})) > \height(w_{2}w_{1}(r_{U,J})) > \height(w_{3}w_{2}w_{1}(r_{U,J})) > \dots\]
    where $r_{U,J}$, $ w_{1}(r_{U,J})$, $ w_{2}w_{1}(r_{U,J})$, $ \dots$ are pairwise distinct positive roots. 

    By Proposition \ref{height},   there exists some $ \epsilon > 0$ such that $ \height(\rho) > \epsilon \ell(s_{\rho})  $ for all $ \rho \in \Phi^{+}$. Hence,
\[ \height(r_{U,J}) > \height(w_{i}w_{i-1} \dots w_{1}(r_{U,J})) > \epsilon \ell(s_{w_{i}w_{i-1} \dots w_{1}(r_{U,J})}) \textrm{ for all } i.\]
    This gives
\[ \frac{1}{\epsilon}\height(r_{U,J}) > \ell(s_{w_{i}w_{i-1} \dots w_{1}(r_{U,J})}) \textrm{ for all } i\]
    The above equation implies that there are infinitely many distinct reflections whose lengths are all less than $\frac{1}{\epsilon}\height(r_{U,J})$. This is impossible in a finite rank Coxeter system, a contradiction which   completes the proof in the case $\Pi$ is linearly independent.
    
    Now consider the general case, with $\Pi$ possibly linearly dependent. We sketch an argument involving compatibility of the construction of realized root systems of Brink-Howlett groupoids with lifts, to reduce  to the case in which $\Pi$ is linearly independent.
    The argument is straightforward but somewhat lengthy, and we omit many  routine details.  
    
    Consider a universal lift $L\colon \wh D\to D$ of $D=(V,B,\Phi,\Pi)$ where $\wh D=(\wh V,\wh B,\wh \Phi,\wh \Pi)$, defined  as in \ref{y2.5} (we write here  $\wh X$ instead of $\wt X$ as in \ref{y2.5}, for typographical reasons).  
    
    In particular, $\wh \Pi$ is a basis of $\wh V$, by definition of universal lifts. Denote the associated Coxeter system of  the based root system $(\wh \Phi,\wh \Pi)$ as  $(\wh W,\wh S)$.
   The  canonical identification $\wh W\to W$, denoted as $\wh w\mapsto w$, makes  $L\colon \wh V\to V$ a $W$-equivariant map and preserves  bilinear forms. Also, $L$ restricts to  bijections $\wh \Phi\to \Phi$ and $\wh \Pi\to \Pi$
    
   Many notions $X$ we have defined for $(W,S)$ from $(\Phi,\Pi)$ on $(V,B)$ have analogues, which we will denote $\wh X$, defined in the same way for 
   $(\wh W,\wh S)$ from $(\wh \Phi,\wh \Pi)$ on $(\wh  V,\wh B)$.
        For example, for any $J\seq \Phi$, we write $\wh J=\mset{\alpha\in \wh \Phi\mid L(\alpha)\in J}$. Corresponding to the Brink-Howlett groupoid $G$ of $(W,S)$,
there is a Brink-Howlett groupoid $\wh G$ of $(\wh  W,\wh S)$ with objects $\mset{\wh J\mid J\in \ob(G)}$ and morphisms
$(\wh K,\wh w,\wh J)$ for morphisms $(K,w,J)$ of $G$.
        
 More importantly for the argument here, there is a realized signed groupoid set $(\wh G,\wh \Upsilon_{X},\wh \mcv,\wh\iota\,)$, where $X$ denotes $R$, $M$ or $P$, defined
 in the same way as $(G,\Upsilon_{X},\mcv,\iota)$.
The assumed validity of conditions  \ref{bf}(i)--(iii) for $(G,V,\Pi)$ implies their validity
 for $(\wh G,\wh V,\wh \Pi)$; to prove this, one observes that, in Proposition \ref{BHred2},  if $K\seq \Pi$ and  $u\in \wh V\sm \wh V_{\wh K}$, then
 $L(u)\in V\sm V_{K}=\eset$, a contradiction if $K\in \ob(G)$.   Hence we may identify $\wh\mcv(\wh J\,)=(\wh J\,)^{\perp}\seq \wh V$, for any $J\in \ob(G)$.

  Henceforward, we identify $(W,S)$ with $(\wh W,\wh S)$ and 
  thereby  $(\wh G,\wh \Upsilon_{X})=(G,\Upsilon_{X})$, for 
  notational simplicity. 
 Let $J\in \ob(G)$. The projections on $(\wh J\,)^{\perp}$ and 
 $J^{\perp}$ are related by the formula  $L\circ \pi_{\wh J}=\pi_{J}\circ 
 L\colon \wh V\to J^{\perp}\seq V$. From the definitions, this 
 implies that   $\iota_{J}=L\circ \wh \iota_{J}\colon \lsub{J}{\Upsilon}\to J^{\perp}$.
 In particular, $\lsub{J}{\wt\Upsilon}^{\im,+}=L( \lsub{J}{\wt{\wh\Upsilon}}^{\im,+})$.  One also can check from the definitions  that 
$L(\lsub{\wt J\,}{\wt\mck})\seq \lsub{J}{\mck}$. 

Now the lemma can be proved as follows. Let $\gamma\in\lsub{J}{\wt \Upsilon}^{\im +}$. Write $\gamma=L(\gamma')$ for some
$\gamma'\in  \lsub{J}{\wt{\wh\Upsilon}}^{\im,+}$. By the case of linearly independent roots, there exists $\wh g=(\wh K,\wh w,\wh J)$ with $(\wh{\mcv}(\wh g))(\gamma')\in \lsub{\wh K\,}{\wh \mck}$.
Applying $L$ to both sides shows that $({\mcv}(g))(\gamma)\in \lsub{K}{\mck}$.   
     \end{proof}

\begin{proof}[Proof of  Theorem \ref{imcone}(d)] 
Let $ \gamma \in {\lsub{J}{\wt\Upsilon}}^{\im,+}$, say $\gamma=r'_{U,J}$ where $U\in R_{J}$ is an infinite index,  corank one reflection overgroup of $W_{J}$.  Recall that we assume that no simple root in $\Pi\sm J $ is joined in the Coxeter graph to infinitely many components of $ J =J_{\fin}$. Then $r_{U,J}\in \Phi^{+}$ and there exists a finite subset $\Pi'$ of $\Pi$ such that
$r_{U,J}\in W_{\Pi'}$, and  $J':=\Pi'\cap J$ is a union of components of $J$ such that $J'$ contains every  vertex of $J$ which is joined in the Coxeter graph of $\Pi$ to an element  of $\Pi'\sm J$. In particular, $\Pi'$ contains every component of $J$ with a vertex which is joined to $r_{U,J}$ in the Coxeter graph of $U$ (regarded as a graph on vertex set $J\cup\set{r_{U,J}}$).

Let $U'$ denote the standard parabolic subgroup of $U$ with $\Pi_{U'}=J'\cup \set{r_{U,J}}$
The component of $\set{r_{U,J}}\cup J$ which contains $r_{U,J}$ is of infinite type (since $[U:W_{J}]$ is infinite) and it is contained in $\set{r_{U,J}}\cup J'$. Hence it is  the component of 
$\set{r_{U,J}}\cup J'$ which contains $r_{U,J}$. It follows that
$U'$ is an infinite index, corank one reflection overgroup of $W_{J'}$ in $W_{\Pi'}$. 
We regard $W_{\Pi'}$ as the (finite rank) Coxeter group associated to a based root system $(\Phi',\Pi')$ with simple roots $\Pi'$ in $(V',B')$ where $V'=(J\sm J')^{\perp}\seq V$ and $B'$ is the restriction of $B$  to a symmetric bilinear form on $V'$.

Consider the component of  $G'$  of the full Brink-Howlett groupoid of $W_{\Pi'}$, such that $G'$  contains the object  $J'$.  Then $G'$ has an associated 
weakly realized root system $(G',\Upsilon', \mcv',\iota')$ defined in the analogous way to $(G,\Upsilon,\mcv,\iota)$. Conditions  \ref{bf}(i)--(iii) for $(G,V,\Pi,J)$ trivially imply the same conditions   for $(G',V',\Pi',J')$. By Proposition \ref{BHred2},  
\ref{bf}(i)--(iii) hold for $(G',V',\Pi', K')$ for any object $K'$ of $G'$. Hence  the orthogonal projection $\pi'_{K'}$ of $V'$ on $(K')^{\perp}\cap V'$ is defined. In particular, $\pi'_{J'}\colon V'\to (J')^{\perp}\cap V'$ is defined. 

Note that $\Pi_{U'}=\set{r_{U,J}}\dot\cup J'$ where $r_{U,J}\in \Phi_{W'}$, so $r_{U',J'}=r_{U,J}$.
By choice of $J'$ and Propositions \ref{fincomp}--\ref{finfacproj}, we  also have $r'_{U,J}=\pi_{J}(r_{U,J})\in r'_{U,J}+\Span(J')$, from which $r'_{U,J}=\pi_{J'}(r_{U,J})$. Hence $r'_{U',J'}:=
\pi'_{J'}(r_{U',J'})=r'_{U,J}=\gamma$.

 For any object $K$ of $G'$, we define
$\lsub{K'\,}{\mcv}'=(K')^{\perp}\cap V'$ for $K'\in \ob(G')$, with projection denoted $\pi'_{K'}\colon V'\to \lsub{K'\,}{\mcv}'$. 
There is an imaginary cone 
$\mcz'$ for  $(G',\Upsilon', \mcv',\iota')$, defined analogously to 
$\mcz$ for $(G,\Upsilon,\mcv,\iota)$, with fundamental chamber at the object $K'\seq \Pi'$ of $G'$ denoted as $\lsub{K'\,}{\mck}'$.
From Lemma \ref{propimemptyset}, there is a morphism
$g'=(K',w',J')$ of $G'$, with $w'\in W'$, such that $(\mcv'(g'))(\gamma)\in
\lsub{K'\,}{\mck}'$. That is, $w'\gamma\in \lsub{K'\,}{\mck}'$. 

Since $\Pi'$ is $B$-orthogonal to $J\sm J'$, $w'\in W_{\Pi'}$ and  $J=J'\cup(J\sm J')$, there  is a morphism $ g:=(K, w',  J )$ in $G$
where $K:= K'\cup(J\sm J')$. To complete the proof, it will therefore suffice to check that $\lsub{K'\,}{\mck}'\seq \lsub{K}{\mck}$.

We have $\lsub{K}{\mck}=\cone(\lsub{K}{\Delta})\cap (-\lsub{K}{\mcc})$ where $\lsub{K}{\mcc}=\mset{v\in \lsub{K}{\mcv}\mid B_{K}(v,
\lsub{K}{\Delta})\seq \mathbb{R}_{\geq 0}}$ and $\lsub{K}{K}=\pi_{K}(\Pi\sm K)$.

Similarly, write $\lsub{K'}{\Delta}':=\pi'_{K'}(\Pi'\sm K')$. 
Then $\lsub{K'\,}{\mck}'=\cone(\lsub{K'}{\Delta}')\cap (-\lsub{K'}{\mcc}')$ where $\lsub{K'}{\mcc}':=\mset{v\in \lsub{K'}{\mcv}'\mid B'_{K'}(v,
\lsub{K'}{\Delta}')\seq \mathbb{R}_{\geq 0}}$ with $B'_{K'}$ denoting the restriction of  $B'$ to $ \lsub{K'}{\mcv}'=(K')^{\perp}\cap V'$.

Note $\Pi'\sm K'\seq \Pi\sm K$, since $K=K'\dot\cup(J\sm J')$
where $\Pi'\cap (J\sm J')=\eset$.
Let $\alpha\in \Pi'\sm K'$ and write $\alpha':=\pi'_{K'}(\alpha)$.
Then $\alpha'- \alpha\in \Span(K')$ and $B(\alpha', K')\seq \set{0}$. 
Hence $\alpha,\alpha'\in \Span(\Pi')$ where $\Pi'\perp J\sm J'$.
This implies that $B(\alpha',K)\seq \set{0}$ and 
$\alpha'-\alpha\in \Span(K)$, so $\alpha'=\pi_{K}(\alpha)$. We conclude that
\begin{equation}\label{eqa1}\pi'_{K'}(\alpha)=\pi_{K}(\alpha) \text{ \rm if $\alpha\in \Pi'\sm K'$,}\quad 
 \lsub{K'}{\Delta}'\seq\lsub{K}{\Delta},\quad
\cone(\lsub{K'}{\Delta}')\seq\cone(\lsub{K}{\Delta}). 
\end{equation}

Note  that $B(\alpha,K)\seq -\mathbb{R}_{\geq 0}$ for $\alpha\in 
\Pi\sm K$, by Theorem \ref{refcan}. By Proposition \ref{finfacproj}, 
this implies  $\pi_{K}(\alpha)\in \alpha+\cone(K)\seq \cone(\Pi)$ for 
$\alpha\in \Pi\sm K$. The same argument shows
\begin{equation}\label{eqa2}
\pi'_{K'}(\alpha)\in \alpha+\cone(K')\seq \cone(\Pi'),  \qquad \alpha\in \Pi'\sm K'.
\end{equation}

Let $v\in \cone(\lsub{K'}{\Delta}')\cap (-\lsub{K'}{\mcc}')$.  
Then we may write $v=\sum_{\alpha\in \Pi'\sm K' }c_{\alpha}
\pi_{K'}(\alpha)$ where all $c_{\alpha}\geq 0$.
We  also  have $v\in V'=(J\sm J')^{\perp}$, and, since   
$v\in -\lsub{K'}{\mcc}'$,   $v\in (K')^{\perp}$  and  
$B(v, \pi'_{K'}(\alpha))=B(v,\alpha)\leq 0$ for all 
$\alpha\in \Pi'\sm K'$. 

We are required to show that 
$v\in \cone(\lsub{K}{\Delta})\cap (-\lsub{K}{\mcc})$. From \eqref{eqa1}, we have $v\in
\cone(\lsub{K'}{\Delta}')\seq\cone(\lsub{K}{\Delta})$, so it remains to prove  $v\in - \lsub{K}{\mcc}$. Equivalently, we have to show that $v\in K^{\perp}$ and $B(v,\alpha)\leq 0$ for all $\alpha\in 
\Pi\sm K$. 

We have  $v\in K^{\perp}$ since $K=K'\cup(J\sm J')$, $v\in (K')^{\perp} $ and $v\in (J\sm J')\perp$. Also $B(v,\Pi'\sm K')\seq \mathbb{R}_{\leq 0}$ from above.
Also,  \[B(v, \Pi\sm \Pi')\seq B(\cone(\Pi'),\cone(\Pi\sm \Pi'))
\seq \cone(B(\Pi',\Pi\sm \Pi'))\seq \mathbb{R}_{\leq 0}\] by \eqref{eqa2} and Theorem \ref{refcan}. This gives 
$B(v,\Pi\sm K)\seq B(v,\Pi\sm K')\seq \mathbb{R}_{\leq 0}$
since $\Pi\sm K\seq \Pi\sm K'= (\Pi\sm \Pi')\cup (\Pi'\sm K')$.
\end{proof}
   
%\bibliography{roots_bh.bib}
%\bibliographystyle{plain} 

\end{document}